\documentclass[10pt]{amsart}

\allowdisplaybreaks

\usepackage{latexsym,exscale,enumerate,amsfonts,amssymb,xparse, mathtools}
\usepackage[normalem]{ulem}
\usepackage{amsmath,amsthm,amsfonts,amssymb,amscd, stmaryrd,textcomp,bbm,euscript}
\usepackage[bookmarks=false]{hyperref}
\usepackage{bbold}
\usepackage[usenames,dvipsnames]{xcolor}
\usepackage{enumitem}
\usepackage{verbatim}
\usepackage{ytableau}

\usepackage{bbm}

\usepackage{tikz}
\usetikzlibrary{positioning}
\usetikzlibrary{decorations.pathreplacing}
\usetikzlibrary{cd}
\tikzset{>=stealth}

\calclayout

\newtheorem{theorem}{Theorem}[section]

\newtheorem{proposition}[theorem]{Proposition}

\newtheorem{lemma}[theorem]{Lemma}

\newtheorem{corollary}[theorem]{Corollary}

\newtheorem{conjecture}[theorem]{Conjecture}

\theoremstyle{definition}
\newtheorem{definition}[theorem]{Definition}

\theoremstyle{remark}
\newtheorem{remark}[theorem]{Remark}

\newtheorem{example}[theorem]{Example}

\tikzset{anchorbase/.style={baseline={([yshift=-0.5ex]current bounding box.center)}},
  tinynodes/.style={font=\tiny,text height=0.5ex,text depth=0.1ex},
  webs/.style={line width=.9,color=black}}
  
\definecolor{revisions}{RGB}{0,0,0}
  
\begin{document}
 
% Title and author info
\title{Row-column mirror symmetry for colored torus knot homology}
\author{Luke Conners}
\date{}

\begin{abstract}
We give a recursive construction of the categorified Young symmetrizer introduced by Abel-Hogancamp in \cite{AH17} corresponding to the single-column partition. As a consequence, we obtain new expressions for the uncolored $y$-ified HOMFLYPT homology of positive torus links and the $y$-ified column-colored HOMFLYPT homology of positive torus knots. In the latter case, we compare with the row-colored homology of positive torus knots computed by Hogancamp-Mellit in \cite{HM19}, verifying the mirror symmetry conjectures of \cite{GS12} and \cite{GGS18} in this case.
\end{abstract}

\maketitle

\setcounter{tocdepth}{1}
\tableofcontents

\section{Introduction}

Knots and links are central objects of study in low-dimensional topology with deep relationships to representation theory. One particularly fruitful result of this interplay has been the development of quantum link polynomials. These invariants, due to Reshetikhin-Turaev (\cite{RT90}), assign a Laurent polynomial to the input data of a (framed, oriented) link $\mathcal{L} = \mathcal{L}_1 \sqcup \dots \sqcup \mathcal{L}_n$, a semisimple Lie algebra $\mathfrak{g}$, and an assignment to each component $\mathcal{L}_i$ of $\mathcal{L}$ an irreducible representation $V_i$ of the quantized universal enevloping algebra $U_q(\mathfrak{g})$ called the \textit{color} of the component. When $\mathfrak{g} = \mathfrak{sl}_2$ and all components of $\mathcal{L}$ are colored by the defining representation, we recover the \textit{Jones polynomial} of $\mathcal{L}$. The same choice of $\mathfrak{g}$ and different representations $V_i$ recover the various colored Jones polynomials of $\mathcal{L}$. More generally, setting $\mathfrak{g} = \mathfrak{sl}_n$ and coloring each component of $\mathcal{L}$ by the defining representation recovers a specialization of the two-variable HOMFLYPT polynomial; varying the representations $V_i$ again recovers specializations of the colored HOMFLYPT polynomials.

The last two decades have seen a great deal of success in categorifying these polynomial invariants to link homology theories, beginning with Khovanov's categorification of the Jones polynomial. The primary subject of this paper is Khovanov and Rozanky's categorification of the HOMFLYPT polynomial to a triply-graded invariant (\cite{KhR08}, \cite{Kh07}) and its colored variants (\cite{WW17}, \cite{MSV11}, \cite{Cau17}). This package of invariants comes with a wealth of conjectured algebraic structures coming from physical considerations (\cite{DGR06}, \cite{GS12}, \cite{GGS18}) that have inspired a great deal of recent research. Examples of such structures that have been explicitly realized include specialization spectral sequences to colored $\mathfrak{sl_n}$ homology (\cite{Ras15}) and color-reducing spectral sequences on HOMFLPYT homology colored by miniscule representations (\cite{Wed19}).

One such unresolved conjecture concerns the behavior of colored HOMFLYPT homology under a change in color induced by transposition of Young diagrams. Recall that irreducible representations of $U_q(\mathfrak{sl}_n)$ are indexed by partitions $\lambda = (\lambda_1, \dots, \lambda_k) \in \mathbb{Z}_{\geq 1}^k$ with $k < n$. These representations stabilize, in the sense that for each such partition $\lambda$ with $k$ parts, there is an irreducible representation $V_{\lambda, m}$ of $U_q(\mathfrak{sl}_m)$ for each $m > k$. Given a (framed, oriented) knot $\mathcal{K}$, the family of Reshetikhin-Turaev invariants $f_{V_{\lambda, m}}(\mathcal{K})(Q)$ colored by these representations arise as specializations of the two-variable HOMFLYPT polynomial $f^{\lambda}(\mathcal{K})(A, Q)$. Let $\lambda^t$ denote the partition obtained from $\lambda$ by reflecting the corresponding Young diagram across a diagonal axis. There is a \textcolor{revisions}{symmetry relation at this decategorified level due to Liu-Peng \cite{LP09, LP11}, arising from rank-level duality of Chern-Simons theory, which relates} the $\lambda$-colored HOMFLYPT polynomial of $\mathcal{K}$ with its $\lambda^t$-colored analog:

\begin{equation} \label{eq: decat_sym}
    f^{\lambda}(\mathcal{K})(A, Q) = f^{\lambda^t}(\mathcal{K})(A, Q^{-1})
\end{equation}

Equation \eqref{eq: decat_sym} suggests a categorified analog between the corresponding colored HOMFLYPT homology theories. More precisely, for each partition $\lambda$, let $HHH^{\lambda}(\mathcal{K})$ denote the $\lambda$-colored HOMFLYPT homology of $\mathcal{K}$ (we postpone a precise definition of this invariant until \S \ref{sec: coloring}). This is a triply-graded vector space, and its graded dimension can be expressed as a rational polynomial in variables $a, q, t$ (these gradings are described in \S \ref{sec: red_yify} below). After a change of variables 

\begin{align} \label{eq: geom_variables}
	\textcolor{revisions}{(a, q, t) \mapsto (A := aq^{-2}, Q := q^2, T := q^{-2}t^2),}
\end{align}
setting $t = -1$ in this graded dimension recovers the $\lambda$-colored HOMFLYPT polynomial $f^{\lambda}(\mathcal{K})$. This specialization takes $T$ to $Q^{-1}$; as a consequence, replacing $\lambda$ with $\lambda^t$ and interchanging $Q$ and $T$ would categorify the relationship in Equation \eqref{eq: decat_sym}.

\textcolor{revisions}{This conjectural $Q \leftrightarrow T$ symmetry of colored HOMFLYPT homology was first suggested by Gukov-Sto\v{s}i\'{c} in \cite{GS12}. They call this phenomenon ``mirror symmetry" in reference to its interpretation as a consequence of CPS symmetry of certain fivebrane theories; we refer the interested reader to Section 5.3 of that work for more details.}

The main result of this paper is a proof of an appropriate form of \textcolor{revisions}{this mirror symmetry conjecture} for $\mathcal{K}$ a positive torus knot and $\lambda$ a single column (or single row) partition. We delay a precise statement of this theorem until after our description of colored HOMFLYPT homology below. In the meantime, we remark that the most na\"ive formulation of this conjecture cannot be true even in the uncolored case $\lambda = (1)$ without modification; indeed, after the appropriate change of variables, the uncolored HOMFLYPT homology of the unknot \textcolor{revisions}{$\mathcal{U} = \bigcirc$} has graded dimension

\begin{align} \label{eq: unknot_invariant}
\textcolor{revisions}{\text{dim}(HHH(\mathcal{U})) = \cfrac{1 + A}{1 - Q}}
\end{align}

This expression is clearly not symmetric under the interchange $Q \leftrightarrow T$. There are two prevalent solutions to this symmetry breaking, roughly corresponding to killing asymmetric terms and adding symmetrizing terms. The resulting invariants are referred to as \textit{reduced} and \textit{$y$-ified} homology, respectively. Perhaps unsurprisingly, complete knowledge of the $y$-ified (enlarged) invariant often implies complete knowledge of the reduced invariant in a sense that is made precise below. More surprisingly, certain dimensional phenomena often lead to easier computations after passing to $y$-ified homology (c.f. \S \ref{big_sec: yification} and \S \ref{sec: link_hom} of this work). For these reasons, we concern ourselves with $y$-ified homology in this paper. We describe both the reduced and the $y$-ified (uncolored) theory below.

\subsection{Reduction vs. y-ification} \label{sec: red_yify}

We recall the construction of the (unreduced, uncolored) HOMFLYPT homology of a link $\mathcal{L}$. First, choose an $n$-strand braid $\beta \in Br_n$ such that the closure $\hat{\beta}$ recovers $\mathcal{L}$ (this is always possible by Alexander's Theorem, \cite{Al23}). Next, assign to $\beta$ a \textit{Rouquier complex} $F(\beta)$ in the homotopy category $K^b(SBim_n)$ of Soergel bimodules; these are certain $\mathbb{Z}$-graded bimodules over the polynomial ring $R[x_1, \dots, x_n]$. We call the internal grading of these bimodules the \textit{quantum} grading, defined by setting $\text{deg}(x_i) = 2$ for each $i$. (Note: we will often denote the degree of an element multiplicatively; we use the variable $q$ for quantum degree, so that e.g. this requirement becomes $\text{deg}(x_i) = q^2 = Q$. We will similarly reference the \textit{homological grading} on the chain complex $F(\beta)$ using the variable $t$.)

Next, apply Hochschild cohomology termwise to $F(\beta)$ to obtain a complex $HH(F(\beta)) \in K^b(\mathcal{D}(SBim_n))$. \textcolor{revisions}{The terms of $HH(F(\beta))$ in each homological degree are themselves chain complexes in the derived category $\mathcal{D}(SBim_n)$; the internal homological grading of each such complex is distinct from the homological grading of $F(\beta)$. We refer to this third grading} as the \textit{Hochschild grading}, denoted by the variable $a$. By the \textcolor{revisions}{aforementioned} work of Khovanov and Rozansky, the homology of $HH(F(\beta))$ \textcolor{revisions}{depends only on the braid closure $\mathcal{L} = \hat{\beta}$ (not the specific braid presentation $\beta$) and} is exactly the HOMFLYPT homology of the link $\mathcal{L}$, denoted $HHH(\mathcal{L})$. This is a triply-graded vector space, and its dimension is a Laurent polynomial in the variables $a, q, t$.

\subsubsection{Reduced HOMFLYPT Homology} \label{subsubsec: reduce}

The bimodule structure on each chain group in $F(\beta)$ lifts to a bimodule structure on $F(\beta)$ itself. Letting
\textcolor{revisions}{$\sigma_{\beta} \in \mathfrak{S}^n$ denote the permutation associated to $\beta$},
the right action of $x_i$ and the left action of \textcolor{revisions}{$x_{\sigma_{\beta}(i)}$} are homotopic endomorphisms of $F(\beta)$ for each $1 \leq i \leq n$.
\textcolor{revisions}{Upon passing to Hochschild cohomology, the induced right action of $x_i$ and left action of $x_{\sigma_{\beta}(i)}$ on $HH(F(\beta))$ remain homotopic for each $i$, and the induced left and right actions of $x_i$ on $HH(F(\beta))$ for each $i$ become equal on the nose. All homotopic actions on $HH(F(\beta))$ induce \textit{identical} actions on the homology $HHH(\mathcal{L})$, resulting in a well-defined action of $R[\mathbb{X}_{\mathcal{L}}] := R[x_i]_{i \in \pi_0(\mathcal{L})}$ on $HHH(\mathcal{L})$.}

\textcolor{revisions}{We can view this module structure on $HHH(\mathcal{L})$ over $R[\mathbb{X}_{\mathcal{L}}] \simeq \bigotimes_{i \in \pi_0(\mathcal{L})} R[x_i]$ as a module structure over a polynomial ring assigned to each component of $\mathcal{L}$. In fact, more is true; given a choice of component $i \in \pi_0(\mathcal{L})$, the action of $R[x_i]$ on $HHH(\mathcal{L})$ lifts to an action of $HH(R[x_i]) \cong R[x_i] \otimes_R \bigwedge[\eta_i]$, where $\mathrm{deg}(\eta_i) = aq^{-2}$.}

\textcolor{revisions}{The algebra $HH(R[x_i])$ is commonly called the \textit{derived sheet algebra}. We point out that its dimension agrees exactly with the dimension of the unknot invariant $\mathrm{dim}(HHH(\mathcal{U}))$ of Equation \eqref{eq: unknot_invariant}. The presence of an algebra with dimension matching that of the unknot invariant acting on the invariant of an \textit{arbitrary} link is common to many link homology theories\footnote{As is the dependence of this action on a choice of component of $\mathcal{L}$.}; we direct the interested reader to the Introduction of \cite{HRW21} for more details.} The \textit{reduced} HOMFLYPT homology of $\mathcal{L}$ is the triply-graded vector space $\overline{HHH}(\mathcal{L})$ obtained by killing the action of \textcolor{revisions}{$HH(R[x_i])$ in non-zero degrees}.

\begin{remark} \label{rem: red_dim}
\textcolor{revisions}{The careful reader will have noticed that our definition of reduced HOMFLYPT homology depends on a choice of component of $\mathcal{L}$.} In fact, standard arguments show that $HHH(\mathcal{L})$ is free over the derived sheet algebra of a single strand; see \cite{Ras15}\textcolor{revisions}{, particularly the discussion following Definitions 2.13 and 2.14 of that work}. This immediately implies a straightforward relationship between the (graded dimensions of the) reduced and unreduced invariants:

\[
\text{dim}(\overline{HHH}(\mathcal{L})) = \left( \cfrac{1 + A}{1 - Q} \right)^{\textcolor{revisions}{-1}} \text{dim}(HHH(\mathcal{L}))
\]

\textcolor{revisions}{In particular, the dimension of $\overline{HHH}(\mathcal{L})$ is independent of the choice of component of $\mathcal{L}$.}
\end{remark}

Note that $\text{dim}(\overline{HHH}(\textcolor{revisions}{\mathcal{U}})) = 1$, which is clearly symmetric in $Q$ and $T$. The authors in \cite{GS12} consider such reduced constructions, but they deal only with conjectural colored invariants at the physical level of rigor. We return to this point in our discussion of colored homology in \S \ref{sec: coloring}.

\subsubsection{Y-ified Homology}

We now turn to our other option for restoring the conjectural $Q \leftrightarrow T$ symmetry; this is the theory of $y$-ification as developed by Gorsky-Hogancamp in \cite{GH22}. By the discussion above, there is a family of degree $q^2t^{-1}$ homotopies $\{h_i\}_{i = 1}^n$ acting on $F(\beta)$ relating the right action of $x_i$ with the left action of $x_{\beta(i)}$. A direct computation shows that these homotopies can be chosen to satisfy the following identities:

\begin{align} \label{eq: dot_slide_ext1}
\textcolor{revisions}{h_i^2 = 0; \quad h_ih_j + h_jh_i = 0 \quad \text{for } i \neq j}
\end{align}

In particular, this family of homotopies assemble to give an exterior action of $\bigwedge[h_1, \dots, h_n]$ on $F(\beta)$. By Koszul duality, we can transform this exterior action into a polynomial action on a \textit{curved} complex (see \cite{GH22} and \S \ref{big_sec: yification} below). More precisely, we let $R[\mathbb{Y}]$ denote a polynomial algebra in formal variables $y_1, y_2, \dots, y_n$ of degree $\text{deg}(y_i) = q^{-2}t^2 = T$. We consider $R[\mathbb{Y}]$ as a chain complex with trivial differential. Define $F^y(\beta)$ to be the curved complex with the same underlying chain groups as $F(\beta) \otimes_R R[\mathbb{Y}]$ together with the degree $t$ endomorphism (called the \textit{connection})

\[
\delta_{F(\beta)} := d_{F(\beta)} \otimes 1 + \sum_{i = 1}^n h_i \otimes y_i
\]

Letting $x_i'$ denote the right action of $x_i$ on $F^y(\beta)$, a straightforward computation using \textcolor{revisions}{Equation \eqref{eq: dot_slide_ext1}} gives

\[
\delta_{F(\beta)}^2 = \sum_{i = 1}^n (x_{\beta(i)} - x_i') \otimes y_i
\]

Upon applying Hochschild cohomology and identifying $y$ variables associated to the same link component, what remains is a \textit{bona fide} chain complex. The (triply-graded) homology of this complex is again an invariant of the braid closure $\mathcal{L} = \hat{\beta}$, denoted $HHH^y(\mathcal{L})$. In the special case of the unknot, we obtain

\[
\text{dim}(HHH^y(\textcolor{revisions}{\mathcal{U}})) = \cfrac{1 + A}{(1 - Q)(1 - T)}
\]

Notice that this expression is again symmetric under $Q \leftrightarrow T$. The factors of $(1 - Q)^{-1}$ and $(1 - T)^{-1}$ in this expression come from the free polynomial actions of $R[x_1]$ and $R[y_1]$, respectively; in this context, one can think of the conjectural mirror symmetry as arising from an automorphism on \textcolor{revisions}{the triply-graded module $HHH^y(\mathcal{L})$ interchanging the actions of the alphabets} $\mathbb{X}$ and $\mathbb{Y}$. Indeed, this vision has been realized in the uncolored setting by recent work of Gorsky-Hogancamp-Mellit and Oblomkov-Rozansky (\cite{GHM21}, \cite{OR17}, \cite{OR18}, \cite{OR19a}, \cite{OR19b}, \cite{OR20}).

\begin{remark} \label{rem: yify_free_dim}
Just as the reduced homology can be recovered from the unreduced homology by killing the polynomial action by $R_{\mathcal{L}}$, the usual, un $y$-ified homology can be recovered from the $y$-ified homology by killing the action of $R[\mathbb{Y}]$. More precisely, by the same arguments as above, $HHH^y(\mathbb{L})$ carries both an action of $R_{\mathcal{L}}$ and a polynomial action of one $y$ variable for each component of $\mathcal{L}$; we refer to this latter polynomial ring in $y$ variables as $R[\mathbb{Y}_{\mathcal{L}}]$. The $R_{\mathcal{L}}$-module obtained from this larger module structure by setting each $y$ variable to $0$ is exactly $HHH(\mathcal{L})$.

Notice that this reduction from $HHH^y$ to $HHH$ depends on the module structure of $HHH^y$. In \cite{GH22}, those authors give a variety of conditions under which $HHH^y(\mathcal{L})$ is a free module over the alphabet $\mathbb{Y}$; in those cases, analogously to the passage from $HHH$ to $\overline{HHH}$, the graded dimensions of the two theories are simply related by dividing by the dimension of $R[\mathbb{Y}_{\mathcal{L}}]$.
\end{remark}

A direct application of Remarks \ref{rem: red_dim} and \ref{rem: yify_free_dim} gives the following proposition.

\begin{proposition} \label{prop: yify_to_red}
Let $\mathcal{L}$ be an $r$-component link, and suppose $\textcolor{revisions}{HHH^y(\mathcal{L}})$ is a free module over $R[\mathbb{Y}_{\mathcal{L}}]$. Then we have

\[
\text{dim}(\overline{HHH}(\mathcal{L})) = \left(\cfrac{1 + A}{1 - Q}\right)^{\textcolor{revisions}{-1}} (1 - T)^r \text{dim}(HHH^y(\mathcal{L}))
\]
\end{proposition}

In particular, let $r = 1$, so that $\mathcal{L}$ is a \textit{knot}. \textcolor{revisions}{Then $HHH^y(\mathcal{L})$ is known to be a free $R[\mathbb{Y}_{\mathcal{L}}]$-module by results of \cite{GH22}. As a consequence, if $HHH^y(\mathcal{L})$ is symmetric under the interchange $Q \leftrightarrow T$, $\overline{HHH}(\mathcal{L})$ is also symmetric under this interchange}.

In this uncolored setting, our work provides an explicit computation of $HHH^y(T(m, n))$ for all positive torus links $T(m, n)$. While the resulting expression is not manifestly symmetric under \textcolor{revisions}{the interchange }$Q \leftrightarrow T$, it does arise from a known expression for $HHH^y(T(m, n))$ as computed by Hogancamp-Mellit in \cite{HM19} upon  application of exactly this symmetry (when working in a coefficient field). As a consequence, we obtain a new proof of the uncolored variant of our main theorem:

\begin{theorem}
Let $R$ be a field. Then for all $m, n \geq 0$, the $y$-ified HOMFLYPT homology of the positive torus link $T(m, n)$ satisfies

\[
\text{dim}(HHH^y(T(m, n)))(A, Q, T) = \text{dim}(HHH^y(T(m, n)))(A, T, Q)
\]

up to an overall normalization.
\end{theorem}

This result is not new in the uncolored setting and follows from the \textcolor{revisions}{aforementioned} work of Gorsky-Hogancamp-Mellit and Oblomkov-Rozansky. We also point out that $HHH^y(T(m, n))$ is known to be a free module over $R[\mathbb{Y}_{T(m, n)}]$ for $R$ a field; hence by Proposition \ref{prop: yify_to_red}, we obtain the corresponding symmetry statement for reduced homology.

\subsection{Colored HOMFLYPT Homology} \label{sec: coloring}

To motivate our discussion of colored HOMFLYPT homology, we first recall the construction of the colored HOMFLYPT polynomial. We let $H_n$ denote the Type $A_{n - 1}$ Hecke algebra throughout. This is a $\mathbb{C}(q)$-algebra generated by elements $s_1, \dots, s_{n - 1}$ subject to the relations

\begin{align*}
    s_i^2 & = (q - q^{-1})s_i + 1; \\
    s_is_{i + 1}s_i & = s_{i + 1}s_is_{i + 1} \quad \text{for } 1 \leq i \leq n - 2; \\
    s_is_j & = s_js_i \quad \text{for } |i - j| \geq 2
\end{align*}

Notice that $H_n$ is a quotient of the braid group algebra $\mathbb{C}(q)[Br_n]$ by the first quadratic relation above; in particular, $H_n$ is a representation of $Br_n$ under the map $\pi$ sending each positive Artin generator $\sigma_i \in Br_n$ to $s_i \in H_n$. There exists a linear map $\chi \colon H_n \to \mathbb{C}(a, q)$ called the \textit{Jones-Ocneanu trace} satisfying properties compatible with the relationship between braids and link closures (\cite{Jon87}); this is the decategorified analog of applying Hochschild cohomology to obtain HOMFLYPT homology. Given a braid representative $\beta$ of $\mathcal{L}$, the HOMFLYPT polynomial of $\mathcal{L}$ is defined as $f(\mathcal{L}) = \chi(\pi(\beta))$ (up to an overall normalization and change of variables).

To define the \textit{colored} HOMFLYPT polynomial, recall that there exist a family of central idempotents $p_{\lambda} \in H_n$ called \textit{Young symmetrizers} indexed by partitions\footnote{These Young symmetrizers decompose as sums $p_{\lambda} = \sum_{T} p_T$ of ``primitive" idempotents $p_T$, where the sum is indexed by standard tableaux $T$ of shape $\lambda$. Since we are interested here in the cases in which $\lambda$ is a single row or a single column, for which there is only a single such choice $T$, we do not concern ourselves with this distinction.} $\lambda$ of $n$. Given a (framed) link $\mathcal{L} = \mathcal{L}_1 \sqcup \dots \sqcup \mathcal{L}_r$, the $(\lambda_1, \dots, \lambda_r)$-colored HOMFLYPT polynomial is obtained by the following procedure.

\begin{enumerate}
    \item Pick a braid representation $\beta$ of $\mathcal{L}$.
    \item Choose a collection of $r$ marked points, one on each component of $\mathcal{L}$, away from the crossings of $\beta$.
    \item Replace each $\lambda_i$-colored strand with $|\lambda_i|$ parallel strands.
    \item Insert a copy of $p_{\lambda_i}$ at the corresponding marked points of $\pi(\beta)$.
    \item Apply $\chi$ to the result of the previous step.
\end{enumerate}

The resulting Laurent polynomial can be shown to be independent of choice of $\beta$ and marked points; up to normalization, this is the colored HOMFLYPT polynomial $f^{\lambda_1, \dots, \lambda_r}(\mathcal{L})$.

Colored HOMFLYPT homology is obtained from a categorification of this \textcolor{revisions}{cabling and insertion process, as outlined in e.g.} \cite{CK12} for the $\mathfrak{sl}_2$ case of Khovanov homology \textcolor{revisions}{and \cite{Cau17} for a different variation of colored HOMFLYPT homology}. Most of the required technology is provided by the Rouquier complex and Hoschschild cohomology mentioned above; the cabling procedure proceeds exactly as in the decategorified case. What remains is to identify a categorified analog of the Young symmetrizers $p_\lambda$. These \textit{categorified projectors} are semi-infinite complexes $P_{\lambda} \in \textcolor{revisions}{K^-(SBim_n)}$ of Soergel bimodules.
\textcolor{revisions}{The family of Young symmetrizers $\{p_{\lambda}\}$ are idempotent, pairwise orthogonal, and add to the identity element of $H_n$; the family of projectors $\{P_{\lambda}\}$ enjoy similar categorical properties in $K^-(SBim_n)$. W}e refer the interested reader to (\cite{EH17a}, \cite{EH17b}) for details. \textcolor{revisions}{In particular}, the complexes $P_{\lambda}$ enjoy certain idempotency and centrality properties that enable a coherent, choice-independent construction of colored homology theories $HHH_{\lambda}$ by exactly the categorical analog of the procedure described above.

In this work, we deal with the cases in which $\lambda = (n)$ is a single row partition and $\lambda = (1^n)$ a single column partition; we refer to these as \textit{row-colored} and \textit{column-colored} homology, respectively. For \textcolor{revisions}{$\mathcal{L}$ a colored link}, we label the corresponding homology theories $HHH_{\text{Sym}^n}(\textcolor{revisions}{\mathcal{L}})$ and $HHH_{\bigwedge^n}(\textcolor{revisions}{\mathcal{L}})$ to highlight the connection with the corresponding representations. The corresponding categorified projectors in the setting of HOMFLYPT homology were first constructed and studied by Hogancamp in \cite{Hog18} in the case of $P_{(n)}$ and by Abel-Hogancamp\footnote{In fact we work primarily with the dual projector $P^\vee_{(1^n)}$ for reasons explained below.} in \cite{AH17} in the case of $P_{(1^n)}$. The projectors $P_{\lambda}$ were constructed for general $\lambda$ by Elias-Hogancamp in (\cite{EH17a}, \cite{EH17b}).

To formulate a plausible mirror symmetry conjecture for this colored homology package again requires either reduction or $y$-ification. We concern ourselves with $y$-ified homology in this work; we precisely formulate the mirror symmetry conjecture of interest to us below.

\begin{conjecture} \label{conj: mirr_sym}
There exist $y$-ifications $P^y_{\lambda}$ of the categorified projectors $P_{\lambda}$. Given an $r$-component link $\mathcal{L} = \mathcal{L}_1 \sqcup \dots \sqcup \mathcal{L}_r$, these $y$-ified projectors give rise to $y$-ified colored HOMFLYPT homology theories $HHH^y_{\lambda_1, \dots, \lambda_r}(\mathcal{L})$. These colored homology theories satisfy

\[
\text{dim}(HHH^y_{\lambda_1, \dots, \lambda_r}(\mathcal{L}))(A, Q, T) = \text{dim}(HHH^y_{\lambda_1^t, \dots, \lambda_r^t}(\mathcal{L}))(A, T, Q)
\]
\end{conjecture}

Partial progress towards $y$-ified colored homology was made by Elbehiry in \cite{Elb22} in the form of a $y$-ification of $P_{(1^n)}$, though that author does not verify the properties required to formulate a well-defined homology theory (namely, that $P_{(1^n)}^y$ is a counital idempotent \textcolor{revisions}{in the sense of \cite{Hog17}} and slides past crossings). We carry out that verification in this work, obtaining the following result.

\begin{theorem}
There exists a $y$-ification $P^y_{(1^n)}$ of the column projector $P_{(1^n)}$. This $y$-ification is a counital idempotent, and its dual $(P^y_{(1^n)})^{\vee}$ gives rise to a column-colored $y$-ified link homology theory $HHH^y_{\bigwedge^n}$.
\end{theorem}

In the case that $\mathcal{K}$ is a positive torus knot, we are able to explicitly compute $HHH^y_{\bigwedge^n}(\mathcal{K})$. We reproduce the result below. In what follows, \textcolor{revisions}{$\pi_{(10^l)}$ denotes a shuffle permutation; see Definition \ref{def: shuffle}.}

\begin{theorem} \label{thm: mainrecursion_intro}
For each pair $v \in \{0, 1\}^{l + m}$, $w \in \{0, 1\}^{l + n}$ and each $\sigma \in \mathfrak{S}^l$, let $p_{\sigma}(v, w) \in \mathbb{N}[[a, q, t]]$ be the family of power series satisfying the following properties:

\begin{enumerate}
	\item \textcolor{revisions}{$p_e(0^m, \emptyset) = \left( \cfrac{1 + A}{(1 - Q)(1 - T)} \right)^m$ and $p_e(\emptyset, 0^n) = \left( \cfrac{1 + A}{(1 - Q)(1 - T)} \right)^n$.}
	
    \item $p_e(0^m1, 0^n1) = \left( \cfrac{1 + A}{(1 - Q)(1 - T)} \right) p_e(0^m, 0^n)$ for all $m, n \geq 0$;
    
    \item $p_{\sigma}(v1, w1) = \left( \cfrac{Q^l + A}{1 - Q} \right) p_{Tr_{l + 1}(\sigma)}(v, w)$ if $|v| = |w| = l \geq 1$ and $\sigma(n) = n$;
    
    \item $p_{\sigma}(v1, w1) = \left( Q^l + A \right) p_{Tr_{l + 1}(\sigma)}(v, w)$ if $|v| = |w| = l \geq 1$ and $\sigma(n) \neq n$;
    
    \item $p_{\sigma}(v0, w1) = p_{\sigma \textcolor{revisions}{\pi_{(10^{l - 1})}}}(v, 1w)$ if $|v| = |w| + 1 = l \geq 1$;
    
    \item $p_{\sigma}(v1, w0) = p_{\textcolor{revisions}{\pi_{(10^{l - 1})}} \sigma}(1v, w)$ if $|v| + 1 = |w| = l \geq 1$;
    
    \item $p_e(0^m, 0^n) = p_e(10^{m - 1}, 10^{n - 1})$ \textcolor{revisions}{for all $m, n \geq 1$};
    
    \item $p_{\sigma}(v0, w0) = Q^{-l}p_{\textcolor{revisions}{e_1 \sqcup \sigma}}(1v, 1w) + Q^{-l}Tp_{\sigma}(0v, 0w)$ if $|v| = |w| = l \geq 1$.
\end{enumerate}

These relations uniquely determine $p_{\sigma}(v, w)$ for all possible input data $v, w, \sigma$. Moreover, for each $k, m, n \geq 1$ with $m, n$ coprime, the positive torus knot $T(m, n)$ satisfies

\[
\text{dim}(HHH^y_{\bigwedge^k}(T(m, n))) = \left( \prod_{i = 2}^k \cfrac{1}{1 - Q^{1 - i}T} \right) p_{e}(1^k0^{k(m - 1)}, 1^k0^{k(n - 1)})
\]

More generally, for any $m, n \geq 1$, let $d = \text{gcd}(m, n)$. Then the positive torus link $T(m, n)$ satisfies

\[
\text{dim}(HHH^y_{\bigwedge^k, \bigwedge^1, \dots, \bigwedge^1}(T(m, n))) = \left( \prod_{i = 2}^k \cfrac{1}{1 - Q^{1 - i}T} \right) p_{e}(1^k0^{md^{-1}(d - 1) + k(md^{-1} - 1)}, 1^k0^{nd^{-1}(d - 1) + k(nd^{-1} - 1)})
\]
\end{theorem}

We also construct a $y$-ified row projector $P^y_{(n)}$. We expect that $P^y_{(n)}$ is a unital idempotent, though we do not prove this fact. Nevertheless, our construction suffices to determine the candidate theory $HHH^y_{\text{Sym}^n}(\mathcal{K})$ for $\mathcal{K}$ a positive torus knot via a modest extension of work of Hogancamp-Mellit \cite{HM19} (again, for $R$ a field). Comparing with Theorem \ref{thm: mainrecursion_intro}, we immediately obtain our primary theorem:

\begin{theorem} \label{thm: maintheorem_intro}
Let $R$ be a field. Then for all $k, m, n \geq 1$, the positive torus link $T(m, n)$ satisfies

\[
\text{dim}(HHH^y_{\text{Sym}^k, \text{Sym}^1, \dots, \text{Sym}^1}(T(m, n)))(A, Q, T) = \text{dim}(HHH^y_{\bigwedge^k, \bigwedge^1, \dots, \bigwedge^1}(T(m, n)))(A, T, Q)
\]

up to an overall normalization.
\end{theorem}

To our knowledge, this is the first verification of Conjecture \ref{conj: mirr_sym} for nontrivial $\lambda$.

\begin{example} \label{example: unkknot_mirsym}
We illustrate the mirror symmetry phenomenon of Theorem \ref{thm: maintheorem_intro} for the $k$-colored \textcolor{revisions}{homology of the} unknot \textcolor{revisions}{$\mathcal{U}$}. As a direct consequence of Theorem \ref{thm: mainrecursion_intro}, the $y$-ified, $\bigwedge^k$-colored \textcolor{revisions}{homology of $\mathcal{U}$} has graded dimension

\begin{align}
\text{dim}(HHH^y_{\bigwedge^k})(\textcolor{revisions}{\mathcal{U}}) = \left( \cfrac{1}{1 - Q} \right)^k \left( \prod_{i = 1}^k \cfrac{Q^{i - 1} + A}{1 - Q^{1 - i}T} \right) \label{eq: col_unknot_grdim}
\end{align}

On the other hand, by Example 4.7 in \cite{HM19}, the (un $y$-ified) $\text{Sym}^k$-colored \textcolor{revisions}{homology of $\mathcal{U}$} has graded dimension

\[
\text{dim}(HHH_{\text{Sym}^k})(\textcolor{revisions}{\mathcal{U}}) = \prod_{i = 1}^k \cfrac{T^{i - 1} + A}{1 - QT^{i - 1}}
\]

In particular, this invariant is concentrated in even homological degrees. As a consequence (see our Proposition \ref{prop: row_color_free}), the corresponding $y$-ified invariant is a free module over $R[y_1, \dots, y_k]$. Since each of these $y$ variables has degree $\text{deg}(y_i) = q^{-2}t^2 = T$, we easily obtain the graded dimension of the $y$-ified invariant by multiplication:

\begin{align}
\text{dim}(HHH^y_{\text{Sym}^k})(\textcolor{revisions}{\mathcal{U}}) = \left(\cfrac{1}{1 - T}\right)^k \left(\prod_{i = 1}^k \cfrac{T^{i - 1} + A}{1 - QT^{i - 1}}\right) \label{eq: row_unknot_grdim}
\end{align}

It is clear that interchanging $Q$ and $T$ in either of Equations \eqref{eq: col_unknot_grdim} and \eqref{eq: row_unknot_grdim} recovers the other.
\end{example}

\begin{remark}
There are several distinct categorifications of the colored HOMFLYPT polynomial in the literature. In the case of $\lambda = (1^k)$ a single column, one such theory is obtained by replacing each $(1^k)$-labelled strand with a \textit{colored} braid strand labelled with the positive integer $k$ and replacing the Rouquier complex with the Rickard complex of colored braids. This approach was originally proposed by Mackaay-\textcolor{revisions}{Sto\v{s}i\'{c}}-Vaz in \cite{MSV11} and constructed in general by Webster-Williamson in \cite{WW17}; for clarity, we refer to this colored homology theory within this Remark as \textit{Webster-Williamson} homology. This is the invariant studied in \cite{RT21} and \cite{HRW21}; the authors of the latter work also construct a $y$-ified version of this invariant. A reduced version of Webster-Williamson homology is constructed and studied in \cite{Wed19}.

\textcolor{revisions}{In \cite{Cau17}, Cautis gives a construction of colored HOMFLYPT homology for $\lambda$ an arbitrary partition. This construction also uses colored strands, but proceeds by cabling and inserting a \textit{HOMFLY clasp} ${\sf{P}}_{\lambda}^-$ built out of infinite colored full twist braids (these HOMFLY clasps are not to be confused with the Elias-Hogancamp projectors $P_{\lambda}$). The construction of HOMFLY clasps was extended by Abel-Willis in \cite{AW19}, \cite{Wil20}. We refer to the invariant of \cite{Cau17} as \textit{Cautis homology}; in the case that $\lambda$ is a single column, it agrees with Webster-Williamson homology.}

In the finite rank case of colored $\mathfrak{sl}_n$ homology, many known categorifications were shown to be equivalent by Cautis in \cite{Cau15}. In contrast, for colored HOMFLYPT homology, even for the un $y$-ified column-colored unknot these invariants disagree. Webster-Williamson and Cautis homology assign the $(1^k)$-colored unknot a vector space \textcolor{revisions}{$HHH_{WW, \bigwedge^k}(\mathcal{U})$} of dimension

\[
\textcolor{revisions}{\mathrm{dim}(HHH_{WW, \bigwedge^k}(\mathcal{U})) = }\prod_{i = 1}^k \cfrac{1 + aq^{-2i}}{1 - q^{2i}}
\]

On the other hand, the dimension of the (un $y$-ified) Elias-Hogancamp homology considered in this paper evaluated on the $(1^k)$-colored unknot is

\[
\textcolor{revisions}{\mathrm{dim}(HHH_{\bigwedge^k}(\mathcal{U})) = } \left( \cfrac{1 - t^2q^{-2}}{1 - q^2} \right)^k \prod_{i = 1}^k \cfrac{1 + aq^{-2i}}{1 - t^2q^{-2i}}
\]

After $y$-ifying, we may similarly compare our invariant obtained from the $y$-ified projector to that of Hogancamp-Rose-Wedrich in \cite{HRW21}, obtained from a family of higher $y$-ifications on Rickard complexes. For convenience, we denote their $y$-ified, $\bigwedge^k$-colored invariant of a given knot $\mathcal{L}$ by $HRW_k(\mathcal{L})$. Reading from Example 5.31 in \cite{HRW21}, the invariant of the $k$-colored unknot in this theory has graded dimension

\[
\text{dim}(HRW_k(\bigcirc)) = \prod_{i = 1}^k \cfrac{1 + AQ^{1 - i}}{(1 - Q^i)(1 - TQ^{1 - i})}
\]

A straightforward comparison with Equation \eqref{eq: col_unknot_grdim} gives the relationship

\[
\text{dim}(HHH^y_{\bigwedge^k}(\bigcirc)) = q^{-\frac{3}{2}k(k - 1)} [k]! \  \text{dim}(HRW_k(\bigcirc))
\]

Here $[k]!$ is the product of the first $k$ quantum natural numbers, as explained below. For now, we briefly comment that this factor of $[k]!$ is exactly the factor obtained by digon removal when completely splitting and merging a $k$-labelled strand. See Section \ref{sec: inv_compare} for further discussion.

Categorified Young symmetrizers have also been constructed in the finite rank $\mathfrak{sl}_n$ homology setting by a variety of authors; see (\cite{Cau15}, \cite{CK12}, \cite{Ros14}, \cite{Roz10}, \cite{SS11}) and the references therein.
\end{remark}

\subsection{Torus Link Homology via Column Projectors}

The technology enabling our proof of Theorem \ref{thm: mainrecursion_intro} fits within a recursive framework originally implemented by Elias-Hogancamp-Mellit in their computations of (un $y$-ified) uncolored and row-colored homology (\cite{EH19}, \cite{Mel22}, \cite{HM19}). Their work relies on a recursive construction of the family of projectors $P_{(n)}$ for varying $n$ developed by Hogancamp in \cite{Hog18} emphasizing the projectors' behavior under concatenation with the Jucys-Murphy braid $J_{n + 1}$. In particular, letting $v_n$ be a formal variable of degree $q^{2n}t^{2 - 2n} = QT^{1 - n}$, Hogancamp proves that $P_{(n)}$ satisfies a recursion of the form

\begin{align}
P_{(n)} \simeq \left(q^{2n}t^{1 - 2n} \mathbb{Z}[v_n] \otimes P_{(n - 1)} \rightarrow \mathbb{Z}[v_n] \otimes P_{(n - 1)}J_n \right) \label{eq: intro_inf_rec} 
\end{align}

Upon setting $P_{(1)} := R_1$, this recursion lays bare a $\mathbb{Z}[v_2, \dots, v_n]$-periodicity in $P_{(n)}$. This periodicity is reflected in an action of the polynomial ring $\mathbb{Z}[v_2, \dots, v_n]$ on the endomorphism algebra $\text{End}(P_{(n)})$; ``killing the action" of these periodic endomorphisms $v_i$ by passing to a Koszul complex results in a finite analogue $K_{(n)}$ of the projector $P_{(n)}$. This finite projector inherits from $P_{(n)}$ a recursive description similar to that of Equation \eqref{eq: intro_inf_rec}. Upon rotating triangles, this allows one to absorb all the homological complexity of the Jucys-Murphy braid in the complex $K_{(n - 1)}J_n$ into a convolution involving only finite projectors $K_{(n)}$ and $K_{(n - 1)}$. This method is particularly well-suited to torus links, allowing for a complete decomposition of their Rouquier complexes into convolutions of finite projectors and disjoint strands.

The problem of computing link homology is thus reduced to the two problems of computing Hochschild cohomology of the finite projectors $K_{(n)}$ and reassembling the Hocshchild cohomology of a convolution from the Hochschild cohomologies of its constituents. The former is accomplished by way of the (finite analogue of the) topological recursion \eqref{eq: intro_inf_rec}. The latter task is greatly simplified by the observation that the relative homological shifts between these summands are all of even degree; since the differential must be of \textit{odd} homological degree (in particular, degree $1$), this convolution is forced to be a direct sum. In particular, the Hochschild cohomology of the convolution is the direct sum of the Hochschild cohomologies of its summands.

We recreate this procedure for the column projector $P_{(1^n)}$. This projector was originally constructed by Abel-Hogancamp in \cite{AH17}; their construction realizes $P_{(1^n)}$ as an explicit periodic complex modeled on a certain finite Koszul complex $K_n$. Our first task is to adapt their construction to the topological setting. Instead of working directly with the infinite complex, we begin by establishing the desired recursion for the finite complex $K_n$. We then extend to the infinite projector $P_{(1^n)}$ by means of a long, explicit verification that our convolution can be perturbed to be compatible with the periodic structure on $P_{(1^n)}$. Leveraging uniqueness of counital idempotents, we arrive at the following (Theorem \ref{thm: projtop} below):

\begin{theorem} \label{thm: intro_recursion}
Let $u_n$ be a formal variable of degree $q^{-2n}t^2 = Q^{1 - n}T$. Then the infinite projector $P_{(1^n)}$ satisfies a recursion of the form

\begin{align*}
P_{(1^n)} \simeq \left( \mathbb{Z}[u_n^{-1}] \otimes P_{(1^{n - 1})} \rightarrow q^2t^{-1} \mathbb{Z}[u_n^{-1}] \otimes P_{(1^{n - 1})}J_n \right)
\end{align*}
\end{theorem}

\begin{remark} \label{rem: intro_duals}
Notice that it is $\mathbb{Z}[u_n^{-1}]$, not $\mathbb{Z}[u_n]$, that appears in this construction. This reflects the status of $P_{(1^n)}$ as a counital idempotent rather than a unital idempotent. To construct the colored homology theory of interest to us, we will dualize to the corresponding unital idempotent $P^{\vee}_{(1^n)}$. This restores the action of the polynomial ring $\mathbb{Z}[u_2, \dots, u_n]$ on $\text{End}(P^{\vee}_{(1^n)})$, allowing for a similar reduction to consideration of finite column projectors as in \cite{EH19}. Notice that hints of mirror symmetry are already present in these periodic constructions, in that $\text{deg}(u_i)(Q, T) = \text{deg}(v_i)(T, Q)$ for all $2 \leq i \leq n$.
\end{remark}

Unfortunately, the finite column projectors $K_n$ do not satisfy the same homological parity properties as do the finite row projectors $K_{(n)}$; c.f. Section \ref{sec: finite} below. To remedy this, we explicitly $y$-ify each step of the above recursive construction. The appropriate $y$-ified complexes were originally constructed by Elbehiry in \cite{Elb22}, though we require slight modifications of that author's constructions here. Another explicit computation verifies that the family of $y$-ified projectors $P_{(1^n)}^y$ enjoys the same recursive structure as does the un $y$-ified family. The following is Theorem \ref{thm: yprojtop} below:

\begin{theorem} \label{thm: intro_yprojtop}
There exist (strict) $y$-ifications $P^y_{(1^n)}$ of the projectors $P_{(1^n)}$ satisfying

\begin{align*}
P_{(1^n)}^y \simeq \left( \mathbb{Z}[u_n^{-1}] \otimes P^y_{(1^{n - 1})} \rightarrow q^2t^{-1} \mathbb{Z}[u_n^{-1}] \otimes P^y_{(1^{n - 1})}J_n^y \right)
\end{align*}
\end{theorem}

\begin{remark}
We pause here to remark that the convolutions involved in computing link homology are finite. In fact, for the computation of uncolored homology, the full strength of the infinite projector is not required, and we could content ourselves with establishing Theorem \ref{thm: intro_yprojtop} for the finite $y$-ified projectors $K^y_n$. The advantage of the infinite version of Theorem \ref{thm: intro_yprojtop} is the ease it furnishes in proving that $P^y_{(1^n)}$ is a counital idempotent and hence gives rise to a well-defined colored homology theory. The corresponding fact for the un $y$-ified column projector was established in \cite{AH17} using general theory of projective resolutions. This toolkit is not available to us after $y$-ifying, as we now deal with \textit{curved} complexes. Instead, we are forced to adapt the corresponding argument for the row projector from \cite{Hog18} to this setting.
\end{remark}

To explicitly compute the Hochschild cohomology of the $y$-ified finite column projectors $K^y_n$ requires us to develop a robust toolkit for computing $y$-ified Hochschild cohomology in general. In particular, we construct $y$-ified ``partial trace" functors analogous to the un $y$-ified partial trace functors developed in \cite{Hog18}. With this technology in hand, we are able to localize our computations to individual strands.

By replacing the finite row projector $K_{(n)}$ with the $y$-ified finite column projector $K^y_n$, we are able to adapt the computations of uncolored homology for arbitrary positive torus links carried out by Hogancamp-Mellit in \cite{HM19} to our setting. After dualizing as in Remark \ref{rem: intro_duals}, we can pass from computations involving the finite projector $(K^y_n)^{\vee}$ to those involving the infinite projector $(P^y_{1^n})^{\vee}$ by restoring periodicity exactly as in the row-colored case, finally arriving at Theorem \ref{thm: mainrecursion_intro}.

\subsection{Outlook} \label{sec: outlook}

In this section, we briefly describe potential extensions of this work in perceived order of difficulty.

\subsubsection{Comparing Colored Homologies} \label{sec: inv_compare}

In \cite{HRW21}, those authors compute their $y$-ified Webster-Williamson homology of the Hopf link $T(2, 2)$ for arbitrary positive integer colors. Numerical comparison with the results of this paper suggests that the relationship between their results and our $y$-ified Elias-Hogancamp invariant for these links should be similar to the corresponding relationship for the unknot of Remark \ref{rem: red_dim}, though the statement of a precise conjecture would require an appropriate interpretation of the deformation alphabets of \cite{HRW21} on each colored strand. We expect that one could leverage our Propositions \ref{prop: kosz_base} and \ref{prop: inf_reduction} to pass directly from Elias-Hogancamp homology to Webster-Williamson homology whenever the former is concentrated in even homological degree. In particular, this would allow for a proof of parity for Webster-Williamson colored homology of positive torus links from a generalization of our Theorem \ref{thm: mainrecursion_intro} to multiple colored strands. After passing to reduced homology, it is reasonable to expect that such a relationship would also suggest an analogue of the exponential growth properties established by Wedrich in \cite{Wed19} for our invariant. We plan to investigate this problem in future work.

\subsubsection{Module Structure over (Derived) Sheet Algebra} \label{sec: mod_sheet_spec}

As discussed in \S \ref{subsubsec: reduce}, many link homology theories carry an additional module structure over the corresponding sheet algebra; in particular, this is true of the $y$-ified invariants constructed in \cite{GH22} and \cite{HRW21}. In fact more is true: The $y$-ified homology of \cite{GH22} carries a well-defined action of the \textit{derived} sheet algebra in that theory, and the invariant of \cite{HRW21} carries an action of the sheet algebra at the level of chain complexes, well-defined up to quasi-isomorphism, which conjecturally extends to an action of the derived sheet algebra. We construct an action of the relevant derived sheet algebra on $\bigwedge^k$-colored $y$-ified homology in \S \ref{sec: mod_sheet}. In light of the discussion in \S \ref{sec: inv_compare}, it would be interesting to investigate the relationship between this module structure and that of \cite{HRW21}.

\subsubsection{$y$-ified Row-colored Homology} \label{subsec: yify_row}

Our procedure for computing $y$-ified $\bigwedge^k$-colored homology passes first from uncolored, un $y$-ified homology to uncolored $y$-ified homology by direct computation, then to $y$-ified column-colored homology by directly accounting for periodicity. In contrast, we establish Theorem \ref{thm: maintheorem_intro} by passing from un $y$-ified row-colored homology as computed in \cite{HM19} to $y$-ified row-colored homology, appealing to parity arguments developed in \cite{GH22} to simplify our computations. This approach has two disadvantages over its reversal. First, the parity arguments of \cite{GH22} require one to work over a coefficient \textit{field}; this prevents us from claiming Theorem \ref{thm: maintheorem_intro} over the integers. More fundamentally, our $y$-ification of the row projector $P_{(n)}$ in Theorem \ref{thm: yify_row} does not demonstrate the same recursive structure as our construction of $P_{(1^n)}^y$. As such, we are unable to directly apply our arguments to prove that $P_{(n)}^y$ is a unital idempotent and hence unable to construct a general $y$-ified row-colored homology theory. One could imagine instead carrying out the procedure of this paper to obtain an explicit $y$-ification of the recursion of \cite{Hog18}. See Remark \ref{rem: field_coeff} for more details.

\subsubsection{Mirror Symmetry for Hooks}

More generally, it was suggested to the author by Matt Hogancamp that the arguments of this work should be adaptable to the case of \textit{hook} partitions of the form $\lambda = (n, 1, 1, \dots, 1)$. This would generalize both the column-colored computation of this paper and the row-colored computations of \S \ref{subsec: yify_row}. Since hooks are stable under transposition $\lambda \mapsto \lambda^t$, these computations would establish Conjecture \ref{conj: mirr_sym} in this case.

\subsection{Structure}

This paper is organized as follows. In Section \ref{sec: cat_back}, we recall the categorical background required for our constructions and introduce notation; none of that material is original, and the expert may safely skip to Section \ref{sec: inf_proj_big}. We begin Section \ref{sec: inf_proj_big} by recalling the construction of \cite{AH17} and reducing the proof of Theorem \ref{thm: intro_recursion} to the existence of certain distinguished morphisms as in \cite{Hog18}. We construct the first of these morphisms in Section \ref{sec: const_beta}. We construct the second such morphism for the finite projector in Section \ref{sec: fork_slide} and extend to the infinite projector in Section \ref{sec: alpha_build}. In Section \ref{big_sec: yification} we $y$-ify this whole procedure, culminating in Theorem \ref{thm: yprojtop} in Section \ref{sec: yify_inf_proj}. We silo a long technical piece of this computation to Appendix \ref{app: dot_slide_nat}. We develop the $y$-ified partial trace toolkit referenced above in Section \ref{sec: yify_hoch_coh}. Finally, in Section \ref{sec: link_hom}, we apply our results to compute column-colored $y$-ified homology, culminating in Theorem \ref{thm: row_col_mirror_sym} in Section \ref{sec: col_tor_hom}.

\subsection{Acknolwedgements}

\textcolor{revisions}{We} would like to thank Eugene Gorsky for many helpful comments on an earlier draft of this paper, Ben Elias for his explanation of categorical diagonalization, and Matt Hogancamp for many helpful discussions during the preparation of this work. \textcolor{revisions}{We} also send thanks to Andrew Adair, Reed Hubbard, and Emma Crawford, whose continuing support throughout the long preparation of this work were instrumental in its eventual completion. \textcolor{revisions}{We} would especially like to thank \textcolor{revisions}{our} advisor, David Rose, for countless helpful discussions and his enduring guidance and encouragement during the long preparation of this work. \textcolor{revisions}{Finally, we would like to thank an anonymous referee, whose many suggestions regarding an earlier version of this paper have dramatically improved its clarity.}

During the preparation of this work, the author was partially supported by Simons Collaboration Grant 523992: “Research on knot invariants, representation theory, and categorification,” NSF CAREER Grant DMS-2144463, and NSF Grant DMS-1954266. This project was conceived of and partially carried out during the AIM Research Community on Link Homology.

\section{Categorical Background} \label{sec: cat_back}

\subsection{DG Categories}

Let $R$ be a commutative ring, and let $\mathcal{A}$ be an $R$-linear category. Denote by $\text{Seq}(\mathcal{A})$ the category whose objects are sequences $\{X^i\}_{i \in \mathbb{Z}}$ of objects in $\mathcal{A}$ indexed by $\mathbb{Z}$ and whose morphism spaces are $\mathbb{Z}$-graded $R$-modules with components

$$
\text{Hom}_{\text{Seq}(\mathcal{A})}^k(X, Y) := \prod_{i \in \mathbb{Z}} \text{Hom}_{\mathcal{A}}(X^i, Y^{i + k})
$$

We call $x \in X^j$ an element of \textit{degree} $j$ (or \textit{homological degree} $j$\textcolor{revisions}{,} when we wish to emphasize the indexing set $\mathbb{Z}$) and write $|x| = j$ or $\text{deg}(x) = j$; similarly, we say $f$ is a morphism of (homological) \textit{degree} $j$ and write $|f| = j$ or $\text{deg}(f) = j$ if $f \in \text{Hom}_{\text{Seq}(\mathcal{A})}^j(X, Y)$. Denote by $\text{Ch}(\mathcal{A})$ the category whose objects consist of pairs $X \in \text{Seq}(\mathcal{A})$ and morphisms $d_X \in \text{End}^1_{\text{Seq}(\mathcal{A})}(X)$ satisfying $d_X^2 = 0$, and whose morphisms are inherited from those of $\text{Seq}(\mathcal{A})$. We call objects of Ch($\mathcal{A}$) \textit{chain complexes} and $d_X$ the \textit{differential} on $X$. In the special case that $\mathcal{A} = R$-Mod, we call objects of Ch($\mathcal{A}$) \textit{dg $R$-modules}. 

Observe that given $X, Y \in \text{Ch}(\mathcal{A})$, the space of morphisms $\text{Hom}_{\text{Ch}(\mathcal{A})}(X, Y)$ is a dg $R$-module via the differential below:

$$d_{\text{Hom}_{\text{Ch}(\mathcal{A})}(X, Y)}: f \mapsto [d, f] := d_Y \circ f - (-1)^{|f|} f \circ d_X$$

Moreover, if $\mathcal{A}$ is a monoidal category, then Ch($\mathcal{A}$) is a monoidal category also, with tensor product structure\textcolor{revisions}{\footnote{\textcolor{revisions}{Note that we do \textit{not} negate the term $\mathrm{id}_X \otimes d_Y$ in odd degrees; the usual sign present in this formula is instead accounted for by the sign involved in the middle interchange law.}}} given by

$$
(X \otimes Y )^k := \bigoplus_{i + j = k} X^i \otimes Y^j; \quad d_{X \otimes Y} := d_X \otimes \text{id}_Y + \text{id}_X \otimes d_Y
$$

By convention the evaluation of the tensor product of two homogeneous morphisms $f, g$ on two homogeneous elements $x, y$ is defined as

$$
(f \otimes g)(x \otimes y) := (-1)^{|g||x|} f(x) \otimes g(y)
$$

Evaluation in general is given by additively extending the above rule. As a result of the above convention, we have the following \textit{middle interchange law} for homogeneous morphisms $f, f', g, g'$:

$$
(f \otimes g) \circ (f' \otimes g') = (-1)^{|g||f'|} (f \circ f') \otimes (g \circ g')
$$

\begin{definition}
Let $(X, d_X)$ be a dg $R$-module. We call $X$ a \textit{dg $R$-algebra} if $X$ is equipped with an $R$-bilinear multiplication giving $X$ the structure of a graded $R$-algebra with respect to homological degree and $d_X$ satisfies the \textit{graded Leibniz rule}:

\[
d_X(ab) = (d_X(a))b + (-1)^{|a|}a(d_X(b))
\]
\end{definition}

Given two dg $R$-algebras $(X, d_X)$, $(Y, d_Y)$, their tensor product $X \otimes Y$ also has the structure of a dg $R$-algebra with multiplication

\[
(x_1 \otimes y_1)(x_2 \otimes y_2) := (-1)^{|y_1||x_2|}(x_1x_2) \otimes (y_1y_2)
\]

\begin{example}
Let $A$ be any $R$-algebra; then $A$ has a trivial dg $R$-algebra structure obtained by declaring all elements of $A$ to live in degree $0$ and setting $d_A = 0$.
\end{example}

\begin{example}
Let $X \in \text{Ch}(\mathcal{A})$ be given; then $\text{End}_{\text{Ch}(\mathcal{A})}(X)$ is a dg $R$-algebra with multiplication given by function composition. The only nontrivial check is the graded Leibniz rule, which is easily verified.

% \begin{align*}
%     [d, f \circ g] & = d_X \circ (f \circ g) - (-1)^{|f \circ g|}(f \circ g) \circ d_X \\
%     & = d_X \circ f \circ g - (-1)^{|f|} f \circ d_X \circ g + (-1)^{|f|}f \circ d_X \circ g - (-1)^{|f| + |g|} f \circ g \circ d_X \\
%     & = [d, f] \circ g + (-1)^{|f|}f \circ [d, g]
% \end{align*}
\end{example}

\begin{example}
Let $\theta_1, \dots, \theta_n$ be formal variables of homological degree $-1$. Then the exterior algebra $\bigwedge[\theta_1, \dots, \theta_n]$ over $R$ has the structure of a dg algebra with trivial differential. For each $1 \leq j \leq n$, we have morphisms $\theta_j \in \text{End}^{-1}_{\text{Ch}(R-\text{Mod})}(\bigwedge[\theta_1, \dots, \theta_n])$ and $\theta_j^{\vee} \in \text{End}^{1}_{\text{Ch}(R-\text{Mod})}(\bigwedge[\theta_1, \dots, \theta_n])$ given by multiplication and contraction, respectively. Each of these endomorphisms squares to $0$, and $\theta_i\theta_j + \theta_j\theta_i = \theta_i^{\vee}\theta_j^{\vee} + \theta_j^{\vee}\theta_i^{\vee} = \theta_i\theta_j^{\vee} + \theta_j^{\vee}\theta_i = 0$ for $i \neq j$. We also have $\theta_i\theta_i^{\vee} + \theta_i^{\vee}\theta_i = \text{id}$ for each $i$.

Given any dg $R$-algebra $M$, the tensor product $M \otimes \bigwedge[\theta_1, \dots, \theta_n]$ is a dg $R$-algebra with trivial differential by the monoidal structure discussed above. Given any elements $f_1, \dots, f_n \in M$, there is an alternative dg $R$-algebra structure on the underlying chain complex $M \otimes \bigwedge[\theta_1, \dots, \theta_n]$ via the differential

\[
d_{M \otimes \bigwedge[\theta_1, \dots, \theta_n]} := \sum_{j = 1}^n f_j \otimes \theta_j^{\vee}
\]

That this differential satisfies $d^2 = 0$ is an easy consequence of the middle interchange law; the graded Leibniz rule follows from the fact that contraction is a derivation. The dg $R$-algebra $M \otimes \bigwedge[\theta_1, \dots, \theta_n]$ equipped with this differential is often called the \textit{Koszul complex} for the elements $f_1, \dots, f_n$ on $M$.
\end{example}

We call morphisms $f \in \text{Hom}_{\text{Ch}(\mathcal{A})}(X, Y)$ \textit{closed} if $[d, f] = 0$ and \textit{exact} (or \textit{nullhomotopic}) if $f = [d, g]$ for some $g \in \text{Hom}(X, Y)$. Two homogeneous morphisms $f, g$ with $|f| = |g|$ are called \textit{homotopic} if $f - g$ is exact; in this case, we write $f \sim g$. A closed map of degree $0$ is called a \textit{chain map}.

\begin{definition}
We call $\mathcal{A}$ a \textit{dg category over $R$} (or just a dg category, if there is no confusion) if $\text{Hom}_{\mathcal{A}}(X, Y)$ is an object of Ch($R$-Mod) for each $X, Y \in \mathcal{A}$ and composition of morphisms defines a chain map

$$
\text{Hom}_{\mathcal{A}}(X, Y) \otimes \text{Hom}_{\mathcal{A}}(Y, Z) \rightarrow \text{Hom}_{\mathcal{A}}(X, Z)
$$
\end{definition}

In this case, we will often denote the differential on $\text{Hom}_{\mathcal{A}}(X, Y)$ simply by $d_{\mathcal{A}}$. That composition of morphisms defines a chain map is easily checked to be equivalent to the \textit{Leibniz rule} for homogeneous morphisms $f, g$:

$$
d_{\mathcal{A}}(f \circ g) = d_{\mathcal{A}}(f) \circ g + (-1)^{|f|}f \circ d_{\mathcal{A}}(g)
$$

\begin{example}
It is easily checked that the differential $d_{\text{Hom}_{\text{Ch}(\mathcal{A})}}$ satisfies the Leibniz rule, so $\text{Ch}(\mathcal{A})$ is a dg category. As above, we denote the differential on morphism spaces in this category by $d_{\text{Ch}(\mathcal{A})}$ or by $[d, -]$ when there is no confusion.
\end{example}

Let $\mathcal{C}, \mathcal{D}$ be any dg categories. A \textit{dg functor} $F: \mathcal{C} \rightarrow \mathcal{D}$ is a functor from $\mathcal{C}$ to $\mathcal{D}$ that induces a degree zero chain map on hom spaces; that is, $d_{\mathcal{D}}(F(f)) = F(d_{\mathcal{C}}(f))$ for all $f$.

We let $Z^0(\mathcal{C})$ denote the category with the same objects as $\mathcal{C}$ and morphisms restricted to the closed degree zero morphisms of $\mathcal{C}$. We restrict the use of the word \textit{isomorphism} for objects in $\mathcal{C}$ to be the usual notion of isomorphism in $Z^0(\mathcal{C})$ (that is, relation via invertible morphisms in $Z^0(\mathcal{C})$). We write $X \cong Y$ to denote isomorphic objects of $\mathcal{C}$.

Let $H^0(\mathcal{C})$ denote the category with the same objects and morphisms as $Z^0(\mathcal{C})$ but declare two morphisms $f, g$ equivalent in $Z^0(\mathcal{C})$ if $f \sim g$. We will often call $H^0(\mathcal{C})$ the \textit{homotopy category} of $\mathcal{C}$. We call two complexes $X, Y$ \textit{homotopy equivalent} if $X$ and $Y$ are isomorphic in $H^0(\mathcal{C})$, and write $X \simeq Y$. If $X \simeq 0$, we say $X$ is \textit{contractible}. Unwinding definitions, we see that the data of a homotopy equivalence is a collection of maps:

\begin{center}
\begin{tikzcd}
X \arrow[rr, "f", shift left] \arrow["h"', loop, distance=2em, in=215, out=145] &  & Y \arrow[ll, "g", shift left] \arrow["k"', loop, distance=2em, in=35, out=325]
\end{tikzcd}
\end{center}

such that $|f| = |g| = 0$, $|h| = |k| = 1$, $d_{\mathcal{C}}(f) = d_{\mathcal{C}}(g) = 0$, $d_{\mathcal{C}}(h) = \text{id}_X - g \circ f$, and $d_{\mathcal{C}}(k) = \text{id}_Y - f \circ g$. A homotopy equivalence satisfying $k = 0$ is called a \textit{strong deformation retraction} from $X$ to $Y$.

We will often denote $Z^0(\text{Ch}(\mathcal{A}))$ by $C(\mathcal{A})$ and $H^0(\text{Ch}(\mathcal{A}))$ by $K(\mathcal{A})$. If $\mathcal{A}$ is abelian, then a morphism $f \in \text{Hom}_{C(\mathcal{A})}(X, Y)$ induces a map $f^* \in \text{Hom}_{\text{Seq}(\mathcal{A})}(H^{\bullet}(X), H^{\bullet}(Y))$ in cohomology. We call $f$ a \textit{quasi-isomorphism} if $f^*$ is an isomorphism.

We will often use superscripts to denote bounds on sequences in categories above. For example, $K^b(\mathcal{A})$ (resp. $K^+(\mathcal{A})$, $K^-(\mathcal{A})$) will denote the homotopy category of bounded (resp. bounded below, bounded above) chain complexes on $\mathcal{A}$. $\text{Seq}^b(\mathcal{A}), \text{Ch}^b(\mathcal{A})$, etc. are defined similarly.

Our primary tool in establishing strong deformation retractions between complexes will be the following \textcolor{revisions}{(see e.g. Lemma 3.2 of \cite{BN07} for a proof)}:

\begin{proposition}[Gaussian Elimination] \label{prop: gauss_elim}
Let $\mathcal{A}$ be an $R$-linear category, let $A, B_1, B_2, D, E, F \in \mathcal{A}$ be given, and let $\varphi: B_1 \rightarrow B_2$ be an isomorphism. Then whenever the top row of the diagram below is a piece of a chain complex, the morphisms specified by the vertical arrows of the diagram below on that piece and identity maps elsewhere constitute a strong deformation retraction from the top row to the bottom row.

\begin{center}
\begin{tikzcd}[ampersand replacement=\&]
    A \arrow[rr, "\begin{pmatrix} \alpha \\ \beta \end{pmatrix}"] \arrow[dd, "\text{id}_A", harpoon, shift left]
        \& \& B_1 \oplus D \arrow[dd, "{\setlength\arraycolsep{2pt} \begin{pmatrix} 0 & \text{id}_D \end{pmatrix}}", harpoon, shift left] \arrow[rr, "{\begin{pmatrix} \varphi & \kappa \\ \gamma & \epsilon \end{pmatrix}}"]
            \& \& B_2 \oplus E \arrow[dd, "{\setlength\arraycolsep{1pt} \begin{pmatrix} -\gamma \varphi^{-1} & \text{id}_E \end{pmatrix}}", harpoon, shift left] \arrow[rr, "{\begin{pmatrix} \rho & \psi \end{pmatrix}}"]
                \& \& F \arrow[dd, "\text{id}_F", harpoon, shift left]\\
    \\
    A \arrow[rr, "\beta"'] \arrow[uu, "\text{id}_A", harpoon, shift left]
        \& \& D \arrow[rr, "\epsilon - \gamma \varphi^{-1}\kappa"'] \arrow[uu, "\begin{pmatrix} -\varphi^{-1}\kappa \\ \text{id}_D \end{pmatrix}", harpoon, shift left]
            \& \& E \arrow[uu, "\begin{pmatrix} 0 \\ \text{id}_E \end{pmatrix}", harpoon, shift left] \arrow[rr, "\psi"']
                \& \& F \arrow[uu, "\text{id}_F", harpoon, shift left]
\end{tikzcd}
\end{center}

\end{proposition}

\subsection{Graded Categories} \label{sec: gr_cat}
Let $\Gamma$ be an abelian group, and let $\mathcal{C}$ be a dg category. We say that $\mathcal{C}$ has a $\Gamma$-\textit{action} (or that $\mathcal{C}$ is $\Gamma$-\textit{graded}) if \textcolor{revisions}{for each $\gamma \in \Gamma$,} there exist dg functors \textcolor{revisions}{$\Sigma_{\gamma}: \mathcal{C} \rightarrow \mathcal{C}$ satisfying $\Sigma_{e} = \text{Id}_{\mathcal{C}}$, $\Sigma_{\gamma + \gamma'} = \Sigma_{\gamma} \circ \Sigma_{\gamma'}$}. In this setting, given $X, Y \in \mathcal{C}$, we define \textit{graded} hom spaces $\text{Hom}^{\Gamma}_{\mathcal{C}}(X, Y)$ by

$$
\text{Hom}^{\gamma}_{\mathcal{C}}(X, Y) := \text{Hom}_{\mathcal{C}}(\Sigma_{\gamma}X, Y)
$$

\begin{example} \label{ex: hom_shift}
The category Ch($\mathcal{A}$) is $\mathbb{Z}$-graded via the \textit{homological shift} (or \textit{suspension}) functors $\Sigma_j = t^j$ defined \textcolor{revisions}{on sequences by $(t^jX)^k = X^{j + k}$ and on differentials by $d_{t^j X} = (-1)^j d_X$. This sign convention ensures that $t^j$ is a dg functor and commutes with tensor products, in the sense that $(t^j X) \otimes Y \cong X \otimes (t^j Y) \cong t^j (X \otimes Y)$}.
\end{example}

\begin{example}
If $R$ is a $\Gamma$-graded ring, then we may consider the category $R$-gMod, with objects graded $R$-modules and morphisms all homogeneous (not necessarily degree-preserving!) graded module homomorphisms, as a dg category with morphism spaces concentrated in homological degree $0$ and trivial differential. Then $R$-gMod has a $\Gamma$-action given by formal degree shift of graded modules.
\end{example}

\begin{example}
If $R$ is a $\Gamma$-graded ring as above, then let $\text{Ch}(R-\text{gMod})$ denote the category of chain complexes over $R-\text{gMod}$ with degree-preserving differentials. The $\Gamma$-action on $R$-gMod extends to Ch($R$-gMod) and commutes with the $\mathbb{Z}$-action discussed \textcolor{revisions}{in Example \ref{ex: hom_shift}}, giving rise to a $\Gamma \times \mathbb{Z}$-action \textcolor{revisions}{on Ch($R$-gMod)}. Given $X, Y \in \text{Ch}(R-\text{gMod})$, we have a \textit{bigraded} morphism space defined component-wise as follows:

$$
\text{Hom}^{\gamma, k}_{\text{Ch}(R-\text{gMod})}(X, Y) := \text{Hom}^k_{\text{Ch}(R-\text{gMod})} (\Sigma_{\gamma}X, Y)
$$
\end{example}

By convention, when morphism spaces in categories of chain complexes (and their quotients, and the bounded variations of these) come equipped with additional gradings, we restrict notation with a single superscript to indicate homological grading as on the right side of the above definition. We also employ this convention with no subscript when the category in which morphisms are taken is clear.

We say a homogeneous morphism $f \in \text{Hom}^{\gamma, k}_{\text{Ch}(R-\text{gMod})}(X, Y)$ has degree $(\gamma, k)$. Notice that it is the \textit{homological} degree, not the $\Gamma$-degree, that determines signs in all definitions above, and the differential $d_{\text{Ch}(R-\text{gMod})}$ preserves $\Gamma$-degree of morphisms. This action also descends to a $\Gamma \times \mathbb{Z}$-action on $C(R-\text{gMod})$ and $K(R-\text{gMod})$ and all the bounded variations thereof.

\begin{remark}
Strictly speaking, extending the $\Gamma$-action on $\text{Ch}(R-\text{gMod})$ to its quotients $C(R-\text{gMod})$ and $K(R-\text{gMod})$ requires extending the morphism spaces of the latter categories to include homogeneous morphisms of all degrees. We make this extension implicitly throughout, as we will only ever be interested in these graded categories.
\end{remark}

\subsection{Twisted Complexes and Convolutions} \label{sec: twist}

In this section we offer a brief introduction to twisted complexes and homological perturbation theory. Our exposition largely follows that of \cite{GH22} and \cite{Hog20}; the reader familiar with that material can safely skip this section.

\begin{definition}
Let $(X, d_X) \in \text{Ch}(\mathcal{A})$ be a chain complex, and let $\alpha \in \text{End}^1(X)$ be given so that $\alpha$ satisfies the \textit{Maurer-Cartan} identity $[d_X, \alpha] + \alpha^2 = 0$. Then $(X, d_X + \alpha)$ is also a chain complex; we denote this complex by $\text{tw}_{\alpha}(X)$. In this case, we call $\text{tw}_{\alpha}(X)$ a \textit{twist} of $X$ and $\alpha$ a \textit{Maurer-Cartan element} (or, by  abuse of notation, a \textit{twist}).
\end{definition}

\begin{definition} \label{def: convolution}
Let $I$ be an indexing poset, and for each $i \in I$, let $(X_i, d_i)$ be a chain complex. Let $X = \bigoplus_{i \in I} (X_i, d_i)$, and let $\alpha \in \text{End}^1(X)$ be a Maurer-Cartan element such that the component $\alpha_{ij}$ of $\alpha$ \textcolor{revisions}{from $X_j$ to $X_i$} is $0$ for $i \leq j$. Then we call $\alpha$ a \textit{one-sided twist} of $X$, and $\text{tw}_{\alpha}(X)$ a (one-sided) \textit{convolution} of the complexes $X_i$ (indexed by $I$).
\end{definition}

\textcolor{revisions}{Definition \ref{def: convolution} recovers many typical constructions in homological algebra as special cases; we detail two examples of particular importance to us below.}
\textcolor{revisions}{
\begin{example} \label{ex: complexes_as_convolutions}
Let $(A, d_A)$ be any chain complex of objects in $\mathcal{A}$ with underlying sequence $A = \bigoplus_{i \in \mathbb{Z}} A^i$. We consider $(A, d_A)$ as a convolution as follows. For each $i \in \mathbb{Z}$, we may consider $A^i$ as a chain complex concentrated in homological degree $0$ with zero differential. Setting $X_i := t^iA^i$, we may consider the direct sum $X := \bigoplus_{i \in \mathbb{Z}} X_i$ in Ch($\mathcal{A}$); this is a chain complex of objects of $\mathcal{A}$ with the same underlying sequence as $A$ but zero differential. Define $\alpha \in \mathrm{End}^1(X)$ by setting $\alpha_{i + 1, i} = (d_A)|_{A^i}$ for each $i \in \mathbb{Z}$ and $\alpha_{ij} = 0$ for all $i \neq j + 1$. Then $\mathrm{tw}_{\alpha}(X) \cong (A, d_A)$ as chain complexes over $\mathcal{A}$.
\end{example}}

\textcolor{revisions}{
\begin{example} \label{ex: cones_as_convolutions}
Let $(A, d_A)$ and $(B, d_B)$ be chain complexes, and suppose $f \colon A \to B$ is a chain map. Set $X_1 := t^{-1}A$, $X_2 = B$, and consider the endomorphism $\alpha \in \mathrm{End}^1(X_1 \oplus X_2)$ with $\alpha_{12} = -f$ and all other components $0$. Then the convolution $\mathrm{tw}_{\alpha}(X_1 \oplus X_2)$ is called the \textit{mapping cone of $f$}, denoted $\mathrm{Cone}(f)$.
\end{example}}

\textcolor{revisions}{
\begin{remark} \label{rem: convolution_drawing}
We will often depict convolutions explicitly by drawing the twist $\alpha$ as an arrow between the constituent complexes. For example, we depict the convolution $(A, d_A) = \mathrm{tw}_{d_A}(\bigoplus_{i \in \mathbb{Z}} t^i A^i)$ of Example \ref{ex: complexes_as_convolutions} as
\end{remark}}

\begin{center}
\tikzcdset{color=revisions, arrows={color=revisions}}
\begin{tikzcd}
(A, d_A) = \dots \arrow[r, "d_A"] & t^{-2} A^{-2} \arrow[r, "d_A"] & t^{-1} A^{-1} \arrow[r, "d_A"] & A^0 \arrow[r, "d_A"] & t A^1 \arrow[r, "d_A"] & t^2 A^2 \arrow[r, "d_A"] & \dots
\end{tikzcd}
\end{center}

\textcolor{revisions}{We point out that this convention specifies the homological degree of each term $A^i$ explicitly using coefficients $t^i$ rather than the more conventional \underline{underlining} of terms in homological degree $0$. Similarly, we depict the convolution $\mathrm{Cone}(f) = \mathrm{tw}_{f}(t^{-1}A \oplus B)$ as}

\begin{center}
\tikzcdset{color=revisions, arrows={color=revisions}}
\begin{tikzcd}
\mathrm{Cone}(f) = t^{-1}A \arrow[r, "f"] & B
\end{tikzcd}
\end{center}

\textcolor{revisions}{Given two homotopy equivalent chain complexes $X \simeq Y$ and a twist $\mathrm{tw}_{\alpha}(X)$ of $X$, it is natural to ask whether there exists a Maurer-Cartan element $\beta$ for $Y$ satisfying $\mathrm{tw}_{\alpha}(X) \simeq \mathrm{tw}_{\beta}(Y)$. This is the subject of \textit{homological perturbation theory}, a detailed investigation of which can be found in \cite{Hog20}. We will not make use of the full strength of this theory, but we will use the following special case.}

\begin{definition}
\textcolor{revisions}{We say a poset $I$ is \textit{upper finite} if for each $i \in I$, there exist finitely many $j \in I$ such that $i < j$. Similarly, we say a poset $I$ is \textit{lower finite} if for each $i \in I$, there exist finitely many $j \in I$ such that $i > j$.}
\end{definition}

\begin{definition}
\textcolor{revisions}{Let $\mathrm{tw}_{\alpha}(X)$ be a convolution of a family $\{X_i\}$ over some indexing set $I$. We say $\mathrm{tw}_{\alpha}(X)$ is \textit{homologically locally finite} if, for all $j \in \mathbb{Z}$, $(X_i)^j = 0$ for all but finitely many $i \in I$.}
\end{definition}

\begin{proposition} \label{prop: useful_hpt}
Let $I$ be a poset\textcolor{revisions}{, and let \textcolor{revisions}{$\{X_i\}, \{Y_i\}$ be families of complexes} indexed by $I$. S}uppose $X_i \simeq Y_i$ for each $i$. Let $\text{tw}_{\alpha}(X)$ be a one-sided convolution of \textcolor{revisions}{the family $\{X_i\}$}. \textcolor{revisions}{If $I$ is upper finite, or if $I$ is lower finite and $\mathrm{tw}_{\alpha}(X)$ is homologically locally finite, t}hen there exists a one-sided convolution $\text{tw}_{\beta}(Y)$ of the \textcolor{revisions}{family} $\{Y_i\}_{i \in I}$ with $\text{tw}_{\alpha}(X) \simeq \text{tw}_{\beta}(Y)$.
\end{proposition}

\begin{proof}
\textcolor{revisions}{See Proposition 4.20 of \cite{EH17a}.}
\end{proof}

\subsection{Soergel Bimodules} \label{sec: SoergBim}

\textcolor{revisions}{In this section we introduce the primary category of interest to us: \textit{Soergel bimodules}. We draw heavily from the conventions and exposition of \cite{HRW21}; the reader familiar with that work can safely skip this section, returning only to reference notation.}

Let $R$ be a commutative ring, and \textcolor{revisions}{denote by $R_n$ the} polynomial ring in the alphabet $\mathbb{X} = \{x_1, x_2, \dots, x_n\}$ over $R$. We endow \textcolor{revisions}{$R_n$ with a $\mathbb{Z}$-grading}, called the \textit{quantum grading}, by declaring each variable to have quantum degree 2. Then the category $R_n$-Bim of graded $R_n$-bimodules is equivalent to the category $(R_n \otimes_R R_n)$-gMod, and hence inherits a $\mathbb{Z}$-grading by the discussion \textcolor{revisions}{of Section \ref{sec: gr_cat}; we} refer to this grading also as the \textit{quantum grading}. We call the functors $\Sigma_j$ \textit{quantum shift} functors and denote them by $q^j$ (so that e.g. $(q^jM)^k = M^{k - j}$).

The symmetric \textcolor{revisions}{group $\mathfrak{S}^n$} acts on $R_n$ by permuting indices of variables in $\mathbb{X}$. \textcolor{revisions}{For each $1 \leq i \leq n - 1$, we denote by $R_n^i$ the subring of $R_n$ which is invariant under the action of the transposition $s_i = (i, i + 1) \in \mathfrak{S}^n$; explicitly, this is the subring of $R_n$ consisting of polynomials which are symmetric in the variables $x_i, x_{i + 1}$. We denote by $B_i$ the graded $R_n$-bimodule}

\[
\textcolor{revisions}{B_i := q^{-1} R_n \otimes_{R_n^i} R_n}
\]

\begin{definition}
Let $SBim_n$ denote the full \textcolor{revisions}{monoidal} subcategory of \textcolor{revisions}{$R_n$-Bim} generated by the objects $B_i$ under \textcolor{revisions}{tensor product $\otimes_{R_n}$}, quantum shift, direct sum, and \textcolor{revisions}{taking direct summands}. We call $SBim_n$ the (type $A_{n - 1}$) \textit{Soergel category} and refer to its objects as \textit{Soergel bimodules}. The collection of monoidal categories $SBim := \bigsqcup_{n \geq 1} SBim_n$ is itself a monoidal category \textcolor{revisions}{under the external tensor product $\otimes_R$; we denote this external product by} $\sqcup$. \textcolor{revisions}{For each $n \geq 1$, there is an inclusion functor $(-) \sqcup R[x_{n + 1}] \colon SBim_n \hookrightarrow SBim_{n + 1}$ taking an $R_n$-bimodule $M$ to the $R_{n + 1}$-bimodule $M \otimes_R R[x_{n + 1}]$; we make use of this inclusion throughout without further comment.}
\end{definition}

\textcolor{revisions}{When the index $n$ is clear, we will often denote the tensor product $M \otimes_{R_n} N \in SBim_n$ of two Soergel bimodules $M$ and $N$ just by $MN$.} Bimodules of the form $q^j B_{i_1}B_{i_2} \dots B_{i_k}$ for some $1 \leq i_1, i_2, \dots, i_k \leq n - 1$ are called \textit{Bott-Samelson bimodules}.

There is an isomorphism $K_0(SBim_n) \cong H_n$ between the Grothendieck group of $SBim_n$ and the (type $A_{n - 1}$) Hecke algebra $H_n$ sending (isomorphism classes of) indecomposable objects in $SBim_n$ to elements of the Kazhdan-Lusztig basis of $H_n$ (c.f. \cite{EW14}). The Hecke algebra $H_n$ is generated as an $R[q, q^{-1}]$-module by a collection of objects $b_{\sigma}$ indexed by permutations\footnote{More generally (i.e. in other types), the Hecke algebra is generated by elements of the corresponding Weyl group.} $\sigma \in \mathfrak{S}^n$; it follows that indecomposable Soergel bimodules are indexed up to shift by permutations as well. \textcolor{revisions}{We denote by} $B_{\sigma}$ the indecomposable Soergel bimodule corresponding to the permutation $\sigma$ \textcolor{revisions}{in this indexing scheme}.

\textcolor{revisions}{A few} indecomposable Soergel bimodules \textcolor{revisions}{$B_{\sigma}$} will be of particular importance to us. Unsurprisingly, we have $B_e = R_n$ \textcolor{revisions}{and} $B_{s_i} = B_i$. Let $w_0 \in \mathfrak{S}^n$ denote the longest word of that group \textcolor{revisions}{and} $\ell(n)$ its length. \textcolor{revisions}{We include $\mathfrak{S}^{n - 1}$ into $\mathfrak{S}^n$ in the standard way by sending generators $s_i = (i, i + 1) \in \mathfrak{S}^{n - 1}$ to the corresponding generators $s_i \in \mathfrak{S}^n$ (that is, we include $\mathfrak{S}^{n - 1}$ as permutations of $\{1, \dots, n\}$ which fix $n$). We denote by $w_1$} the longest word of $\mathfrak{S}^{n - 1}$ considered as an element of $\mathfrak{S}^n$ \textcolor{revisions}{in this way}. Then we have:

$$
B_{w_1} = q^{-\ell(n - 1)} R_n \otimes_{R_n^{\mathfrak{S}^{n - 1}}} R_n; \quad B_{w_0} = q^{-\ell(n)} R_n \otimes_{R_n^{\mathfrak{S}^n}} R_n
$$

\textcolor{revisions}{We will often consider} $R_n$-bimodules as left $R_n \otimes_R R_n$-modules. \textcolor{revisions}{When we do this, we typically} identify the latter ring with the polynomial ring $R[\mathbb{X}, \mathbb{X}'] = R[x_1, x_2, \dots, x_n, x_1', x_2', \dots, x_n']$ over two \textcolor{revisions}{disjoint alphabets $\mathbb{X}, \mathbb{X}'$ of size $n$} via the isomorphism $x_1 \otimes 1 \mapsto x_i$, $1 \otimes x_i \mapsto x_i'$. \textcolor{revisions}{Under this identification, we can describe the bimodules $R_n$, $B_i$, $B_{w_1}$, and $B_{w_0}$ explicitly as quotients of $R[\mathbb{X}, \mathbb{X}']$}:

\begin{gather}
	R_n \cong R[\mathbb{X}, \mathbb{X}'] / (x_1 - x_1', x_2 - x_2', \dots, x_n - x_n'); \label{eq: R_n} \\
	B_i \cong q^{-1} R[\mathbb{X}, \mathbb{X}'] / (x_i + x_{i + 1} - x_i' - x_{i + 1}', x_ix_{i + 1} - x_i'x_{i + 1}'); \label{eq: B_i}\\
    B_{w_1} \cong \textcolor{revisions}{q^{-\ell(n - 1)}} R[\mathbb{X}, \mathbb{X}']/(e_1(\mathbb{X} - \{x_n\}) - e_1(\mathbb{X}' - \{x_n'\}), \dots, e_{n - 1}(\mathbb{X} - \{x_n\}) - e_{n - 1}(\mathbb{X}' - \{x_n'\}), \textcolor{revisions}{x_n - x'_n}); \label{eq: B_w1}\\
    B_{w_0} \cong \textcolor{revisions}{q^{-\ell(n)}} R[\mathbb{X},\mathbb{X}']/(e_1(\mathbb{X}) - e_1(\mathbb{X}'), \dots, e_n(\mathbb{X}) - e_n(\mathbb{X}')) \label{eq: B_w0}
\end{gather}

\textcolor{revisions}{Here we have denoted by $e_i(\mathbb{X})$ the $i^{th}$ elementary symmetric polynomial in the alphabet $\mathbb{X}$.} 

The \textcolor{revisions}{bimodule $B_{w_0}$ tends to absorb other bimodules up to direct sums and grading shifts. To describe this phenomenon, it is convenient to introduce the the \textit{quantum integers} and \textit{quantum factorial}}

\[
[j] := q^{j - 1} + q^{j - 3} + \dots + q^{-j + 3} + q^{-j + 1}; \quad [j]! := [j] [j - 1] \dots [2] [1]
\]

for all $j \in \mathbb{Z}_{\geq 1}$. Given a Laurent polynomial $f \in \mathbb{N}[q, q^{-1}]$ and a \textcolor{revisions}{Soergel bimodule $M$, we denote by $fM$} a direct sum of quantum shifts of $M$ indicated by the corresponding coefficients of $f$\textcolor{revisions}{. For example, we have $3[2]M = (3q + 3q^{-1})M := qM \oplus qM \oplus qM \oplus q^{-1}M \oplus q^{-1}M \oplus q^{-1}M$.}

\begin{proposition} \label{prop: tensid}
For each $\sigma \in \mathfrak{S}^n$, $B_\sigma B_{w_0}$ and $B_{w_0}B_{\sigma}$ are direct sums of quantum shifts of $B_{w_0}$. In particular, we have $B_{w_0}B_{w_0} \cong [n]!B_{w_0}$ and $B_{w_1}B_{w_0} \cong B_{w_0}B_{w_1} \cong [n - 1]! B_{w_0}$.
\end{proposition}

\begin{proof}
\textcolor{revisions}{See e.g. Proposition 2.5 of \cite{AH17}}.
\end{proof}

\textcolor{revisions}{For each $1 \leq i \leq n - 1$, there are distinguished degree $q$ morphisms}

\[
unzip \colon B_i \to R_n; \quad zip \colon R_n \to B_i
\]

\textcolor{revisions}{We can succinctly describe these morphisms using the polynomial realizations of these bimodules in \eqref{eq: R_n}, \eqref{eq: B_i}. Each is uniquely defined by specifying the image of $1$ and extending $R[\mathbb{X}, \mathbb{X}']$-linearly. Explicitly, these images are as follows:}

\begin{align*}
    unzip & : \ 1 \mapsto 1 \\
    zip & : \ 1 \mapsto x_i - x_{i + 1}'.
\end{align*}

It is an easy exercise to show that each of these maps is a well-defined homomorphism of $R_n$-bimodules. \textcolor{revisions}{We may think of $unzip$ as the quotient map obtained by killing $x_{i + 1} - x_{i + 1}' \in B_i$. Similarly, the quotient map obtained by killing $x_n - x_n'$ in $B_{w_0}$ is a degree $q^{n - 1}$ bimodule homomorphism}

\[
\textcolor{revisions}{unzip \colon B_{w_0} \to B_{w_1}.}
\]

We will often be interested in computing the (graded) dimension of morphism spaces \textcolor{revisions}{$\text{Hom}_{SBim}(M, N)$ between two bimodules $M, N$}. We are aided in this task by two contravariant functors

$$
^{\vee}(-), \ (-)^{\vee} \colon \textcolor{revisions}{SBim_n \to SBim_n}
$$

defined by

$$
\textcolor{revisions}{^{\vee}M := \text{Hom}_{R_n-\text{Mod}} (R_n, M); \quad M^{\vee} := \text{Hom}_{\text{Mod}-R_n} (M, R_n)}
$$

and satisfying natural adjunctions

\[
\textcolor{revisions}{\text{Hom}_{SBim_n}(XY, Z) \cong \text{Hom}_{SBim_n}(Y, \ ^{\vee}XZ); \quad \text{Hom}_{SBim_n}(XY, Z) \cong \text{Hom}_{SBim_n}(X, ZY^{\vee})}.
\]

Note that $^{\vee}(q^a M) \cong q^a (^{\vee} M)$ and $(q^a M)^{\vee} \cong q^{-a} (M^{\vee})$ for all $a \in \mathbb{Z}$ and all Soergel bimodules $M$.

\begin{proposition} \label{prop: duals}
\textcolor{revisions}{Each of the bimodules $R_n, B_i, B_{w_1}, B_{w_0} \in SBim_n$ is self-dual under the operations $^{\vee}(-), (-)^{\vee}$.}
\end{proposition}

\begin{proof}
\textcolor{revisions}{See e.g. \cite{HRW21}, Proposition 3.19.}
\end{proof}

There is a useful graphical calculus for Soergel bimodules\footnote{\textcolor{revisions}{This calculus is better adapted to the $2$-category of \textit{singular} Soergel bimodules, of which Soergel bimodules form a distinguished endomorphism category. The diagrammatics we outline here can be considered as a special case of that broader calculus. We employ the full singular calculus in Appendices \ref{app: ssbim} and \ref{app: dot_slide_nat}; see the discussion there for more detail on this point.}} analogous to that of Type A webs developed by Cautis-Kamnitzer-Morrison in \cite{CKM14}. \textcolor{revisions}{This calculus depicts type $A_{n - 1}$ Soergel bimodules as certain graphs embedded in the the unit square $[0, 1] \times [0,1]$ with boundary consisting of $n$ endpoints in each of the intervals $[0, 1] \times \{0\}$ and $[0, 1] \times \{1\}$. We denote the bimodules $R_n$, $B_i$, $B_{w_1}$, $B_{w_0} \in SBim_n$ in this scheme as follows:}

\begin{gather*}
R_n = 
\begin{tikzpicture}[baseline=(current bounding box.center),scale=.5,tinynodes]
	\draw[color=white] (-.5,-.9) to (.5,-.9);
    \draw[webs] (-.5,0) -- (-.5,1.5);
    \draw[webs] (.5,0) -- (.5,1.5);
    \node at (0, .75) {$\dots$};
    \draw [decorate,decoration = {brace}] (-.5,1.6) -- (.5,1.6) node[pos=.5,above]{$n$};
\end{tikzpicture}
\quad \quad 
B_i = 
\begin{tikzpicture}[baseline=(current bounding box.center),scale=.5,tinynodes]
	\draw[color=white] (-.5,-.9) to (.5,-.9);
    \draw[webs] (-1.5,0) -- (-1.5,1.5);
    \node at (-1,.75) {$\dots$};
    \draw[webs] (-.5,0) -- (-.5,1.5);
    \draw [decorate,decoration = {brace}] (-1.5,1.6) -- (-.5,1.6) node[pos=.5,above]{$i - 1$};
    \draw[webs] (0,0) to[out=90,in=180] (.5,.5);
    \draw[webs] (1,0) to[out=90,in=0] (.5,.5);
    \draw[webs] (.5,.5) -- (.5,1);
    \draw[webs] (.5,1) to[out=180,in=270] (0,1.5);
    \draw[webs] (.5,1) to[out=0,in=270] (1,1.5);
    \draw[webs] (1.5,0) -- (1.5,1.5);
    \node at (2,.75) {$\dots$};
    \draw[webs] (2.5,0) -- (2.5,1.5);
    \draw [decorate,decoration = {brace}] (1.5,1.6) -- (2.5,1.6) node[pos=.5,above]{$n - (i + 1)$}; 
\end{tikzpicture}
\quad \quad
B_{w_1} = 
\begin{tikzpicture}[anchorbase,scale=.5,tinynodes]
	\draw[webs] (-.5,0) to[out=90,in=180] (.5,.5);
	\draw[webs] (0,0) to[out=90,in=225] (.5,.5);
	\node at (.5,.1) {$\dots$};
	\draw[webs] (1,0) to[out=90,in=315] (.5,.5);
	\draw[webs] (1.5,0) to[out=90,in=0] (.5,.5);
	\draw [decorate,decoration = {brace,mirror}] (-.5,-.1) -- (1.5,-.1) node[pos=.5,below]{$n - 1$};
	\draw[webs] (.5,.5) -- (.5,1);
	\draw[webs] (.5,1) to[out=180,in=270] (-.5,1.5);
	\draw[webs] (.5,1) to[out=135,in=270] (0,1.5);
	\node at (.5,1.4) {$\dots$};
	\draw[webs] (.5,1) to[out=45,in=270] (1,1.5);
	\draw[webs] (.5,1) to[out=0,in=270] (1.5,1.5);
	\draw [decorate,decoration = {brace}] (-.5,1.6) -- (1.5,1.6) node[pos=.5,above]{$n - 1$};
	\draw[webs] (2,0) -- (2,1.5);
\end{tikzpicture}
\quad \quad
B_{w_0} = 
\begin{tikzpicture}[anchorbase,scale=.5,tinynodes]
	\draw[webs] (-.5,0) to[out=90,in=180] (.5,.5);
	\draw[webs] (0,0) to[out=90,in=225] (.5,.5);
	\node at (.5,.1) {$\dots$};
	\draw[webs] (1,0) to[out=90,in=315] (.5,.5);
	\draw[webs] (1.5,0) to[out=90,in=0] (.5,.5);
	\draw [decorate,decoration = {brace,mirror}] (-.5,-.1) -- (1.5,-.1) node[pos=.5,below]{$n$};
	\draw[webs] (.5,.5) -- (.5,1);
	\draw[webs] (.5,1) to[out=180,in=270] (-.5,1.5);
	\draw[webs] (.5,1) to[out=135,in=270] (0,1.5);
	\node at (.5,1.4) {$\dots$};
	\draw[webs] (.5,1) to[out=45,in=270] (1,1.5);
	\draw[webs] (.5,1) to[out=0,in=270] (1.5,1.5);
	\draw [decorate,decoration = {brace}] (-.5,1.6) -- (1.5,1.6) node[pos=.5,above]{$n$};
\end{tikzpicture}
\end{gather*}

We depict external tensor product $\sqcup$ in \textcolor{revisions}{$SBim$} by horizontal concatenation of such diagrams and \textcolor{revisions}{tensor product $\otimes_{R_n}$ of two bimodules in $SBim_n$ by vertical concatenation, read from bottom to top and right to left}. More complicated diagrams are built from the above by these operations. For example, we may depict the Bott-Samelson bimodule $B_2B_1 \in SBim_3$ as

\begin{gather*}
B_1B_2B_1 = 
\begin{tikzpicture}[anchorbase,scale=.5,tinynodes]
\draw[webs] (0,0) to[out=90,in=180] (.5,.5);
\draw[webs] (1,0) to[out=90,in=0] (.5,.5);
\draw[webs] (.5,.5) to (.5,1);
\draw[webs] (.5,1) to[out=180,in=270] (0,1.5);
\draw[webs] (.5,1) to[out=0,in=270] (1,1.5);
\draw[webs] (2,0) to (2,1.5);
\draw[webs] (1,1.5) to[out=90,in=180] (1.5,2);
\draw[webs] (2,1.5) to[out=90,in=0] (1.5,2);
\draw[webs] (1.5,2) to (1.5,2.5);
\draw[webs] (1.5,2.5) to[out=180,in=270] (1,3);
\draw[webs] (1.5,2.5) to[out=0,in=270] (2,3);
\draw[webs] (0,1.5) to (0,3);
\end{tikzpicture}
\end{gather*}

\begin{remark} \label{rem: chain_graphs}
We will routinely extend the graphical calculus outlined above for \textcolor{revisions}{$SBim_n$} to various categories of chain complexes over \textcolor{revisions}{$SBim_n$}. The duality functors $^{\vee}(-)$, $(-)^{\vee}$ also extend to these categories of chain complexes by termwise application to chain bimodules and enjoy analogous adjunction relations and behavior under formal homological shifts.
\end{remark}

\subsection{Rouquier Complexes} \label{sec: rouq}

Let $Br_n$ denote the (type $A$) braid group on $n$ strands. We refer to the generators $\sigma_i, 1 \leq i \leq n - 1$ of $Br_n$ as \textit{positive crossings} and their inverses $\sigma_i^{-1}$ as \textit{negative crossings}. There is a homomorphism from $Br_n$ to $\mathfrak{S}^n$ taking $\sigma_i^{\pm}$ to $s_i$; given $\beta \in Br_n$, we call its image under this homomorphism the \textit{associated permutation} to $\beta$ \textcolor{revisions}{and denote this permutation by $\sigma_{\beta}$}. We say $\beta$ is a \textit{pure braid} if its associated permutation \textcolor{revisions}{$\sigma_{\beta}$} is trivial.

We employ a graphical calculus for $Br_n$ consisting of $n$ visibly braided parallel strands. Positive and negative crossings are depicted as follows:

\begin{gather*}
\sigma_i =
\begin{tikzpicture}[anchorbase,scale=.5,tinynodes]
    \draw[color=white] (0,-1.9) to (1,-1.9);
    \draw[webs] (-1.5,-1) to (-1.5,1);
    \node at (-1,0) {$\dots$};
    \draw[webs] (-.5,-1) to (-.5,1);
    \draw [decorate,decoration = {brace}] (-1.5,1.1) -- (-.5,1.1) node[pos=.5,above]{$i - 1$};
	\draw[webs] (1,-1) to [out=90,in=270] (0,1);
	\draw[line width=5pt,color=white] (0,-1) to [out=90,in=270] (1,1);
	\draw[webs] (0,-1) to [out=90,in=270] (1,1);
	\draw[webs] (1.5,-1) to (1.5,1);
    \node at (2,0) {$\dots$};
    \draw[webs] (2.5,-1) to (2.5,1);
    \draw [decorate,decoration = {brace}] (1.5,1.1) -- (2.5,1.1) node[pos=.5,above]{$n - (i + 1)$};
\end{tikzpicture}
;
\quad \quad
\sigma_i^{-1} =
\begin{tikzpicture}[anchorbase,scale=.5,tinynodes]
    \draw[color=white] (0,-1.9) to (1,-1.9);
    \draw[webs] (-1.5,-1) to (-1.5,1);
    \node at (-1,0) {$\dots$};
    \draw[webs] (-.5,-1) to (-.5,1);
    \draw [decorate,decoration = {brace}] (-1.5,1.1) -- (-.5,1.1) node[pos=.5,above]{$i - 1$};
    \draw[webs] (0,-1) to [out=90,in=270] (1,1);
	\draw[line width=5pt,color=white] (1,-1) to [out=90,in=270] (0,1);
	\draw[webs] (1,-1) to [out=90,in=270] (0,1);
	\draw[webs] (1.5,-1) to (1.5,1);
    \node at (2,0) {$\dots$};
    \draw[webs] (2.5,-1) to (2.5,1);
    \draw [decorate,decoration = {brace}] (1.5,1.1) -- (2.5,1.1) node[pos=.5,above]{$n - (i + 1)$};
\end{tikzpicture}
\end{gather*}

We depict multiplication of braid words (read from right to left) as vertical composition (read from bottom to top)\footnote{This somewhat dyslexic convention is standard and ensures good behavior under composition when considering tensoring with Rouquier complexes as endofunctors on $K(SBim_n)$.}; this is consistent with the above convention for interpreting webs. Given a braid word $\beta$, we assign a complex $F(\beta)$ of Soergel bimodules, called the \textit{Rouquier complex}, defined by

\begin{center}
\begin{tikzcd} \label{eq: rouquier}
F(\sigma_i) &[-3em] := &[-3em] B_i \arrow[r, "unzip"] & q^{-1}tR_n; \quad \quad
F(\sigma_i^{-1}) &[-3em] := &[-3em] qt^{-1}R_n \arrow[r, "zip"] & B_i
\end{tikzcd}
\end{center}

together with the rule $F(\beta\beta') = F(\beta)F(\beta')$. Note that we have made all homological shifts of objects of these complexes explicit by depicting them as convolutions, and the quantum shifts above ensure that all differentials are of quantum degree $0$. It is a theorem of Rouquier \cite{Rou04} that these complexes respect the braid relations up to homotopy equivalence; as a consequence, we have a well-defined homomorphism $F: Br_n \rightarrow K^b(SBim_n)$.

Given a braid $\beta \in B_n$, we will often depict $F(\beta)$ diagramatically according to the graphical calculus for $Br_n$ described above. We consider this as an extension of the graphical calculus for $K^b(SBim)$. In the special case that a complex $C \in K^b(SBim)$ is depicted by a single diagram (i.e. \textit{not} explicitly as a convolution), we call the corresponding diagram a \textit{braided web}. For example, the diagram on the left below is a braided web, while the diagram on the right is not.

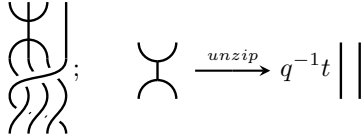
\begin{gather*}
\begin{tikzpicture}[anchorbase,scale=.5,tinynodes]
    \draw[webs] (0,0) to[out=90,in=180] (.5,.5);
    \draw[webs] (.5,0) to (.5,.5);
    \draw[webs] (1,0) to[out=90,in=0] (.5,.5);
    \draw[webs] (.5,.5) to (.5,1);
    \draw[webs] (.5,1) to[out=180,in=270] (0,1.5);
    \draw[webs] (.5,1) to (.5,1.5);
    \draw[webs] (.5,1) to[out=0,in=270] (1,1.5);
    \draw[webs] (1.5,0) to (1.5,1.5);
    \draw[webs] (0,0) to[out=270,in=90] (.5,-1);
    \draw[webs] (.5,0) to[out=270,in=90] (1,-1); 
    \draw[webs] (1,0) to[out=270,in=90] (1.5,-1);
    \draw[line width=5pt,color=white] (1.5,0) to[out=270,in=90] (0,-1);
    \draw[webs] (1.5,0) to[out=270,in=90] (0,-1);
    \draw[webs] (0,-1) to[out=270,in=90] (1.5,-2);
    \draw[line width=5pt,color=white] (.5,-1) to[out=270,in=90] (0,-2);
    \draw[webs] (.5,-1) to[out=270,in=90] (0,-2);
    \draw[line width=5pt,color=white] (1,-1) to[out=270,in=90] (.5,-2);
    \draw[webs] (1,-1) to[out=270,in=90] (.5,-2);
    \draw[line width=5pt,color=white] (1.5,-1) to[out=270,in=90] (1,-2);
    \draw[webs] (1.5,-1) to[out=270,in=90] (1,-2);
\end{tikzpicture}
;\quad \quad
\begin{tikzpicture}[anchorbase,scale=.5,tinynodes]
    \draw[webs] (0,0) to[out=90,in=180] (.5,.5);
    \draw[webs] (1,0) to[out=90,in=0] (.5,.5);
    \draw[webs] (.5,.5) to (.5,1);
    \draw[webs] (.5,1) to[out=180,in=270] (0,1.5);
    \draw[webs] (.5,1) to[out=0,in=270] (1,1.5);
    \draw[thick] (1.5,.75) to (2.5,.75) node[above]{$unzip$};
    \draw[thick,->] (2.5,.75) to (3.5,.75);
\end{tikzpicture}
\ q^{-1}t \ 
\begin{tikzpicture}[anchorbase,scale=.5,tinynodes]
\draw[webs] (0,0) to (0,1.5);
\draw[webs] (.5,0) to (.5,1.5);
\end{tikzpicture}
\end{gather*}

Let $C \in K^b(SBim)$ be given, and let $\mathcal{D}$ be a braided web representing $C$. \textcolor{revisions}{At each vertical slice of $\mathcal{D}$ corresponding to a tensor product $\otimes_{R_n}$, there is a} collection of well-defined central endomorphisms of $C$ given by the action of \textcolor{revisions}{$R_n$ in that position.} We will often describe this action on extremal segments as multiplication by polynomials in $R[\mathbb{X}, \mathbb{X}']$ or $R[\mathbb{X}, \mathbb{X}'']$ depending on the context; note that this agrees with the pointwise action of $R[\mathbb{X}, \mathbb{X}']$ on chain $R_n$-bimodules. The following well-known proposition ensures that this action is reasonably well-behaved with respect to crossings:

\begin{proposition}[Dot Sliding] \label{prop: dot_slide}
For each $1 \leq i \leq n - 1$, the actions of $x_i$ and $x_{i + 1}'$ are homotopic as endomorphisms of $F(\sigma_i)$ and of $F(\sigma_i')$.
\end{proposition}

\begin{proof}
We include the proof primarily as an excuse to introduce notation. We wish to identify some homotopies $h_i^+ \in \text{End}^{-1}_{K^b(SBim)}(F(\sigma_i))$, $h_i^- \in \text{End}^{-1}_{K^b(SBim)}(F(\sigma_i^{-1}))$ such that $[d, h_i^{\pm}] = x_i - \textcolor{revisions}{x_{i + 1}'}$. Mirroring our notational convention for $d$, we will denote $h_i^{\pm}$ by \textit{backwards} arrows between chain groups labelled by quantum degree $q^2$ bimodule maps. Explicitly, we have:

\begin{center}
\begin{tikzcd}
B_i \arrow[rr, shift left, "unzip"] & & q^{-1}tR_n \arrow[ll, shift left, "h_i^+ = zip"]
\end{tikzcd}
; \quad \quad
\begin{tikzcd}
q^{-1}tR_n \arrow[rr, shift left, "zip"] & & B_i \arrow[ll, shift left, "h_i^- = unzip"]
\end{tikzcd}
\end{center}
\end{proof}

Notice that $f(\mathbb{X}) = f(\mathbb{X}') \in \text{End}(F(\sigma_i))$ for any $f \in R_n^i$. In particular, we have

\[
\textcolor{revisions}{x_{i + 1} - x_i' = (x_{i + 1} + x_i) - x_i - x_i' = (x_{i + 1}' + x_i') - x_i - x_i' = -(x_i - x_{i + 1}')} \in \text{End}(F(\sigma_i)),
\]

so \textcolor{revisions}{$[d, -h_i^{\pm}] = x_{i + 1} - x_i'$}. We refer to $h_i$ in the proof above and its counterpart for the other strand as \textit{dot-sliding} homotopies. By summing these dot-sliding homotopies over various crossings, we obtain the following:

\begin{proposition}
Let $\beta \in Br_n$ be given. Then for each $1 \leq j \leq n$, we have $x_{\beta(j)} \sim x_{j}' \in \text{End}^0(F(\beta))$. Let $h_j \in \text{End}^{-1}(F(\beta))$ be such that $[d, h_j] = x_{\beta(j)} - x_j'$; then the collection $\{h_j\}_{j = 1}^n$ square to $0$ and pairwise anticommute. 
\end{proposition}

Since $SBim_n$ is a subcategory of the abelian category $R_n-\text{Bim}$, we can consider the cohomology $H^{\bullet}(F(\beta))$ of Rouquier complexes in the latter (note that this is independent of a choice of braid word representing $\beta$, as Rouquier complexes are well-defined up to homotopy equivalence). We conclude this section by quoting that result for positive \textcolor{revisions}{and negative} braids; see Chapter 19 of \cite{EMTW20} for further details.

\begin{definition}
For each $\sigma \in \mathfrak{S}^n$, let $R[\mathbb{X}]_{\sigma}$ denote the $R_n$-bimodule which equals $R[\mathbb{X}]$ as a left $R[\mathbb{X}]$-module and whose right $R[\mathbb{X}]$-module action is twisted by $\sigma$. In other words, viewing $R[\mathbb{X}]_{\sigma}$ as a $R[\mathbb{X}, \mathbb{X}']$-module, we have $x_i \cdot f(\mathbb{X}) \cdot x_j' = x_ix_{\sigma(j)} f(\mathbb{X})$ for all $f(\mathbb{X}) \in R_n$ and all $i, j$.
\end{definition}

\begin{proposition} \label{prop: rouqhom}
\textcolor{revisions}{For each braid $\beta \in Br_n$,} let $e(\beta)$ denote \textcolor{revisions}{its} braid exponent (i.e. the number of positive crossings minus the number of negative crossings in a fixed expression for $\beta$). \textcolor{revisions}{If $\beta$ is a positive braid, t}hen there is a quasi-isomorphism $\psi_{\beta} \colon q^{e(\beta)}R[\mathbb{X}]_{\textcolor{revisions}{\sigma_{\beta}}} \to F(\beta)$ in $Ch^b(R_n-\text{Bim})$. \textcolor{revisions}{Similarly, if $\beta$ is a negative braid, then there is a quasi-isomorphism $\phi_{\beta} \colon F(\beta) \to q^{e(\beta)} R[\mathbb{X}]_{\sigma_{\beta}}$.} In particular, \textcolor{revisions}{in either case,} we have $H^{\bullet}(F(\beta)) \cong q^{e(\beta)}R[\mathbb{X}]_{\textcolor{revisions}{\sigma_{\beta}}}$, concentrated in degree $0$, as graded $R_n$-bimodules. 
\end{proposition}

\subsection{Hochschild Cohomology} \label{section: HoCh}

For each $n$, let $\mathcal{D}_n := D^b(R_n-\text{Bim})$ denote the bounded derived category of graded $R_n$-bimodules. (We set $\mathcal{D}_0 = D^b(R-\text{Bim})$ for notational convenience.) Since quasi-isomorphisms in $\text{Ch}(R_n-\text{Bim})$ preserve quantum degree, the $\mathbb{Z} \times \mathbb{Z}$- grading on this category naturally extends to a grading on $\mathcal{D}_n$. We refer to the derived homological grading in this case as the \textit{Hochschild} grading and denote degree shifts in this grading by $a^j$. As above, given $X, Y \in \mathcal{D}_n$, we have bigraded morphism spaces:

$$
\text{Hom}^{i, j}_{\mathcal{D}_n} (X, Y) := \text{Hom}^j_{\mathcal{D}_n} (q^iX, Y)
$$

Note that $i$ and $j$ track quantum degree and Hochschild degree above, respectively. As usual, we will often indicate the degree of a homogeneous morphism $f \in \text{Hom}^{i, j}_{\mathcal{D}_n} (X, Y)$ multiplicatively as $\text{deg}(f) = a^jq^i$ to avoid ambiguity in ordering of degrees.

Again $\mathcal{D}_n$ is a monoidal category via the usual derived tensor product, which we again denote by $\otimes$. There is a fully faithful inclusion $R_n$-Bim $\hookrightarrow \mathcal{D}_n$ of monoidal categories given by inclusion in homological degree $0$; from now on, we will suppress this inclusion and regard graded $R_n$-bimodules as objects of $\mathcal{D}_n$ without further comment. 

\begin{definition}

Given $M \in R_n$-Bim, let $HH^{\bullet, \bullet}(M)$ denote the $\mathbb{Z} \times \mathbb{Z}$-graded $R$-module given in each degree by $HH^{i, j}(M) := \text{Hom}^{i, j}_{\mathcal{D}_n} (R_n, M)$. Notice that $HH^{i, j}(M) \cong \text{Ext}^j (q^iR_n, M) \cong q^i\text{Ext}^j(R_n, q^{-i}M)$ is (a shift of) the classical graded $j^{th}$ Hochschild cohomology of $R_n$ with coefficients in $q^{-i}M$; for this reason, we refer to the functor $HH^{\bullet, \bullet}$ as \textit{Hochschild cohomology}. Since we deal exclusively with Hochschild \textit{co}homology in this paper, we will often drop the superscript bullets and denote this functor simply by $HH$.

\end{definition}

Now, for each $n$, let $\mathcal{C}_n := K(\mathcal{D}_n)$ denote the homotopy category of chain complexes on $\mathcal{D}_n$. (As usual, we let $\mathcal{C}_n^{b, +, -}$ denote the corresponding categories of bounded complexes). The inclusion above extends to a fully faithful monoidal functor $K(R_n-Bim) \hookrightarrow \mathcal{C}_n$ given by termwise inclusion of chain bimodules into $\mathcal{D}_n$. Morphism spaces in this category are $\mathbb{Z}^3$-graded with quantum, Hochschild, and homological grading. Explicitly, we have:

$$
\text{Hom}^{i, j, k}_{\mathcal{C}_n}(X, Y) := \text{Hom}^k_{\mathcal{C}_n}(a^jq^iX, Y)
$$

where $i, j, k$ denote quantum, Hochschild, and homological gradings, respectively. Again, we will indicate the degree of a homogeneous element $f \in \text{Hom}_{\mathcal{C}_n}^{i, j, k}(X, Y)$ mutliplicatively as $\text{deg}(f) = a^jq^it^k$.

Given a complex $C \in K(R_n-\text{Bim})$ of $R_n$-bimodules, let $HH(C)$ denote the complex of $q,a$-graded $R$-modules obtained by termwise application of Hochschild cohomology. We denote the (triply-graded!) cohomology of the resulting complex by $HHH(C) := H^{\bullet}(HH(C))$. Unraveling definitions, we see that $HHH(C) \cong \text{Hom}_{\mathcal{C}_n}(R_n, C)$ as triply-graded $R$-modules and $HHH(C)|_{a = 0} \cong H^{\bullet}(\text{Hom}_{\text{Ch}(SBim_n)}(R_n, C))$ as doubly-graded $R$-modules. Combined with the duality functor $^{\vee}(-)$ of Section \ref{sec: SoergBim}, this characterization gives a convenient method of computing graded dimensions of morphism spaces between bimodules.

\begin{lemma} \label{lem: morph_dim}
Let $M, N \in SBim_n$ be given. Then there is an isomorphism of graded vector spaces

\[
\text{Hom}_{SBim_n}(M, N) \cong HHH(^{\vee}M N)|_{a = 0}
\]
\end{lemma}

Our primary interest in Hochschild cohomology comes from its relation to link invariants. The following is due to Khovanov and Rozansky (\cite{KhR08}, \cite{Kh07}).

\begin{theorem} \label{thm: hhhinvar}
Let $\beta \in Br_n$ be given. Then up to an overall normalization, the triply-graded $R$-module $HHH(F(\beta))$ is an invariant of the braid closure $\hat{\beta}$ categorifying the HOMFLYPT polynomial of $\hat{\beta}$.
\end{theorem}

We refer to this invariant as \textit{triply-graded Khovanov-Rozansky homology}, or often just \textit{triply-graded homology}. Given a link $\mathcal{L}$, we will often abuse notation by denoting the triply-graded homology of $\mathcal{L}$ as $HHH(\mathcal{L})$ without reference to the Rouquier complex of a braid presentation of $\mathcal{L}$. We will use the notation $\mathcal{P}(\mathcal{L})$ to denote the graded dimension of $HHH(\mathcal{L})$; note that $\mathcal{P}(\mathcal{L})$ is a power series in the variables $a, q, t$.

In light of this relation to braid closures, it is standard practice to depict $HHH(C)$ in the graphical calculus by "closing up" all strands. For example, we have:

\begin{gather*}
HHH(F(\sigma_1^2)) := 
\begin{tikzpicture}[anchorbase,scale=.35,tinynodes]
    \draw[webs] (1,0) to[out=90,in=270] (0,1);
    \draw[line width=5pt,color=white] (0,0) to[out=90,in=270] (1,1);
    \draw[webs] (0,0) to[out=90,in=270] (1,1);
    \draw[webs] (1,1) to[out=90,in=270] (0,2);
    \draw[line width=5pt,color=white] (0,1) to[out=90,in=270] (1,2);
    \draw[webs] (0,1) to[out=90,in=270] (1,2);
    \draw[webs] (1,2) to[out=90,in=180] (1.5,2.5);
    \draw[webs] (1.5,2.5) to[out=0,in=90] (2,2);
    \draw[webs] (0,2) to[out=90,in=180] (1.5,3.5);
    \draw[webs] (1.5,3.5) to[out=0,in=90] (3,2);
    \draw[webs] (2,2) to (2,0);
    \draw[webs] (3,2) to (3,0);
    \draw[webs] (1,0) to[out=270,in=180] (1.5,-.5);
    \draw[webs] (1.5,-.5) to[out=0,in=270] (2,0);
    \draw[webs] (0,0) to[out=270,in=180] (1.5,-1.5);
    \draw[webs] (1.5,-1.5) to[out=0,in=270] (3,0);
\end{tikzpicture}
\ ; \quad HHH(B_2) := \ 
\begin{tikzpicture}[anchorbase,scale=.35,tinynodes]
    \draw[webs] (0,0) to (0,1.5);
    \draw[webs] (1,0) to[out=90,in=180] (1.5,.5);
    \draw[webs] (2,0) to[out=90,in=0] (1.5,.5);
    \draw[webs] (1.5,.5) to (1.5,1);
    \draw[webs] (1.5,1) to[out=180,in=270] (1,1.5);
    \draw[webs] (1.5,1) to[out=0,in=270] (2,1.5);
    \draw[webs] (2,1.5) to[out=90,in=180] (2.5,2);
    \draw[webs] (2.5,2) to[out=0,in=90] (3,1.5);
    \draw[webs] (1,1.5) to[out=90,in=180] (2.5,3);
    \draw[webs] (2.5,3) to[out=0,in=90] (4,1.5);
    \draw[webs] (0,1.5) to[out=90,in=180] (2.5,4);
    \draw[webs] (2.5,4) to[out=0,in=90] (5,1.5);
    \draw[webs] (5,1.5) to (5,0);
    \draw[webs] (4,1.5) to (4,0);
    \draw[webs] (3,1.5) to (3,0);
    \draw[webs] (2,0) to[out=270,in=180] (2.5,-.5);
    \draw[webs] (2.5,-.5) to[out=0,in=270] (3,0);
    \draw[webs] (1,0) to[out=270,in=180] (2.5,-1.5);
    \draw[webs] (2.5,-1.5) to[out=0,in=270] (4,0);
    \draw[webs] (0,0) to[out=270,in=180] (2.5,-2.5);
    \draw[webs] (2.5,-2.5) to[out=0,in=270] (5,0);
\end{tikzpicture}
\end{gather*}

Implicit in the statement of Theorem \ref{thm: hhhinvar} is invariance of $HHH(F(\beta))$ under conjugation by a Rouquier complex. In fact more is true: it is well known that $HH$ is \textit{trace-like}, in the sense that 

\begin{align*}
    HH(C \otimes D) \cong HH(D \otimes C)
\end{align*}

for all $C, D \in \mathcal{D}_n$. We depict arbitrary complexes as coupons within a diagram in the graphical calculus; the above isomorphism then takes the form

\begin{gather*}
\begin{tikzpicture}[anchorbase, scale=.5, every node/.append style={draw, fill=white}]
    \draw[webs] (0,0) to (0,3);
    \draw[webs] (1,0) to (1,3);
    \draw[webs] (.5,0) to node[pos=.25,minimum width=2.5em]{$C$} node[pos=.75,minimum width=2.5em]{$D$} (.5,3);
    \draw[webs] (1,3) to[out=90,in=180] (1.5,3.5);
    \draw[webs] (1.5,3.5) to[out=0,in=90] (2,3);
    \draw[webs] (.5,3) to[out=90,in=180] (1.5,4);
    \draw[webs] (1.5,4) to[out=0,in=90] (2.5,3);
    \draw[webs] (0,3) to[out=90,in=180] (1.5,4.5);
    \draw[webs] (1.5,4.5) to[out=0,in=90] (3,3);
    \draw[webs] (1,0) to[out=270,in=180] (1.5,-.5);
    \draw[webs] (1.5,-.5) to[out=0,in=270] (2,0);
    \draw[webs] (.5,0) to[out=270,in=180] (1.5,-1);
    \draw[webs] (1.5,-1) to[out=0,in=270] (2.5,0);
    \draw[webs] (0,0) to[out=270,in=180] (1.5,-1.5);
    \draw[webs] (1.5,-1.5) to[out=0,in=270] (3,0);
    \draw[webs] (2,0) to (2,3);
    \draw[webs] (2.5,0) to (2.5,3);
    \draw[webs] (3,0) to (3,3);
\end{tikzpicture}
\ \cong \ 
\begin{tikzpicture}[anchorbase, scale=.5, every node/.append style={draw, fill=white}]
    \draw[webs] (0,0) to (0,3);
    \draw[webs] (1,0) to (1,3);
    \draw[webs] (.5,0) to node[pos=.25,minimum width=2.5em]{$D$} node[pos=.75,minimum width=2.5em]{$C$} (.5,3);
    \draw[webs] (1,3) to[out=90,in=180] (1.5,3.5);
    \draw[webs] (1.5,3.5) to[out=0,in=90] (2,3);
    \draw[webs] (.5,3) to[out=90,in=180] (1.5,4);
    \draw[webs] (1.5,4) to[out=0,in=90] (2.5,3);
    \draw[webs] (0,3) to[out=90,in=180] (1.5,4.5);
    \draw[webs] (1.5,4.5) to[out=0,in=90] (3,3);
    \draw[webs] (1,0) to[out=270,in=180] (1.5,-.5);
    \draw[webs] (1.5,-.5) to[out=0,in=270] (2,0);
    \draw[webs] (.5,0) to[out=270,in=180] (1.5,-1);
    \draw[webs] (1.5,-1) to[out=0,in=270] (2.5,0);
    \draw[webs] (0,0) to[out=270,in=180] (1.5,-1.5);
    \draw[webs] (1.5,-1.5) to[out=0,in=270] (3,0);
    \draw[webs] (2,0) to (2,3);
    \draw[webs] (2.5,0) to (2.5,3);
    \draw[webs] (3,0) to (3,3);
\end{tikzpicture}
\end{gather*}

Computing the triply-graded homology of a link directly from the definition is challenging. The Rouquier complex of a braid presentation grows exponentially in the number of crossings, and Hochschild cohomology must be applied not only to each chain module, but also to each differential. Performing these computations by hand quickly becomes infeasible even for braids on $3$ strands. Fortunately, there is a computational tool introduced in \cite{Hog18} that allows us to localize these computations, closing a braid presentation one strand at a time.

\begin{definition} \label{def: trace}
For each $n \geq 1$, let $Tr_n: \mathcal{D}_n \rightarrow \mathcal{D}_{n - 1}$ be the functor taking a complex $C$ to the convolution

\begin{center}
\begin{tikzcd}[sep=large]
Tr_n(C) := C \arrow[r, "x_n - x_n'"] & aq^{-2}C
\end{tikzcd}
\end{center}

Applying $Tr_n$ termwise gives a functor which we also denote $Tr_n: \mathcal{C}_n \rightarrow \mathcal{C}_{n - 1}$. We call each version of $Tr_n$ a \textit{partial trace functor}. We set $Tr_0 := \text{Hom}_{\mathcal{C}_0}(R, -)$ for notational convenience.
\end{definition}

Given any $M \in R_n-\text{Bim}$, the differential above gives $Tr_n(M)$ the structure of a Koszul complex for $x_n - x_n'$; this should be thought of as the derived equivalent of equating the actions of $x_n$ and $x_n'$ on $M$. \textcolor{revisions}{Indeed, the induced endomorphism $x_n - x_n' \in \mathrm{End}_{\mathcal{D}_{n - 1}}(Tr_n(M))$ is nullhomotopic in the \textit{derived} sense (with respect to the \textit{Hochschild} differential), and therefore $0$ on the nose in the derived category $\mathcal{D}_{n - 1}$ (since the latter can be viewed as a localization of the homotopy category $K(R_{n - 1}-\text{Bim})$).}

Our primary interest in partial trace functors is a result of the following adjunction:

\begin{proposition} \label{prop: tradj}
For each $n \geq 1$, $Tr_n$ is right adjoint to the natural inclusion functor $I_n: \mathcal{C}_{n - 1}^b \rightarrow \mathcal{C}_n^b$. In other words, for each $M \in \mathcal{C}_{n - 1}^b$, $N \in \mathcal{C}_n^b$, there is an isomorphism $\text{Hom}_{\mathcal{C}_n^b}(I_n(M), N) \cong \text{Hom}_{\mathcal{C}_{n - 1}^b}(M, Tr_n(M))$ that is natural in both $M$ and $N$.
\end{proposition}

\begin{proof}
\textcolor{revisions}{See \cite{Hog18}, Corollary 3.7.}
\end{proof}

Repeated application of this isomorphism gives $HHH = Tr_0 \circ Tr_1 \circ \dots \circ Tr_n$. The advantage in computing $HHH$ via partial trace functors is their \textit{locality}:

\begin{proposition} \label{prop: trlocal}
We have $\text{Tr}_n(I_n(M) \otimes C \otimes I_n(N)) \cong M \otimes Tr_n(C) \otimes N$ for any $M, N \in \mathcal{C}_{n - 1}^b$, $C \in \mathcal{C}_n$.
\end{proposition}

\begin{proof}
\textcolor{revisions}{Immediate from Definition \ref{def: trace}.}
\end{proof}

This locality allows a convenient depiction of partial trace functors in the graphical calculus consistent with $HHH$ by closing up one strand at a time. In this notation, Proposition \ref{prop: trlocal} becomes

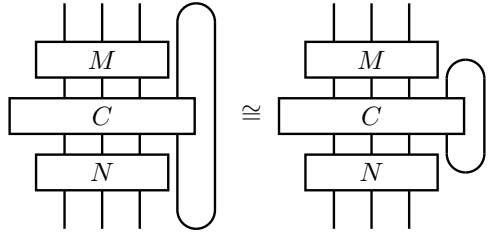
\begin{gather*}
\begin{tikzpicture}[anchorbase, every node/.append style={draw, fill=white}]
    \draw[webs] (0,0) to (0,3);
    \draw[webs] (1,0) to (1,3);
    \draw[webs] (1.5,.25) to (1.5,2.75);
    \draw[webs] (.5,0) to node[pos=.25,minimum width=5em]{$N$} node[pos=.5,minimum width=7em]{$C$} node[pos=.75,minimum width=5em]{$M$} (.5,3);
    \draw[webs] (1.5,2.75) to[out=90,in=180] (1.75,3);
    \draw[webs] (1.75,3) to[out=0,in=90] (2,2.75);
    \draw[webs] (2,2.75) to (2,.25);
    \draw[webs] (1.5,.25) to[out=270,in=180] (1.75,0);
    \draw[webs] (1.75,0) to[out=0,in=270] (2,.25);
\end{tikzpicture}
\ \ \cong \ 
\begin{tikzpicture}[anchorbase, every node/.append style={draw, fill=white}]
    \draw[webs] (0,0) to (0,3);
    \draw[webs] (1,0) to (1,3);
    \draw[webs] (1.5,1) to (1.5,2);
    \draw[webs] (.5,0) to node[pos=.25,minimum width=5em]{$N$} node[pos=.5,minimum width=7em]{$C$} node[pos=.75,minimum width=5em]{$M$} (.5,3);
    \draw[webs] (1.5,2) to[out=90,in=180] (1.75,2.25);
    \draw[webs] (1.75,2.25) to[out=0,in=90] (2,2);
    \draw[webs] (2,2) to (2,1);
    \draw[webs] (1.5,1) to[out=270,in=180] (1.75,.75);
    \draw[webs] (1.75,.75) to[out=0,in=270] (2,1);
\end{tikzpicture}
\end{gather*}

Also implicit in the statement of Theorem \ref{thm: hhhinvar} is invariance up to normalization under Markov stabilization. This invariance is encoded more precisely in the following list of identities, which we refer to as ``Partial trace Markov moves":

\begin{proposition} \label{prop: trmarkov}
Let $F(\sigma_i^{\pm})$ denote the Rouquier complexes for the braid group generators $\sigma_i^{\pm}$. We have:
\begin{enumerate}
    \item $Tr_n(R_n) \cong \left(\cfrac{1 + aq^{-2}}{1 - q^2}\right) R_{n - 1}$;
    
    \item $Tr_n(F(\sigma_{n - 1})) \cong q^{-1}t R_{n - 1}$;
    
    \item $Tr_n(F(\sigma_{n - 1}^{-1})) \cong aq^{-3} R_{n - 1}$
\end{enumerate}

Here the denominator in identity (1) should be considered as a formal power series $(1 - q^2) = 1 + q^2 + q^4 + \dots$, and the right-hand side of that identity should be interpreted as an infinite direct sum.
\end{proposition}

\begin{proof}
\textcolor{revisions}{See \cite{Hog18}, Proposition 3.10.}
\end{proof}

Graphically, these identities take the form

\begin{gather*}
\begin{tikzpicture}[anchorbase,scale=.5,tinynodes]
    \draw[webs] (0,-.5) to (0,1.5);
    \draw[webs] (.5,0) to (.5,1);
    \draw[webs] (.5,1) to[out=90,in=180] (1,1.5);
    \draw[webs] (1,1.5) to[out=0,in=90] (1.5,1);
    \draw[webs] (.5,0) to[out=270,in=180] (1,-.5);
    \draw[webs] (1,-.5) to[out=0,in=270] (1.5,0);
    \draw[webs] (1.5,0) to (1.5,1);
\end{tikzpicture}
\ \cong \left(\cfrac{1 + aq^{-2}}{1 - q^2}\right) \ 
\begin{tikzpicture}[anchorbase,scale=.5,tinynodes]
    \draw[webs] (0,-.5) to (0,1.5);
\end{tikzpicture}
\ ;
\end{gather*}

\begin{gather*}
\begin{tikzpicture}[anchorbase,scale=.5,tinynodes]
    \draw[webs] (.75,.25) to (-.5,1.5);
    \draw[line width=5pt, color=white] (0,0) to (1,1);
    \draw[webs] (-.5,-.5) to (.75,.75);
    \draw[webs] (.75,.75) to[out=45,in=180] (1.5,1.5);
    \draw[webs] (1.5,1.5) to[out=0,in=90] (2,1);
    \draw[webs] (.75,.25) to[out=315,in=180] (1.5,-.5);
    \draw[webs] (1.5,-.5) to[out=0,in=270] (2,0);
    \draw[webs] (2,0) to (2,1);
\end{tikzpicture}
\ \cong q^{-1}t \ 
\begin{tikzpicture}[anchorbase,scale=.5,tinynodes]
    \draw[webs] (0,-.5) to (0,1.5);
\end{tikzpicture}
\ ;
\end{gather*}

\begin{gather*}
\begin{tikzpicture}[anchorbase,scale=.5,tinynodes]
    \draw[webs] (-.5,-.5) to (.75,.75);
    \draw[line width=5pt, color=white] (1,0) to (0,1);
    \draw[webs] (.75,.25) to (-.5,1.5);
    \draw[webs] (.75,.75) to[out=45,in=180] (1.5,1.5);
    \draw[webs] (1.5,1.5) to[out=0,in=90] (2,1);
    \draw[webs] (.75,.25) to[out=315,in=180] (1.5,-.5);
    \draw[webs] (1.5,-.5) to[out=0,in=270] (2,0);
    \draw[webs] (2,0) to (2,1);
\end{tikzpicture}
\ \cong aq^{-3} \ 
\begin{tikzpicture}[anchorbase,scale=.5,tinynodes]
    \draw[webs] (0,-.5) to (0,1.5);
\end{tikzpicture}
\
\end{gather*}

We conclude this section with the following result, due to Rose-Tubbenhauer \textcolor{revisions}{(Lemma 4.10 in \cite{RT21})}:

\begin{proposition} \label{prop: trb}
$Tr_n(B_{w_0}) \cong \left( \cfrac{q^{n - 1} + aq^{-1 - n}}{1 - q^2} \right) B_{w_1}$ in $\mathcal{D}_{n - 1}$.
\end{proposition}

\section{The Infinite Projector} \label{sec: inf_proj_big}

In this section we show that the categorified Young antisymmetrizer $P_{1^n}$ introduced in \cite{AH17} satisfies a useful topological recursion. We fix $n \geq 1$ throughout and let $\mathcal{I}_n$ denote the full subcategory of $K(SBim_n)$ consisting of complexes whose chain modules are isomorphic to (direct sums of shifts of) $B_{w_0}$. It follows easily from Proposition \ref{prop: tensid} that $\mathcal{I}_n$ is a two-sided tensor ideal. We let $\mathcal{I}_n^{\perp}$, resp. $^{\perp}\mathcal{I}_n$, denote the full subcategory of $K(SBim_n)$ of complexes $C$ such that $CB_{w_0} \simeq 0$, resp. $B_{w_0}C \simeq 0$. As usual, we will allow decorations $\mathcal{I}_n^b$, $\mathcal{I}_n^{\pm}$ to denote the corresponding bounded subcategories; each of these is also a two-sided tensor ideal of its native bounded category.

The following characterization of $P_{1^n}$ is from \cite{AH17}:

\begin{theorem} \label{thm: AHchar}
Let $P_{1^n}$ denote the desired projector. Then:
\begin{itemize}
    \item[(P1)] $P_{1^n} \in \mathcal{I}_n^-$;
    \item[(P2)] There exists a chain map $\epsilon_n \colon P_{1^n} \to R_n$ such that $\text{Cone}(\epsilon_n) \in ^{\perp}(\mathcal{I}_n^-) \cap (\mathcal{I}_n^-)^{\perp}$.
\end{itemize}
These two properties uniquely characterize $P_{1^n}$ up to homotopy equivalence. That is, if $(Q, \nu)$ are a complex and morphism satisfying (P1) and (P2), then there is a unique (up to homotopy) map $\phi \colon Q \to P_{1^n}$ such that $\nu = \epsilon \circ \phi$; moreover, $\phi$ is a homotopy equivalence.
\end{theorem}

In \cite{AH17}, the authors obtain a model for $P_{1^n}$ by explicitly constructing a chain complex categorifying the expression $p_{1^n} = \frac{1}{[n]!} b_{w_0}$ from the Hecke algebra $H_n$ using a periodic convolution of Koszul complexes on $B_{w_0}$\footnote{This approach also works in other Coxeter groups; see \cite{Elb22} for details.}. We begin by recalling their construction.

\begin{definition}
Let $\theta_2, \dots, \theta_n$ be formal variables of degree $q^2t^{-1}$. We denote by $K_n$ the Koszul complex for the elements $x_2 - x_2', \dots, x_n - x_n'$ on $B_{w_0}$. Then $K_n := B_{w_0} \otimes \bigwedge[\theta_2, \dots, \theta_n]$ as a dg $R$-algebra with differential

\begin{equation} \label{eq: d_kn}
d_{K_n} = \sum_{j = 2}^n (x_j - x_j') \otimes \theta_k^{\vee}.
\end{equation}

Note that setting $\text{deg}(\theta_j) = q^2t^{-1}$ ensures that $\text{deg}(d_{K_n}) = t$.
\end{definition}

\textcolor{revisions}{In Definition 2.40 of \cite{AH17}, Abel-Hogancamp define a family of polynomials\footnote{\textcolor{revisions}{Which they denote $a_{ij}$.}} $f_{ij}(\mathbb{X}, \mathbb{X}')$ of degree $i - 1$ for each $2 \leq i \leq j$. We will not require explicit expressions for those polynomials here; the only properties we use are that they are chosen in such a way that $\sum_{j = i}^n f_{ij}(\mathbb{X}, \mathbb{X}')(x_j - x_j')$ acts by $0$ on $B_{w_0}$ and that they satisfy Proposition \ref{prop: AHequiv} below.} For convenience, we declare $f_{ij} = 0$ for $i > j$.

For each $2 \leq i \leq n$, \textcolor{revisions}{we set}

\begin{equation} \label{eq: xi_i}
\xi_i := \sum_{k = 2}^n f_{ik}(\mathbb{X}, \mathbb{X}') \otimes \theta_k \in \textup{End}^{-1}(K_n).
\end{equation}

We point out three properties this family of endomorphisms satisfies:

1) $\xi_i^2 = 0$ for each $i$;

2) $\xi_i\xi_j + \xi_j\xi_i = 0$ for each $i \neq j$;

3) $[d_{K_n}, \xi_i] = 0$ for each $i$.

The properties above guarantee that the family $\{\xi_i\}_{i = 2}^n$ form an exterior action of closed endomorphisms on $K_n$. We can pass to a polynomial action by the usual yoga of Koszul duality:

\begin{proposition}
Let $u_2, \dots, u_n$ be formal variables of degree $\text{deg}(u_i) = q^{-2i}t^2$, and consider $\mathbb{Z}[u_2^{-1}, \dots, u_n^{-1}]$ as a dg $R$-algebra with trivial differential. Set $A_n := K_n \otimes \mathbb{Z}[u_2^{-1}, \dots, u_n^{-1}]$, and let 

\begin{equation} \label{eq: dan}
d_{A_n} := d_{K_n} \otimes 1 + \sum_{i = 2}^n \xi_i \otimes u_i \in \text{End}^1(A_n).
\end{equation}

Then $d_{A_n}^2 = 0$, so $(A_n, d_{A_n}) \in \text{Ch}^-(SBim_n)$.
\end{proposition}

\begin{proof}
\textcolor{revisions}{Straightforward computation using the middle interchange rule and the properties 1) - 3) above.}
\end{proof}

The following is due to Abel-Hogancamp (\cite{AH17}):

\begin{proposition} \label{prop: AHequiv}
There exists a map $\textcolor{revisions}{\epsilon} \colon A_n \to R_n$ satisfying the conditions of Theorem \ref{thm: AHchar}.
\end{proposition}

While explicit, this description of $P_{1^n}$ does not reflect the recursive structure of categorified projectors in Type A or their deep relationship with the full twist braid that is present in \cite{Hog18} and \cite{EH17a}. We take a middle road between these two perspectives, beginning with the explicit model of \cite{AH17} and showing it satisfies similar recursive relationships to those of \cite{Hog18} and \cite{EH17a}. This perspective will vastly simplify computations of the associated link invariant; see Section \S \ref{sec: link_hom} below.

Throughout, let $X_n := \sigma_{n - 1} \sigma_{n - 2} \dots \sigma_2 \sigma_1 \in Br_n$, $Y_n := \sigma_1 \sigma_2 \dots \sigma_{n - 2} \sigma_{n - 1} \in Br_n$. Let $J_n$ denote the Jucys-Murphy braid $J_n = X_nY_n$. We abuse notation throughout, denoting the Rouquier complex $F(J_n) \in K^b(SBim_n)$ by $J_n$ as well.

\begin{theorem} \label{thm: projtop}
Suppose that there exist closed morphisms $\alpha, \beta: A_{n - 1} \rightarrow A_{n - 1}J_n$ of degrees $q^{-2}t^2$, $q^{2(n - 1)}$, respectively, such that $\textup{Cone}(\alpha) \in \mathcal{I}_n^-$, $\textup{Cone}(\beta) \in (\mathcal{I}_n^-)^{\perp} \ \cap \ ^{\perp}(\mathcal{I}_n^-)$. Let $u_n$ be a formal variable of degree $q^{-2n}t^2$, and consider $\mathbb{Z}[u_n^{-1}]$ as a dg $R$-algebra with trivial differential. Consider the map

\[
\Phi \colon t\mathbb{Z}[u_n^{-1}] \otimes A_{n - 1} \to q^2t^{-1}\mathbb{Z}[u_n^{-1}] \otimes A_{n - 1}J_n; \quad \quad \Phi = 1 \otimes \alpha + u_n \otimes \beta
\]

Set $P_{1^n} = \textup{Cone}(\Phi)$. Then there exists a chain map $\epsilon_n \colon P_{1^n} \to R_n$ such that $(P_{1^n}, \epsilon_n)$ satisfy the conditions of Theorem \ref{thm: AHchar}. In particular, $P_{1^n} \simeq A_n$.
\end{theorem}

\begin{proof}
\textcolor{revisions}{Following \cite{Hog18}, we depict $\mathrm{Cone}(\Phi)$ as an infinite ladder diagram:}

% Define shapes for each column
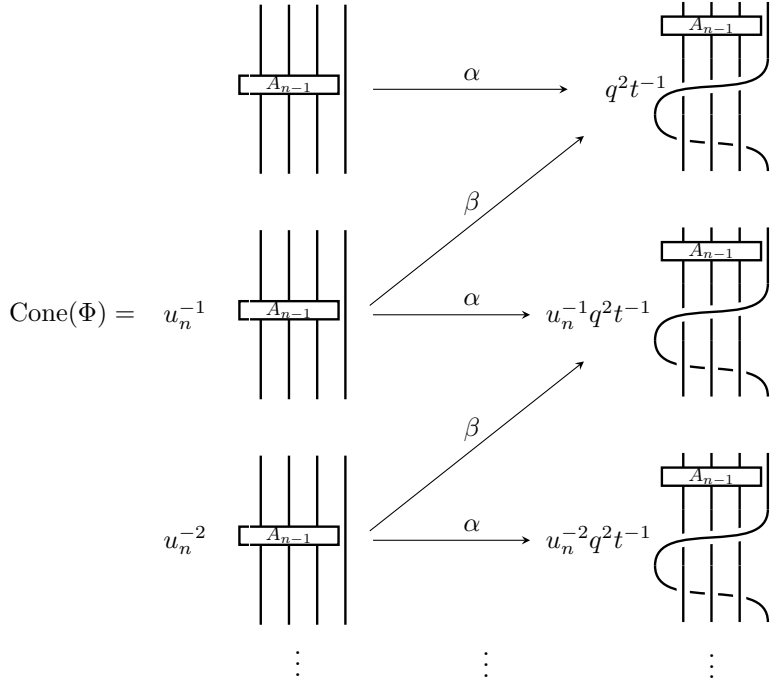
\newcommand{\columnAShape}{
    \begin{tikzpicture}[anchor=base, every node/.append style={draw, fill=white}, scale=0.75]
        \draw[webs] (0,0) to (0,3);
        \draw[webs] (1,0) to (1,3);
        \draw[webs] (.5,0) to node[pos=.5,scale=0.75,minimum width=5em]{$A_{n - 1}$} (.5,3);
        \draw[webs] (1.5,0) to (1.5,3);
        \draw[color=white] (-.2,0) to (-.2,3);
        \draw[color=white] (1.7,0) to (1.7,3);
    \end{tikzpicture}
}

\newcommand{\columnBShape}{
    \begin{tikzpicture}[anchor=base, every node/.append style={draw, fill=white}, scale=0.75]
        \draw[webs] (0,2) to (0,3);
        \draw[webs] (1,2) to (1,3);
        \draw[webs] (.5,2) to node[pos=.5,scale=0.75,minimum width=5em]{$A_{n - 1}$} (.5,3);
        \draw[webs] (1.5,2) to (1.5,3);
        \draw[webs] (0,2) to (0,1);
        \draw[webs] (.5,2) to (.5,1);
        \draw[webs] (1,2) to (1,1);
        \draw[color=white,line width=5pt] (1.5,2) to[out=270,in=90] (-.5,1);
        \draw[webs] (1.5,2) to[out=270,in=90] (-.5,1);
        \draw[webs] (-.5,1) to[out=270,in=90] (1.5,0);
        \draw[color=white,line width=5pt] (0,1) to (0,0);
        \draw[color=white,line width=5pt] (.5,1) to (.5,0);
        \draw[color=white,line width=5pt] (1,1) to (1,0);
        \draw[webs] (0,1) to (0,0);
        \draw[webs] (.5,1) to (.5,0);
        \draw[webs] (1,1) to (1,0);
        \draw[color=white] (-.7,0) to (-.7,3);
        \draw[color=white] (1.7,0) to (1.7,3);
    \end{tikzpicture}
}

% Define a macro for drawing a single column of shapes
\newcommand{\drawColumn}[4]{
    % Draw shapes for column A
    \ifx#2\columnAShape
        \foreach \i in {1,...,#1} {
            \node[inner sep=0pt] (nodeA\i) at (0,-\i*3) {#2};
            % Add coefficient for the second and third entry
            \ifnum\i=2
                \node at (-1.5, -\i*3) {$u_n^{-1}$};
            \fi
            \ifnum\i=3
                \node at (-1.5, -\i*3) {$u_n^{-2}$};
            \fi
        }
    \fi
    
    % Draw shapes for column B
    \ifx#3\columnBShape
        \foreach \i in {1,...,#1} {
            \node[inner sep=0pt] (nodeB\i) at (5.5cm,-\i*3) {#3};
            % Add coefficients
            \ifnum\i=1
                \node at (4.5cm,-\i*3) {$q^2t^{-1}$};
            \fi
            \ifnum\i=2
                \node at (4cm, -\i*3) {$u_n^{-1}q^2t^{-1}$};
            \fi
            \ifnum\i=3
                \node at (4cm, -\i*3) {$u_n^{-2}q^2t^{-1}$};
            \fi
        }
    \fi
}

\begin{center}
\begin{tikzpicture}
    % Draw the first column (column A)
    \drawColumn{3}{\columnAShape}{}
    
    % Draw the second column (column B)
    \drawColumn{3}{}{\columnBShape}
    
    % Draw arrows between corresponding entries
    \draw[->, shorten <=2mm, shorten >=1cm] (nodeA1.east) -- (nodeB1.west) node[pos=.4, above] {$\alpha$};
    \draw[->, shorten <=2mm, shorten >=1.5cm] (nodeA2.east) -- (nodeB2.west) node[pos=.4, above] {$\alpha$};
	\draw[->, shorten <=2mm, shorten >=1.5cm] (nodeA3.east) -- (nodeB3.west) node[pos=.4, above] {$\alpha$};    
    \draw[->, shorten <=2mm, shorten >=1cm] (nodeA2.east) -- (nodeB1.west) node[pos=.4,above] {$\beta$};
    \draw[->, shorten <=2mm, shorten >=1cm] (nodeA3.east) -- (nodeB2.west) node[pos=.4,above] {$\beta$};
    
    \node at (nodeB3.south) [below] {$\vdots$};
    \node at (nodeA3.south) [below] {$\vdots$};
    \node at (2.5,-10.55) {$\vdots$};
    
    \node at (-3,-6) {$\mathrm{Cone}(\Phi) = $};
\end{tikzpicture}
\end{center}

(P1): Observe that $P_{1^n}$ is a one-sided convolution of complexes of the form $\text{Cone}(\alpha)$ \textcolor{revisions}{(these are the horizontal rows in the diagram above)} indexed by increasing powers of $u_n$. \textcolor{revisions}{This is an upper finite indexing set.} Since $\text{Cone}(\alpha) \in \mathcal{I}_n^-$ \textcolor{revisions}{by assumption, $P_{1^n} \in \mathcal{I}_n^-$ by Proposition \ref{prop: useful_hpt}.}

(P2): Define $\epsilon_n \colon P_{1^n} \to R_n$ by $\epsilon_{n - 1} \textcolor{revisions}{\sqcup} \text{id}_{R_1}$ on the component $A_{n - 1} \sqcup R_1$ of $P_{1^n}$ \textcolor{revisions}{(in the top left of the above diagram)} and zero otherwise. (Note that this is indeed a chain map, as $\epsilon_{n - 1}$ is a chain map by assumption and no component of $\Phi$ has this term as its codomain). Then $\textup{Cone}(\epsilon_n)$ is a convolution of the form

\begin{center}
\begin{equation} \label{eq: cone_epsilon}
\begin{tikzcd}
\textup{Cone}(\epsilon_n) = \Big(\big(A_{n - 1} \arrow[r, "\epsilon_{n - 1} \textcolor{revisions}{\sqcup \text{id}_{R_1}}"] & R_n\big) \arrow[r] & u^{-1}\textup{Cone}(\beta) \arrow[r] & u^{-2}\textup{Cone}(\beta) \arrow[r] & \dots\Big)
\end{tikzcd}
\end{equation}
\end{center}

Note that $B_{w_0}\text{Cone}(\beta) \simeq \text{Cone}(\beta) B_{w_0} \simeq 0$ by hypothesis. We have $B_{w_1} \text{Cone}(\epsilon_{n - 1}) \simeq \text{Cone}(\epsilon_{n - 1})B_{w_1} \simeq 0$ by construction. \textcolor{revisions}{By Proposition \ref{prop: tensid}, this implies}

\[
[n - 1]! B_{w_0} \mathrm{Cone}(\epsilon_{n - 1}) \textcolor{revisions}{\cong B_{w_0} B_{w_1} \mathrm{Cone}(\epsilon_{n - 1})} \simeq 0,
\]

\textcolor{revisions}{and similarly } $[n - 1]! \text{Cone}(\epsilon_{n - 1})B_{w_0} \simeq 0$. \textcolor{revisions}{Then each of $B_{w_0} \mathrm{Cone}(\epsilon_{n - 1})$ and $\text{Cone}(\epsilon_{n - 1})B_{w_0}$ is a summand of a contractible complex and is therefore contractible.}

\textcolor{revisions}{The convolution \eqref{eq: cone_epsilon} is homologically locally finite; this follows from the fact that each of $\mathrm{Cone}(\epsilon_n)$ and $\mathrm{Cone}(\beta)$ is bounded above by assumption and that $u_n^{-i}$ has negative homological degree for each $i \geq 1$. Applying $B_{w_0} \otimes -$ or $- \otimes B_{w_0}$ preserves this property.}
That $B_{w_0} \text{Cone}(\epsilon_n) \simeq \text{Cone}(\epsilon_n) B_{w_0} \simeq 0$ then follows from \textcolor{revisions}{a direct application of Proposition \ref{prop: useful_hpt}}.
\end{proof}

The rest of this chapter is devoted to proving the following:

\begin{theorem} \label{thm: projtopexists}
There exist maps $\alpha, \beta$ with the properties of Theorem \ref{thm: projtop}.
\end{theorem}

\subsection{Constructing Beta} \label{sec: const_beta}

Our first goal is constructing a map $\beta: A_{n - 1} \rightarrow A_{n - 1}J_n$ such that $\text{Cone}(\beta) \in (\mathcal{I}_n^-)^{\perp} \cap ^{\perp} (\mathcal{I}_n^-)$. \textcolor{revisions}{The majority of this construction follows from statements in \cite{AH17}; we cite those statements without proof, indicating where in \cite{AH17} the proofs can be found.}

\begin{proposition} \label{prop: I_acyc}
Let $C \in \mathcal{I}_n^-$ be given. Then $C$ is contractible if and only if $C$ is acyclic.
\end{proposition}

\begin{proof}
See Lemma 2.13 in \cite{AH17}.
\end{proof}

\begin{lemma} \label{lem: bs_proj}
$B_{\sigma}$ is projective as a left or right $R_n$-module for each $\sigma \in \mathfrak{S}^n$. In particular, $B_{w_0}$ and $B_{w_1}$ are projective as left and right $R_n$-modules.
\end{lemma}

\begin{proof}
This is standard; see e.g. \cite{ESW14}.
\end{proof}

\textcolor{revisions}{We call a complex of Soergel bimodules \textit{acyclic} if its underlying complex of $R_n$-bimodules is acyclic.}

\begin{corollary} \label{corr: acyc_perp}
Let $Z \in K^-(SBim_n)$ be given. If $Z$ is acyclic, then $Z \in (\mathcal{I}_n^-)^{\perp} \cap ^{\perp} (\mathcal{I}_n^-)$.
\end{corollary}

\begin{proof}
See Corollary 2.15 in \cite{AH17}.
\end{proof}

Now, recall from Proposition \ref{prop: rouqhom} that $H^i(J_n) \cong q^{2(n - 1)}R_n$ for $i = 0$, $H^i(J_n) = 0$ otherwise. In particular, since $J_n$ is concentrated in positive homological degrees, there is a quasi-isomorphism $j_n: R_n \hookrightarrow q^{-2(n - 1)}J_n$ given by inclusion of $0^{th}$ homology.

\begin{proposition} \label{prop: beta_perp}
The map $\beta = \text{id}_{A_{n - 1}} \otimes j_n \colon A_{n - 1} \to q^{-2(n - 1)}A_{n - 1}J_n$ satisfies $\text{Cone}(\beta) \in \mathcal{I}^{\perp} \cap ^{\perp} \mathcal{I}$.
\end{proposition}

\begin{proof}
\textcolor{revisions}{There is an isomorphism of complexes $\mathrm{Cone}(\beta) \cong A_{n - 1} \otimes \mathrm{Cone}(j_n)$.} Since $j_n$ is a quasi-isomorphism, $\text{Cone}(j_n)$ is acyclic. It is clear that $A_{n - 1} \in \mathcal{I}_{n - 1}$; \textcolor{revisions}{hence by Lemma \ref{lem: bs_proj}, $A_{n - 1}$ consists of flat (right) $R_n$-modules in each degree. It follows that $A_{n - 1} \otimes \mathrm{Cone}(j_n)$ is the tensor product of a chain complex of flat modules with a \textit{bounded} acyclic complex, and is therefore itself acyclic. An immediate application of Corollary \ref{corr: acyc_perp} then gives that $\text{Cone}(\beta) \in \mathcal{I}^{\perp} \cap ^{\perp} \mathcal{I}$.}
\end{proof}

\subsection{Fork Sliding} \label{sec: fork_slide}

In this subsection, we take our first steps towards constructing $\alpha$ by identifying a closed morphism between \textit{finite} complexes $\overline{\alpha}: K_{n - 1} \rightarrow K_{n - 1}J_n$ of degree $q^{-2}t^2$ satisfying $\text{Cone}(\overline{\alpha}) \in \mathcal{I}_n^-$. Since $K_{n - 1}J_n$ is a convolution of complexes of the form $B_{w_1}J_n$, it will be helpful to begin with a more thorough analysis of the latter.

\begin{proposition} \label{prop: mu_nu}
\textcolor{revisions}{Let $C_n$ denote the three term complex below:}

\begin{center}
\begin{equation} \label{eq: c_n}
\begin{tikzcd}[sep=large]
C_n := q^{n - 1}B_{w_0} \arrow[r, "x_n - x_n'"] & q^{n - 3}tB_{w_0} \arrow[r, "unzip"] & q^{-2}t^2B_{w_1} .
\end{tikzcd}
\end{equation}
\end{center}

\textcolor{revisions}{Then t}here is a strong deformation retraction of the form

\begin{center}
\begin{tikzcd}
B_{w_1}J_n \arrow[r, "\nu", harpoon, shift left] & C_n \arrow[l, "\mu", harpoon, shift left].
\end{tikzcd}
\end{center}

\end{proposition}

\begin{proof}
\textcolor{revisions}{See Appendix \ref{app: dot_slide_nat}.}
\end{proof}

There is an obvious degree $q^{-2}t^2$ map $\iota$ from $B_{w_1}$ to \textcolor{revisions}{$C_n$} given by including $B_{w_1}$ as a subcomplex; Gaussian elimination along this inclusion shows that $\text{Cone}(\iota) \in \mathcal{I}_n^-$. Since $\nu$ is a strong deformation retraction, we obtain a homotopy equivalence $\text{Cone}(\mu \iota) \simeq \text{Cone}(\iota)$, so that $\text{Cone}(\mu \iota) \in \mathcal{I}_n^-$ as well; see the diagram below.

\begin{center}
\begin{tikzcd}
t^{-1}B_{w_1} \arrow[rr, "\iota", hook] \arrow[d, "\text{id}", harpoon, shift left] & & q^2t^{-2}C_n \arrow[d, "\mu", harpoon, shift left] \\
t^{-1}B_{w_1} \arrow[rr, "\mu \iota"] \arrow[u, "\text{id}", harpoon, shift left] & & q^2t^{-2}B_{w_1}J_n \arrow[u, "\nu", harpoon, shift left]
\end{tikzcd}
\end{center}

Unfortunately, the \textcolor{revisions}{na{\"i}ve} extension $\mu \iota \otimes 1 \colon K_{n - 1} \rightarrow K_{n - 1}J_n$ is not closed. Indeed, letting $\mathbb{X}', \mathbb{X}''$ denote the `top' and `bottom' alphabets on $J_n$ (so that in particular, the suppressed tensor product is taken over $R[\mathbb{X}']$), we have $K_{n - 1}J_n \cong B_{w_1}J_n \otimes_R \bigwedge[\theta_2, \dots, \theta_{n - 1}]$ with differential $d_{K_{n - 1}J_n} = d_{J_n} \otimes 1 + \sum_{j = 2}^{n - 1} (x_j - x_j') \otimes \theta_j$. Since the fork-sliding homotopy $\mu$ is not $\mathbb{X}'$-linear, $\mu \iota \otimes 1$ is not guaranteed to commute with the latter component of this differential.

To circumvent this difficulty, observe that both $\mu$ and $\iota$ are morphisms of bimodule complexes, so they must commute with the external $\mathbb{X}''$ action on $K_{n - 1}J_n$. In particular, $\mu \iota \otimes 1$ \textit{is} a closed morphism from $K_{n - 1}$ to a complex $K_n'$ with the same underlying chain bimodules as $K_{n - 1}J_n$ but with differential 

\begin{equation} \label{eq: dkn'}
d_{K_n'} = d_{B_{w_1}J_n} \otimes 1 + \sum_{j = 2}^{n - 1} (x_j - x_j'') \otimes \textcolor{revisions}{\theta_j^{\vee}}.
\end{equation}

To complete our construction of $\overline{\alpha}$, we would like to show that $K_n' \cong K_{n - 1}J_n$. We accomplish this by way of the following ``change of basis" theorem for Koszul complexes:

\begin{proposition} \label{prop: kosz_base}
Let $(C, d_C)$ be a bounded chain complex, and let $f_1, \dots, f_r$ \textcolor{revisions}{and $g$} be central closed degree 0 endomorphisms of $C$. Suppose $[d_C, \textcolor{revisions}{h}] = f_j - \textcolor{revisions}{g}$ for some $1 \leq j \leq r$, $\textcolor{revisions}{h} \in \text{End}^{-1}(C)$ satisfying $\textcolor{revisions}{h^2} = 0$. Let $\theta_1, \dots, \theta_r$ be odd formal variables of homological degree $t^{-1}$. Then the chain complex $K_f := C \otimes \bigwedge[\theta_1, \dots, \theta_r]$ with differential $d_{K_f} = d_C \otimes 1 + \sum_{i = 1}^r f_i \otimes \theta_i^{\vee}$ is isomorphic to the chain complex $K_g := C \otimes \bigwedge[\theta_1, \dots, \theta_r]$ with differential $d_{K_g} = d_C \otimes 1 + \sum_{i \neq j} f_i \otimes \theta_i^{\vee} + \textcolor{revisions}{g} \otimes \theta_j^{\vee}$.
\end{proposition}

\begin{proof}
Let $\Psi := 1 \otimes 1 + \textcolor{revisions}{h} \otimes \theta_j^{\vee} \in \text{Hom}^0(K_f, K_g)$. To see that $\Psi$ is a chain map is a straightforward computation (note the continued use of the middle interchange law):

\begin{align*}
    d_{K_g}\Psi - \Psi d_{K_f} & = (d_C \otimes 1 + \sum_{i \neq j} f_i \otimes \theta_i^{\vee} + \textcolor{revisions}{g} \otimes \theta_j^{\vee})(1 \otimes 1 + \textcolor{revisions}{h} \otimes \theta_j^{\vee}) - (1 \otimes 1 + \textcolor{revisions}{h} \otimes \theta_j^{\vee})(d_C \otimes 1 + \sum_{i = 1}^r f_i \otimes \theta_i^{\vee}) \\
    & = \left(d_C \otimes 1 + (d_C\textcolor{revisions}{h} + \textcolor{revisions}{g}) \otimes \theta_j^{\vee} + \sum_{i \neq j} (f_i \otimes \theta_i^{\vee} - f_i\textcolor{revisions}{h} \otimes \theta_i^{\vee}\theta_j^{\vee})\right) - \\
    & \quad \left(d_C \otimes 1 + (-\textcolor{revisions}{h}d_C + f_j) \otimes \theta_j^{\vee} + \sum_{i \neq j} (f_i \otimes \theta_i^{\vee} + \textcolor{revisions}{h}f_i \otimes \theta_j^{\vee}\theta_i^{\vee})\right) \\
    & = ([d_C, h_j] - (f_j - \textcolor{revisions}{g})) \otimes \theta_j^{\vee} - \left(\sum_{i \neq j} f_i\textcolor{revisions}{h} \otimes (\theta_i^{\vee}\theta_j^{\vee} + \theta_j^{\vee}\theta_i^{\vee})\right) \\
    & = 0.
\end{align*}

By the same computation as above with the roles of $f_j$ and $\textcolor{revisions}{g}$ reversed, we see that $\Psi' := 1 \otimes 1 - \textcolor{revisions}{h} \otimes \theta_j^{\vee}$ is a chain map from $K_g$ to $K_f$. We claim that $\Psi\Psi' = 1 \otimes 1$, so that $\Psi$ is an isomorphism. Again, this is a direct computation:

\begin{align*}
    \Psi \Psi' & = (1 \otimes 1 + \textcolor{revisions}{h} \otimes \theta_j^{\vee})(1 \otimes 1 - \textcolor{revisions}{h} \otimes \theta_j^{\vee}) = 1 \otimes 1 + (\textcolor{revisions}{h} - \textcolor{revisions}{h}) \otimes \theta_j^{\vee} = 1 \otimes 1
\end{align*}

\end{proof}

\begin{corollary} \label{corr: building_psi}
For each $2 \leq i \leq n - 1$, let $h_i \in \text{End}^{-1}(J_n)$ be the dot-sliding homotopy \textcolor{revisions}{of Proposition \ref{prop: dot_slide}} satisfying $[d, h_i] = x_i' - x_i''$. Then there is an isomorphism $\Psi: K_{n - 1}J_n \rightarrow K'_n$ given by 

\begin{equation} \label{eq: psi}
\Psi := \prod_{j = 2}^{n - 1} (1 \otimes 1 - (\text{id}_{B_{w_1}} \otimes h_j) \otimes \theta_j^{\vee})
\end{equation}

with inverse

\begin{equation} \label{eq: psi'}
\Psi' = \prod_{j = 2}^{n - 1} (1 \otimes 1 + (\text{id}_{B_{w_1}} \otimes h_j) \otimes \theta_j^{\vee}).
\end{equation}
\end{corollary}

\begin{proof}
Recall that the $R[\mathbb{X}, \mathbb{X}', \mathbb{X}'']$ action given by the bimodule structures of each chain group is central in the endomorphism algebra of the braided web $B_{w_1}J_n$. Since $[d, -\text{id}_{B_{w_1}} \otimes h_i] = x_i'' - x_i' = (x_i - x_i') - (x_i - x_i'')$ for each $i$, that $\Psi$ is an isomorphism follows immediately from Proposition \ref{prop: kosz_base}.

It remains to show that $\Psi'\Psi = 1$. For this, it suffices to show that the defining composition of $\Psi'$ is independent of the order of its factors. This is a straightforward computation:

\begin{align*}
    (1 \otimes 1 + h_i \otimes \theta_i^{\vee})(1 \otimes 1 + h_j \otimes \theta_j^{\vee}) & = 1 \otimes 1 + h_i \otimes \theta_i^{\vee} + h_j \otimes \theta_j^{\vee} - h_ih_j \otimes \theta_i^{\vee}\theta_j^{\vee} \\
    & = 1 \otimes 1 + h_j \otimes \theta_j^{\vee} + h_i \otimes \theta_i^{\vee} - h_jh_i \otimes \theta_j^{\vee}\theta_i^{\vee} \\
    & = (1 \otimes 1 + h_j \otimes \theta_j^{\vee})(1 \otimes 1 + h_i \otimes \theta_i^{\vee})
\end{align*}

Here we have used the fact that both $h_i, h_j$ and $\theta_i^{\vee}, \theta_j^{\vee}$ anticommute in the second equality.
\end{proof}

\begin{proposition} \label{prop: krec}
The map $\overline{\alpha}: K_{n - 1} \rightarrow q^2t^{-2}K_{n - 1}J_n$ given by $\overline{\alpha} := \Psi'(\mu \iota \otimes 1)$ is a chain map satisfying $\text{Cone}(\overline{\alpha}) \simeq q^{n - 1}t^{-1}K_n \in \mathcal{I}_n^-$.
\end{proposition}

\begin{proof}
We denote the external polynomial action on all complexes by $R[\mathbb{X}, \mathbb{X}'']$. Consider $C_n \otimes \bigwedge[\theta_2, \dots, \theta_{n - 1}]$ as a chain complex with differential $d_{C_n \otimes \bigwedge[\theta_2, \dots, \theta_{n - 1}]} := d_C \otimes 1 + \sum_{j = 2}^{n - 1} (x_j - x_j'') \otimes \theta_j^{\vee}$. Then $K_{n - 1}$ appears as a subcomplex of $q^2t^{-2}C_n \otimes \bigwedge[\theta_2, \dots, \theta_{n - 1}]$ via the inclusion $\iota \otimes 1$.

Gaussian elimination along this inclusion leaves the quotient complex below:

\begin{center}
\begin{tikzcd}[sep=large]
\text{Cone}(\iota \otimes 1) \simeq q^2t^{-2}(q^{n - 1}B_{w_0} \arrow[r, "x_n - x_n''"] & q^{n - 3}tB_{w_0}) \otimes \bigwedge[\theta_2, \dots, \theta_{n - 1}]
\end{tikzcd}
\end{center}

After absorbing the overall grading shift, the two term complex in parentheses is exactly the Koszul complex for $x_n - x_n''$ on $q^{n - 1}t^{-1}B_{w_0}$. Absorbing this Koszul differential into the exterior algebra gives $\text{Cone}(\iota \otimes 1) \simeq q^{n - 1}t^{-1}K_n$.

Finally, we obtain $\text{Cone}(\overline{\alpha}) \simeq \text{Cone}(\iota \otimes 1)$ by the chain of homotopy equivalences depicted below:

\begin{center}
\begin{tikzcd}
t^{-1}K_{n - 1} \arrow[rr, "\overline{\alpha} = \Psi'(\mu \iota \otimes 1)"] \arrow[d, "\text{id}", harpoon, shift left] & & q^2t^{-2}K_{n - 1}J_n \arrow[d, "\Psi", harpoon, shift left] \\
t^{-1}K_{n - 1} \arrow[rr, "\mu \iota \otimes 1"] \arrow[u, "\text{id}", harpoon, shift left] \arrow[d, "\text{id}", harpoon, shift left] & & q^2t^{-2}K_n' \arrow[u, "\Psi'", harpoon, shift left] \arrow[d, "\nu \otimes 1", harpoon, shift left] \\
t^{-1}K_{n - 1} \arrow[rr, "\iota \otimes 1"] \arrow[u, "\text{id}", harpoon, shift left] & & q^2t^{-2} C_n \otimes \bigwedge[\theta_2, \dots, \theta_{n - 1}] \arrow[u, "\mu \otimes 1", harpoon, shift left]
\end{tikzcd}
\end{center}
\end{proof}

\subsection{Constructing Alpha} \label{sec: alpha_build}

Our next task is to lift the construction of $\overline{\alpha}$ above to a map $\alpha: A_{n - 1} \rightarrow A_{n - 1}J_n$ satisfying $\text{Cone}(\alpha) \in \mathcal{I}_n^-$ by incorporating the periodic action of $\mathbb{Z}[u_2^{-1}, \dots, u_{n - 1}^{-1}]$. \textcolor{revisions}{We do this by explicitly computing components of a candidate map $\alpha$ in each periodic degree, then verifying that the result is closed and satisfies $\mathrm{Cone}(\alpha) \in \mathcal{I}_n^-$.}

\subsubsection{Change of Basis}

\textcolor{revisions}{To construct $\alpha$, w}e will again find it useful \textcolor{revisions}{first} to \textcolor{revisions}{change basis in the Koszul direction} from $K_{n - 1}J_n$ to $K_n'$. \textcolor{revisions}{We begin by computing the interaction of this change of basis with the periodic action.}

\begin{proposition} \label{prop: xi_conj}
\textcolor{revisions}{Let $\xi_j \in \mathrm{End}^{-1}(K_{n - 1}J_n)$ be as in Equation \eqref{eq: xi_i} and $\Psi, \Psi'$ be as in Equations \eqref{eq: psi} and \eqref{eq: psi'}. Then we have}

\begin{equation} \label{eq: xi_conj}
\Psi \xi_j \Psi' = \sum_{k = 2}^{n - 1} f_{jk}(\mathbb{X}, \mathbb{X}')(1 \otimes \theta_k - h_k \otimes 1).
\end{equation}
\end{proposition}

\begin{proof}
Since composition of morphisms is linear, it suffices to compute the effect of conjugation on each summand $\xi_{jk} := f_{jk}(\mathbb{X}, \mathbb{X}') \otimes \theta_k$ of $\xi_j$. We conjugate by one factor $\Psi_i = 1 \otimes 1 - h_i \otimes \theta_i^{\vee}$ at a time:

\begin{align*}
    \Psi_i \xi_{jk} \Psi_i' & = (1 \otimes 1 - h_i \otimes \theta_i^{\vee})(f_{jk}(\mathbb{X}, \mathbb{X}') \otimes \theta_k)(1 \otimes 1 + h_i \otimes \theta_i^{\vee}) \\
    & = f_{jk}(\mathbb{X}, \mathbb{X}')(1 \otimes \theta_k - h_i \otimes (\theta_k\theta_i^{\vee} + \theta_i^{\vee}\theta_k)) \\
    & = \xi_{jk} + \delta_{ik}(f_{jk}(\mathbb{X}, \mathbb{X}')h_i \otimes 1)
\end{align*}

Here $\delta_{ik}$ denotes the Kronecker delta (i.e. $\delta_{ik} = 1$ if $i = k$, $0$ otherwise). Since $\Psi, \Psi'$ are independent of the ordering of the homotopies $h_i$, we are free to conjugate by $\Psi_k$ last. Then our expression for $\xi_{jk}$ is unchanged by conjugation until the final step, at which point we obtain

\[
\Psi \xi_{jk} \Psi' = f_{jk}(\mathbb{X}, \mathbb{X}')(1 \otimes \theta_k - h_k \otimes 1)
\]

\textcolor{revisions}{Equation \eqref{eq: xi_conj}} follows immediately after summing over $k$.
\end{proof}

\textcolor{revisions}{As an immediate consequence of Proposition \ref{prop: xi_conj}, applying the isomorphism $\Psi$, $\Psi'$ in each periodic degree gives the following result.}

\begin{corollary}
Let $A'_n := K'_n \otimes \mathbb{Z}[u_2^{-1}, \dots, u_{n - 1}^{-1}]$ with differential

\begin{equation} \label{eq: da'n}
d_{A'_n} := d_{K'_n} \otimes 1 + \sum_{i = 2}^{n - 1} \textcolor{revisions}{\left( \sum_{j = 2}^{n - 1} f_{ij}(\mathbb{X}, \mathbb{X}')(1 \otimes \theta_j - h_j \otimes 1) \right)} \otimes u_i.
\end{equation}

\textcolor{revisions}{Then there is an isomorphism of chain complexes}

\begin{center}
\begin{tikzcd}
A_{n - 1}J_n \arrow[r, harpoon, shift left, "\Psi \otimes 1"] & A_n' \arrow[l, harpoon, shift left, "\Psi' \otimes 1"]
\end{tikzcd}
\end{center}
\end{corollary}

\textcolor{revisions}{Next, we show that to construct $\alpha$ from $A_{n - 1}$ to $A_{n - 1}J_n$, it suffices to construct a map $\tilde{\alpha}$ to the better behaved complex $A'_n$.}

\begin{proposition} \label{prop: reduction_to_tilde_alpha}
\textcolor{revisions}{Suppose $\tilde{\alpha} \colon A_{n - 1} \to A'_n$ is a degree $q^{-2}t^2$ chain map satisfying $\mathrm{Cone}(\tilde{\alpha}) \in \mathcal{I}^-_n$, and set $\alpha := (\Psi' \otimes 1) \tilde{\alpha} \colon A_{n - 1} \to A_{n - 1}J_n$. Then $\mathrm{Cone}(\alpha) \in \mathcal{I}^-_n$.}
\end{proposition}

\begin{proof}
\textcolor{revisions}{We furnish an explicit isomorphism between $\mathrm{Cone}(\tilde{\alpha})$ in the top row of the diagram below and $\mathrm{Cone}(\alpha)$ in the bottom row:}

\begin{center}
\begin{tikzcd}
t^{-1}A_{n - 1} \arrow[r, "\tilde{\alpha}"] \arrow[d, shift left, harpoon, "\mathrm{id}"] & q^2t^{-2} A'_n \arrow[d, shift left, harpoon, "\Psi' \otimes 1"] \\
t^{-1} A_{n - 1} \arrow[r, "\tilde{\alpha}"] \arrow[u, shift left, harpoon, "\mathrm{id}"] & q^2t^{-2} A_{n - 1}J_n \arrow[u, shift left, harpoon, "\Psi \otimes 1"]
\end{tikzcd}
\end{center}
\end{proof}

\textcolor{revisions}{As a consequence of Proposition \ref{prop: reduction_to_tilde_alpha}, to prove Theorem \ref{thm: projtopexists}, it suffices to exhibit an appropriate map $\tilde{\alpha}$.}

\begin{theorem} \label{thm: alphasquig}
There exists a closed map $\tilde{\alpha} \colon A_{n - 1} \to A'_n$ of degree $q^{-2}t^2$ such that $\text{Cone}(\tilde{\alpha}) \in \mathcal{I}_n^-$.
\end{theorem}

The proof of this theorem occupies the remainder of this \textcolor{revisions}{section}.

\subsubsection{Components of $\tilde{\alpha}$}

We are aided in our construction of $\tilde{\alpha}$ by the following two observations:

\textbf{Observation 1}: The domain and codomain of $\tilde{\alpha}$ are filtered by polynomial degree in each variable $u_i$. We refer to this as the ``periodic degree". The subquotients with respect to this filtration are $K_{n - 1}$ and $K'_n$ in the domain and codomain of $\tilde{\alpha}$, respectively.

\textbf{Observation 2}: The complexes $K_{n - 1}$, $K'_n$ are filtered by decreasing degree of each exterior variable $\theta_j$. We refer to this as the ``Koszul degree". \textcolor{revisions}{The subquotients with respect to this filtration are} $B_{w_1}$ and $B_{w_1}J_n$ in $K_{n - 1}$ and $K'_n$, respectively.

\textcolor{revisions}{We construct $\tilde{\alpha}$ as a sum of terms which are homogeneous with respect to both the periodic degree and the Koszul degree. The coefficient on each such homogeneous component will be a map from $B_{w_1}$ to $B_{w_1}J_n$, the degree of which can be pinned down by insisting that $\tilde{\alpha}$ be degree zero. Our next task is to identify these components and establish some useful properties.}

\begin{lemma}
\textcolor{revisions}{For each triple $(i,j,k)$ satisfying $2 \leq i, j \leq n - 1$ and $1 \leq k \leq n - 1$, there exists a polynomial $g_{ijk} \in R[\mathbb{X}, \mathbb{X}', \mathbb{X}'']$ of degree $i - 2$ satisfying}

\begin{equation} \label{eq: g_ijk_def}
f_{ij}(\mathbb{X}, \mathbb{X}') - f_{ij}(\mathbb{X}, \mathbb{X}'') = \sum_{k = 1}^{n - 1} g_{ijk} (x_k' - x_k'').
\end{equation}
\end{lemma}

\begin{proof}
\textcolor{revisions}{Observe that the left-hand side of Equation \eqref{eq: g_ijk_def} vanishes upon setting $x_k' = x_k''$ for each $1 \leq k \leq n - 1$. It follows immediately that the left-hand side is contained in the ideal generated by the terms $x_k' - x_k''$.}
\end{proof}

\textcolor{revisions}{Now for each pair $(i, j)$ satisfying $2 \leq i, j \leq n - 1$, we set}

\begin{equation}
\tilde{\varphi}_{ij} := \sum_{k = 1}^{n - 1} -g_{ijk}h_k \in \mathrm{End}^{-1}(B_{w_1}J_n); \quad \varphi_{ij} := \tilde{\varphi}_{ij} \mu \iota \in \mathrm{Hom}^1(B_{w_1}, B_{w_1}J_n).
\end{equation}

\begin{lemma} \label{lem: varphis_diff}
\textcolor{revisions}{The families $\{\tilde{\varphi}_{ij}\}$ and $\{\varphi_{ij}\}$ each pairwise anticommute and satisfy}

\begin{equation}  \label{eq: varphis_diff}
\textcolor{revisions}{[d, \varphi_{ij}] = \left(f_{ij}(\mathbb{X}, \mathbb{X}'') - f_{ij}(\mathbb{X}, \mathbb{X}') \right) \mu \iota}; \quad \textcolor{revisions}{[d, \tilde{\varphi}_{ij}] = f_{ij}(\mathbb{X}, \mathbb{X}'') - f_{ij}(\mathbb{X}, \mathbb{X}')}
\end{equation}
\end{lemma}

\begin{proof}
\textcolor{revisions}{Graded anticommutativity follows directly from the same property of the family of dot-sliding homotopies $\{h_i\}$. Equation \eqref{eq: varphis_diff} is a direct computation; by the graded Leibniz rule and Equation \eqref{eq: g_ijk_def}, we immediately obtain}

\[
[d, \varphi_{ij}] = - \sum_{k = 1}^{n - 1} g_{ijk} [d, h_k] \mu \iota = - \sum_{k = 1}^{n - 1} g_{ijk} (x_k' - x_k'') \mu \iota = \left(f_{ij}(\mathbb{X}, \mathbb{X}'') - f_{ij}(\mathbb{X}, \mathbb{X}') \right) \mu \iota.
\]
\end{proof}

\textcolor{revisions}{To identify the next component of $\tilde{\alpha}$ requires a brief investigation of the \textit{chain complex} structure on $\mathrm{Hom}^{\bullet}_{Ch(SBim_n)}(B_{w_1}, B_{w_1}J_n)$.}

\begin{lemma} \label{lem: hopf_hom}
\textcolor{revisions}{$H^1 \left( \mathrm{Hom}^{\bullet}_{Ch(SBim_n)}(B_{w_1}, B_{w_1}J_n) \right) = 0$.}
\end{lemma}

\begin{proof}
By the adjunction relation for $^{\vee}(-)$ of Section \ref{sec: SoergBim} \textcolor{revisions}{and Propositions \ref{prop: duals} and \ref{prop: tensid}}, we obtain a \textcolor{revisions}{sequence of isomorphisms}

\[
    \text{Hom}^{\bullet}(B_{w_1}, B_{w_1}J_n) \cong \text{Hom}^{\bullet}(R_n, ^{\vee}B_{w_1}B_{w_1}J_n) \cong \text{Hom}^{\bullet}(R_n, B_{w_1}B_{w_1}J_n) \cong [n - 1]! \ \text{Hom}^{\bullet}(R_n, B_{w_1}J_n).
\]

\textcolor{revisions}{Recall from Proposition \ref{prop: mu_nu} that} $B_{w_1}J_n \simeq C_n$. This homotopy equivalence induces a homotopy equivalence of morphism complexes, and hence an isomorphism on homology

\[
H^{\bullet}(\text{Hom}(R_n, B_{w_1}J_n)) \cong H^{\bullet}(\text{Hom}(R_n, C_n))
\]

Now, the right-hand side of this isomorphism is by definition the Hochschild degree $0$ piece of $HHH(C_n)$. By computations in \cite{HRW21}, $HHH(C_n)$ is concentrated in even homological degrees\footnote{In fact the authors in \cite{HRW21} compute the colored HOMFLYPT homology (as defined in \cite{WW17}) of the $(n - 1, 1)$-colored Hopf link\textcolor{revisions}{, which is related to $HHH(C_n)$ by a factor of $[n - 1]!$.}}.  In particular, we have $H^1(\text{Hom}^{\bullet}(B_{w_1}, B_{w_1}J_n)) = 0$.
\end{proof}

\begin{corollary} \label{cor: rhos}
\textcolor{revisions}{For each $2 \leq i \leq n - 1$, set}

\begin{equation}
\textcolor{revisions}{\zeta_i := \sum_{j = 2}^{n - 1} (x_j - x_j'') \varphi_{ij} + f_{ij}(\mathbb{X}, \mathbb{X}') h_j \mu \iota \in \mathrm{Hom}^1(B_{w_1}, B_{w_1}J_n).}
\end{equation}

\textcolor{revisions}{Then there exists $\rho_i \in \mathrm{Hom}^0(B_{w_1}, B_{w_1}J_n)$ satisfying $[d, \rho_i] = \zeta_i$.}
\end{corollary}

\begin{proof}
\textcolor{revisions}{By Lemma \ref{lem: hopf_hom}, it suffices to show that $\zeta_i$ is closed. This follows from a straightforward computation involving the graded Leibniz rule and Lemma \ref{lem: varphis_diff}.}
\end{proof}

\textcolor{revisions}{There is one more component of $\tilde{\alpha}$ we will need. For each $4$-tuple $(i, j, k, \ell)$ satisfying $2 \leq i, j, k, \ell \leq n - 1$, we set}

\begin{equation} \label{eq: omegas_i_neq_j}
\tilde{\omega}_{ijk\ell} := \tilde{\varphi}_{i\ell}\tilde{\varphi}_{jk} \in \text{End}^{-2}(B_{w_1}J_n); \quad \omega_{ijk\ell} := \tilde{\omega}_{ijk\ell} \mu \iota \in \mathrm{Hom}^0(B_{w_1}, B_{w_1}J_n).
\end{equation}

\begin{lemma} \label{lem: omega_ij_diff}
\textcolor{revisions}{The family $\{\omega_{ijk\ell}\}$ satisfy}

\begin{equation} \label{eq: omega_i_neq_j_dif}
\textcolor{revisions}{\omega_{ijk\ell} = -\omega_{ji\ell k};} \quad \textcolor{revisions}{[d, \omega_{ijk\ell}] = (f_{i\ell}(\mathbb{X}, \mathbb{X}'') - f_{i\ell}(\mathbb{X}, \mathbb{X}')) \varphi_{jk} + (f_{jk}(\mathbb{X}, \mathbb{X}') - f_{jk}(\mathbb{X}, \mathbb{X}''))\varphi_{i\ell}}.
\end{equation}

\end{lemma}

\begin{proof}
\textcolor{revisions}{The first property follows immediately from anticommutativity of the family $\{\tilde{\varphi}_{ij}\}$. The second is a direct computation using the graded Leibniz rule and Lemma \ref{lem: varphis_diff}.}
\end{proof}

\subsubsection{The Map $\tilde{\alpha}$}

\textcolor{revisions}{All the pieces are now in place to define the map $\tilde{\alpha}$. In the computations below, the first tensor factor is always a map from $B_{w_1}$ to $B_{w_1}J_n$, the middle tensor factor is always multiplication by some Koszul degree, and the last tensor factor is always multiplication by some periodic degree.}

\begin{proposition}
\textcolor{revisions}{Set}

\begin{equation} \label{eq: alpha_squig}
\textcolor{revisions}{\tilde{\alpha} := (\mu \iota \otimes 1) \otimes 1 + \sum_{i = 2}^{n - 1} \left(\rho_i \otimes 1 + \sum_{j = 2}^{n - 1} \varphi_{ij} \otimes \theta_j \right) \otimes u_i + \sum_{i = 2}^{n - 1} \sum_{j = 2}^{n - 1} \sum_{k = 2}^{n - 1} \sum_{\ell = k + 1}^{n - 1} (\omega_{ijk\ell} \otimes \theta_k \theta_{\ell}) \otimes u_iu_j.}
\end{equation}

\textcolor{revisions}{Then $\tilde{\alpha}$ is a degree $q^{-2}t^2$ chain map from $A_{n - 1}$ to $A'_n$.}
\end{proposition}

\begin{proof}
\textcolor{revisions}{This is a long, explicit computation. To make our lives easier, we verify that $[d, \tilde{\alpha}]$ in each periodic degree separately. For each vector $\textbf{v} = (v_2, \dots, v_{n - 1}) \in \mathbb{Z}^{n - 2}$ and each morphism $F$ we consider, we set $u^{\textbf{v}} := u_2^{v_2} \dots u_{n - 1}^{v_{n - 1}}$ and denote by $F_{u^{\textbf{v}}}$ the periodic degree $u^\textbf{v}$ component of $F$.}

\textbf{Periodic degree \textcolor{revisions}{$0$}}:

That $[d, \tilde{\alpha}]_{u^\textbf{0}} = 0$ is just the statement that $\mu \iota \otimes 1$ is closed.

\textbf{Periodic degree $u_i$}:

There are two contributions to $[d, \tilde{\alpha}]$ in this \textcolor{revisions}{degree}: $[d, \tilde{\alpha}]_{u_i} = [d_{u_i}, \tilde{\alpha}_{u^{\textbf{0}}}] + [d_{u^{\textbf{0}}}, \tilde{\alpha}_{u_i}]$. We \textcolor{revisions}{begin by computing} the first contribution\textcolor{revisions}{, reading off the terms $(d_{A_{n - 1}})_{u_i}$ and $(d_{A'_n})_{u_i}$ directly from Equations \eqref{eq: dan}}\footnote{\textcolor{revisions}{Note that even though Equation \eqref{eq: xi_i} involves the alphabet $\mathbb{X}'$ as written, here we always denote the right-module action of $R_n$ using the alphabet $\mathbb{X}''$.}} \textcolor{revisions}{and \eqref{eq: da'n}}:

\begin{align}
    [d_{u_i}, \tilde{\alpha}_{u^{\textbf{0}}}] & = \textcolor{revisions}{ \left( \sum_{j = 2}^{n - 1} f_{ij}(\mathbb{X}, \mathbb{X}')(1 \otimes \theta_j - h_j \otimes 1) \right) (\mu \iota \otimes 1) - (\mu \iota \otimes 1)  \left( \sum_{j = 2}^{n - 1} f_{ij}(\mathbb{X}, \mathbb{X}'') \otimes \theta_j \right)} \nonumber \\
    & = \sum_{j = 2}^{n - 1} \Big((f_{ij}(\mathbb{X}, \mathbb{X}') - f_{ij}(\mathbb{X}, \mathbb{X}'')) \mu \iota\Big) \otimes \theta_j - (f_{ij}(\mathbb{X}, \mathbb{X}')h_j \mu \iota) \otimes 1. \label{eq: temp_d1}
\end{align}

\textcolor{revisions}{Meanwhile, the second contribution gives}

\begin{align*}
\textcolor{revisions}{[d_{u^{\textbf{0}}}, \tilde{\alpha}_{u_i}]} & \textcolor{revisions}{= \left( d_{B_{w_1}J_n} \otimes 1 + \sum_{j = 2}^{n - 1} (x_j - x_j'') \otimes \theta_j^{\vee} \right) \left( \rho_i \otimes 1 + \sum_{j = 2}^{n - 1} \varphi_{ij} \otimes \theta_j \right)} \\
& \quad \textcolor{revisions}{- \left( \rho_i \otimes 1 + \sum_{j = 2}^{n - 1} \varphi_{ij} \otimes \theta_j \right) \left( \sum_{j = 2}^{n - 1} (x_j - x_j'') \otimes \theta_j^{\vee} \right)} \\
\end{align*}

\textcolor{revisions}{We combine terms of the same Koszul degree using the middle interchange law:}

\begin{align}
\textcolor{revisions}{[d_{u^{\textbf{0}}}, \tilde{\alpha}_{u_i}]} & \textcolor{revisions}{= [d, \rho_i] \otimes 1 + \sum_{j = 2}^{n - 1} [d, \varphi_{ij}] \otimes \theta_j + \sum_{j = 2}^{n - 1} (x_j - x_j'') \rho_i \otimes \theta_j^{\vee} - \sum_{j = 2}^{n - 1} (x_j - x_j'') \rho_i \otimes \theta_j^{\vee}} \nonumber \\
& \quad \textcolor{revisions}{- \left( \sum_{j, k = 2}^{n - 1} (x_j - x_j'') \varphi_{ik} \otimes (\theta_j^{\vee} \theta_k + \theta_k \theta_j^{\vee}) \right)} \nonumber \\
& \textcolor{revisions}{= \left( [d, \rho_i] - \sum_{j = 2}^{n - 1} (x_j - x_j'') \varphi_{ij} \right) \otimes 1 + \sum_{j = 2}^{n - 1} [d, \varphi_{ij}] \otimes \theta_j.} \label{eq: temp_d2}
\end{align}

\textcolor{revisions}{We can simplify \eqref{eq: temp_d2} using Lemma \ref{lem: varphis_diff} and Corollary \ref{cor: rhos}, obtaining}

\begin{align*}
\textcolor{revisions}{[d_{u^{\textbf{0}}}, \tilde{\alpha}_{u_i}]} & \textcolor{revisions}{= \left( \sum_{j = 2}^{n - 1} f_{ij}(\mathbb{X}, \mathbb{X}') h_j \mu \iota \right) \otimes 1 + \sum_{j = 2}^{n - 1} \left(f_{ij}(\mathbb{X}, \mathbb{X}'') - f_{ij}(\mathbb{X}, \mathbb{X}') \right) \mu \iota \otimes \theta_j.}
\end{align*}

\textcolor{revisions}{This exactly cancels the contribution from \eqref{eq: temp_d1}.}

\textbf{Periodic degree $u_iu_j$}:

\textcolor{revisions}{Since neither $d_{A_{n - 1}}$ nor $d_{A'_n}$ has any terms of quadratic or higher periodic degree, when $i \neq j$, w}e see three contributions to $[d, \tilde{\alpha}]$ in this \textcolor{revisions}{degree}:

\[
[d, \tilde{\alpha}]_{u_iu_j} = [d_{u_i}, \tilde{\alpha}_{u_j}] + [d_{u_j}, \tilde{\alpha}_{u_i}] + [d_{u^{\textbf{0}}}, \tilde{\alpha}_{u_iu_j}].
\]

\textcolor{revisions}{When $i = j$, we see only the first and third of these contributions. We again begin by computing the first contribution:}

\begin{align*}
\textcolor{revisions}{[d_{u_i}, \tilde{\alpha}_{u_j}]} & \textcolor{revisions}{= \left( \sum_{k = 2}^{n - 1} f_{ik}(\mathbb{X}, \mathbb{X}') (1 \otimes \theta_k - h_k \otimes 1) \right) \left(\rho_j \otimes 1 + \sum_{k = 2}^{n - 1} \varphi_{jk} \otimes \theta_k \right)} \\
& \quad \textcolor{revisions}{- \left(\rho_j \otimes 1 + \sum_{k = 2}^{n - 1} \varphi_{jk} \otimes \theta_k \right) \left( \sum_{k = 2}^{n - 1} f_{ik}(\mathbb{X}, \mathbb{X}'') \otimes \theta_k \right)} \\
& \textcolor{revisions}{= \sum_{k = 2}^{n - 1} f_{ik}(\mathbb{X}, \mathbb{X}') h_k \rho_j \otimes 1 + \sum_{k = 2}^{n - 1} \left( \left( f_{ik}(\mathbb{X}, \mathbb{X}') - f_{ik}(\mathbb{X}, \mathbb{X}'') \right) \rho_j - \sum_{\ell = 2}^{n - 1} f_{i\ell}(\mathbb{X}, \mathbb{X}')h_{\ell} \varphi_{jk} \right) \otimes \theta_k} \\
& \quad \textcolor{revisions}{- \sum_{k = 2}^{n - 1} \sum_{\ell = 2}^{n - 1} \left( f_{ik}(\mathbb{X}, \mathbb{X}') \varphi_{j\ell} + \varphi_{jk} f_{i\ell}(\mathbb{X}, \mathbb{X}'') \right) \otimes \theta_k \theta_{\ell}}.
\end{align*}

\textcolor{revisions}{Since $h_k \rho_j \in \mathrm{Hom}^{-1}(B_{w_1}, B_{w_1}J_n) = 0$, the first sum above vanishes. We may also rewrite the second sum using Lemma \ref{lem: varphis_diff}. In total, we obtain}

\begin{equation} \label{eq: temp_d3}
\textcolor{revisions}{[d_{u_i}, \tilde{\alpha}_{u_j}] = - \sum_{k = 2}^{n - 1} \left( [d, \tilde{\varphi}_{ik}] \rho_j + \sum_{\ell = 2}^{n - 1} f_{i\ell}(\mathbb{X}, \mathbb{X}')h_{\ell} \varphi_{jk} \right) \otimes \theta_k - \sum_{k = 2}^{n - 1} \sum_{\ell = 2}^{n - 1} \left( f_{ik}(\mathbb{X}, \mathbb{X}') \varphi_{j\ell} + \varphi_{jk} f_{i\ell}(\mathbb{X}, \mathbb{X}'') \right) \otimes \theta_k \theta_{\ell}.}
\end{equation}

\textcolor{revisions}{Now by the graded Leibniz rule, we know $[d, \tilde{\varphi}_{ik}] \rho_j = [d, \tilde{\varphi}_{ik} \rho_j] + \tilde{\varphi}_{ik} [d, \rho_j]$. Since $\tilde{\varphi}_{ik} \rho_j \in \mathrm{Hom}^{-1}(B_{w_1}, B_{w_1}J_n) = 0$, the first of these terms vanishes; we may rewrite the other using Corollary \ref{cor: rhos}. Meanwhile, the Koszul degree $\theta_k \theta_{\ell}$ component of Equation \eqref{eq: temp_d3} contains some redundancy due to the anticommutativity relation $\theta_k \theta_{\ell} = - \theta_{\ell} \theta_k$, which we can eliminate by restricting to indices $k + 1 \leq \ell \leq n - 1$. We again rewrite Equation \eqref{eq: temp_d3} making both of these adjustments:}

\begin{align}
\textcolor{revisions}{[d_{u_i}, \tilde{\alpha}_{u_j}]} & \textcolor{revisions}{= - \sum_{k = 2}^{n - 1} \left( \tilde{\varphi}_{ik} [d, \rho_j] + \sum_{\ell = 2}^{n - 1} f_{i\ell}(\mathbb{X}, \mathbb{X}')h_{\ell} \varphi_{jk} \right) \otimes \theta_k} \nonumber \\
& \quad \textcolor{revisions}{ - \sum_{k = 2}^{n - 1} \sum_{\ell = k + 1}^{n - 1} \left( f_{ik}(\mathbb{X}, \mathbb{X}') \varphi_{j\ell} + f_{i\ell}(\mathbb{X}, \mathbb{X}'') \varphi_{jk} - f_{i\ell}(\mathbb{X}, \mathbb{X}') \varphi_{jk} - f_{ik}(\mathbb{X}, \mathbb{X}'') \varphi_{j\ell} \right) \otimes \theta_k \theta_{\ell}} \nonumber \\
& \textcolor{revisions}{= - \sum_{k = 2}^{n - 1} \left( \tilde{\varphi}_{ik} \left( \sum_{\ell = 2}^{n - 1} (x_{\ell} - x_{\ell}'') \varphi_{j\ell} + f_{j\ell}(\mathbb{X}, \mathbb{X}') h_{\ell} \mu \iota \right) + \sum_{\ell = 2}^{n - 1} f_{i\ell}(\mathbb{X}, \mathbb{X}')h_{\ell} \varphi_{jk} \right) \otimes \theta_k} \nonumber \\
& \quad \textcolor{revisions}{- \sum_{k = 2}^{n - 1} \sum_{\ell = k + 1}^{n - 1} \left( (f_{ik}(\mathbb{X}, \mathbb{X}') - f_{ik}(\mathbb{X}, \mathbb{X}'')) \varphi_{j\ell} + ( f_{i\ell}(\mathbb{X}, \mathbb{X}'') - f_{i\ell} (\mathbb{X}, \mathbb{X}')) \varphi_{jk} \right) \otimes \theta_k \theta_{\ell}} \nonumber \\
& \textcolor{revisions}{= - \sum_{k = 2}^{n - 1} \left( \sum_{\ell = 2}^{n - 1} (x_{\ell} - x_{\ell}'') \omega_{ij \ell k} + (f_{j\ell}(\mathbb{X}, \mathbb{X}')\tilde{\varphi}_{ik}h_{\ell} + f_{i\ell}(\mathbb{X}, \mathbb{X}') h_{\ell} \tilde{\varphi}_{jk}) \mu \iota \right) \otimes \theta_k} \nonumber \\
& \quad \textcolor{revisions}{- \sum_{k = 2}^{n - 1} \sum_{\ell = k + 1}^{n - 1} \left( (f_{ik}(\mathbb{X}, \mathbb{X}') - f_{ik}(\mathbb{X}, \mathbb{X}'')) \varphi_{j\ell} + ( f_{i\ell}(\mathbb{X}, \mathbb{X}'') - f_{i\ell} (\mathbb{X}, \mathbb{X}')) \varphi_{jk} \right) \otimes \theta_k \theta_{\ell}.} \nonumber
\end{align}

\textcolor{revisions}{From here, we break into two cases. First, suppose $i = j$. Since $\tilde{\varphi}_{ik}$ and $h_{\ell}$ anticommute, the term $f_{j\ell}(\mathbb{X}, \mathbb{X}')\tilde{\varphi}_{ik}h_{\ell} + f_{i\ell}(\mathbb{X}, \mathbb{X}') h_{\ell} \tilde{\varphi}_{jk}$ vanishes in this case. Also in this case, by Lemma \ref{lem: omega_ij_diff}, the Koszul degree $\theta_k \theta_{\ell}$ component is exactly $-[d, \omega_{iik\ell}]$. In total, we end up with a contribution of the form}

\begin{equation} \label{eq: temp_d4}
\textcolor{revisions}{[d_{u_i}, \tilde{\alpha}_{u_i}] = - \sum_{k = 2}^{n - 1} \sum_{\ell = 2}^{n - 1} (x_k - x_k'') \omega_{iik\ell} \otimes \theta_{\ell} - \sum_{k = 2}^{n - 1} \sum_{\ell = k + 1}^{n - 1} [d, \omega_{iik\ell}] \otimes \theta_k \theta_{\ell}.}
\end{equation}

\textcolor{revisions}{Otherwise, when $i \neq j$, we obtain $[d_{u_j}, \tilde{\alpha}_{u_i}]$ by swapping $i$ and $j$ in the above expression. In this case, the terms of the form $f_{j\ell}(\mathbb{X}, \mathbb{X}')\tilde{\varphi}_{ik}h_{\ell} + f_{i\ell}(\mathbb{X}, \mathbb{X}') h_{\ell} \tilde{\varphi}_{jk}$ again cancel upon swapping $i$ and $j$. Similarly, the contributions in Koszul degree $\theta_k \theta_{\ell}$ add to produce $-[d, \omega_{ijk\ell}] - [d, \omega_{jik\ell}]$. In total, we end up with a contribution of the form}

\begin{equation} \label{eq: temp_d5}
\textcolor{revisions}{[d_{u_i}, \tilde{\alpha}_{u_j}] + [d_{u_j}, \tilde{\alpha}_{u_i}] = - \sum_{k = 2}^{n - 1} \sum_{\ell = 2}^{n - 1} (x_k - x_k'') (\omega_{ijk\ell} + \omega_{jik\ell}) \otimes \theta_{\ell} - \sum_{k = 2}^{n - 1} \sum_{\ell = k + 1}^{n - 1} ([d, \omega_{ijk\ell}] + [d, \omega_{jik\ell}]) \otimes \theta_k \theta_{\ell}.}
\end{equation}

\textcolor{revisions}{In this periodic degree, it remains to compute the contribution from $[d_{u^{\textbf{0}}}, \tilde{\alpha}_{u_iu_j}]$. We again break into cases. When $i = j$, we have}

\begin{align}
\textcolor{revisions}{[d_{u^{\textbf{0}}}, \tilde{\alpha}_{u_iu_i}]} & \textcolor{revisions}{= \left( d_{B_{w_1}J_n} \otimes 1 + \sum_{k = 2}^{n - 1} (x_k - x_k'') \otimes \theta_k^{\vee} \right) \left( \sum_{k = 2}^{n - 1} \sum_{\ell = k + 1}^{n - 1} \omega_{iik\ell} \otimes \theta_k \theta_{\ell} \right)} \nonumber \\
& \quad \textcolor{revisions}{- \left( \sum_{k = 2}^{n - 1} \sum_{\ell = k + 1}^{n - 1} \omega_{iik\ell} \otimes \theta_k \theta_{\ell} \right) \left( \sum_{k = 2}^{n - 1} (x_k - x_k'') \otimes \theta_k^{\vee} \right)} \nonumber \\
& \textcolor{revisions}{= \sum_{k = 2}^{n - 1} \sum_{\ell = k + 1}^{n - 1} [d, \omega_{iik\ell}] \otimes \theta_k \theta_{\ell} + \sum_{m = 2}^{n - 1} \sum_{k = 2}^{n - 1} \sum_{\ell = k + 1}^{n - 1} (x_m - x_m'') \omega_{iik\ell} \otimes \left( \theta_m^{\vee} \theta_k \theta_{\ell} - \theta_k \theta_{\ell} \theta_m^{\vee} \right).} \label{eq: temp_d6}
\end{align}

\textcolor{revisions}{The first term in \eqref{eq: temp_d6} exactly cancels the second term in \eqref{eq: temp_d4}. Meanwhile, $\theta_m^{\vee} \theta_k \theta_{\ell} - \theta_k \theta_{\ell} \theta_m^{\vee}$ vanishes unless either $m = k$, in which case it simplifies to $\theta_{\ell}$, or $m = \ell$, in which case it simplifies to $-\theta_k$. In total, the second term of \eqref{eq: temp_d6} simplifies to}

\begin{equation} \label{eq: temp_d7}
\textcolor{revisions}{\sum_{k = 2}^{n - 1} \sum_{\ell = k + 1}^{n - 1} (x_k - x_k'') \omega_{iik\ell} \otimes \theta_{\ell} - \sum_{k = 2}^{n - 1} \sum_{\ell = k + 1}^{n - 1} (x_{\ell} - x_{\ell}'') \omega_{iik\ell} \otimes \theta_k}
\end{equation}

\textcolor{revisions}{Using the property $\omega_{iik\ell} = - \omega_{ii\ell k}$ from Lemma \ref{lem: omega_ij_diff}, we can rewrite \eqref{eq: temp_d7} as}

\[
\textcolor{revisions}{\sum_{k = 2}^{n - 1} \sum_{\ell = 2}^{n - 1} (x_k - x_k'') \omega_{iik\ell} \otimes \theta_{\ell} - \left( \sum_{k = 2}^{n - 1} (x_k - x_k') \omega_{iikk} \otimes \theta_k \right).}
\]

\textcolor{revisions}{The first term exactly cancels the second term of \eqref{eq: temp_d4}, and the second term vanishes since $\omega_{iikk} = - \omega_{iikk} = 0$.}

\textcolor{revisions}{When $i \neq j$, the analysis is almost identical; the main difference is that now $\tilde{\alpha}_{u_iu_j}$ has two terms with coefficients $\omega_{ijk\ell}$ and $\omega_{jik\ell}$. In total, we see}

\begin{align*}
\textcolor{revisions}{[d_{u^{\textbf{0}}}, \tilde{\alpha}_{u_iu_j}]} & \textcolor{revisions}{= \sum_{k = 2}^{n - 1} \sum_{\ell = k + 1}^{n - 1} \left([d, \omega_{ijk\ell}] + [d, \omega_{jik\ell}] \right) \otimes \theta_k \theta_{\ell}} \nonumber \\
& \quad \textcolor{revisions}{+ \sum_{k = 2}^{n - 1} \sum_{\ell = 2}^{n - 1} (x_k - x_k'')(\omega_{ijk\ell} + \omega_{jik\ell}) \otimes \theta_{\ell} - \left( \sum_{k = 2}^{n - 1} (x_k - x_k'') (\omega_{ijkk} + \omega_{jikk}) \otimes \theta_k \right).}
\end{align*}

\textcolor{revisions}{This exactly cancels the contribution from \eqref{eq: temp_d5}.}

\textcolor{revisions}{\textbf{Periodic degree $u_iu_ju_k$}: In general, contributions in this periodic degree will be sums of terms of the following form:}

\begin{align}
\textcolor{revisions}{\left[ d_{u_i}, \sum_{\ell = 2}^{n - 1} \sum_{m = \ell + 1}^{n - 1} \omega_{ijk\ell} \otimes \theta_j \theta_k \right]} & \textcolor{revisions}{ = \left( \sum_{\ell = 2}^{n - 1} f_{i\ell}(\mathbb{X}, \mathbb{X}') (1 \otimes \theta_{\ell} - h_{\ell} \otimes 1) \right) \left( \sum_{\ell = 2}^{n - 1} \sum_{m = \ell + 1}^{n - 1} \omega_{ijk\ell} \otimes \theta_j \theta_k \right)} \nonumber \\
& \quad \textcolor{revisions}{- \left( \sum_{\ell = 2}^{n - 1} \sum_{m = \ell + 1}^{n - 1} \omega_{ijk\ell} \otimes \theta_j \theta_k \right) \left( \sum_{\ell = 2}^{n - 1} f_{i\ell}(\mathbb{X}, \mathbb{X}'') \otimes \theta_{\ell} \right)} \nonumber \\
& \textcolor{revisions}{ = \sum_{\ell = 2}^{n - 1} \sum_{m = 2}^{n - 1} \sum_{p = m + 1}^{n - 1} \left(f_{i\ell}(\mathbb{X}, \mathbb{X}') - f_{i\ell}(\mathbb{X}, \mathbb{X}'') \right) \omega_{jkmp} \otimes \theta_{\ell} \theta_m \theta_p} \nonumber \\
& \textcolor{revisions}{ = \sum_{\ell = 2}^{n - 1} \sum_{m = 2}^{n - 1} \sum_{p = m + 1}^{n - 1} -[d, \tilde{\varphi}_{i\ell}] \tilde{\varphi}_{jp} \tilde{\varphi}_{km} \mu \iota \otimes \theta_{\ell} \theta_m \theta_p}. \label{eq: temp_d8}
\end{align}

\textcolor{revisions}{Notice that we drop all terms involving $h_{\ell} \omega_{ijmp}$, which vanishes for degree reasons. The final line \eqref{eq: temp_d8} again contains some redundancies in Koszul degree. Upon reindexing this sum over the triples satisfying $\ell < m < p$ and applying anticommutativity of the $\theta$ variables, we may rewrite \eqref{eq: temp_d8}}

\begin{equation} \label{eq: temp_d9}
\textcolor{revisions}{\sum_{\ell < m < p} \left( -[d, \tilde{\varphi}_{i \ell}] \tilde{\varphi}_{jp} \tilde{\varphi}_{km} + [d, \tilde{\varphi}_{im}] \tilde{\varphi}_{jp} \tilde{\varphi}_{k \ell} - [d, \tilde{\varphi}_{ip}] \tilde{\varphi}_{jm} \tilde{\varphi}_{k\ell} \right) \mu \iota \otimes \theta_{\ell} \theta_m \theta_p}.
\end{equation}

\textcolor{revisions}{The total contribution to $[d, \tilde{\alpha}]$ in this periodic degree is given by summing \eqref{eq: temp_d9} over all distinct permutations of the triple $(i, j, k)$, resulting in a contribution which is symmetric in the indices $i, j, k$. Consequently, it suffices to show the contribution from terms involving only the maps $\tilde{\varphi}_{i\ell}, \tilde{\varphi}_{jp}$, and $\tilde{\varphi}_{km}$ vanish in each Koszul degree. This contribution is exactly}

\begin{align*}
\textcolor{revisions}{\left( -[d, \tilde{\varphi}_{i \ell}] \tilde{\varphi}_{jp} \tilde{\varphi}_{km} + [d, \tilde{\varphi}_{km}] \tilde{\varphi}_{jp} \tilde{\varphi}_{i\ell} - [d, \tilde{\varphi}_{jp}] \tilde{\varphi}_{km} \tilde{\varphi}_{i\ell} \right) \mu \iota} & \textcolor{revisions}{= \left( -[d, \tilde{\varphi}_{i \ell}] \tilde{\varphi}_{jp} \tilde{\varphi}_{km} + \tilde{\varphi}_{i\ell} [d, \tilde{\varphi}_{jp}] \tilde{\varphi}_{km} - \tilde{\varphi}_{i\ell} \tilde{\varphi}_{jp} [d, \tilde{\varphi}_{km}] \right) \mu \iota} \\
& \textcolor{revisions}{ = -[d, \tilde{\varphi}_{i\ell} \tilde{\varphi}_{jp} \tilde{\varphi}_{km} \mu \iota]}
\end{align*}

\textcolor{revisions}{where here we have used the anticommutativity of the family $\{ \tilde{\varphi}_{ij}\}$ and the graded Leibniz rule. Since $\tilde{\varphi}_{i\ell} \tilde{\varphi}_{jp} \tilde{\varphi}_{km} \mu \iota \in \mathrm{Hom}^{-1}(B_{w_1}, B_{w_1}J_n) = 0$, this contribution vanishes.}

\textcolor{revisions}{Finally, since $[d, \tilde{\alpha}]$ contains no terms of higher periodic degree, we have $[d, \tilde{\alpha}] = 0$ in all degrees.}

\end{proof}

\begin{proposition}
\textcolor{revisions}{The map $\tilde{\alpha}$ defined in \eqref{eq: alpha_squig} satisfies $\mathrm{Cone}(\tilde{\alpha}) \in \mathcal{I}_n^-$.}
\end{proposition}

\begin{proof}
\textcolor{revisions}{First, note that $\mathrm{Cone}(\tilde{\alpha})$ is certainly bounded above, since $\tilde{\alpha}$ is a map between two complexes which are bounded above. Additionally, since $\tilde{\alpha}$ respects the filtration by periodic degree on $A_{n - 1}$ and $A'_n$, $\mathrm{Cone}(\tilde{\alpha})$ is similarly filtered. It follows that we have}

\[
\textcolor{revisions}{\mathrm{Cone}(\tilde{\alpha}) = \mathrm{tw}_{\beta} \left( \mathrm{Cone}(\mu \iota \otimes 1) \otimes \mathbb{Z}[u_2^{-1}, \dots, u_{n - 1}^{-1}] \right)}
\]

\textcolor{revisions}{for some twist $\beta$. We have seen in Proposition \ref{prop: krec} that $\mathrm{Cone}(\mu \iota \otimes 1) \simeq K_n$ up to an overall degree shift. Since $\mathbb{Z}[u_2^{-1}, \dots, u_{n - 1}^{-1}]$ is an upper finite indexing set, we may apply Proposition \ref{prop: useful_hpt} to rewrite $\mathrm{Cone}(\tilde{\alpha})$ as a convolution of copies of $K_n$. But $K_n \in \mathcal{I}_n^-$ by construction.}
\end{proof}

\section{Y-ification} \label{big_sec: yification}

\subsection{The Finite Projector} \label{sec: finite}

Methods developed by Elias-Hogancamp-Mellit over the last ten years use (bounded versions of) the projectors $P_n$ to compute triply graded homology of positive torus links (\cite{EH19}, \cite{Mel22}, \cite{HM19}). As in their work, we prefer to deal with finite versions of our projectors to ensure that the complexes involved in computing homology are bounded\footnote{This approach will necessarily be abandoned when computing \textit{colored} homology later in this work, though our computations with bounded complexes will lay the groundwork for that approach as well.}. These finite projectors will be appropriately shifted copies of the Koszul complex $K_n$ constructed above.

\begin{remark}
In the computations of (\cite{EH19}, \cite{Mel22}, \cite{HM19}) utilizing row projectors $P_n$, the finite projector was constructed inductively by taking a tensor product with the fundamental domain of their infinite projector. The primary advantage of that approach is an easy proof of the appropriate analog of Proposition \ref{prop: krec}. This reduction also works in our setting, though we will find the explicit description of $K_n$ as a Koszul complex more useful for partial trace computations; see Proposition \ref{prop: ylessktrace} below.
\end{remark}

We collect the properties of $K_n$ relevant to the computation of link homology below:

\begin{proposition} \label{prop: k.markov.unnorm.uny}
The family of bounded complexes $K_n$ considered above satisfy:

(1) $K_1 = R_1$;

(2) $K_n \in \mathcal{I}_n^-$;

(3) $K_{n - 1}J_n \simeq (q^{n - 3}t K_n \longrightarrow q^{-2}t^2K_{n - 1})$.
\end{proposition}

\begin{proof}
The first two properties were discussed above. Property (3) follows from the homotopy equivalence $\text{Cone}(\overline{\alpha}) \simeq q^{n - 1}t^{-1} K_n$ of Proposition \ref{prop: krec} after rotating triangles.
\end{proof}

The general strategy for utilizing these finite projectors to compute triply graded homology goes as follows. Identify a full twist on two strands as a subword of a braid presentation of the given link. Using Properties (1) and (3) above, we may rewrite the Rouquier complex associated to that braid as a convolution of complexes built out of simpler braids and the finite projector $K_2$. Repeatedly apply Property (3) in this way, decomposing the Rouquier complex for the original braid as a (potentially very large) convolution of finite projectors together with simple braids.

Now, the problem of computing $HHH$ has been reduced to the two problems of computing Hochschild cohomology of the simpler complexes and computing the homology of the resulting convolution as a garden-variety chain complex. With any luck, repeated application of the partial trace Markov moves above are sufficient for the first task. While the second task might seem forbiddingly complicated (indeed, keeping track of the differentials involved in this decomposition would be intractable for large braids), we are aided by the following significant simplification:

\begin{proposition} \label{prop: paritymiracle}
Let $C = \text{tw}_{\alpha}(\bigoplus_{j \in J} F_j)$ be a one-sided convolution of the complexes $F_j$, and suppose $H^i(F_j) = 0$ for all odd $i$. Then $H^i(C) \cong \bigoplus_{j \in J} H^i(F_j)$ for each $i$.
\end{proposition}

\begin{proof}
Notice that the indexing set $J$ gives a bounded filtration on $C$. The $E_1$ page of the spectral sequence computing $H^{\bullet}(C)$ induced by this filtration is given by $\bigoplus_{j \in J} H^{\bullet}(F_j)$; in particular, this page is concentrated in even homological degree. Since the differential on this page has homological degree $1$, no component of the differential can map between nonzero terms. It follows that the spectral sequence degenerates at this page, so $H^{\bullet}(C) \cong \bigoplus_{j \in J} H^{\bullet}(F_j)$.
\end{proof}

An appropriately normalized version of Property (3) of Proposition \ref{prop: k.markov.unnorm.uny} will contain only even homological shifts. For this reason, Proposition \ref{prop: paritymiracle} has the potential to significantly simplify our computations. The only remaining ingredient is the computation of the Hochschild cohomology of the simpler complexes appearing in the resulting convolution. Because these will contain finite projectors $K_j$ as tensor factors, it is essential to understand how $K_j$ behaves under partial trace.

\begin{proposition} \label{prop: ylessktrace}
For all $n > 1$, we have $Tr_n(K_n) \simeq (1 + q^2t^{-1})\left(\cfrac{q^{n - 1} + aq^{-n - 1}}{1 - q^2}\right) K_{n - 1}$.
\end{proposition}

\begin{proof}
We apply $Tr_n$ termwise to the complex $K_n \cong B_{w_0} \otimes \bigwedge[\theta_2, \dots, \theta_n]$. On each chain group, we have $Tr_n(B_{w_0}) \cong \left( \cfrac{q^{n - 1} + aq^{-n - 1}}{1 - q^2} \right) B_{w_1} \in \mathcal{D}_{n - 1}$ by Proposition \ref{prop: trb}. Applying $Tr_n$ to the differential $d_{K_n}$ kills the term $(x_n - x_n') \otimes \theta_n^{\vee}$. What remains is the complex $\left( \cfrac{q^{n - 1} + aq^{-n - 1}}{1 - q^2} \right) B_{w_1} \otimes \bigwedge[\theta_2, \dots, \theta_n]$ with differential $d = \sum_{j = 2}^{n - 1} (x_j - x_j') \otimes \theta_j$. This is exactly $\left( \cfrac{q^{n - 1} + aq^{-n - 1}}{1 - q^2} \right) (K_{n - 1} \oplus \theta_n K_{n - 1}) \cong (1 + q^2t^{-1})\left( \cfrac{q^{n - 1} + aq^{-n - 1}}{1 - q^2} \right)K_{n - 1}$.
\end{proof}

Notice that $Tr_n(K_n)$ contains two copies of $K_{n - 1}$ with an odd degree relative homological shift between them. This odd relative shift cannot be renormalized away, so we cannot hope to apply the methods of (\cite{EH19}, \cite{Mel22}, \cite{HM19}) \textit{mutatis mutandis} to compute with this projector. Recent work of Gorsky-Hogancamp-Mellit nevertheless establishes a mirror symmetry relationship for uncolored links similar to the expected relationship between $P_{1^n}$ and $P_n$ after passing to \textit{$y$-ified} homology (\cite{GHM21}). Inspired by their results, we turn to the task of $y$-ifying our projector; as we will see, this resolves the parity issue of Proposition \ref{prop: ylessktrace}.

\subsection{Y-ification} \label{sec: yification}

In this section we review the theory of $y$-ification as introduced in \cite{GH22}. Our exposition and notation largely mirrors that of \cite{GH22}; the reader familiar with that work can safely skip this section.

\begin{definition}
Let $\Gamma$ be an abelian group and $S$ a $\Gamma \times \mathbb{Z}$-graded ring. Let $\mathcal{A}$ be a graded $S$-linear category (so that Hom spaces in $\mathcal{A}$ are graded $S$-modules). Let $Z \in S$ be a homogeneous element with $\text{deg}(Z) = (0, 2)$. A \textit{$Z$-factorization} in $\mathcal{A}$ is an ordered pair $(C, \delta)$ with $C \in \mathcal{A}$, $\delta \in \text{End}^1(C)$ satisfying $\delta^2 = Z$. We often refer to $Z$-factorizations as \textit{curved complexes} with curvature $Z$ and connection $\delta$.
\end{definition}

For each such $Z \in S$, we let $\text{Fac}(\mathcal{A}, Z)$ denote the category with objects $Z$-factorizations in $\mathcal{A}$ and morphism spaces inherited from $\mathcal{A}$. Observe that $Z$ is central in $\text{Hom}_{\text{Fac}(\mathcal{A}, Z)}$ by virtue of $S$-linearity of $\mathcal{A}$.

\begin{proposition}
$\text{Fac}(\mathcal{A}, Z)$ is a $\Gamma \times \mathbb{Z}$-graded dg category over $S$.
\end{proposition}

\begin{proof}
Let $(X, \delta_X)$, $(Y, \delta_Y) \in \text{Fac}(\mathcal{A}, Z)$ be given. As usual, given a morphism $f$ in $\mathcal{A}$, we reserve the notation $|f| $ for the disinguished "homological" $\mathbb{Z}$-component of the grading on morphism spaces. We declare the differential on $\text{Hom}_{\text{Fac}(\mathcal{A}, Z)}(X, Y)$ to be 

$$d_{\text{Fac}(\mathcal{A}, Z)} \colon f \mapsto [\delta, f] := \delta_Y \circ f - (-1)^{|f|} f \circ \delta_X$$

\textcolor{revisions}{That this differential squares to $0$ and satisfies the graded Leibniz rule are easy computations.} The $\Gamma \times \mathbb{Z}$-grading on $\text{Fac}(\mathcal{A}, Z)$ is given by formal grading shifts.
\end{proof}

\begin{example}
There is a dg-equivalence of categories $\text{Fac}(\text{Seq}(\mathcal{C}), 0) \cong \text{Ch}(\mathcal{C})$.
\end{example}

Our primary tool for establishing homotopy equivalences, Gaussian elimination, extends naturally to the curved setting.

\begin{proposition}[Gaussian Elimination for Curved Complexes]
\textcolor{revisions}{In the situtation of Proposition \ref{prop: gauss_elim}, the vertical arrows still constitute a strong deformation retraction when the top row is a piece of a $Z$-factorization in $\mathrm{Fac}(\mathrm{Seq}(\mathcal{A}, Z))$.}
\end{proposition}

All curved complexes we consider will be of the following form.

\begin{example} \label{ex: zfac}
Let $\mathbb{X}$, $\mathbb{X}'$ be alphabets of $n$ variables each as usual. Let $R_n = R[\mathbb{X}]$ as a $\mathbb{Z} \times \mathbb{Z}$-graded ring, where the first factor is the usual quantum grading and the second factor is the homological grading. In particular, we have $\text{deg}(x_i) = q^2 = (2,0)$ for each $x_i$. Let $\mathbb{Y} := \{y_1, \dots, y_n\}$ be an alphabet of formal variables of degree $\text{deg}(y_i) = q^{-2}t^2$ for each $i$, and consider the polynomial ring $R[\mathbb{Y}]$ as a dg $R$-algebra with trivial differential.

Given any complex $C \in \text{Ch}(R_n-\text{Bim})$ \textcolor{revisions}{and considering $R[\mathbb{Y}]$ as a chain complex with trivial differential as in the previous paragraph, we denote the tensor product of these two complexes by} $C[\mathbb{Y}] := C \otimes_R R[\mathbb{Y}]$.\footnote{Note that the number of variables in the alphabet $\mathbb{Y}$ depends on the size of the alphabets $\mathbb{X}, \mathbb{X}'$. This number of variables will always be clear in context; as such, we decline to distinguish among the variously-sized alphabets to avoid unnecessary notational clutter.} Note that $C[\mathbb{Y}]$ has a natural structure as a dg $R[\mathbb{X}, \mathbb{X}', \mathbb{Y}]$-module with the action of $\mathbb{Y}$ given by multiplication in the second tensor factor. Because of this, we will often denote elements of $C[\mathbb{Y}]$ multiplicatively, writing e.g. $y_i^2C^j = C^jy_i^2 = C^j \otimes y_i^2$. Note also that $C$ naturally includes into $C[\mathbb{Y}]$ as a subcomplex $C \otimes 1$; we will make implicit use of this inclusion without further comment.

Let $\text{Ch}_y(R_n-\text{Bim})$ denote the category with objects chain complexes $C[\mathbb{Y}]$ for some $C \in \text{Ch}(R_n-\text{Bim})$ and graded morphism spaces 

\[
\text{Hom}^k_{\text{Ch}_y(R_n-\text{Bim})}(C[\mathbb{Y}], D[\mathbb{Y}]) := \prod_{i \in \mathbb{Z}} \left( \bigoplus_{j = k - i} \text{Hom}_{R_n-\text{Bim}}(C^i, D^{i + k - j}) \otimes_R (R[\mathbb{Y}])^j \right)
\]

Then $\text{Ch}_y(R_n-\text{Bim})$ is a graded $R[\mathbb{X}, \mathbb{X}', \mathbb{Y}]$-linear category, and given any $Z \in R[\mathbb{X}, \mathbb{X}', \mathbb{Y}]$ with $\text{deg}(Z) = t^2$, we may consider the category $Z-\text{Fac}(R_n-\text{Bim}) := \text{Fac}(\text{Ch}_y(R_n-\text{Bim}), Z)$ of $Z$-factorizations.
\end{example}

\begin{remark} \label{rem: decomp}
For each $\textbf{v} = (v_1, \dots, v_n) \in \mathbb{Z}^n_{\geq 0}$, we let $y^{\textbf{v}} := y_1^{v_1} \dots v_n^{v_n}$, $|\textbf{v}| := v_1 + \dots + v_n$. Let $(C[\mathbb{Y}], \delta_C)$, $(D[\mathbb{Y}], \delta_D) \in Z-\text{Fac}(R_n-\text{Bim})$ be given. Then given any homogeneous morphism $f \in \text{Hom}_{\text{Ch}_y(R_n-\text{Bim})}(C[\mathbb{Y}], D[\mathbb{Y}])$ of degree $q^rt^s$, we may separate $f$ into its $\mathbb{Y}$-components:

\[
f = \sum_{\textbf{v} \in \mathbb{Z}^n_{\geq 0}} f_{y^\textbf{v}} y^{\textbf{v}}
\]

where $f_{y^\textbf{v}} \in \text{Hom}_{\text{Ch}(R_n-\text{Bim})}(C, D)$ is a homogeneous morphism of degree $q^{r + 2|\textbf{v}|} t^{s - 2|\textbf{v}|}$ for each $\textbf{v}$.
\end{remark}

\begin{remark} \label{rem: convergence}
The reader familiar with $y$-ification will notice that our convention for morphism spaces differs from the usual definition $\text{Hom}_{\text{Ch}_y(R_n-\text{Bim})}(C[\mathbb{Y}], D[\mathbb{Y}]) = \text{Hom}_{\text{Ch}(R_n-\text{Bim}}(C, D)[\mathbb{Y}]$ from e.g. \cite{GH22}. When dealing with semi-infinite complexes, our convention allows for formal power series in the alphabet $\mathbb{Y}$ in the decomposition of Remark \ref{rem: decomp} but restricts to morphisms for which this decomposition is finite on each homological degree in the domain. This atypical convention affords necessary leeway when dealing with semi-infinite complexes but introduces less pathological divergence than allowing for all formal power series. When $C$ and $D$ are bounded complexes, there is no difference between these two conventions.
\end{remark}

\begin{definition}
Retain notation as in Remark \ref{rem: decomp}. We call $f$ a \textit{curved lift} of the morphism $f_{y^{\textbf{0}}} \in \text{Hom}_{\text{Ch}(R_n-\text{Bim})}$ when we wish to emphasize its degree $0$ component. We call $(C[\mathbb{Y}], \delta_C)$ a \textit{curved lift} of $(C, d_C)$ if $\delta_C$ is a curved lift of $d_C$. We often make this property explicit by writing $\delta_C = d_C + \Delta_C$, where $\Delta_C = \sum_{|\textbf{v}| \geq 1} (\delta_C)_{y^\textbf{v}} y^\textbf{v}$. We refer to $\Delta_C$ as a \textit{curved twist} of $C$ and will often denote the curved complex $(C[\mathbb{Y}], \delta_C)$ by $\text{tw}_{\Delta_C}(C[\mathbb{Y}])$ (notice the similarities with Section \S \ref{sec: twist}).
\end{definition}

One naturally wonders whether the Homological Perturbation Lemma extends to the setting of curved complexes. This is the subject of homological perturbation theory with curvature; see \cite{Hog20} for details. We will not need the full strength of that theory, but we will use the following result:

\begin{proposition} \label{prop: hpt_iso_yified}
Let $f \in \text{Hom}^0(C, D)$, $g \in \text{Hom}^0(D, C)$ be the data of an isomorphism $C \cong D$ in $\text{Ch}(R_n-\text{Bim})$, and let $\text{tw}_{\Delta_C}(C[\mathbb{Y}])$ be a curved lift of $C$ with curvature $Z$. Then $f, g$ are isomorphisms $\text{tw}_{\Delta_C}(C[\mathbb{Y}]) \cong \text{tw}_{f\Delta_C g}(D[\mathbb{Y}])$ in $Z-\text{Fac}(R_n-\text{Bim})$.
\end{proposition}

\begin{proof}
We begin by showing that $\text{tw}_{f\Delta_Cg}(D[\mathbb{Y}])$ is a $Z$-factorization via direct computation.

\begin{align*}
    (d_D + f\Delta_Cg)^2 & = d_D^2 + d_Df\Delta_Cg + f\Delta_Cgd_D + (f\Delta_Cg)(f\Delta_Cg) \\
    & = f(d_C \Delta_C + \Delta_C d_C + \Delta_C^2)g \\
    & = f(d_C + \Delta_C)^2g \\
    & = Z
\end{align*}

A similar computation shows that $[\delta, f] = [\delta, g] = 0$, so that $f, g$ lift to chain maps of curved complexes; we omit the details, as they are easily checked. That $f$ and $g$ are mutually inverse in $Z-\text{Fac}(R_n-\text{Bim})$ follows immediately from the corresponding fact in $\text{Ch}(R_n-\text{Bim})$.
\end{proof}

Note that the only elements of $R[\mathbb{X}, \mathbb{X}', \mathbb{Y}]$ of degree $t^2$ are of the form $Z = \sum_{i = 1}^n f_i(\mathbb{X}, \mathbb{X}')y_i$ for some homogeneous degree $1$ polynomials $f_i \in R[\mathbb{X}, \mathbb{X}']$. We will primarily be interested in the special case 

\begin{equation} \label{eq: yify_curv}
Z_{\sigma} := \sum_{i = 1}^n (x_{\sigma(i)} - x_i')y_i
\end{equation}

for some $\sigma \in \mathfrak{S}^n$.

\begin{definition}
Let $(C, d_C) \in \text{Ch}(R_n-\text{Bim})$ be given. A \textit{$y$-ification} of $C$ is an ordered triple $(C[\mathbb{Y}], \sigma, \delta_C)$ such that $(C[\mathbb{Y}], \delta_C)$ is a curved lift of $(C, d_C)$ with curvature $Z_{\sigma}$ \textcolor{revisions}{as defined in Equation \eqref{eq: yify_curv}}. For each $\sigma \in \mathfrak{S}^n$, let $\mathcal{Y}_{\sigma}(R_n-\text{Bim})$ denote the full subcategory of $Z_{\sigma}-\text{Fac}(R_n-\text{Bim})$ generated by (direct sums of shifts of) y-ifications $(C[\mathbb{Y}], \sigma, \delta_C)$. There is a natural inclusion functor $\mathcal{Y}_{\sigma}(R_{n - 1}-\text{Bim}) \hookrightarrow \mathcal{Y}_{\sigma}(R_n-\text{Bim})$ for each $\sigma \in \mathfrak{S}^{n - 1}$; we consider objects in the former also as objects in the latter under this inclusion without further comment. 
\end{definition}

The monoidal structure on $R_n-\text{Bim}$ extends to $y$-ifications, though some care must be taken with permutations. Let two $y$-ifications $(C[\mathbb{Y}], \sigma, \delta_C)$ and $(D[\mathbb{Y}], \rho, \delta_D)$ be given. We regard $C[\mathbb{Y}]$ and $D[\mathbb{Y}]$ as $R[\mathbb{Y}]$-bimodules, where the left and right actions on $C[\mathbb{Y}]$ and the right action on $D[\mathbb{Y}]$ are given by the usual multiplication, but the \textit{left} action on $D[\mathbb{Y}]$ is given by multiplication by $\rho^{-1}(\mathbb{Y})$. Set $(C \otimes D)[\mathbb{Y}] := C[\mathbb{Y}] \otimes_{R_n[\mathbb{Y}]} D[\mathbb{Y}]$ with the usual tensor product differential $\delta_C \otimes 1 + 1 \otimes \delta_D$.

\begin{proposition}
Retain notation from the previous paragraph. Then $((C \otimes D)[\mathbb{Y}], \sigma \rho, \delta_C \otimes 1 + 1 \otimes \delta_D)$ is a $y$-ification of $C \otimes D$.
\end{proposition}

We call a curved lift $f$ of a morphism $f_{y^\textbf{0}}$ \textit{strict} if $(f)_{y^\textbf{v}} = 0$ for all $|\textbf{v}| > 1$; that is, $f$ is linear in the alphabet $\mathbb{Y}$. Similarly, we call a $y$-ification $(C[\mathbb{Y}], \sigma, \delta_C)$ \textit{strict} if $\delta_C$ is a strict lift of $d_C$. Note that a strict $y$-ification $(C[\mathbb{Y}], \sigma, \delta_Y)$ is equivalent to a collection of homogeneous morphisms $\Delta_i \in \text{End}_{\text{Ch}(R_n-\text{Bim})}(C)$ of degree $q^2t^{-1}$ satisfying $(d_C + \sum_{i = 1}^n \Delta_i \otimes y_i)^2 = Z_{\sigma}$. Separating this equation into $\mathbb{Y}$-components, we obtain the following conditions:

\begin{align*}
    [d_C, \Delta_i] & = x_{\sigma(i)} - x'_i \\
    \Delta_i \Delta_j + \Delta_j \Delta_i & = 0 \quad \text{for} \ i \neq j \\
    \Delta_i^2 & = 0
\end{align*}

Given a braid $\beta \in Br_n$, the dot-sliding homotopies $\{h_i\}_{i = 1}^n$ relating the actions of $x_{\beta(i)}$ and $x_i'$ on $F(\beta)$ satisfy these identities for \textcolor{revisions}{the induced permutation $\sigma_{\beta} \in \mathfrak{S}^n$}. As an immediate consequence, we obtain

\begin{proposition} \label{prop: ybraid}
For each braid $\beta \in Br_n$, the triple $F^y(\beta) := (F(\beta)[\mathbb{Y}], \textcolor{revisions}{\sigma_{\beta}}, d_{F(\beta)} + \sum_{i = 1}^n h_i \otimes y_i)$ is a strict $y$-ification of $F(\beta)$.
\end{proposition}

In fact $y$-ifications of Rouquier complexes are unique up to homotopy equivalence; we will not require this result here.

Next, we recall an obstruction-theoretic technique for lifting morphisms of chain complexes to morphisms of $y$-ifications developed in \cite{GH22}. Aside from explicit computations, this will be our primary tool for constructing morphisms.

\begin{proposition} \label{prop: GHobs}
Let $C, D \in \text{Ch}(R_n-\text{Bim})$ be given, and suppose $H^i(\text{Hom}_{\text{Ch}(R_n-\text{Bim})}(C, D)) = 0$ for all $i < N$ for some fixed $N$. Let $(C[\mathbb{Y}], \sigma, \delta_C)$, $(D[\mathbb{Y}], \sigma, \delta_D)$ be $y$-ifications of $C$ and $D$, respectively, and let some homogeneous $g \in \text{Hom}_{\mathcal{Y}_{\sigma}(R_n-\text{Bim})}(C[\mathbb{Y}], D[\mathbb{Y}])$ be given with degree $\text{deg}(g) = q^rt^s$ for some $s \leq N + 1$. Then there exists a formal power series

\begin{align*}
    f = \sum_{\textbf{v} \in \mathbb{Z}_{\geq 0}^n} f_{y^{\textbf{v}}} \otimes y^{\textbf{v}}; \quad \quad f_{\textbf{v}} \in \text{Hom}^{s - 1 - 2\textbf{v}}_{Ch(R_n-\text{Bim})}(C, D)
\end{align*}

satisfying $[\delta, f] = g$ if and only if there exists $f_0 \in \text{Hom}^{r, s - 1}_{\text{Ch}(R_n-\text{Bim})}(C, D)$ such that $[d, f_0] = g_{y^\textbf{0}}$. In this case, we have $f_{y^\textbf{0}} = f_0$.
\end{proposition}

We omit the proof; see \cite{GH22} for details\footnote{In fact the authors in \cite{GH22} restrict their attention to the cases of Corollaries \ref{corr: GHobs} and \ref{corr: GHobs2}. Their argument goes through with only cosmetic adjustments in the generality of Proposition \ref{prop: GHobs} above.}. The following two corollaries will almost always suffice for our purposes.

\begin{corollary} \label{corr: GHobs}
Let $C$, $D \in \text{Ch}^b(R_n-\text{Bim})$ be given, and suppose $\text{Hom}_{\text{Ch}(R_n-\text{Bim})}(C, D)$ has homology concentrated in non-negative degrees. Let $f: C \rightarrow D$ be a chain map, and let $(C[\mathbb{Y}], \sigma, \delta_C)$, $(D[\mathbb{Y}], \sigma, \delta_D)$ be $y$-ifications of $C$ and $D$, respectively. Then $f$ lifts to a chain map $\tilde{f}: C[\mathbb{Y}] \rightarrow D[\mathbb{Y}]$ in $\mathcal{Y}_{\sigma}(R_n-\text{Bim})$ satisfying $\tilde{f}_{y^\textbf{0}} = f$.
\end{corollary}

\begin{proof}
Take $N = 0$, $g = 0 \in \text{Hom}^1_{\mathcal{Y}_{\sigma}(R_n-\text{Bim})}(C[\mathbb{Y}], D[\mathbb{Y}])$ in Proposition \ref{prop: GHobs}, and observe that boundedness of $C$ and $D$ implies $f_{y^{\textbf{v}}} = 0$ for $|\textbf{v}| >> 0$.
\end{proof}

\begin{corollary} \label{corr: GHobs2}
Retain notation as in Corollary \ref{corr: GHobs}, and again suppose $\text{Hom}_{\text{Ch}(R_n-\text{Bim})}(C, D)$ has homology concentrated in non-negative degrees. Then a morphism $g \in \text{Hom}^0_{\mathcal{Y}_{\sigma}(R_n-\text{Bim})}(C[\mathbb{Y}], D[\mathbb{Y}])$ is nullhomotopic if and only if $g_{y^{\textbf{0}}} \in \text{Hom}^0_{\text{Ch}(R_n-\text{Bim})}(C, D)$ is nullhomotopic. 
\end{corollary}

\begin{proof}
Take $N = 0$ in Proposition \ref{prop: GHobs}.
\end{proof}

We conclude this section with another application of homological perturbation theory with curvature; this is Lemma 2.19 in \cite{GH22} (appropriately adapted to our convention for morphisms; see Remark \ref{rem: convergence}).

\begin{lemma} \label{lem: y_contract}
If $C \in \text{Ch}^+(R_n-\text{Bim})$ is contractible, then any $y$-ification of $C$ is contractible.
\end{lemma}

\subsection{Y-ified Fork Slide}

We now turn to the task of $y$-ifying the complexes constructed in Section 3. In \cite{Elb22}, the author constructed explicit $y$-ifications for the projectors $P_{1^n}$ in all Coxeter groups. As in the uncurved case, we take this explicit description as the starting point for our recursive construction, beginning with the finite projector.

\begin{proposition}
\textcolor{revisions}{Set} 

\begin{equation} \label{eq: kn_conn}
\delta_{K_n} := \sum_{j = 2}^n (x_j - x_j') \otimes \theta_j^{\vee} + (y_j - y_1) \otimes \theta_j \in \text{End}_{\text{Ch}_y(R_n-\text{Bim})}(K_n[\mathbb{Y}]).
\end{equation}

Then $K_n^y := (K_n[\mathbb{Y}], e, \delta_{K_n})$ is a (strict) $y$-ification of $K_n$.
\end{proposition}

\begin{proof}
\textcolor{revisions}{What needs to be verified is that $\delta_{K_n}^2 = \sum_{j = 1}^n (x_j - x_j')y_j$. This is a straightforward computation; see \cite{Elb22} for the details.}
\end{proof}

Let $J_n^y := F^y(J_n)$. Then $K_{n - 1}^yJ_n^y = (K_{n - 1}J_n[\mathbb{Y}], e, \delta_{K_{n - 1}J_n})$ is a (strict) $y$-ification of $K_{n - 1}J_n$ with connection

\begin{equation} \label{eq: knjn_conn}
\delta_{K_{n - 1}J_n} = \text{id}_{B_{w_1}}(d_{J_n} + \sum_{i = 1}^n h_i y_i) \otimes 1 + \sum_{j = 2}^{n - 1} (x_j - x_j') \otimes \theta_j^{\vee} + (y_j - y_1) \otimes \theta_j
\end{equation}

By Proposition \ref{prop: hpt_iso_yified}, the isomorphisms $\Psi, \Psi'$ from Corollary \ref{corr: building_psi} lift without modification to isomorphisms $K_{n - 1}^yJ_n^y \cong K^{\prime y}_n := (K'_n[\mathbb{Y}], e, \delta_{K'_n})$, where $\delta_{K'_n} = d_{K'_n} + \Psi \Delta_{K_{n - 1}J_n} \Psi'$.

We compute $\Delta_{K'_n} = \Psi \Delta_{K_{n - 1}J_n} \Psi'$ explicitly using the expression $\Delta_{K_{n - 1}J_n} = (\sum_{i = 1}^n h_iy_i \otimes 1) + (\sum_{j = 2}^{n - 1} (y_j - y_1) \otimes \theta_j)$, which can be read off directly \textcolor{revisions}{from \eqref{eq: knjn_conn}}. We treat each summand separately. Since $y_j - y_1$ has even homological degree and is central in $\text{Ch}_y(R_n-\text{Bim})$, the computation of $\Psi (\sum_{i = 1}^n (y_i - y_1) \otimes \theta_i) \Psi'$ is nearly identical to that of Proposition \ref{prop: xi_conj}. We only quote the result:

$$\Psi (\sum_{j = 1}^{n - 1} (y_j - y_1) \otimes \theta_i) \Psi' = \sum_{j = 2}^{n - 1} (y_j - y_1)(1 \otimes \theta_j - h_j \otimes 1)$$

In the other component, upon conjugating a given summand $h_jy_j \otimes 1$ by \textcolor{revisions}{a single factor $1 \otimes 1 - h_i \otimes \theta_i^{\vee}$ of $\Psi$}, we see the following:

\begin{align*}
    (1 \otimes 1 - h_i \otimes \theta_i^{\vee})(h_jy_j \otimes 1)(1 \otimes 1 + h_i \otimes \theta_i^{\vee})
    & = h_jy_j \otimes 1 + h_jy_jh_i \otimes \theta_i^{\vee} + h_ih_jy_j \otimes \theta_i^{\vee} \\
    & = h_jy_j \otimes 1 + (h_jh_i + h_ih_j)y_j \otimes \theta_i^{\vee} \\
    & = h_jy_j \otimes 1
\end{align*}

\textcolor{revisions}{Conjugating by each of these factors and summing over $i$ results in} $\Psi (\sum_{i = 1}^n h_iy_i \otimes 1) \Psi' = \sum_{i = 1}^n h_iy_i \otimes 1$. Summing both contributions, we obtain

\begin{align*}
    \Delta_{K_n'} & = \left(\sum_{j = 1}^n h_jy_j\right) \otimes 1 + \sum_{j = 2}^{n - 1} (y_j - y_1)(1 \otimes \theta_j - h_j \otimes 1) \\
    & = \left(h_1y_1 + h_ny_n + \sum_{j = 2}^{n - 1} h_jy_1 \right) \otimes 1 + \sum_{j = 2}^{n - 1} (y_j - y_1) \otimes \theta_j \\
    & = h_n(y_n - y_1) \otimes 1 + \sum_{j = 2}^{n - 1} (y_j - y_1) \otimes \theta_j
\end{align*}

In the last equality we have used the identity $\sum_{j = 1}^n h_j = 0$. We record the final result below.

\begin{proposition} \label{prop: y_forkslide_i}
The maps $\Psi, \Psi'$ of Corollary \ref{corr: building_psi} lift without modification to isomorphisms $K_{n - 1}^yJ_n^y \cong K^{\prime y}_n$ in $\mathcal{Y}_e(R_n-\text{Bim})$. Here

\begin{equation} \label{eq: k'n_conn}
\delta_{K'_n} = \text{id}_{B_{w_1}}(d_{J_n} + h_n(y_n - y_1)) \otimes 1 + \sum_{j = 2}^{n - 1} (x_j - x_j'') \otimes \theta_j^{\vee} + (y_j - y_1) \otimes \theta_j
\end{equation}
\end{proposition}

% Computation that shows that $\Delta^2 = Z$:

% \begin{align*}
%     \Delta^2 & = B_{w_1}\Delta_{J_n}^2 \otimes 1 + \sum_{j, k = 2}^{n - 1} (y_j - y_1)(x_k - x_k'') \otimes (\theta_j\theta_k^{\vee} + \theta_k^{\vee}\theta_j) + \sum_{j = 2}^{n - 1} -B_{w_1}(y_j - y_1)(\Delta_{J_n}h_j + h_j\Delta_{J_n}) \otimes 1 \\
%     & = \sum_{j = 1}^n (x_j' - x_j'')y_j \otimes 1 + \sum_{j = 2}^{n - 1} (x_j'' - x_j')(y_j - y_1) \otimes 1 + (x_j - x_j'')(y_j - y_1) \otimes 1 \\
%     & = \sum_{j = 1}^n (x_j' - x_j'')y_j \otimes 1 + \sum_{j = 2}^{n - 1} (x_j - x_j')(y_j - y_1) \otimes 1 \\
%     & = \left((x_1' - x_1'')y_1 + (x_n' - x_n'')y_n  - (e_1(\mathbb{X}) - x_1 - x_n - (e_1(\mathbb{X}') - x_1' - x_n'))y_1 + \sum_{j = 2}^{n - 1} (x_j - x_j'')y_j\right) \otimes 1 \\
%     & = \sum_{j = 1}^n (x_j - x_j'')y_j \otimes 1
% \end{align*}

To continue our program of $y$-ifying $P_{1^n}$, we would like to lift the homotopy equivalence $B_{w_1}J_n \simeq C_n$ of Section \S \ref{sec: fork_slide} to a the $y$-ified setting. Let $B_{w_1}J_n^{y_n}$ denote the curved lift of $B_{w_1}J_n$ with twist $\Delta_{B_{w_1}J_n} = h_n(y_n - y_1)$; this curved complex has curvature $Z_{y_n} := (x_n' - x_n'')(y_n - y_1) = (x_n - x_n'')(y_n - y_1)$. To lift the results of Section \S \ref{sec: fork_slide}, we will also need some curved lift $C_n^{y_n}$ of $C_n$ in $Z_{y_n}-\text{Fac}(R_n-\text{Bim})$. The correct connection on $C_n^{y_n}$ is fairly easy to guess and is indicated below. Here the curved twist $\Delta_{C_n}$ is depicted by a backwards red arrow.

\begin{center}
\begin{tikzcd}[sep=large]
C_n^{y_n} := q^{n - 1}B_{w_0}[\mathbb{Y}] \arrow[r, "x_n - x_n''", harpoon, shift left] & q^{n - 3}tB_{w_0}[\mathbb{Y}] \arrow[r, "unzip", harpoon, shift left] \arrow[l, "y_n - y_1", harpoon, shift left, red] & q^{-2}t^2B_{w_1}[\mathbb{Y}] \arrow[l, "0", harpoon, shift left]
\end{tikzcd}
\end{center}

\begin{proposition} \label{prop: forkslyde}
The map $\nu \colon B_{w_1}J_n \to C_n$ of Proposition \ref{prop: mu_nu} lifts without modification to a chain map $\nu \colon B_{w_1}J_n^{y_n} \to C_n^{y_n}$ of curved complexes. The map $\mu$ of the same proposition has a strict lift to a chain map $\tilde{\mu} \colon C_n^{y_n} \to B_{w_1}J_n^{y_n}$ of curved complexes. These maps constitute a strong deformation retraction from $B_{w_1}J_n^{y_n}$ to $C_n^{y_n}$; that is, $\nu \tilde{\mu} = \text{id}_{C_n^{y_n}}$, $\tilde{\mu} \nu \sim \text{id}_{B_{w_1}J_n^{y_n}}$.
\end{proposition}

\begin{proof}
\textcolor{revisions}{See Appendix \ref{app: dot_slide_nat}.}
\end{proof}

Assuming the result of Proposition \ref{prop: forkslyde}, all that remains in constructing the $y$-ified finite projector is to lift Proposition \ref{prop: krec} to the $y$-ified setting.

\begin{proposition} \label{prop: fynite}
\textcolor{revisions}{Set}

\begin{equation}
\overline{\alpha}^y := \Psi' (\tilde{\mu} \iota \otimes 1) \colon K_{n - 1}^y \to q^2t^{-2}K_{n - 1}^y J_n^y.
\end{equation}

\textcolor{revisions}{Then $\overline{\alpha}^y$ is a strict lift of $\overline{\alpha}$} satisfying $\text{Cone}(\overline{\alpha}^y) \simeq q^{n - 1}t^{-1} K_n^y$.
\end{proposition}

\begin{proof}
Consider the curved complex $C_n^{y_n} \otimes \bigwedge[\theta_2, \dots, \theta_n]$ with connection

\[
\delta_{C_n \otimes \bigwedge[\theta_2, \dots, \theta_n]} = \delta_{C_n} \otimes 1 + \sum_{j = 2}^n (x_j - x_j'') \otimes \theta_j^{\vee} + (y_j - y_1) \otimes \theta_j
\]

It is easily verified that $(C_n^{y_n} \otimes \bigwedge[\theta_2, \dots, \theta_n], e, \delta_{C_n \otimes \bigwedge[\theta_2, \dots, \theta_n]})$ is a (strict) $y$-ification of $C_n \otimes \bigwedge[\theta_2, \dots, \theta_n]$. \textcolor{revisions}{As in the un $y$-ified setting,} $q^{-2}t^2 K_{n - 1}^y J_n^y$ appears as a subcomplex with inclusion map $\iota \otimes 1$. Gaussian elimination along this inclusion gives a homotopy equivalence

\begin{align*}
    \text{Cone}(\iota \otimes 1) \simeq q^{n - 1}t^{-1} K_n^y
\end{align*}

Since $\tilde{\mu}, \nu$ are $R[\mathbb{X}, \mathbb{X}'', \mathbb{Y}]$-linear, they extend to a strong deformation retraction $\tilde{\mu} \otimes 1$, $\nu \otimes 1$ from $K_n^{\prime y}$ to $C_n^{y_n} \otimes \bigwedge[\theta_2, \dots, \theta_n]$. Completely analogously to Proposition \ref{prop: krec}, we obtain a chain of homotopy equivalences

\begin{align*}
    \text{Cone}(\overline{\alpha}^y) = \text{Cone}(\Psi'(\tilde{\mu} \iota \otimes 1)) \simeq \text{Cone}(\tilde{\mu} \iota \otimes 1) \simeq \text{Cone}(\iota \otimes 1) \simeq q^{n - 1}t^{-1} K_n^y \in \mathcal{Y}_e(R_n-\text{Bim})
\end{align*}

\end{proof}

\subsection{Y-ified Infinite Projector} \label{sec: yify_inf_proj}

We now turn to a $y$-ification of our recursion for the infinite projector. We begin by $y$-ifying the explicit combinatorial projector $A_n$.

\begin{proposition} \label{prop: combyproj}
Let $A_n := K_n \otimes \mathbb{Z}[u_2^{-1}, \dots, u_n^{-1}]$ as above, and \textcolor{revisions}{set} 

\begin{equation}
\delta_{A_n} := \delta_{K_n} \otimes 1 + \sum_{i = 2}^n \xi_i \otimes u_i \in \text{End}_{\text{Ch}_y(R_n-\text{Bim})}(A_n[\mathbb{Y}])
\end{equation}

Then $A_n^y := (A_n[\mathbb{Y}], e, \delta_{A_n})$ is a (strict) $y$-ification of $A_n$.
\end{proposition}

\begin{proof}
Notice that $\delta_{A_n} = d_{A_n} + \Delta_{K_n} \otimes 1$\textcolor{revisions}{; in particular, the only components of $\delta_{A_n}$ with nonzero $y$-degree have periodic degree $0$. With this in mind, it suffices to show that $\Delta_{K_n}$ commutes with the nonzero periodic degree component $\sum_{i = 2}^n \xi_i \otimes u_i$ of $d_{A_n}$. This is easily verified.}
\end{proof}

Our next task is to lift each of the maps $\alpha, \beta$ from Theorem \ref{thm: projtopexists} to maps of $y$-ifications. 

\begin{proposition}
The map $j \colon R_n \hookrightarrow J_n$ given by inclusion in $0^{th}$ homology lifts without modification to a chain map $j \colon R_n[\mathbb{Y}] \hookrightarrow J_n^y$ of $y$-ifications.
\end{proposition}

\begin{proof}
It suffices to verify that $\Delta_{J_n}j = 0$. Let $\{h_i\}_{i = 1}^n$ denote the dot-sliding homotopies on $J_n$; then $\Delta_{J_n} = \sum_{i = 1}^n h_iy_i$. Observe that $h_ij \in \text{Hom}^{-1}_{\text{Ch}(R_n-\text{Bim})}(R_n, J_n)$ for each $i$. Since $J_n$ is concentrated in non-negative homological degrees, this morphism space is $0$, so $h_ij = 0$ for each $i$. Then $\Delta_{J_n}j = \sum_{i = 1}^n (h_ij)y_i = 0$ as well.
\end{proof}

\begin{corollary}
The map $\beta = \text{id}_{A_{n - 1}} \otimes j \colon A_{n - 1} \to A_{n - 1}J_n$ lifts without modification to a chain map $\beta \colon A_{n - 1}^y \to A_{n - 1}^yJ_n^y$.
\end{corollary}

The remainder of this subsection is dedicated to the proof of the following.

\begin{theorem} \label{thm: yalpha}
The map $\alpha$ of Section \S \ref{sec: alpha_build} has a strict lift to a chain map $\alpha^y \colon A_{n - 1}^y \to q^2t^{-2}A_{n - 1}^yJ_n^y$.
\end{theorem}

Our proof of Theorem \ref{thm: yalpha} will follow the same framework as our construction of $\alpha$. We first pass through a $y$-ification of the complex $A_n'$. We have already verified that $\Psi, \Psi'$ lift without modification to chain maps of $y$-ifications. Since the curved twist $\Delta_{A_{n - 1}}$ has no periodic components, applying $\Psi' \otimes 1$ to $A_{n - 1}^yJ_n^y$ immediately gives the following result.

\begin{proposition} \label{prop: inf_y_base_change}
\textcolor{revisions}{Set}

\begin{equation} \label{eq: a'n_conn}
\delta_{A'_n} := \delta_{K'_{n - 1}} \otimes 1 + \sum_{j = 2}^{n - 1} (\Psi \xi_j \Psi') \otimes u_j \in \text{End}_{\text{Ch}_y(R_n-\text{Bim})}(A'_n[\mathbb{Y}]).
\end{equation}

Then $A^{\prime y}_n := (A'_n[\mathbb{Y}], e, \delta_{A'_n})$ is a (strict) $y$-ification of $A'_n$. Moreover, we have $A^{\prime y}_{n - 1} \cong A_{n - 1}^yJ_n^y$ via isomorphisms $\Psi \otimes 1, \Psi' \otimes 1$.
\end{proposition}

Theorem \ref{thm: yalpha} then immediately follows from the following result.

\begin{proposition}
The map $\tilde{\alpha}$ of Theorem \ref{thm: alphasquig} has a strict lift to a chain map $\tilde{\alpha}^y \colon A^y_{n - 1} \to q^2t^{-2} A^{\prime y}_n$.
\end{proposition}

\begin{proof}
We set $\tilde{\alpha}^y := \tilde{\alpha} + \sum_{i = 1}^n ((\overline{\alpha})_{y_i} \otimes 1) y_i$; this is nothing more than the map $\tilde{\alpha}$ with the periodic degree $0$ component $\overline{\alpha} \otimes 1$ replaced by its lift $\overline{\alpha}^y \otimes 1$ from Proposition \ref{prop: fynite}. For notational convenience, we refer to the nontrivial $y$-degree piece of $\tilde{\alpha}^y$ as $\tilde{\alpha}_y$, so that $\tilde{\alpha}^y = \tilde{\alpha} + \tilde{\alpha}_y$. By the results of Appendix \ref{app: dot_slide_nat}, we can read off $\tilde{\alpha}_y = (y_n - y_1) (\overline{\mu} \tilde{\chi} \iota \otimes 1) \otimes 1$.

We compute $[\delta, \tilde{\alpha}^y]$ explicitly:

\begin{align*}
    [\delta, \tilde{\alpha}^y] & = (d_{A'_n} + \Delta_{A'_n})(\tilde{\alpha} + \tilde{\alpha}_y) - (\tilde{\alpha} + \tilde{\alpha}_y) (d_{A_{n - 1}} + \Delta_{A_{n - 1}}) \\
    & = [d, \tilde{\alpha}] + (\Delta_{A'_n} \tilde{\alpha} - \tilde{\alpha} \Delta_{A_{n - 1}}) + (d_{A'_n} \tilde{\alpha}_y - \tilde{\alpha}_y d_{A_{n - 1}}) + (\Delta_{A'_n} \tilde{\alpha}_y - \tilde{\alpha}_y \Delta_{A_{n - 1}})
\end{align*}

Since $\tilde{\alpha}$ is already known to be a chain map of uncurved complexes, the term $[d, \tilde{\alpha}]$ above vanishes. To see that $\tilde{\alpha}^y$ is a chain map, we check explicitly that the contributions from the other three grouped terms above cancel. As in our proof of Theorem \ref{thm: alphasquig}, we \textcolor{revisions}{collect our computations by periodic degree}. Note that none of $\Delta_{A'_n}$, $\Delta_{A_{n - 1}}$, or $\tilde{\alpha}_y$ has any nonzero periodic degree components, so the third grouped term above will only play a role in periodic degree $u^\textbf{0}$ computations.

\textbf{Periodic degree $u^\textbf{0}$}:

That $[\delta, \tilde{\alpha}^y]_{u^\textbf{0}} = 0$ follows immediately from Proposition \ref{prop: fynite}.

\textbf{Periodic degree $u_i$}:

Since $\Delta_{A_{n - 1}}$ and $\Delta_{A'_n}$ have no nonzero periodic degree components, the only contribution from $\Delta_{A'_n} \tilde{\alpha} - \tilde{\alpha} \Delta_{A_{n - 1}}$ in this periodic degree is of the form

\begin{align*}
    (\Delta_{A'_n} \tilde{\alpha} - \tilde{\alpha} \Delta_{A_{n - 1}})_{u_i} & = \Delta_{A'_n} (\tilde{\alpha})_{u_i} - (\tilde{\alpha})_{u_i} \Delta_{A_{n - 1}} \\
    & = \left( \left( h_n(y_n - y_1) \otimes 1 + \sum_{j = 2}^{n - 1} (y_j - y_1) \otimes \theta_j \right) \otimes 1 \right) (H_i \otimes u_i) \\
    & \quad - (H_i \otimes u_i) \left( \left( \sum_{j = 2}^{n - 1} (y_j - y_1) \otimes \theta_j \right) \otimes 1 \right) \\
    & = \left( (h_n (y_n - y_1) \otimes 1)H_i + \sum_{j = 2}^{n - 1} (y_j - y_1) ((1 \otimes \theta_j)H_i - H_i(1 \otimes \theta_j)) \right) \otimes u_i \\
    & = \Bigg( (h_n (y_n - y_1) \otimes 1) \left( \rho_i \otimes 1 + \sum_{k = 2}^{n - 1} \tilde{\varphi}_{ik} \mu \iota \otimes \theta_k \right) \\
    & \quad + \sum_{j = 2}^{n - 1} (y_j - y_1) \Bigg((1 \otimes \theta_j) \left( \rho_i \otimes 1 + \sum_{k = 2}^{n - 1} \tilde{\varphi}_{ik} \mu \iota \otimes \theta_k \right) \\
    & \quad - \left( \rho_i \otimes 1 + \sum_{k = 2}^{n - 1} \tilde{\varphi}_{ik} \mu \iota \otimes \theta_k \right) (1 \otimes \theta_j)\Bigg) \Bigg) \otimes u_i \\
    & = \Bigg( h_n \rho_i (y_n - y_1) \otimes 1 + \sum_{j = 2}^{n - 1} (y_n - y_1) h_n\tilde{\varphi}_{ij} \mu \iota \otimes \theta_j \\
    & \quad - \sum_{j, k = 2}^{n - 1} (y_j - y_1) \tilde{\varphi}_{ik}\mu \iota \otimes (\theta_k \theta_j + \theta_j \theta_k) \Bigg) \otimes u_i \\
    & = \left(h_n \rho_i (y_n - y_1) \otimes 1 + \sum_{j = 2}^{n - 1} (y_n - y_1) h_n \tilde{\varphi}_{ij} \mu \iota \otimes \theta_j \right) \otimes u_i
\end{align*}

Since $h_n \rho_i \in \text{Hom}^{-1}(B_{w_1}, B_{w_1}J_n) = 0$, the first term above vanishes. Since $\tilde{\varphi}_{ij}$ is a polynomial in the alphabets $\mathbb{X}$, $\mathbb{X}'$ and $h_n$ is $R[\mathbb{X}, \mathbb{X}']$-linear, we have $h_n \tilde{\varphi}_{ij} \mu \iota = \tilde{\varphi}_{ij} h_n \mu \iota$.

We factor $\mu$ as $\mu = \overline{\mu} g$ as in the proof of Proposition \ref{prop: forkslyde} in Appendix \ref{app: dot_slide_nat}. Then by the results of that section, we have

\[
h_n \mu \iota = h_n \overline{\mu} g \iota = \overline{\mu} h g \iota = \overline{\mu} (gh' - [d, \tilde{\chi}]) \iota
\]

It is easily verified using the explicit description of $h'$ given in Proposition \ref{prop: forkslyde} that $h' \iota = 0$. What remains of the above expression is

\[
h_n \mu \iota = \overline{\mu} \tilde{\chi} d_{C_n} \iota - \overline{\mu} d_{STM_n} \tilde{\chi} \iota
\]

Again, it is easily checked that $d_{C_n} \iota = 0$. In total, we see

\[
(\Delta_{A'_n} \tilde{\alpha} - \tilde{\alpha} \Delta_{A_{n - 1}})_{u_i} = (y_1 - y_n) \left( \sum_{j = 2}^{n - 1} \tilde{\varphi}_{ij} \overline{\mu} d_{STM_n} \tilde{\chi} \iota \otimes \theta_j \right) \otimes u_i
\]

Moving on to the term $d_{A'_n} \tilde{\alpha}_y - \tilde{\alpha}_y d_{A_{n - 1}}$, we are again aided by the fact that $\tilde{\alpha}_y$ has no terms of nonzero periodic degree. Then the only contribution in this periodic degree from this term is of the form

\begin{align*}
    (d_{A'_n} \tilde{\alpha}_y - \tilde{\alpha}_y d_{A_{n - 1}})_{u_i} & = (d_{A'_n})_{u_i} \tilde{\alpha}_y - \tilde{\alpha}_y (d_{A_{n - 1}})_{u_i} \\
    & = (\Psi \xi_i \Psi' \otimes u_i) ((\overline{\mu} \tilde{\chi} \iota \otimes 1) \otimes 1) (y_n - y_1) - ((\overline{\mu} \tilde{\chi} \iota \otimes 1) \otimes 1) (\xi_i \otimes u_i) (y_n - y_1) \\
    & = \Big((\Psi \xi_i \Psi' (\overline{\mu} \tilde{\chi} \iota \otimes 1) - (\overline{\mu} \tilde{\chi} \iota \otimes 1) \xi_i) \otimes u_i \Big) (y_n - y_1)
\end{align*}

Substituting known expressions for $\xi_i$ and $\Psi \xi_i \Psi'$:

\begin{align*}
    \Psi \xi_i \Psi' (\overline{\mu} \tilde{\chi} \iota \otimes 1) - (\overline{\mu} \tilde{\chi} \iota \otimes 1) \xi_i & = \left( \sum_{j = 2}^{n - 1} f_{ij}(\mathbb{X}, \mathbb{X}') (1 \otimes \theta_j - h_j \otimes 1) \right) (\overline{\mu} \tilde{\chi} \iota \otimes 1) \\
    & \quad - (\overline{\mu} \tilde{\chi} \iota \otimes 1) \left( \sum_{j = 2}^{n - 1} f_{ij}(\mathbb{X}, \mathbb{X}'') \otimes \theta_j \right) \\
    & = \sum_{j = 2}^n (f_{ij}(\mathbb{X}, \mathbb{X}') - f_{ij}(\mathbb{X}, \mathbb{X}'')) \overline{\mu} \tilde{\chi} \iota \otimes \theta_j - f_{ij}(\mathbb{X}, \mathbb{X}') h_j \overline{\mu} \tilde{\chi} \iota \otimes 1
\end{align*}

Since $h_j \overline{\mu} \tilde{\chi} \iota \in \text{Hom}^{-1}_{\text{Ch}(SBim_n)}(B_{w_1}, B_{w_1}J_n) = 0$, the second term above vanishes. We rewrite the remainder in terms of the maps $\tilde{\varphi}_{ij}$:

\begin{align*}
    (f_{ij}(\mathbb{X}, \mathbb{X}') - f_{ij}(\mathbb{X}, \mathbb{X}'')) \overline{\mu} \tilde{\chi} \iota \otimes \theta_j & = -[d, \tilde{\varphi}_{ij}] \overline{\mu} \tilde{\chi} \iota \otimes \theta_j \\
    & = (\tilde{\varphi}_{ij} d_{B_{w_1}J_n} - d_{B_{w_1}J_n} \tilde{\varphi}_{ij}) \overline{\mu} \tilde{\chi} \iota \otimes \theta_j
\end{align*}

Again, since $\tilde{\varphi}_{ij} \overline{\mu} \tilde{\chi} \iota \in \text{Hom}^{-1}_{\text{Ch}(SBim_n)}(B_{w_1}, B_{w_1}J_n) = 0$, the second term in parentheses vanishes. Since $\overline{\mu}$ is a chain map, we again rewrite the remainder as

\begin{align*}
    \tilde{\varphi}_{ij} d_{B_{w_1}J_n} d_{B_{w_1}J_n} \tilde{\varphi}_{ij}) \overline{\mu} \tilde{\chi} \iota \otimes \theta_j & = \tilde{\varphi}_{ij} \overline{\mu} d_{STM_n} \tilde{\chi} \iota \otimes \theta_j
\end{align*}

Collecting our results, we obtain

\[
(d_{A'_n} \tilde{\alpha}_y - \tilde{\alpha}_y d_{A_{n - 1}})_{u_i} = (y_n - y_1) \left( \sum_{j = 2}^{n - 1} \tilde{\varphi}_{ij} \overline{\mu} d_{STM_n} \tilde{\chi} \iota \otimes \theta_j \right) \otimes u_i
\]

This exactly cancels the contribution from $\Delta_{A'_n} \tilde{\alpha} - \tilde{\alpha} \Delta_{A_{n - 1}}$ in this periodic degree.

\textbf{Periodic degree $u_iu_j$, $i \leq j$}:

We again begin by checking the term $\Delta_{A'_n} \tilde{\alpha} - \tilde{\alpha} \Delta_{A_{n - 1}}$ in this degree:

\begin{align*}
    (\Delta_{A'_n}\tilde{\alpha} - \tilde{\alpha}\Delta_{A_{n - 1}})_{u_iu_j} & = \Delta_{A'_n}(\tilde{\alpha})_{u_iu_j} - (\tilde{\alpha})_{u_iu_j} \Delta_{A_{n - 1}} \\
    & = \left( \left( h_n(y_n - y_1) \otimes 1 + \sum_{k = 2}^{n - 1} (y_k - y_1) \otimes \theta_k \right) \otimes 1 \right) (H_{ij} \otimes u_iu_j) \\
    & \quad - (H_{ij} \otimes u_iu_j) \left( \left( \sum_{k = 2}^{n - 1} (y_k - y_1) \otimes \theta_k \right) \otimes 1 \right) \\
    & = \Bigg( \left( h_n(y_n - y_1) \otimes 1 + \sum_{k = 2}^{n - 1} (y_k - y_1) \otimes \theta_k \right) \Bigg(\sum_{2 \leq \ell < m}^{n - 1} \tilde{\omega}_{ij \ell m} \mu \iota \otimes \theta_{\ell} \theta_m \Bigg) \\
    & \quad - \Bigg(\sum_{2 \leq \ell < m}^{n - 1} \tilde{\omega}_{ij \ell m} \mu \iota \otimes \theta_{\ell} \theta_m \Bigg) \Bigg( \sum_{k = 2}^{n - 1} (y_k - y_1) \otimes \theta_k \Bigg) \Bigg) \otimes u_iu_j \\
    & = \Bigg( \sum_{2 \leq \ell < m}^{n - 1} (y_n - y_1) h_n \tilde{\omega}_{ij\ell m} \otimes \theta_{\ell} \theta_m \\
    & \quad + \sum_{k = 2}^{n - 1} \sum_{2 \leq \ell < m}^{n - 1} (y_k - y_1) \tilde{\omega}_{ij\ell m} \mu \iota \otimes (\theta_k \theta_{\ell} \theta_m - \theta_{\ell} \theta_m \theta_k) \Bigg) \otimes u_iu_j \\
    & = 0
\end{align*}

Here the final equality follows from anticommutativity of the variables $\theta$ and $h_n \tilde{\omega}_{ij\ell m} \in \text{Hom}^{-1}(B_{w_1}, B_{w_1}J_n) = 0$. Since $d_{A'_n}$ and $d_{A_{n - 1}}$ have no components of periodic degree larger than $u_i$, we see no contribution from $d_{A'_n} \tilde{\alpha}_y - \tilde{\alpha}_u d_{A_{n - 1}}$ in this periodic degree.

\textbf{Higher periodic degrees}: None of $\Delta_{A'_n}, \Delta_{A_{n - 1}}$, $\tilde{\alpha}$, or $\tilde{\alpha}_y$ have any terms of higher periodic degrees, so there is nothing to check.
\end{proof}

We collect our results so far.

\begin{theorem} \label{thm: yprojtop}
Consider the map

\begin{align*}
    \Phi^y \colon t\mathbb{Z}[u_n^{-1}] \otimes A_{1^{n - 1}}^y \to q^{2}t^{-1} \mathbb{Z}[u_n^{-1}] \otimes A_{1^{n - 1}}^yJ_n^y; \quad \quad \Phi^y = 1 \otimes \alpha^y + u_n \otimes \beta
\end{align*}

Set $P_{1^n}^y := \text{Cone}(\Phi^y)$\textcolor{revisions}{, and let $\delta_{P_{1^n}}$ denote the connection on this cone}. Then $(P_{1^n}^y, e, \textcolor{revisions}{\delta_{P_{1^n}}})$ is a (strict) $y$-ification of $P_{1^n}$.
\end{theorem}

Note the similarities between Theorem \ref{thm: yprojtop} and Theorem \ref{thm: projtop}. 

\subsection{Categorical Idempotents}

In order to use the projector $P_{1^n}^y$ to define a colored, $y$-ified link homology theory, we require that theory to be independent of the location of the projector on a given strand. This is most easily established by showing that the projector used to color a strand is a categorical idempotent in the sense of \cite{Hog17}. In the un $y$-ified setting, this was accomplished in \cite{AH17} by realizing $A_n$ explicitly as a projective resolution of $R_n$. Our Theorem \ref{thm: projtop} gives an alternative proof of this claim for the model $P_{1^n}$ in the un $y$-ified setting.

After $y$-ifying, the transition from chain complexes to \textit{curved} complexes prevents us from employing the usual language of quasi-isomorphisms and projective resolutions to show $A_n^y$ is a counital idempotent by mimicking the proof given in \cite{AH17}.  In lieu of this, we show that the $y$-ified projector is a categorical idempotent inductively as in Theorem \ref{thm: projtop}. In fact, we will find it easier to work with the \textit{dual} projectors $(P^y_{1^n})^{\vee}$ and $(A^y_n)^{\vee}$. More precisely, applying the contravariant duality functor $(-)^{\vee}$ of Remark \ref{rem: chain_graphs} to $P_{1^n}$ and $A_n$ gives two bounded \textit{below} complexes $P_{1^n}^{\vee}, A_n^{\vee} \in K^+(SBim_n)$ which are \textit{unital} idempotents in this category, in the sense described below.

\begin{theorem} \label{thm: unital_idem}
Let $C = P_{1^n}^{\vee}$ or $A_n^{\vee}$. Then:

\begin{itemize}
    \item[(P1)] $C \in \mathcal{I}_n^+$;
    \item[(P2)] There exists a chain map $\nu_n \colon R_n \to C$ such that $\text{Cone}(\nu_n) \in ^{\perp}(\mathcal{I}_n^+) \cap (\mathcal{I}_n^+)^{\perp}$.
\end{itemize}
These two properties uniquely characterize $C$ up to homotopy equivalence. That is, if $(Q, \psi)$ are a complex and morphism satisfying (P1) and (P2), then there is a unique (up to homotopy) map $\phi \colon C \to Q$ such that $\psi = \phi \circ \nu_n$; moreover, $\phi$ is a homotopy equivalence.
\end{theorem}

\textcolor{revisions}{Note the similarity between Theorem \ref{thm: unital_idem}, which characterizes $P_{1^n}^{\vee}$ and $A_n^{\vee}$ as \textit{unital} idempotents, and Theorem \ref{thm: AHchar}, which characterizes their duals $P_{1^n}$ and $A_n$ as \textit{counital} idempotents.}

We refer the interested reader to \cite{AH17} and \cite{Hog17} for more details of this construction and the relation between unital and counital idempotents in general. Our primary motivations for dualizing are twofold. First, the bounded below category $\mathcal{Y}^+_e(R_n-\text{Bim})$ is better behaved with respect to our convention for morphisms; see Remark \ref{rem: convergence}. Second, previous results for row-colored homology deal with \textit{unital} idempotents, and comparison to these results necessitates our dealing with unital idempotent-colored invariants as well.

We can extend $(-)^{\vee}$ to curved complexes in an essentially tautological way; given a curved complex $(C[\mathbb{Y}], \sigma, \delta_C) \in \mathcal{Y}_{\sigma}(R_n-\text{Bim})$, we set

\[
\delta_{C^{\vee}} := \sum_{\textbf{v} \in \mathbb{Z}^n_{\geq 0}} (f_{y^\textbf{v}})^{\vee} \otimes y^\textbf{v}
\]

Functoriality (and $R[\mathbb{X}, \mathbb{X}']$-linearity) of $(-)^{\vee}$ then ensures that $(C^{\vee}[\mathbb{Y}], \sigma, \delta_{C^{\vee}}) \in \mathcal{Y}_{\sigma}(R_n-\text{Bim})$. In particular, we have (strict) $y$-ifications $(A^y_n)^{\vee}$, $(P^y_{1^n})^{\vee} \in \mathcal{Y}_e(R_n-\text{Bim})^+$ of the unital idempotents $A_n^{\vee}$, $P_{1^n}^{\vee}$. We obtain explicit descriptions of these $y$-ifications by reversing all arrows and grading shifts in our explicit descriptions of $A_n^y$ and $P_{1^n}^y$. In particular, dualizing preserves the $\mathbb{Z}[u_2, \dots, u_n]$-periodicity of these curved complexes. We record this result below.

\begin{proposition} \label{prop: dual_y_struc}
Let $u_2, \dots, u_n$ be formal variables of degree $\text{deg}(u_i) = q^{-2i}t^2$. Then there exist $\mathbb{Y}$-degree \textcolor{revisions}{$0$} twists such that $(P_{1^n}^y)^{\vee} \cong \text{tw}(\text{Cone}(\overline{\alpha}^y)^{\vee} \otimes \mathbb{Z}[u_2, \dots, u_n])$ and $(A_n^y)^{\vee} \cong \text{tw}((K_n^y)^{\vee} \otimes \mathbb{Z}[u_2, \dots, u_n])$. These twists endow $(P_{1^n}^y)^{\vee}$, respectively $(A_n^y)^{\vee}$, with filtrations over $\mathbb{Z}[u_2, \dots, u_n]$ with subquotients $\text{Cone}(\overline{\alpha}^y)^{\vee}$, respectively $(K_n^y)^{\vee}$.
\end{proposition}

In what follows, we will rely on the fact that the explicit combinatorial model $(A^y_{n - 1})^{\vee}$ of our $y$-ified projector is a unital idempotent whenever $(P^y_{n - 1})^{\vee}$ is a unital idempotent. We prove this by furnishing an explicit homotopy equivalence $(A^y_n)^{\vee} \simeq (P^y_{1^n})^{\vee}$.

\begin{theorem} \label{thm: top_comb_equiv}
Let $F_0 \colon P_{1^n}^{\vee} \to A_n^{\vee}$ and $G_0 \colon A_n^{\vee} \to P_{1^n}^{\vee}$ be the data of a homotopy equivalence. Then there exist curved lifts $F \colon (P_{1^n}^y)^{\vee} \to (A_n^y)^{\vee}$, $G \colon (A_n^y)^{\vee} \to (P_{1^n}^y)^{\vee}$ of $F_0, G_0$ to homotopy equivalences between the $y$-ifications constructed above.
\end{theorem}

\begin{proof}
By results in \cite{AH17}, $\text{Hom}^{\bullet}(P_{1^n}^{\vee}, A_n^{\vee}) \simeq \text{End}^{\bullet}(P_{1^n}^{\vee})$ has homology concentrated in non-negative degrees. Applying Proposition \ref{prop: GHobs} to $F_0$ with $N = 0$, $g = 0 \in \text{Hom}^1_{\mathcal{Y}_e (R_n-\text{Bim})}((P_{1^n}^y)^{\vee}, (A_n^y)^{\vee})$, we obtain a formal power series

\begin{align*}
    F = \sum_{\textbf{v} \in \mathbb{Z}_{\geq 0}^n} F_{y^\textbf{v}} \otimes y^{\textbf{v}}
\end{align*}

in $\text{Hom}^0_{\text{Ch}(R_n-\text{Bim})}(P_{1^n}^{\vee}, A_n^{\vee})[[\mathbb{Y}]]$ satisfying $F_{y^\textbf{0}} = F_0$ and $[\delta, F] = 0$. For each $\textbf{v} \in \mathbb{Z}_{\geq 0}^n$, we have $\text{deg}(F_{y^\textbf{v}}) = q^{2|\textbf{v}|}t^{-2|\textbf{v}|}$. Since $A_n^{\vee}$ is bounded below in homological degree, $F$ must restrict to a finite sum on each homological degree of $P_{1^n}^{\vee}$, so $F$ is well-defined.

An exactly analogous argument gives a well-defined lift $G$ of $G_0$. That $F, G$ are the data of a homotopy equivalence then follows from a direct application of Corollary \ref{lem: y_contract} to $\text{Cone}(F)$ and $\text{Cone}(G)$ and the well-known fact that contractibility of mapping cones characterizes homotopy equivalence.
\end{proof}

Everything is in place for our main theorem of this section.

\begin{definition} \label{def: yideal}
For each $\sigma \in \mathfrak{S}^n$, let $\textcolor{revisions}{\mathcal{Y}} \mathcal{I}_{n, \sigma}$ denote the full subcategory of $\mathcal{Y}_{\sigma}(R_n-\text{Bim})$ consisting of objects of the form $(C[\mathbb{Y}], \sigma, \delta_C)$ for some $C \in \mathcal{I}_n$. We define $\textcolor{revisions}{\mathcal{Y}} \mathcal{I}_{n, \sigma}^b$, $\textcolor{revisions}{\mathcal{Y}} \mathcal{I}_{n, \sigma}^{\pm}$ as usual.
\end{definition}

\begin{theorem} \label{thm: yified_unital_proj}
Fix $n \geq 1$. Then:

\begin{itemize}
    \item[(P1)] $(P_{1^n}^y)^{\vee} \in \textcolor{revisions}{\mathcal{Y} \mathcal{I}_{n, e}^+}$;
    \item[(P2)] The map $\nu_n \colon R_n \to (P_{1^n})^{\vee}$ lifts to a chain map $\tilde{\nu}_n \colon R_n[\mathbb{Y}] \to (P_{1^n}^y)^{\vee}$ of curved complexes such that, for each $\sigma \in \mathfrak{S}^n$ and each $C[\mathbb{Y}] \in \textcolor{revisions}{\mathcal{Y} \mathcal{I}^+_{n, \sigma}}$, we have $C[\mathbb{Y}] \otimes \text{Cone}(\tilde{\nu}_n) \simeq \text{Cone}(\tilde{\nu}_n) \otimes C[\mathbb{Y}] \simeq 0$.
\end{itemize}

These two properties uniquely characterize $(P_{1^n}^y)^{\vee}$ up to homotopy equivalence. That is, if $(Q, \psi)$ are a curved complex and morphism satisfying (P1) and (P2), then there is a unique (up to homotopy) chain map $\phi \colon (P_{1^n}^y)^{\vee} \to Q$ such that $\psi = \phi \circ \tilde{\nu}_n$; moreover, $\phi$ is a homotopy equivalence.
\end{theorem}

\begin{proof}
We proceed by induction. In the base case, $(P_{1^1}^y)^{\vee} := R_1[\mathbb{Y}]$ and $\nu_1 = \text{id}_{R_1[\mathbb{Y}]}$. Now, suppose the pair $((P_{1^{n - 1}}^y)^{\vee}, \tilde{\nu}_{n - 1})$ satisfy (P1) and (P2). By Theorem \ref{thm: top_comb_equiv}, there is a chain map $\tilde{\psi}_{n - 1} \colon R_{n - 1} \to (A^y_{n - 1})^{\vee}$ satisfying (P1) and (P2) as well; moreover, $\tilde{\psi}_{n - 1} = F \circ \tilde{\nu}_{n - 1}$. It is clear from Proposition \ref{prop: dual_y_struc} that $\tilde{\psi}_{n - 1}$ extends to a chain map $\tilde{\nu}_n \colon R_n[\mathbb{Y}] \to (P_{1^n}^y)^{\vee}$ lifting $\nu_n$ as in the proof of Theorem \ref{thm: projtop}. Fix $\sigma \in \mathfrak{S}^n$, and let $C[\mathbb{Y}] \in \mathcal{I}_{n, \sigma}^+$ be given. Then $C[\mathbb{Y}] \otimes \text{Cone}(\tilde{\nu}_n)$ (resp. $\text{Cone}(\tilde{\nu}_n) \otimes C[\mathbb{Y}]$) is a $y$-ification of $C \otimes \text{Cone}(\nu_n)$ (resp. $\text{Cone}(\nu_n) \otimes C$); this is contractible by Lemma \ref{lem: y_contract}.
\end{proof}

The desired centrality of $(P_{1^n}^y)^{\vee}$ follows as a direct consequence:

\begin{corollary} \label{cor: crossing_slide}
For each braid $\beta \in Br_n$, there is a homotopy equivalence $F^y(\beta) \otimes (P^y_{1^n})^{\vee} \simeq (P^y_{1^n})^{\vee} \otimes F^y(\beta)$.
\end{corollary}

\begin{proof}
\textcolor{revisions}{For each $\sigma \in \mathfrak{S}^n$, let $K \mathcal{Y}_{\sigma}(SBim_n)$ denote the homotopy category of $y$-ifications of complexes of Soergel bimodules with permutation $\sigma$. Set $\mathcal{A} := \sqcup_{\sigma \in \mathfrak{S}^n} K \mathcal{Y}_{\sigma}(SBim_n)$ and $\mathcal{Y}\mathcal{I}^+_{n, \mathfrak{S}^n} := \sqcup_{\sigma \in \mathfrak{S}^n} \mathcal{Y} \mathcal{I}^+_{n, \sigma}$. The category $\mathcal{A}$ equipped with the usual structure of shifts and mapping cones is a triangulated monoidal category with unit $R_n[\mathbb{Y}]$, and $\mathcal{Y}\mathcal{I}^+_{n, \mathfrak{S}^n}$ is a two-sided tensor ideal in this category.}

\textcolor{revisions}{We rephrase Theorem \ref{thm: yified_unital_proj} in the language of the previous paragraph: There is an object $(P_{1^n}^y)^{\vee} \in \mathcal{Y}\mathcal{I}^+_{n, \mathfrak{S}^n}$ and a morphism $\tilde{\nu}_n \colon R_n[\mathbb{Y}] \to (P_{1^n}^y)^{\vee}$ satisfying $\mathrm{Cone}(\tilde{\nu}_n) \in ^{\perp} (\mathcal{Y} \mathcal{I}_{n, \mathfrak{S}^n}^+) \cap (\mathcal{Y} \mathcal{I}_{n, \mathfrak{S}^n}^+)^{\perp}$. By Corollary 4.29 of \cite{Hog17}, it follows that $(P_{1^n}^y)^{\vee}$ is central in $\mathcal{A}$. In particular, $(P_{1^n}^y)^{\vee}$ commutes with all $y$-ified Rouquier complexes up to homotopy equivalence.}
\end{proof}

In using $(P^y_{1^n})^{\vee}$ to formulate a well-defined link invariant, we also require that this projector ``slides past crossings"; this ensures independence of the resulting invariant from the choice of marked point away from crossings on each strand.

\begin{proposition} \label{prop: inf_cross_slide}
For each $n \geq 1$, let $X_{n + 1} := \sigma_n \sigma_{n - 1} \dots \sigma_1 \in \text{Br}_n$, $Y_{n + 1} := \sigma_1 \sigma_2 \dots \sigma_n$. Then $F^y(X_{n + 1})(P^y_{1^n})^{\vee} \simeq (P^y_{1^n})^{\vee} F^y(X_{n + 1})$ and $F^y(Y_{n + 1}) (P^y_{1^n})^{\vee} \simeq (P^y_{1^n})^{\vee} F^y(Y_{n + 1})$.
\end{proposition}

\begin{proof}
Our proof follows the same techniques as in \cite{CK12}. We claim that the map \textcolor{revisions}{$\tilde{\nu}_n \colon R_n[\mathbb{Y}] \to (P^y_{1^n})^{\vee}$} induces homotopy equivalences

\[
\textcolor{revisions}{\text{id} \otimes \tilde{\nu}_n \colon (P^y_{1^n})^{\vee} F^y(X_{n + 1}) \to (P^y_{1^n})^{\vee} F^y(X_{n + 1}) (P^y_{1^n})^{\vee}}
\]

and

\[
\textcolor{revisions}{\tilde{\nu}_n \otimes \text{id} \colon F^y(X_{n + 1}) (P^y_{1^n})^{\vee} \to (P^y_{1^n})^{\vee} F^y(X_{n + 1}) (P^y_{1^n})^{\vee}.}
\]

As in previous cases, it suffices to check that the cone of each map is contractible\textcolor{revisions}{; by Lemma \ref{lem: y_contract}, it in turn suffices to check that the non $y$-ified cones $\mathrm{Cone}(\mathrm{id} \otimes \nu_n)$ and $\mathrm{Cone}(\nu_n \otimes \mathrm{id})$ are contractible. We prove this for the former cone; the latter follows exactly analogously.}

\textcolor{revisions}{First, notice that $\mathrm{Cone}(\mathrm{id} \otimes \nu_n) = (P_{1^n}^{\vee} F(X_{n + 1})) \mathrm{Cone}(\nu_n)$. Using the explicit description of $(P_{1^n})^{\vee}$ from e.g. Proposition \ref{prop: dual_y_struc}, we have that $(P_{1^n})^{\vee}F(X_{n + 1}) = \mathrm{tw}_{\gamma}(B_{w_0}F(X_{n + 1} \otimes \wedge[\theta_2, \dots, \theta_n] \otimes \mathbb{Z}[u_2, \dots, u_n])$ for some twist $\gamma$. It is well known that $B_{w_0} F(X_{n + 1}) \simeq F(X_{n + 1}) B_{w_0}$ (in other words, $B_{w_0}$ slides past crossings). Since this is an upper finite, homologically locally finite convolution, by Proposition \ref{prop: useful_hpt} we have a homotopy equivalence}

\begin{equation*}
\textcolor{revisions}{(P_{1^n})^{\vee}F(X_{n + 1}) \simeq \mathrm{tw}_{\gamma'}(F(X_{n + 1}) B_{w_0} \otimes \wedge[\theta_2, \dots, \theta_n] \otimes \mathbb{Z}[u_2, \dots, u_n])}
\end{equation*}

\textcolor{revisions}{for some twist $\gamma'$. Applying $- \otimes \mathrm{Cone}(\nu)$, we obtain}

\begin{equation} \label{eq: sliding_conv}
\textcolor{revisions}{(P_{1^n})^{\vee}F(X_{n + 1}) \mathrm{Cone}(\nu_n) \simeq \mathrm{tw}_{\gamma'}(F(X_{n + 1}) B_{w_0} \mathrm{Cone}(\nu_n) \otimes \wedge[\theta_2, \dots, \theta_n] \otimes \mathbb{Z}[u_2, \dots, u_n])}
\end{equation}

\textcolor{revisions}{Now $\mathrm{Cone}(\nu_n) \in ^{\perp}(\mathcal{I}_n^+) \cap (\mathcal{I}_n^+)^{\perp}$ by Theorem \ref{thm: unital_idem}, so in particular $F(X_{n + 1}) B_{w_0} \mathrm{Cone}(\nu_n) \simeq 0$. The convolution in \eqref{eq: sliding_conv} is once again upper finite and homologically locally finite, so again by Proposition \ref{prop: useful_hpt}, the whole convolution is contractible.}
\end{proof}

\subsection{Y-ified Hochschild Cohomology} \label{sec: yify_hoch_coh}

In this section we develop the $y$-ified analog of the partial trace technology recalled in Section \S \ref{section: HoCh}. Much of the material of this section is likely well-known to experts, but to our knowledge, this is the first development of this technology appearing in the literature.

\begin{definition}
Let $\mathcal{C}_n$ be as in Section $\S \ref{section: HoCh}$. We regard $R[\mathbb{Y}]$ as an element of $\mathcal{C}_n$ with trivial differential under the usual inclusion functors; note that $\text{deg}(y_i) = q^{-2}t^2 = a^0q^{-2}t^2$ under this inclusion. For each $C \in \mathcal{C}_n$, let $C[\mathbb{Y}] := C \otimes R[\mathbb{Y}]$. We denote by $\mathcal{C}^y_n$ the category with objects $C[\mathbb{Y}]$ for some $C \in \mathcal{C}_n$ and graded morphism spaces defined analogously to those of $\mathcal{Y}_{\sigma}(R_n-\text{Bim})$. Then $\mathcal{C}^y_n$ is a graded $R[\mathbb{X}, \mathbb{X}', \mathbb{Y}]$-linear category via the termwise action of $R[\mathbb{X}, \mathbb{X}']$ and multiplication by $R[\mathbb{Y}]$. Given any $Z \in R[\mathbb{X}, \mathbb{X}', \mathbb{Y}]$ with degree $\text{deg}(Z) = t^2$, we may consider the category $Z-\text{Fac}^y(\mathcal{C}_n) := \text{Fac}(\mathcal{C}^y_n, Z)$ of $Z$-factorizations.
\end{definition}

Most of the technology of Section \ref{sec: yification} carries over to the derived setting. Given any morphism $f$ in $Z-\text{Fac}(\mathcal{C}_n)$, we again have a decomposition $f = \sum_{\textbf{v} \in \mathbb{Z}^n_{\geq 0}} f_{y^\textbf{v}}y^{\textbf{v}}$ as in Remark \ref{rem: decomp} and call $f$ a \textit{curved lift} of $f_{y^\textbf{0}}$. We call $(C[\mathbb{Y}], \delta_C) \in \mathcal{C}^y_n$ a curved lift of $(C, d_C) \in \mathcal{C}_n$ if $\delta_C$ is a curved lift of $d_C$.

\begin{definition} \label{def: yderv}
Let $(C, d_C) \in \mathcal{C}_n$ be given. A \textit{$y$-ification} of $C$ is an ordered triple $(C[\mathbb{Y}], \sigma, \delta_C)$ such that $(C[\mathbb{Y}], \delta_C)$ is a curved lift of $(C, d_C)$ with curvature $Z_{\sigma}$. For each $\sigma \in \mathfrak{S}^n$, let $\mathcal{CY}_{\sigma}(R_n-\text{Bim})$ denote the full subcategory of $Z_{\sigma}-\text{Fac}(\mathcal{C}_n)$ generated by (direct sums of shifts of) $y$-ifications $(C[\mathbb{Y}], \sigma, \delta_C)$. We set $\mathcal{CY}_e(R-\text{Bim}) := \mathcal{C}_0$ for notational convenience. There is a natural inclusion functor $\mathcal{CY}_{\sigma}(R_{n - 1}-\text{Bim}) \hookrightarrow \mathcal{CY}_{\sigma}(R_n-\text{Bim})$ for each $\sigma \in \mathfrak{S}^{n - 1}$; we consider objects in the former also as objects in the latter under this inclusion without further comment.
\end{definition}

Let $(C[\mathbb{Y}], \sigma, \delta_C) \in \mathcal{Y}_{\sigma}(R_n-\text{Bim})$ be given, and consider the triply-graded Hochschild cohomology complex $HH(C) \in \mathcal{C}_n$ of $C$ as in \S \ref{section: HoCh}. Recalling the decomposition $\delta_C = \sum_{\textbf{v} \in \mathbb{Z}^n_{\geq 0}} (\delta_C)_{y^\textbf{v}} y^\textbf{v}$ of Remark \ref{rem: decomp}, we set

\[
\delta_{HH(C)} := \sum_{\textbf{v} \in \mathbb{Z}^n_{\geq 0}} HH((\delta_C)_{y^{\textbf{v}}}) y^{\textbf{v}} \in \text{End}^1_{\mathcal{C}^y_n}(HH(C)[\mathbb{Y}])
\]

\begin{definition}
Retain notation as above. Then the ordered triple $HH^y(C[\mathbb{Y}], \sigma, \delta_C) := (HH(C)[\mathbb{Y}], \sigma, \delta_{HH(C)})$ is called the \textit{$y$-ified Hochschild cohomology} of $(C[\mathbb{Y}], \sigma, \delta_C)$. We will often denote $y$-ified Hochschild cohomology simply by $HH^y(C[\mathbb{Y}])$ when the connection $\delta_C$ and permutation $\sigma$ are clear from context.
\end{definition}

\begin{proposition}
We have $HH^y(C[\mathbb{Y}]) \in \mathcal{CY}_{\sigma}(\mathcal{C}_n)$.
\end{proposition}

\begin{proof}
All that requires proof is that $\delta_{HH(C)}^2 = Z_{\sigma}$. This follows immediately from $\delta_C^2 = Z_{\sigma}$ and the fact that $HH$ is a functor.
\end{proof}

Since $HH(x_i) = HH(x_i')$ as endomorphisms of $HH(C)$ for all $C \in \mathcal{D}_n$, we can rewrite the identity $\delta_{HH(C)}^2 = Z_{\sigma}$ as follows:

\begin{align*}
\delta_{HH(C)}^2 = \sum_{i = 1}^n (HH(x_{\sigma(i)}) - HH(x_i'))y_i = \sum_{i = 1}^n (HH(x_{\sigma(i)}) - HH(x_i))y_i = \sum_{i = 1}^n HH(x_i) (y_{\sigma^{-1}(i)} - y_i)
\end{align*}

The above identity guarantees that $HH^y(C[\mathbb{Y}])$ becomes a chain complex after identifying $y_i$ and $y_{\sigma_i}$ for each $1 \leq i \leq n$.

\textcolor{revisions}{The following is Lemma 3.9 of \cite{GH22}:}

\begin{proposition} \label{prop: hy_tracelike}
\textcolor{revisions}{For any $y$-ifications $C[\mathbb{Y}] \in \mathcal{Y}_{\sigma}(R_n-\text{Bim}), D[\mathbb{Y}] \in \mathcal{Y}_{\rho}(R_n-\text{Bim})$, we have $HH^y(C[\mathbb{Y}] D[\mathbb{Y}]) \cong HH^y(D[\mathbb{Y}] C[\mathbb{Y}])$ up to a permutation of indices in $R[\mathbb{Y}]$ structure.}
\end{proposition}

\begin{definition}
Retain notation as above. Then the \textit{$y$-ified homology} of $(C[\mathbb{Y}], \sigma, \delta_C)$ is the triply-graded $R$-module \footnote{In fact the $y$-ified homology carries an action of a variable $x_c$ and $y_c$ for each orbit $c$ of $\sigma$. Since our primary motivation is comparing the graded dimensions of invariants as vector spaces, we will not concern ourselves with the details of this module structure; see \cite{GH22} for more detail}

\[
HHH^y(C) := H^{\bullet} \big( HH(C)[\mathbb{Y}] \otimes (R[\mathbb{Y}]/(y_1 - y_{\sigma(1)}, \dots, y_n - y_{\sigma(n)})) \big)
\]

We will often write $HHH^y(C[\mathbb{Y}])$ to denote the $y$-ified homology of $(C[\mathbb{Y}], \sigma, \delta_C)$ when the connection and permutation are clear.
\end{definition}

Again, our interest in $HHH^y$ comes from its relation to link invariants. The following is due to Gorsky-Hogancamp (\cite{GH22}):

\begin{theorem} \label{thm: yhom_invariant}
Let $\beta \in Br_n$ be given, and let $F^y(\beta)$ be as in Proposition \ref{prop: ybraid}. Then up to an overall normalization, the triply-graded $R$-module $HHH^y(F^y(\beta))$ is an invariant of the braid closure $\hat{\beta}$.
\end{theorem}

We call this invariant the \textit{$y$-ified triply-graded homology of $\hat{\beta}$}. As in the un $y$-ified case, given a link $\mathcal{L}$, we will often denote its $y$-ified triply-graded homology as $HHH^y(\mathcal{L})$ and the graded dimension of $HHH^y(\mathcal{L})$ as $\mathcal{P}^y(\mathcal{L})$.

As in the computation of $HHH$, our primary tool in computing $y$-ified homology will be an analog of the partial trace functor of Definition \ref{def: trace}. Since the categories $\mathcal{CY}_{\sigma}(R_n-\text{Bim})$ depend on a choice of $\sigma \in \mathfrak{S}^n$, some care must be taken with permutations to define the codomain of these functors.

\begin{definition}
Fix $n \geq 1$, and let $\sigma \in \mathfrak{S}^n$ be given. We define $\text{Tr}_n(\sigma) \in \mathfrak{S}^{n - 1}$ as follows:

\[
\text{Tr}_n(\sigma)(i) :=
\begin{cases}
\sigma(i) \quad \text{if } \sigma(i) \neq n; \\
\sigma(n) \quad \text{if } \sigma(i) = n
\end{cases}
\]

\end{definition}

\textcolor{revisions}{We claim that $\mathrm{Tr}_n(\sigma)$ is actually a bijection on $\{1, \dots, n - 1\}$. Indeed, if $\sigma(n) = n$, then $\mathrm{Tr}_n(\sigma)$ is simply the restriction of $\sigma$ to the set $\{1, \dots, n - 1\}$. Otherwise, $\sigma(n) = m$ for some $1 \leq m < n$, and $\mathrm{Tr}_n(\sigma)$ send $\sigma^{-1}(n)$ to $m$.} Note that $Tr_n(\sigma)$ can be easily obtained from a string diagram for $\sigma$ by ``closing up" one strand in the manner of \S \ref{section: HoCh} and straightening out the resulting string.

\begin{proposition}
Let $(C[\mathbb{Y}], \sigma, \delta_C) \in \mathcal{CY}_{\sigma}(R_n-\text{Bim})$ be given. Define $\delta_{Tr_n(C)}$ by

\[
\delta_{Tr_n(C)} := \left( \sum_{\textbf{v} \in \mathbb{Z}^n_{\geq 0}} Tr_n((\delta_C)_{y^{\textbf{v}}})y^{\textbf{v}} \right) \Bigg|_{y_n = y_{\sigma^{-1}(n)}}
\]

Then $(Tr_n(C)[\mathbb{Y}], Tr_n(\sigma), \delta_{Tr_n(C)})$ is a $y$-ification of $Tr_n(C)$.
\end{proposition}

\begin{proof}
All that requires proof is that $\delta_{Tr_n(C)}^2 = Z_{Tr_n(\sigma)}$. We break into cases analogous to our construction of $Tr_n(\sigma)$.

\textbf{Case 1}: $\sigma(n) = n$. By functoriality of $Tr_n$, we have

\begin{align*}
    \delta_{Tr_n(C)}^2 & = \sum_{i = 1}^n Tr_n(x_{\sigma(i)} - x'_i) y_i \\
    & = \left( \sum_{i = 1}^{n - 1} Tr_n(x_{\sigma(i)} - x'_i) y_i \right) + Tr_n(x_n - x_n')y_n \\
    & = \sum_{i = 1}^{n - 1} Tr_n(x_{\sigma(i)} - x'_i) y_i \\
    & = Z_{Tr_n(\sigma)}
\end{align*}

Here the third equality follows from the identity $Tr_n(x_n) = Tr_n(x_n')$ in $\text{End}_{\mathcal{C}_{n - 1}}(Tr_n(C))$.

\textbf{Case 2}: $\sigma^{-1}(n) = j \neq n$. Again by functoriality of $Tr_n$, we have

\begin{align*}
    \delta_{Tr_n(C)}^2 & = \left( \sum_{i = 1}^n Tr_n(x_{\sigma(i)} - x'_i) y_i \right) \Bigg|_{y_n = y_{\sigma^{-1}(n)}} \\
    & = \left( \sum_{1 \leq i \neq j}^{n - 1} Tr_n(x_{\sigma(i)} - x'_i) y_i \right) + Tr_n(x_n - x_j')y_j + Tr_n(x_{\sigma(n)} - x_n')y_j \\
    & = \left( \sum_{1 \leq i \neq j}^{n - 1} Tr_n(x_{\sigma(i)} - x'_i) y_i \right) + Tr_n(x_{\sigma(n)} - x_j') y_j \\
    & = Z_{Tr_n(\sigma)}
\end{align*}

Again the third equality follows from the identity $Tr_n(x_n) = Tr_n(x_n')$ in $\text{End}_{\mathcal{C}_{n - 1}}(Tr_n(C))$.

\end{proof}

\begin{definition}
For each $\sigma \in \mathfrak{S}^n$, let $Tr_{n, \sigma}^y \colon \mathcal{CY}_{\sigma}(R_n-\text{Bim}) \to \mathcal{CY}_{Tr_n(\sigma)}(R_{n - 1}-\text{Bim})$ be the functor given by

\[
Tr_{n, \sigma}^y(C[\mathbb{Y}], \sigma, \delta_C) := (Tr_n(C)[\mathbb{Y}], Tr_n(\sigma), \delta_{Tr_n(C)})
\]

We set $Tr_{0, e}^y := \text{Hom}_{\mathcal{C}_0}(R, -)$ for notational convenience.

\end{definition}

To formulate $y$-ified analogs of Propositions \ref{prop: tradj} and \ref{prop: trlocal} regarding adjointness and locality of the functors $Tr^y_{n, \sigma}$, we construct a collection of inclusion functors $I^y_{n, \sigma} \colon \mathcal{CY}_{Tr_n(\sigma)}(R_{n - 1}-\text{Bim}) \hookrightarrow \mathcal{CY}_{\sigma}(R_n-\text{Bim})$ that are sensitive to the choice of permutation $\sigma$. As in the uncurved case, $HHH^y$ is a representable functor, though the representing object in $\mathcal{CY}_{\sigma}(R_n-\text{Bim})$ depends on the permutation $\sigma$.

\begin{definition}
For each pair $i \neq j \in \{1, \dots, n\}$, let $\mathbbm{1}^y_{i, j}$ denote the curved complex

\begin{center}
\begin{tikzcd}[sep=large]
\mathbbm{1}^y_{i, j} := q^2t^{-1}R_n[\mathbb{Y}] \arrow[r, "x_i - x_j", harpoon, shift left] & R_n[\mathbb{Y}] \arrow[l, "y_j - y_i", harpoon, shift left]
\end{tikzcd}
\end{center}

Observe that $\mathbbm{1}_{i, j}$ is a (strict) $y$-ification of the Koszul complex for $x_i - x_j$ on $R_n$ with curvature $Z_{(ij)}$; we denote this Koszul complex by $\mathbbm{1}_{i, j}$. By abuse of notation, we also let $\mathbbm{1}^y_{i, j}$ denote the corresponding triple in $\mathcal{CY}_{(ij)}(R_n-\text{Bim})$.
\end{definition}

The following is Proposition 3.27 in \cite{GH22}:

\begin{proposition} \label{prop: hyrep}
Suppose $(C[\mathbb{Y}], \sigma, \delta_C) \in \mathcal{CY}_{\sigma}(R_n-\text{Bim})$ is a strict $y$-ification. Fix a minimal length expression $\sigma = (i_1, j_1) (i_2, j_2) \dots (i_r, j_r)$ for $\sigma$ as a product of transpositions, and set $\mathbbm{1}^y_{\sigma} := \mathbbm{1}^y_{i_1, j_1} \otimes \mathbbm{1}^y_{i_2, j_2} \otimes \dots \otimes \mathbbm{1}^y_{i_r, j_r} \in \mathcal{CY}_{\sigma}(R_n-\text{Bim})$. Then we have

\[
HHH^y(C) \cong \text{Hom}_{\mathcal{CY}_{\sigma}(R_n-\text{Bim})}(\mathbbm{1}^y_{\sigma}, C[\mathbb{Y}])
\]
\end{proposition}

\begin{remark}
Note that $\mathbbm{1}^y_{\sigma}$ is bounded in homological degree. In particular, our convention for morphisms (see Remark \ref{rem: convergence}) has no effect on the conclusion of Proposition \ref{prop: hyrep}. \textcolor{revisions}{We point out also that $\mathbbm{1}^y_{\sigma}$ is independent of the choice of minimal length expression for $\sigma$ up to isomorphism; indeed, such choices differ only by ordering of the transpositions, and the curved complexes $\mathbbm{1}_{i, j}$ all commute up to isomorphism.}
\end{remark}

\begin{lemma} \label{lem: untrword}
Let $\sigma \in \mathfrak{S}^n$ be given, and let $Tr_n(\sigma) = (i_1, j_1)(i_2, j_2) \dots (i_r, j_r)$ be a minimal length expression for $Tr_n(\sigma) \in \mathfrak{S}^{n - 1}$ as a product of transpositions. Then 

\begin{equation} \label{eq: expanding_sigma}
\textcolor{revisions}{\sigma = (n, \sigma(n))} (i_1, j_1)(i_2, j_2) \dots (i_r, j_r)
\end{equation}

is a minimal length expression for $\sigma$ as a product of transpositions.
\end{lemma}

\begin{proof}
Throughout, let $\ell(\sigma)$ and $\ell(Tr_n(\sigma))$ denote the minimal lengths of representations of $\sigma$ and $Tr_n(\sigma)$, respectively, as products of transpositions. \textcolor{revisions}{We begin by showing that Equation \eqref{eq: expanding_sigma} holds. Let $\omega$ denote the product of transpositions on the right-hand side. Certainly $\omega(n) = \sigma(n)$. Moreover, if $s \neq \sigma^{-1}(n)$, then $Tr_n(\sigma)(s) \neq n$, so $\omega(s) = Tr_n(\sigma)(s) = \sigma(s)$. Then $\omega(\sigma^{-1}(n)) = \sigma(n)$ by exhaustion.}

\textcolor{revisions}{Next, l}et $\{c_1, \dots, c_s\}$ denote the cycles of $\sigma$; then

\[
\ell(\sigma) = \sum_{i = 1}^s |c_i| - 1
\]

Let $c_j$ denote the cycle of $\sigma$ containing $n$, and observe that the cycles of $Tr_n(\sigma)$ are exactly the cycles of $\sigma$ with $n$ deleted from $c_j$. If $c_j = \{n\}$, we must have $\sigma(n) = n$; then the transposition $(n, \sigma(n))$ vanishes and the expression in question has length $r$. \textcolor{revisions}{Additionally}, the term $|c_j| - 1$ does not contribute to $\ell(\sigma)$, so

\[
r = \ell(Tr_n(\sigma)) = \sum_{i \neq j} |c_i| - 1 = \ell(\sigma)
\]

Otherwise, $|c_j| \geq 2$, so the same calculation shows that $\ell(\sigma) = r + 1$.
\end{proof}

\begin{definition} \label{def: perminc}
Fix $\sigma \in \mathfrak{S}^n$, and let $I^y_{Tr_n(\sigma)}$ denote the natural inclusion functor $\mathcal{CY}_{Tr_n(\sigma)}(R_{n - 1}-\text{Bim}) \hookrightarrow \mathcal{CY}_{Tr_n(\sigma)}(R_n-\text{Bim})$ of Definition \ref{def: yderv}. We define the functor $I^y_{n, \sigma} \colon \mathcal{CY}_{Tr_n(\sigma)}(R_{n - 1}-\text{Bim}) \to \mathcal{CY}_{\sigma}(R_n-\text{Bim})$ by

\[
I^y_{n, \sigma}(C[\mathbb{Y}]) := I^y_{Tr_n(\sigma)}(C[\mathbb{Y}]) \otimes \mathbbm{1}^y_{n, \sigma(n)}
\]
\end{definition}

\begin{remark} \label{rem: usualincy}
For $\sigma \in \mathfrak{S}^{n - 1}$, $I^y_{n, \sigma}$ is the usual inclusion functor of Definition \ref{def: yderv}.
\end{remark}

\begin{corollary} \label{cor: hyrepinc}
$\mathbbm{1}^y_{\sigma} = I^y_{n, \sigma}(\mathbbm{1}^y_{Tr_n(\sigma)})$ for all $\sigma \in \mathfrak{S}^n$.
\end{corollary}

\begin{proof}
Immediate consequence of Lemma \ref{lem: untrword}.
\end{proof}

\begin{proposition} \label{prop: ytradj}
Fix $\sigma \in \mathfrak{S}^n$, and let $N[\mathbb{Y}] \in \mathcal{CY}^b_{\sigma}(R_n-\text{Bim})$, $M[\mathbb{Y}] \in \mathcal{CY}^ b_{Tr_n(\sigma)}(R_{n - 1}-\text{Bim})$ be given. Then there is an isomorphism

\[
\text{Hom}_{\mathcal{CY}^b_{\sigma}(R_n-\text{Bim})}(I^y_{n, \sigma}(M[\mathbb{Y}]), N[\mathbb{Y}]) \cong \text{Hom}_{\mathcal{CY}^b_{Tr_n(\sigma)}(R_{n - 1}-\text{Bim})}(M[\mathbb{Y}], Tr^y_{n, \sigma}(N[\mathbb{Y}]))
\]

that is natural in both $M[\mathbb{Y}]$ and $N[\mathbb{Y}]$.
\end{proposition}

\begin{proof}
Given any morphism $f \in \text{Hom}_{\mathcal{CY}^b_{\sigma}(R_n-\text{Bim})}(I^y_{n, \sigma}(M[\mathbb{Y}]), N[\mathbb{Y}])$, we can decompose $f$ into components

\[
f = \sum_{\textbf{v} \in \mathbb{Z}^n_{\geq 0}} f_{y^\textbf{v}} y^\textbf{v}
\]

with $f_{y^\textbf{v}} \in \text{Hom}_{\mathcal{C}^b_n}(I_n(M), N)$. Let $\phi \colon \text{Hom}_{\mathcal{C}^b_n}(I_n(M), N) \to \text{Hom}_{\mathcal{C}^b_{n - 1}}(M, Tr_n(N))$ be the isomorphism of Proposition \ref{prop: tradj}. Define $\Phi \colon \text{Hom}_{\mathcal{CY}^b_{\sigma}(R_n-\text{Bim})}(I^y_{n, \sigma}(M[\mathbb{Y}]), N[\mathbb{Y}]) \to \text{Hom}_{\mathcal{CY}^b_{Tr_n(\sigma)}(R_{n - 1}-\text{Bim})}(M[\mathbb{Y}], Tr_{n, \sigma}(N[\mathbb{Y}]))$ by

\[
\Phi(f) := \left( \sum_{\textbf{v} \in \mathbb{Z}^n_{\geq 0}} \phi(f_{y^\textbf{v}}) y^\textbf{v} \right) \Bigg|_{y_n = y_{\sigma^{-1}(n)}}
\]

That $\Phi$ is a natural isomorphism follows from the corresponding facts about $\phi$ and the definition of $\delta_{Tr_n(N)}$; we leave the details as an exercise, as they are similar to previous computations in this section.
\end{proof}

Just as in the un $y$-ified setting, we can compute $HHH^y$ by iteratively applying partial trace.

\begin{proposition} \label{prop: hom_from_tr}
Fix $\sigma \in \mathfrak{S}^n$, and let $(C[\mathbb{Y}], \sigma, \delta_C) \in \mathcal{CY}^b_{\sigma}(R_n-\text{Bim})$ be a strict $y$-ification. \textcolor{revisions}{For each $1 \leq i < n$, set $\sigma_i := Tr_{i + 1} \circ \dots \circ Tr_{n - 1} \circ Tr_n(\sigma) \in \mathfrak{S}^i$, and set $\sigma_0 = e$, $\sigma_n = \sigma$.} Then

\[
\textcolor{revisions}{HHH^y(C[\mathbb{Y}]) = \prod_{i = 0}^n \left( Tr^y_{i, \sigma_i} \right) (C[\mathbb{Y}]).}
\]

\textcolor{revisions}{Here the product $\prod_{i = 0}^n \left( Tr^y_{i, \sigma_i} \right)$ is to be read as the composition of functors $Tr^y_{0, e} \circ Tr^y_{1, \sigma_1} \circ \dots \circ Tr^y_{n - 1, \sigma_{n - 1}} \circ Tr^y_{n, \sigma}$.}
\end{proposition}

\begin{proof}
Immediate application of Proposition \ref{prop: ytradj} and Corollary \ref{cor: hyrepinc}.
\end{proof}

Note that the only object (up to shifts) in $\mathcal{CY}_e(R_0-\text{Bim})$ is $R$, which satisfies $Tr^y_{0, e}(R) \cong R$. We have therefore reduced the task of computing $HHH^y$ to computations of partial traces as in the un $y$-ified case. Here we are aided again by locality:

\begin{proposition} \label{prop: try_local}
Fix $\rho, \varphi \in \mathfrak{S}^{n - 1}$ and $\sigma \in \mathfrak{S}^n$. Let $M[\mathbb{Y}] \in \mathcal{CY}_{\rho}(R_{n - 1}-\text{Bim})$, $N[\mathbb{Y}] \in \mathcal{CY}_{\varphi}(R_{n - 1}-\text{Bim})$, $C[\mathbb{Y}] \in \mathcal{CY}_{\sigma}(R_n-\text{Bim})$ be given. Then there is an isomorphism

\[
Tr^y_{n, \rho \sigma \varphi} (I^y_{n, \rho}(M[\mathbb{Y}]) \otimes C[\mathbb{Y}] \otimes I^y_{n, \varphi}(N[\mathbb{Y}])) \cong M[\mathbb{Y}] \otimes Tr^y_{n, \sigma}(C[\mathbb{Y}]) \otimes N[\mathbb{Y}]
\]

in $\mathcal{CY}_{\rho Tr_n(\sigma) \varphi}(R_{n - 1}-\text{Bim})$ which is natural in $M[\mathbb{Y}]$, $N[\mathbb{Y}]$, and $C[\mathbb{Y}]$.
\end{proposition}

\begin{proof}
Immediate, given Remark \ref{rem: usualincy}.
\end{proof}

We record the $y$-ified partial trace Markov moves below; these are taken from Example 3.5 and Proposition 3.8 in \cite{GH22}.

\begin{proposition} \label{prop: ymark}
Let $F^y(\sigma_i^{\pm 1})$ denote the $y$-ified Rouquier complexes for the braid group generators $\sigma_i^{\pm}$, and let $s_i \in \mathfrak{S}^n$ denote the permutations corresponding to those generators. We have:

\begin{enumerate}
    \item $Tr^y_{n, e}(R_n[\mathbb{Y}]) \simeq \left( \cfrac{1 + aq^{-2}}{(1 - q^2)(1 - q^2t^{-2})} \right) R_{n - 1}[\mathbb{Y}]$;
    
    \item $Tr^y_{n, s_{n - 1}} (F^y(\sigma_{n - 1})) \simeq q^{-1}tR_{n - 1}[\mathbb{Y}]$;
    
    \item $Tr^y_{n, s_{n - 1}}(F^y(\sigma_{n - 1}^{-1})) \simeq aq^{-3}R_{n - 1}[\mathbb{Y}]$
\end{enumerate}
\end{proposition}

Finally, we record the $y$-ified Markov move for the finite projector:

\begin{proposition} \label{prop: yktrace}
The family of $y$-ifications \textcolor{revisions}{$(K_n^y, e, \delta_{K_n})$ of $K_n$} satisfy:

\begin{enumerate}
    \item $K^y_1 = R_1[y]$;
    
    \item $K^y_{n - 1}J_n^y \simeq (q^{n - 3}tK^y_n \longrightarrow q^{-2}t^2 K^y_{n - 1})$;
    
    \item $Tr^y_{n, e}(K_n^y) \simeq \left(\cfrac{q^{n - 1} + aq^{-n - 1}}{1 - q^2}\right) q^2t^{-1} (K_{n - 1}^y)$
\end{enumerate}
\end{proposition}

\begin{proof}
Property (1) is a direct consequence of the definitions. Property (2) follows from Proposition \ref{prop: fynite} completely analogously to the proof of Proposition \ref{prop: k.markov.unnorm.uny}.

Property (3) is a direct computation. Applying $Tr_n$ termwise to $K_n$ \textcolor{revisions}{using Proposition \ref{prop: trb}} gives the curved complex 

\[
\left( \cfrac{q^{n - 1} + aq^{-n-1}}{1 - q^2} \right) B_{w_1}[y_1, \dots, y_n] \otimes \bigwedge[\theta_2, \dots, \theta_n]
\]

with connection $\delta_{Tr(K_n)} = (\sum_{j = 2}^{n - 1} (x_j - x_j') \otimes \theta_j^{\vee} + (y_j - y_1) \otimes \theta_j) + (y_n - y_1) \otimes \theta_n$ as in Proposition \ref{prop: ylessktrace}. We may Gaussian eliminate along the component $(y_n - y_1) \otimes \theta_n$ of the connection; what remains is the curved complex

\[
\left( \cfrac{q^{n - 1} + aq^{-n-1}}{1 - q^2} \right) \theta_n B_{w_1} \otimes R[y_1, \dots, y_n]/(y_n - y_1) \otimes \bigwedge[\theta_2, \dots, \theta_{n - 1}] \cong \left( \cfrac{q^{n - 1} + aq^{-n-1}}{1 - q^2} \right)q^2t^{-1}K^y_{n - 1}
\]
\end{proof}

\begin{remark}
Notice that the copies of $K^y_{n - 1}$ on the right-hand side of Property (3) of Proposition \ref{prop: yktrace} are not separated by a relative homological shift; in particular, they are not separated by an \textit{odd} relative homological shift. \textcolor{revisions}{Contrast this with the factor $(1 + q^2t^{-1})$ present in Proposition \ref{prop: ylessktrace}. Eliminating this odd degree shift is precisely why $y$-ification is necessary for our computations.}
\end{remark}

\subsection{Twisted Projectors} \label{sec: twist_proj}

In this section we analyze the action of the braid group $Br_n$ on $K^y_n$ given by $\beta \cdot K^y_n := F^y(\beta) \otimes K^y_n$. As usual, we \textcolor{revisions}{denote the permutation associated to a braid $\beta \in Br_n$ by $\sigma_{\beta} \in \mathfrak{S}^n$.}

\textcolor{revisions}{Recall the quasi-isomorphisms}

\[
\textcolor{revisions}{\psi_{\beta} \colon q^{e(\beta)} R[\mathbb{X}]_{\sigma_{\beta}} \to F(\beta)}; \quad \textcolor{revisions}{\phi_{\beta} \colon F(\beta) \to q^{e(\beta)} R[\mathbb{X}]_{\sigma(\beta)}}
\]

\textcolor{revisions}{of Proposition \ref{prop: rouqhom} associated to positive and negative braids.} The following is Theorem 3.13 in \cite{AH17}:

\begin{theorem} \label{thm: ab_hog_twisting}
For each positive braid $\beta \in Br_n$, $\psi_{\beta}$ induces isomorphisms of functors

\[
q^{e(\beta)}R[\mathbb{X}]_{\beta} \otimes - \cong F(\beta) \otimes -, \quad - \otimes q^{e(\beta)} R[\mathbb{X}]_{\beta} \cong - \otimes F(\beta) \textcolor{revisions}{\colon \mathcal{I}_n^- \to K^-(R_n-\text{Bim}).}
\]

\textcolor{revisions}{Similarly, for each negative braid $\beta$, $\phi_{\beta}$ induces analogous isomorphisms of functors. Composing these isomorphisms for positive and negative braids induces analogous isomorphisms for arbitrary braids $\beta$.}
\end{theorem}

\textcolor{revisions}{We wish to sketch a proof of the final claim of Theorem \ref{thm: ab_hog_twisting}, as we will make use of the ideas involved below. Before doing this, we need a quick technical lemma.}

\begin{lemma} \label{lem: twisting}
\textcolor{revisions}{For each complex $A \in \mathcal{I}_n$ and each permutation $\sigma \in \mathfrak{S}^n$, there is a natural isomorphism $R[\mathbb{X}]_{\sigma} A \cong A R[\mathbb{X}]_{\sigma}$.}
\end{lemma}

\begin{proof}
\textcolor{revisions}{Suppose $A = B_{w_0}$, considered as a quotient of the polynomial algebra $R[\mathbb{X}, \mathbb{X}']$ as in \eqref{eq: B_w0}. Each of $R[\mathbb{X}]_{\sigma} B_{w_0}$ and $B_{w_0} R[\mathbb{X}]$ has an identical description as a \textit{ring}, but with a twisted $R[\mathbb{X}, \mathbb{X}']$-module structure. Then the desired isomorphism $\psi \colon R[\mathbb{X}]_{\sigma} B_{w_0} \to B_{w_0} R[\mathbb{X}]_{\sigma}$ is the ring homomorphism taking $x_i$ to $x_{\sigma(i)}$ and $x'_j$ to $x'_{\sigma(j)}$. We check that this is a bimodule homomorphism below:}

\[
\textcolor{revisions}{\psi(x_i \cdot f) = \psi (x_{\sigma^{-1}(i)} f) = x_i \psi(f) = x_i \cdot \psi(f);} \quad \textcolor{revisions}{\psi(f \cdot x'_j) = \psi(f x'_j) = \psi(f) x'_{\sigma(j)} = \psi(f) \cdot x'_j}.
\]

\textcolor{revisions}{We extend $\psi$ to a general complex $A$ by applying $\psi$ on each chain bimodule. Nautrality then follows from the fact that any endomorphism of $B_{w_0}$ is given by multiplication by some polynomial $f(\mathbb{X}, \mathbb{X}')$.}
\end{proof}

\textcolor{revisions}{Now, fix a complex $A \in \mathcal{I}_n^-$ and a braid $\beta \in Br_n$. Suppose $\beta$ factors as $\beta = \beta^+ \beta^-$ for some positive braid $\beta^+$ and some negative braid $\beta^- \in Br_n$; the general case follows similarly. Then Abel-Hogancamp prove that the quasi-isomorphism}

\[
\textcolor{revisions}{\mathrm{id} \otimes \phi_{\beta^-} \colon A F(\beta^-) \to q^{e(\beta^-)} A R[\mathbb{X}]_{\sigma_{\beta^-}}}
\]

\textcolor{revisions}{is actually a homotopy equivalence; we denote the up to homotopy inverse of this map by $\phi_{\beta^-}^{-1}$. Similarly, the quasi-isomorphism}

\[
\textcolor{revisions}{\mathrm{id} \otimes \psi_{\beta^+} \colon q^{e(\beta^+)} A R[\mathbb{X}]_{\sigma_{\beta^+}} \to A F(\beta^+)}
\]

\textcolor{revisions}{is a homotopy equivalence with up to homotopy inverse $\psi_{\beta^+}^{-1}$. Using Lemma \ref{lem: twisting}, we build a sequence of homotopy equivalences as follows:}

\begin{tikzcd}[column sep=tiny]
\textcolor{revisions}{q^{e(\beta)} A R[\mathbb{X}]_{\sigma_{\beta}}} \arrow[r, color=revisions,"\textcolor{revisions}{\cong}"] & \textcolor{revisions}{q^{e(\beta^+) + e(\beta^-)} A R[\mathbb{X}]_{\sigma_{\beta^+}} R[\mathbb{X}]_{\sigma_{\beta^-}}} \arrow[r, color=revisions,"\textcolor{revisions}{\cong}"] \arrow[d, phantom, ""{coordinate, name=Z}]
& \textcolor{revisions}{q^{e(\beta^+)} R[\mathbb{X}]_{\sigma_{\beta^+}} \left( q^{e(\beta^-)} A R[\mathbb{X}]_{\sigma_{\beta^-}} \right)} \arrow[dll, color = revisions,
"\textcolor{revisions}{\phi_{\beta^-}^{-1}}",
rounded corners,
to path={ -- ([xshift=2ex]\tikztostart.east)
|- (Z) [near end]\tikztonodes
-| ([xshift=-2ex]\tikztotarget.west)
-- (\tikztotarget)}] & \\
\textcolor{revisions}{q^{e(\beta^+)} R[\mathbb{X}]_{\sigma_{\beta^+}} A F(\beta^-)} \arrow[r, color=revisions, "\textcolor{revisions}{\cong}"]
& \textcolor{revisions}{\left( q^{e(\beta^+)} A R[\mathbb{X}]_{\sigma_{\beta^+}} \right) F(\beta^-)} \arrow[r, color=revisions, "\textcolor{revisions}{\mathrm{id} \otimes \psi_{\beta^+}}"]
& \textcolor{revisions}{AF(\beta^+) F(\beta^-)} \arrow[r, color=revisions, "\textcolor{revisions}{\simeq}"]
& \textcolor{revisions}{AF(\beta)}
\end{tikzcd}

We wish to $y$-ify \textcolor{revisions}{Theorem \ref{thm: ab_hog_twisting}. We begin by $y$-ifying Proposition \ref{prop: rouqhom}.}

\begin{definition}
For each $\sigma \in \mathfrak{S}^n$, let $R[\mathbb{X}, \mathbb{Y}]_{\sigma}$ denote the $R[\mathbb{X}, \mathbb{X}',  \mathbb{Y}]$-module which equals \textcolor{revisions}{$R[\mathbb{X}]_{\sigma} \otimes_R R[\mathbb{Y}]$} as \textcolor{revisions}{an} $R[\mathbb{X}, \mathbb{X}']$-module and whose $R[\mathbb{Y}]$-module action is twisted by $\sigma$.
\end{definition}

\begin{proposition}
Considering $R[\mathbb{X}, \mathbb{Y}]_{\sigma}$ as a curved complex with trivial differential, we have $(R[\mathbb{X}, \mathbb{Y}]_{\sigma}, \sigma, 0) \in \mathcal{Y}_{\sigma}(R_n-\text{Bim})$.
\end{proposition}

\begin{proof}
Observe that $x_{\sigma(i)} - x'_i = 0$ as an endomorphism of $R[\mathbb{X}, \mathbb{Y}]_{\sigma}$ for each $i$.
\end{proof}

\begin{lemma} \label{lem: yified_twisting}
\textcolor{revisions}{For each curved complex $(A[\mathbb{Y}], e, \delta_A) \in \mathcal{Y} \mathcal{I}_{n, e}$ and each permutation $\sigma \in \mathfrak{S}^n$, the natural isomorphism of Lemma \ref{lem: twisting} lifts without modification to a natural isomorphism of curved complexes $R[\mathbb{X}, \mathbb{Y}]_{\sigma} A[\mathbb{Y}] \cong A[\mathbb{Y}] R[\mathbb{X}, \mathbb{Y}]_{\sigma}$.}
\end{lemma}

\begin{proof}
\textcolor{revisions}{Since the $R[\mathbb{Y}]$-module structure is the same on the left-hand and right-hand sides, there is nothing to check. The requirement that $A[\mathbb{Y}]$ be a $y$-ification with \textit{trivial} permutation is required for the left-hand and right-hand sides to be $y$-ifications with the same permtuation.}
\end{proof}

\begin{proposition} \label{prop: yify_rouqhom}
\textcolor{revisions}{If $\beta$ is a positive braid, then the quasi-isomorphism $\psi_{\beta}$ of Proposition \ref{prop: rouqhom} lifts without modification to a closed morphism}

\[
\textcolor{revisions}{\psi_{\beta} \colon q^{e(\beta)} R[\mathbb{X}, \mathbb{Y}]_{\sigma_{\beta}} \to F^y(\beta)}
\]

\textcolor{revisions}{in $\mathcal{Y}_{\sigma_{\beta}}(R_n-\mathrm{Bim})$. Similarly, if $\beta$ is a negative braid, then the quasi-isomorphism $\phi_{\beta}$ lifts without modification to a closed morphism}

\[
\textcolor{revisions}{\phi_{\beta} \colon F^y(\beta) \to q^{e(\beta)} R[\mathbb{X}, \mathbb{Y}]_{\sigma_{\beta}}}
\]

\textcolor{revisions}{in $\mathcal{Y}_{\sigma_{\beta}}(R_n-\mathrm{Bim})$.}
\end{proposition}

\begin{proof}
For a positive braid $\beta$, the only component of the connection $\delta_{F(\beta)}$ of the $y$-ified Rouquier complex that does not vanish on $F(\beta)^0[\mathbb{Y}]$ is $d_{F(\beta)} \otimes 1$. \textcolor{revisions}{This component already annihilates $\phi_{\beta}$ by Proposition \ref{prop: rouqhom}, so the entire connection $\delta_{F(\beta)}$ annihilates $\phi_{\beta}$.}

\textcolor{revisions}{Similarly, for a negative braid $\beta$, the only component of the connection $\delta_{F(\beta)}$ which enters $F(\beta)^0[\mathbb{Y}]$ is $d_{F(\beta)} \otimes 1$. This component is already annihilated by $\psi_{\beta}$, so the entire connection $\delta_{F(\beta)}$ is annihilated by $\psi_{\beta}$.}
\end{proof}

\textcolor{revisions}{Everything is now in place to state the $y$-ification of Theorem \ref{thm: ab_hog_twisting}:}

\begin{theorem} \label{thm: yifytwist}
\textcolor{revisions}{Fix a $y$-ification $(A[\mathbb{Y}], e, \delta_A) \in \mathcal{Y} \mathcal{I}_{n, e}^b.$ Then for each positive braid $\beta \in Br_n$, $\psi_{\beta}$ induces homotopy equivalences of curved complexes}

\[
\textcolor{revisions}{q^{e(\beta)} R[\mathbb{X}, \mathbb{Y}]_{\sigma_{\beta}} A[\mathbb{Y}] \simeq F^y(\beta) A[\mathbb{Y}]}, \quad \textcolor{revisions}{q^{e(\beta)} A[\mathbb{Y}] R[\mathbb{X}, \mathbb{Y}]_{\sigma_{\beta}} \simeq A[\mathbb{Y}] \otimes F^y(\beta)}
\]

\textcolor{revisions}{Similarly, for each negative braid $\beta$, $\phi_{\beta}$ induces analogous homotopy equivalences. Composing these equivalences for positive and negative braids induces analogous equivalences for arbitrary braids $\beta$.}
\end{theorem}

\begin{proof}
\textcolor{revisions}{First, observe that exactly as in the un $y$-ified case, we may use Lemma \ref{lem: yified_twisting} to reduce the proof to the case in which $\beta$ is a positive or a negative braid. The only potential concern is that some intermediate step in this process might require an analog of Lemma \ref{lem: yified_twisting} for some nontrivial permutation, but this does not occur precisely because $A[\mathbb{Y}]$ is a $y$-ification with \textit{trivial} permutation.}

\textcolor{revisions}{We treat the case in which $\beta$ is a positive braid; the negative case follows similarly. To that end, let $\beta \in Br_n$ be a positive braid and $(A[\mathbb{Y}], e, \delta_A) \in \mathcal{Y} \mathcal{I}_{n, e}^b$ a $y$-ification. By Proposition \ref{prop: yify_rouqhom}, we obtain a closed morphism of curved complexes}

\[
\textcolor{revisions}{\mathrm{id} \otimes \psi_{\beta}^y \colon q^{e(\beta)} A[\mathbb{Y}] R[\mathbb{X}, \mathbb{Y}]_{\sigma_{\beta}} \to A[\mathbb{Y}] F^y(\beta).}
\]

\textcolor{revisions}{Here we write $\psi_{\beta}^y$ rather than $\psi_{\beta}$ to emphasize that this is a morphism of \textit{curved} complexes. This distinction comes in handy immediately, as we can express $\mathrm{Cone}(\mathrm{id} \otimes \psi_{\beta}^y)$ as a convolution:}

\[
\textcolor{revisions}{\mathrm{Cone}(\mathrm{id} \otimes \psi_{\beta}^y) = \mathrm{tw}_{\Delta} \left( \mathrm{Cone}(\mathrm{id} \otimes \psi_{\beta}) \otimes_R R[\mathbb{Y}] \right)}
\]

\textcolor{revisions}{for some twist $\Delta$. Because both $A$ and $F(\beta)$ are bounded, this is a homologically locally finite, lower finite convolution. Further, $\mathrm{Cone}(\mathrm{id} \otimes \psi_{\beta})$ is contractible by Theorem \ref{thm: ab_hog_twisting}. It follows from Proposition \ref{prop: useful_hpt} that $\mathrm{Cone}(\mathrm{id} \otimes \psi_{\beta}^y)$ is contractible as well.}
\end{proof}

\begin{remark}
\textcolor{revisions}{In fact, we can promote the statement of Theorem \ref{thm: yifytwist} to the case in which $A$ is bounded below without much difficulty. We first extend Theorem \ref{thm: ab_hog_twisting} to unbounded complexes $A$ by writing $A$ as a convolution between a bounded above and a bounded below complex, then applying Proposition \ref{prop: useful_hpt}. Then the promotion to the $y$-ified case uses the fact that a $y$-ification of a bounded below complex is still a homologically locally finite convolution. We will not use this extension here.}
\end{remark}

We will be particularly interested in the result of Theorem \ref{thm: yifytwist} in the case \textcolor{revisions}{$A[\mathbb{Y}] = K_n^y$}.

\begin{definition}
For each $\sigma, \rho \in \mathfrak{S}^n$, let 

\[
K^y_{\rho, n, \sigma} := R[\mathbb{X}, \mathbb{Y}]_{\rho} K^y_n  R[\mathbb{X}, \mathbb{Y}]_{\sigma}.
\]

We continue to reserve the notation $K^y_n$ for $K^y_{e, n, e}$ and extend this convention by suppressing $\rho$ and $\sigma$ whenever they are trivial. In keeping with the terminology of \cite{AH17}, we call these curved complexes \textit{$y$-ified twisted projectors}.
\end{definition} 

In fact it suffices to consider ``one-sided" twisted projectors, as twisting on the left and right yield isomorphic curved complexes.

\begin{lemma} \label{lem: left_right_twist}
For each $\sigma, \rho \in \mathfrak{S}^n$, we have $K^y_{\sigma, n, \rho} \cong \textcolor{revisions}{K^y_{n, \sigma \rho}}$.
\end{lemma}

\begin{proof}
\textcolor{revisions}{This is a direct application of Lemma \ref{lem: yified_twisting} with $A[\mathbb{Y}] = K^y_n$.}
\end{proof}

\textcolor{revisions}{Theorem \ref{thm: yifytwist} allows us to absorb braids whose strands all feed into a $y$-ified twisted projector up to a change in permutation.}

\begin{proposition} \label{prop: yify_proj_twist}
For each braid $\beta \in Br_n$ \textcolor{revisions}{and each permutation $\rho \in \mathfrak{S}^n$}, we have homotopy equivalences $\textcolor{revisions}{K^y_{n, \rho} F^y(\beta) \simeq q^{e(\beta)} K^y_{n, \rho \sigma_\beta}}$ and $\textcolor{revisions}{F^y(\beta) K^y_{n, \rho} \simeq q^{e(\beta)} K^y_{n, \sigma_\beta \rho}}$.
\end{proposition}

\begin{proof}
\textcolor{revisions}{We treat the first equivalence; the second follows exactly analogously. Let $\beta$, $\rho$ be as above. Applying Theorem \ref{thm: yifytwist} and Lemma \ref{lem: left_right_twist}, we obtain a sequence of homotopy equivalences}

\[
\textcolor{revisions}{K^y_{n, \rho} F^y(\beta) \cong R[\mathbb{X}, \mathbb{Y}]_{\rho} \left(K^y_n F^y(\beta) \right) \simeq R[\mathbb{X}, \mathbb{Y}]_{\rho} \left( q^{e(\beta)} K^y_n R[\mathbb{X}, \mathbb{Y}]_{\sigma_{\beta}} \right) = q^{e(\beta)} K^y_{\rho, n, \sigma_{\beta}} \cong q^{e(\beta)} K^y_{n, \rho \sigma_{\beta}}.}
\]
\end{proof}

\textcolor{revisions}{In our treatment of colored link homology, w}e will also need to know that \textcolor{revisions}{$K^y_{n, \sigma}$} slides past crossings\textcolor{revisions}{. This statement is completely analogous to Proposition \ref{prop: inf_cross_slide} for the infinite projector $(P_{(1^n)}^y)^{\vee}$ and can be viewed as a ``braid relation" for $K^y_{n, \sigma}$, allowing us to manipulate $K^y_{n, \sigma}$ topologically, just as we would a braid on $n$ strands, in the diagrammatic calculus.}

\begin{proposition} \label{prop: ky_cross_slide}
\textcolor{revisions}{Let $X_{n + 1}, Y_{n + 1} \in Br_{n + 1}$ be as in Proposition \ref{prop: inf_cross_slide}. Then for each permutation $\sigma \in \mathfrak{S}^n$, we have homotopy equivalences $F^y(X_{n + 1}) K^y_{n, \sigma} \simeq K^y_{n, \sigma} F^y(X_{n + 1})$ and $F^y(Y_{n + 1}) K^y_{n, \sigma} \simeq K^y_{n, \sigma} F^y(Y_{n + 1})$.}
\end{proposition}

\begin{proof}
\textcolor{revisions}{We treat the first equivalence; the second follows exactly analogously. We begin with the case $\sigma = e$. In this case, we may express the right-hand side explicitly as a convolution:}

\begin{align*}
\textcolor{revisions}{K^y_n F^y(X_{n + 1})} & \textcolor{revisions}{= \mathrm{tw}_{\delta} \left( B_{w_0} F(X_{n + 1})[\mathbb{Y}] \otimes \wedge[\theta_2, \dots, \theta_n] \right);} \\
\textcolor{revisions}{\delta} & \textcolor{revisions}{ = \mathrm{id}_{B_{w_0}} \left(d_{X_n} + \sum_{i = 1}^{n + 1} h_iy_i \right) \otimes 1 + \sum_{j = 2}^n (x_j - x_j') \otimes \theta_j^{\vee} + (y_{j + 1} - y_2) \otimes \theta_j.}
\end{align*}

\textcolor{revisions}{Here $\mathbb{X}$ is  the alphabet corresponding to the \textit{left} action of $R_n$, $\mathbb{X}'$ is the alphabet over which the tensor product of $K^y_n$ and $F^y(X_{n + 1})$ is taken, $\mathbb{X}''$ is the alphabet corresponding to the \textit{right} action of $R_n$, and $h_j \in \mathrm{End}(F(X_j))$ are dot-sliding homotopies satisfying $[d, h_j] = x'_{\sigma_{X_{n + 1}}(j)} - x''_j$.}

\textcolor{revisions}{Completely analogously to the proof of Proposition \ref{prop: y_forkslide_i}, we apply Proposition \ref{prop: kosz_base} to replace terms of the form $(x_j - x_j') \otimes \theta_j^{\vee}$ with $(x_j - x_{j + 1}'') \otimes \theta_j^{\vee}$. The computations involved are identical up to a change in notation for various morphisms. As a result, we obtain}

\begin{align*}
\textcolor{revisions}{K^y_n F^y(X_{n + 1})} & \textcolor{revisions}{= \mathrm{tw}_{\delta'} \left( B_{w_0} F(X_{n + 1})[\mathbb{Y}] \otimes \wedge[\theta_2, \dots, \theta_n] \right);} \\
\textcolor{revisions}{\delta'} & \textcolor{revisions}{= \mathrm{id}_{B_{w_0}} \left(d_{X_n} + h_1(y_1 - y_2) \right) + \sum_{j = 2}^n (x_j - x_{j + 1}'') \otimes \theta_j^{\vee} + (y_{j + 1} - y_2) \otimes \theta_j.}
\end{align*}

\textcolor{revisions}{Next, we wish to apply the homotopy equivalence $B_{w_0} F(X_n) \simeq F(X_n) B_{w_0}$ to each term of this convolution. The morphisms involved automatically commute with every component of $\delta'$ except $h_1(y_1 - y_2)$. That these morphisms do commute with $h_1(y_1 - y_2)$ follows from repeated application of Lemma \ref{lem: dot_slide_fork_slide} from Appendix \ref{app: dot_slide_nat}. By the same argument as in Corollary \ref{cor: half_forkslyde}, applying this homotopy equivalence to each term induces a homotopy equivalence of convolutions}

\begin{align*}
\textcolor{revisions}{K^y_n F^y(X_{n + 1})} & \textcolor{revisions}{= \mathrm{tw}_{\delta'} \left( F(X_{n + 1}) B_{w_0} [\mathbb{Y}] \otimes \wedge[\theta_2, \dots, \theta_n] \right).}
\end{align*}

\textcolor{revisions}{From here, we once again apply Proposition \ref{prop: kosz_base} to replace the terms $(x_j - x_{j + 1}'') \otimes \theta_j^{\vee}$ in $\delta'$ with $(x_{j + 1}' - x_{j + 1}'') \otimes \theta^j$. This exactly reverses the previous application of Proposition \ref{prop: kosz_base}, resulting in the following:}

\begin{align*}
\textcolor{revisions}{K^y_n F^y(X_{n + 1})} & \textcolor{revisions}{= \mathrm{tw}_{\delta''} \left( F(X_{n + 1}) B_{w_0} [\mathbb{Y}] \otimes \wedge[\theta_2, \dots, \theta_n] \right);} \\
\textcolor{revisions}{\delta''} & \textcolor{revisions}{ = \left(d_{X_n} + \sum_{i = 1}^{n + 1} h_iy_i \right) \mathrm{id}_{B_{w_0}} \otimes 1 + \sum_{j = 2}^n (x_{j + 1}' - x_{j + 1}'') \otimes \theta_j^{\vee} + (y_{j + 1} - y_2) \otimes \theta_j.}
\end{align*}

\textcolor{revisions}{This is exactly $F^y(X_{n + 1}) K^y_n$.}

\textcolor{revisions}{To extend to an arbitrary permutation $\sigma \in \mathfrak{S}^n$, let $\beta \in Br_n$ be a braid with associated permutation $\sigma$, and recall the homotopy equivalence $K^y_{n, \sigma} \simeq q^{-e(\beta)} F^y(\beta) K^y_n$ of Proposition \ref{prop: yify_proj_twist}. Then we have a sequence of homotopy equivalences}

\begin{multline*}
\textcolor{revisions}{K^y_{n, \sigma} F^y(X_{n + 1}) \simeq q^{-e(\beta)} F^y(\beta) K^y_n F^y(X_{n + 1}) \simeq q^{-e(\beta)} F^y(\beta) F^y(X_{n + 1}) K^y_n} \\
\textcolor{revisions}{\simeq q^{-e(\beta)} F^y(X_{n + 1}) F^y(\beta) K^y_n \simeq F^y(X_{n + 1}) K^y_{n, \sigma}.}
\end{multline*}

\textcolor{revisions}{Here the third homotopy equivalence follows from the corresponding identity $\beta X_{n + 1} = X_{n + 1} \beta \in Br_{n + 1}$.}

\end{proof}

In what follows, we will need an analog of the $y$-ified Markov Moves of Proposition \ref{prop: yktrace} for $y$-ified twisted projectors.

\begin{proposition} \label{prop: twist_mark_j}
For each $\sigma \in \mathfrak{S}^{n - 1}$, there is a homotopy equivalence $K^y_{n - 1, \sigma}J^y_n \simeq (q^{n - 3}tK^y_{n, \sigma} \rightarrow q^{-2}t^2 K^y_{n - 1, \sigma})$.
\end{proposition}

\begin{proof}
Let $\beta(\sigma) \in Br_{n - 1}$ be any positive braid lift of $\sigma$. By Proposition \ref{prop: yify_proj_twist}, we have a homotopy equivalence $K^y_{n - 1, \sigma} \simeq q^{-e(\beta(\sigma))} K^y_{n - 1} F^y(\beta(\sigma))$. Because $J_n$ commutes with all elements of $Br_{n - 1}$, we have a homotopy equivalence $K^y_{n - 1, \sigma} J^y_n \simeq q^{-e(\beta(\sigma))}K^y_{n - 1}J^y_n F^y(\beta(\sigma))$. The result then follows immediately from identity (2) in Proposition \ref{prop: yktrace} and another application of Proposition \ref{prop: yify_proj_twist}.
\end{proof}

\begin{proposition} \label{prop: twist_mark_tr}
For each $\sigma \in \mathfrak{S}^n$ satisfying $\sigma(n) = n$, we have 
\[
Tr^y_{n, \sigma}(K^y_{n, \sigma}) \simeq \left( \cfrac{q^{n - 1} + aq^{-n - 1}}{1 - q^2} \right) q^2 t^{-1} K^y_{n - 1, Tr_n(\sigma)}.
\]

For all other $\sigma \in \mathfrak{S}^n$, we have

\[
Tr^y_{n, \sigma}(K^y_{n, \sigma}) \simeq (q^{n - 1} + aq^{-n - 1}) K^y_{n - 1, Tr_n(\sigma)}.
\]
\end{proposition}

\begin{proof}
The first assertion follows from a straightforward application of Propositions \ref{prop: yktrace} and \ref{prop: yify_proj_twist} together with linearity of $Tr^y_{n, \sigma}$. To see the second, observe that any $\sigma \in \mathfrak{S}^n$ satisfying $\sigma(n) \neq n$ can be written as $\sigma = \rho_1 s_{n - 1} \rho_2$ for some $\rho_1, \rho_2 \in \mathfrak{S}^{n - 1}$; moreover, under this decomposition, we have $Tr_n(\sigma) = \rho_1 \rho_2$. Combined with linearity of $Tr^y_{n, \sigma}$ and Propositions \ref{prop: yify_proj_twist} and \ref{prop: ky_cross_slide}, this allows us to restrict our attention to the case $\sigma = s_{n - 1}$.

Recall the homotopy equivalence
\[
K^y_n \simeq \big(q^{-n + 1} K^y_{n - 1} \rightarrow q^{-n + 3}t^{-1} K^y_{n - 1} J^y_n\big)
\]

of Proposition \ref{prop: fynite}. Upon taking a tensor product with $F^y(\sigma_{n - 1}^{-1})$ and simplifying the final term via an isotopy, we obtain
\[
K^y_n F^y(\sigma_{n - 1}^{-1}) \simeq \big(q^{-n + 1} K_{n - 1} F^y(\sigma_{n - 1}^{-1}) \rightarrow q^{-n + 3}t^{-1} K^y_{n - 1} F^y(\sigma_{n - 1}) J^y_{n - 1}\big)
\]

By Proposition \ref{prop: yify_proj_twist}, we have $K^y_{n - 1} F^y(\sigma_{n - 1}^{-1}) \simeq q^{-1} K^y_{n - 1, s_{n - 1}}$. Upon making this substitution and applying $Tr^y_{n, s_{n - 1}}$ to all terms, we obtain
\begin{align*}
q^{-1}Tr^y_{n, s_{n - 1}}(K^y_{n, s_{n - 1}}) & \simeq \big(q^{-n + 1} K^y_{n - 1} Tr^y_{n, s_{n - 1}}(F^y(\sigma_{n - 1}^{-1})) \rightarrow q^{-n + 3}t^{-1} K^y_{n - 1} Tr^y_{n, s_{n - 1}}(F^y(\sigma_{n - 1})) J^y_{n - 1}\big) \\
& \simeq \big(aq^{-n - 2} K^y_{n - 1} \rightarrow q^{-n + 2} K^y_{n - 1} J^y_{n - 1}\big) \\
& \simeq \big(aq^{-n - 2} K^y_{n - 1} \rightarrow q^{n - 2} K^y_{n - 1}\big)
\end{align*}

Here the second homotopy equivalence follows from the Markov moves of Proposition \ref{prop: ymark} and the third follows from Proposition \ref{prop: yify_proj_twist}. Now, the twist on this convolution must be a map of total degree $t = a^0q^0t^1$. On the other hand, $K^y_{n - 1}$ is concentrated in Hochschild ($a$) degree $0$. Since there is a relative Hochschild degree shift between the two terms of this convolution, in fact the twist must be equal to $0$, and this convolution degenerates to a direct sum. After regrading by $q$, we obtain

\[
Tr^y_{n, s_{n - 1}}(K^y_{n, s_{n - 1}}) \simeq (aq^{-n - 1} + q^{n - 1})K^y_{n - 1}
\]

as desired.
\end{proof}

\section{Applications to Link Homology} \label{sec: link_hom}

\subsection{Torus Link Homology} \label{sec: uncolor_hom}

We now employ the recursive machinery developed above to compute $y$-ified homology. To do this, we will find it convenient to work with a different convention for the normalization of crossings and finite projectors; this normalization will be in place for the remainder of this paper (excluding \textcolor{revisions}{Appendices \ref{app: ssbim} and \ref{app: dot_slide_nat}}). The choice of normalization for crossings is such that our projectors absorb crossings (up to a twisted left or right polynomial action) without a grading shift; the choice of normalization for finite projectors is such that identities (4) and (5) of Proposition \ref{prop: renorm.markov} have the correct coefficients as predicted by mirror symmetry.

\begin{definition} \label{def: new_braids}
Let $\beta \in Br_n$ be given, and let $F(\beta)$ denote its Rouquier complex. We set $\hat{F}(\beta) := q^{-e(\beta)} F(\beta)$ and use similar notation $\hat{F}^y(\beta)$ for $y$-ified Rouquier complexes.
\end{definition}

\begin{definition}
For each $n \geq 1$ and each $\sigma \in \mathfrak{S}^n$, set 

\[
\hat{K}^y_{n, \sigma} := q^{(2 + \sum_{k = 1}^n k - 3)}t^{n - 1} K^y_{n, \sigma} \textcolor{revisions}{ = q^{(n - 1)(n - 4)/2} t^{n - 1} K^y_{n, \sigma}}.
\]
\end{definition}

\begin{proposition} \label{prop: ky_absorb_renorm}
For each positive braid $\beta \in Br_n$, we have $\hat{K}^y_{n, \sigma} \hat{F}^y(\beta) \simeq \hat{K}^y_{n, \sigma \beta}$ and $\hat{F}^y(\beta) \hat{K}^y_{n, \sigma} \simeq \hat{K}^y_{n, \beta \sigma}$.
\end{proposition}

\begin{proof}
\textcolor{revisions}{Follows immediately from Proposition \ref{prop: yify_proj_twist} after regrading as in Definition \ref{def: new_braids}.}
\end{proof}

We will also employ a change of variables $Q := q^2$, $T := q^{-2}t^2$, $A := aq^{-2}$ in our grading shifts. This change of variables is common in the literature, and the expected mirror symmetry relation in these variables reduces to $Q \leftrightarrow T$.

We record renormalized versions of the various Markov-type relationships relevant to our computations below.

\begin{proposition} \label{prop: renorm.markov} With the normalizations above, the following identities hold:

\begin{enumerate}
    \item $Tr^y_{n, e}(R_n[\mathbb{Y}]) \simeq \left( \cfrac{1 + A}{(1 - Q)(1 - T)} \right) R_{n - 1}[\mathbb{Y}]$;
    
    \item $Tr^y_{n, s_{n - 1}}(\hat{F}^y(\sigma_{n - 1})) \simeq Q^{-1}t R_{n - 1}[\mathbb{Y}]$;
    
    \item $\hat{K}_1^y \cong R_1[\mathbb{Y}]$;
    
    \item $Tr^y_{n, \sigma}(\hat{K}^y_{n, \sigma}) \simeq \left(\cfrac{Q^{n - 1} + A}{1 - Q} \right) \hat{K}^y_{n - 1, Tr_n(\sigma)}$ for each $n \geq 2$, $\sigma(n) = n$;
    
    \item $Tr^y_{n, \sigma}(\hat{K}^y_{n, \sigma}) \simeq Q^{-1}t(Q^{n - 1} + A) \hat{K}^y_{n - 1, Tr_n(\sigma)}$ for each $n \geq 2$, $\sigma(n) \neq n$;
    
    \item $\hat{K}^y_{n - 1, \sigma} \hat{F}^y(J_n) \simeq \left(Q^{-(n - 1)} \hat{K}^y_{n, \sigma} \longrightarrow Q^{-(n - 1)}T \hat{K}^y_{n - 1, \sigma}\right)$.
\end{enumerate}

\end{proposition}

\begin{proof}
Follows immediately from Propositions \ref{prop: ymark}, \ref{prop: twist_mark_j}, and \ref{prop: twist_mark_tr} after regrading.
\end{proof}

Inspired by the computations of \cite{HM19}, we use the (renormalized, curved, twisted, finite) projectors $\hat{K}^y_{j, \sigma}$ to construct a family of complexes that will prove valuable in computing homology.

\begin{definition} \label{def: shuffle}
\textcolor{revisions}{For each binary sequence $v \in \{0, 1\}^r$, we denote by $|v|$ the number of ones in $v$. Then the \textit{shuffle permutation} $\pi_v \in \mathfrak{S}^r$ associated to $v$ is the permutation sending the set $\{1, \dots, r - |v|\}$ to the set of indices $i$ for which $v_i = 0$ and the set $\{r - |v| + 1, \dots, r\}$ to the set of indices $i$ for which $v_i = 1$, both in an order-preserving fashion. We denote by $\alpha_v, \beta_v \in Br_n$ the positive braid lifts of the shuffle permutations $\pi_v$ and $\pi_v^{-1}$, respectively.}
\end{definition}

\begin{definition} \label{def: reccomp}
Let $v \in \{0, 1\}^{m + l}$ and $w \in \{0, 1\}^{n + l}$ be finite sequences with $|v| = |w| = l$. For each pair $(v, w)$ and $\sigma \in \mathfrak{S}^l$, we define the curved complexes

\[
\textbf{C}^y(v, w)_{\sigma} := (R_n[\mathbb{Y}] \sqcup \hat{F}^y(\alpha_v)) \otimes (\hat{F}^y(\textcolor{revisions}{\alpha_{(1^m0^n)}} \sqcup \hat{K}^y_{l, \sigma}) \otimes (R_m[\mathbb{Y}] \sqcup \hat{F}^y(\beta_w))
\]

We set $\textbf{C}^y(\emptyset, \emptyset)_e := R$ for notational convenience.
\end{definition}

\textcolor{revisions}{We can also depict $\textbf{C}^y(v, w)_{\sigma}$ in the graphical calculus as follows. Here we use a strand labeled with a positive integer $k$ to denote a $k$ parallel strands strands for convenience, and we assign to each braid $\beta$ its \textit{renormalized, $y$-ified} Rouquier complex $\hat{F}^y(\beta)$.}

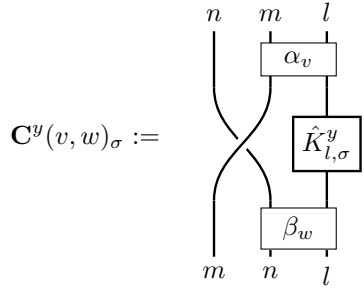
\begin{center}
\begin{gather*}
\textbf{C}^y(v, w)_{\sigma} := \quad
\begin{tikzpicture}[baseline=(current bounding box.center), scale=0.75]
	\draw[webs] (1,0) node[below]{$n$} to (1,1);
	\draw[webs] (1,1) to[out=90,in=270] (0,3);
	\draw[webs] (0,0) node[below]{$m$} to (0,1);
	\draw[color=white, line width=5pt] (0,1) to[out=90,in=270] (1,3);
	\draw[webs] (0,1) to[out=90,in=270] (1,3);
	\draw[webs] (2,0) node[below]{$l$} to (2,1);
	\node[draw, fill=white, minimum width=1cm] at (1.5,.5) {$\beta_w$};
	\draw[webs] (2,1) to node[pos=.5, draw, fill=white]{$\hat{K}^y_{l, \sigma}$} (2,3);
	\draw[webs] (0,3) to (0,4) node[above]{$n$};
	\draw[webs] (1,3) to (1,4) node[above]{$m$};
	\draw[webs] (2,3) to (2,4) node[above]{$l$};
	\node[draw, fill=white, minimum width=1cm] at (1.5,3.5) {$\alpha_v$};
\end{tikzpicture}
\end{gather*}
\end{center}

\begin{definition} \label{def: renorm_reccomp}
For each possible triple $(v, w, \sigma)$, \textcolor{revisions}{the curved complex $\textbf{C}^y(v, w)_{\sigma}$ is a $y$-ification of a complex of Soergel bimodules. Explicitly,} we have $\textbf{C}^y(v, w)_{\sigma} \in \mathcal{Y}_{\rho}(R_k-\text{Bim})$\textcolor{revisions}{for $k = l + m + n$ and}

\[
\textcolor{revisions}{\rho = (e_n \sqcup \pi_v) (\pi_{(1^m 0^n)} \sqcup \sigma) (e_m \sqcup \pi_w^{-1}) \in \mathfrak{S}^k.}
\]

Let $c$ denote the \textcolor{revisions}{number of of disjoint cycles} of $\rho$. We set

\[
\hat{\textbf{C}}^y(v, w)_{\sigma} := \left( Qt^{-1} \right)^{\textcolor{revisions}{c - k}} \textbf{C}^y(v, w)_{\sigma}.
\]

\end{definition}

\textcolor{revisions}{In the spirit of \cite{HM19}, we compute $HH^y(\hat{\textbf{C}}^y(v, w)_{\sigma})$ for all input data $v, w, \sigma$ using the identities of Proposition \ref{prop: renorm.markov}. Our computation will mostly follow the computation given there, with slight modifications accounting for our distinct normalization and the presence of the permutation $\sigma$. We begin with our analog of their Lemma 3.4.}

\begin{lemma} \label{lem: recursion_well_defined}
For each pair $v \in \{0, 1\}^{l + m}$, $w \in \{0, 1\}^{l + n}$ and each $\sigma \in \mathfrak{S}^l$, \textcolor{revisions}{there is a unique power series $p_{\sigma}(v, w) \in \mathbb{N}[[a, q, t]]$ satisfying the following identities:}

\begin{enumerate}
	\item \textcolor{revisions}{$p_e(0^m, \emptyset) = \left( \cfrac{1 + A}{(1 - Q)(1 - T)} \right)^m$ and $p_e(\emptyset, 0^n) = \left( \cfrac{1 + A}{(1 - Q)(1 - T)} \right)^n$.}
	
    \item $p_e(0^m1, 0^n1) = \left( \cfrac{1 + A}{(1 - Q)(1 - T)} \right) p_e(0^m, 0^n)$ for all $m, n \geq 0$;
    
    \item $p_{\sigma}(v1, w1) = \left( \cfrac{Q^l + A}{1 - Q} \right) p_{Tr_{l + 1}(\sigma)}(v, w)$ if $|v| = |w| = l \geq 1$ and $\sigma(n) = n$;
    
    \item $p_{\sigma}(v1, w1) = \left( Q^l + A \right) p_{Tr_{l + 1}(\sigma)}(v, w)$ if $|v| = |w| = l \geq 1$ and $\sigma(n) \neq n$;
    
    \item $p_{\sigma}(v0, w1) = p_{\sigma \textcolor{revisions}{\pi_{(10^{l - 1})}}}(v, 1w)$ if $|v| = |w| + 1 = l \geq 1$;
    
    \item $p_{\sigma}(v1, w0) = p_{\textcolor{revisions}{\pi_{(10^{l - 1})}} \sigma}(1v, w)$ if $|v| + 1 = |w| = l \geq 1$;
    
    \item $p_e(0^m, 0^n) = p_e(10^{m - 1}, 10^{n - 1})$ \textcolor{revisions}{for all $m, n \geq 1$};
    
    \item $p_{\sigma}(v0, w0) = Q^{-l}p_{\textcolor{revisions}{e_1 \sqcup \sigma}}(1v, 1w) + Q^{-l}Tp_{\sigma}(0v, 0w)$ if $|v| = |w| = l \geq 1$.
\end{enumerate}
\end{lemma}

\begin{proof}
\textcolor{revisions}{We define a partial order $\leq$ on the set of all binary sequences by declaring $v \leq v'$ if any of the following conditions hold:}

\begin{enumerate}
\item \textcolor{revisions}{$\ell(v) < \ell(v')$;}

\item \textcolor{revisions}{$\ell(v) = \ell(v')$ and $|v| > |v'|$;}

\item \textcolor{revisions}{$\ell(v) = \ell(v')$, $|v| = |v'|$, and $\mathrm{inv}(v) \leq \mathrm{inv}(v')$.}
\end{enumerate}

\textcolor{revisions}{Here $\ell(v)$ denotes the length of $v$ and $\mathrm{inv}(v)$ denotes the number of pairs of indices $i < j$ with $v_i = 1$, $v_j = 0$. It is pointed out in \cite{HM19} that $\leq$ turns the set of binary sequences into a lower finite poset with unique minimum $\emptyset$. We write $(v, w) \leq (v', w')$ if $v \leq v'$ and $w \leq w'$.}

\textcolor{revisions}{We construct $p_{\sigma}(v, w)$ recursively as follows. If $v = \emptyset$, then we must have $|w| = |v| = 0$, so $w = 0^n$ for some $n$. Similarly, if $w = \emptyset$, then $v = 0^m$ for some $m$. In either case, we must have $\sigma = e \in \mathfrak{S}^0$, and so $p_{e}(v, w)$ is uniquely determined by (1).}

\textcolor{revisions}{Now assume neither $v$ nor $w$ is empty. If $v$ and $w$ both contain entirely zeros, then $l = 0$, so we must have $\sigma = e$, and we can rewrite $p_{e}(v, w)$ using (7). Otherwise, we can rewrite $p_{\sigma}(v, w)$ using exactly one of relations (2) - (6) or (8) in terms of $p_{\sigma'}(v', w')$ with $(v', w') < (v, w)$. This process eventually terminates because each $(v, w)$ has only finitely many pairs $(v', w')$ satisfying $(v', w') < (v, w)$.}
\end{proof}

\begin{remark}
\textcolor{revisions}{Notice that $p_{\sigma}(v,w)$ is not just a power series in $a, q, t$, but a rational function in $A, Q, T$. In particular, $p_{\sigma}(v, w)$ only involves \textit{even} powers of $t$.}
\end{remark}

\textcolor{revisions}{Next, we show that the complexes $\hat{\textbf{C}}^y_{\sigma}(v, w)$ satisfy categorical analogs of Properties (1)-(8) in Lemma \ref{lem: recursion_well_defined}. The following is our analog of Lemma 3.7 in \cite{HM19}.}

\begin{lemma} \label{lem: hm3.7}
\textcolor{revisions}{Let $v \in \{0, 1\}^{m + l}$ and $w \in \{0, 1\}^{n + l}$ be sequences with $|v| = |w| + 1 = l \geq 1$. Then}

\[
\textcolor{revisions}{Tr^y_{k, \rho} \left( \hat{\textbf{C}}^y_{\sigma}(v0, w1) \right) \simeq \hat{\textbf{C}}^y_{\sigma \pi_{(10^{l - 1})}} (v, 1w).}
\]
\end{lemma}

\begin{proof}
\textcolor{revisions}{First, a straightforward computation shows that the underlying permutations in each of $\hat{\textbf{C}}^y_{\sigma}(v0, w1)$ and $\hat{\textbf{C}}^y_{\sigma \pi_{(10^l)}} (v, 1w)$ have the same number of disjoint cycles $c$. The overall number of strands $k$ decreases by one from the left-hand side to the right-hand side, so the overall normalization relative to the unnormalized complexes of Definition \ref{def: reccomp} changes by a factor of $Q^{-1}t$ relative to these unnormalized complexes.}

\textcolor{revisions}{Using Proposition \ref{prop: ky_cross_slide}, we have}

\begin{center}
\begin{gather*}
\textcolor{revisions}{\hat{\textbf{C}}^y_{\sigma}(v0, w1) := (Qt^{-1})^{c - k}}
\begin{tikzpicture}[baseline=(current bounding box.center), scale=0.75]
\draw[webs] (1.5,0) node[below]{$n$} to (1.5,2);
\draw[webs] (2.5,0) node[below]{$l - 1$} to (2.5,4.5) node[above]{$l$};
\node[draw, fill=white, minimum width=1cm] at (2,.5) {$\beta_w$};
\draw[webs] (3.5,0) node[below]{$1$} to (3.5,1.5);
\node[draw, fill=white, minimum width=1.5cm] at (3,2) {$\hat{K}^y_{l, \sigma}$};
\draw[webs] (0,0) node[below]{$m$} to (0,2);
\draw[webs] (.5,0) node[below]{$1$} to (.5,2);
\draw[webs] (1.5,2) to[out=90,in=270] (0,4.5) node[above]{$n$};
\draw[color=white, line width=5pt] (0,2) to[out=90,in=270] (1.5,3.5);
\draw[webs] (0,2) to[out=90,in=270] (1.5,3.5);
\draw[color=white, line width=5pt] (.5,2) to[out=90,in=270] (3.5,3.5);
\draw[webs] (.5,2) to[out=90,in=270] (3.5,3.5);
\draw[webs] (1.5,3.5) to (1.5,4.5) node[above]{$m$};
\node[draw, fill=white, minimum width=1cm] at (2,4) {$\alpha_v$};
\draw[webs] (3.5,3.5) to (3.5,4.5) node[above]{$1$};
\end{tikzpicture}
\enspace \textcolor{revisions}{\simeq (Qt^{-1})^{c - k}} \begin{tikzpicture}[baseline=(current bounding box.center), scale=0.75]
\draw[webs] (1.5,0) node[below]{$n$} to (1.5,2);
\draw[webs] (2.5,0) node[below]{$l - 1$} to (2.5,4.5) node[above]{$l$};
\node[draw, fill=white, minimum width=1cm] at (2,.5) {$\beta_w$};
\draw[webs] (3.5,0) node[below]{$1$} to (3.5,1.5);
\draw[webs] (3.5,1.5) to[out=90,in=270] (3,2.5);
\node[draw, fill=white, minimum width=1cm] at (2.75,3) {$\hat{K}^y_{l, \sigma}$};
\draw[webs] (0,0) node[below]{$m$} to (0,2);
\draw[webs] (.5,0) node[below]{$1$} to (.5,1);
\draw[webs] (1.5,2) to[out=90,in=270] (0,4.5) node[above]{$n$};
\draw[color=white, line width=5pt] (0,2) to[out=90,in=270] (1.5,3.5);
\draw[webs] (0,2) to[out=90,in=270] (1.5,3.5);
\draw[color=white, line width=5pt] (.5,1) to[out=90,in=270] (3.5,2.5);
\draw[webs] (.5,1) to[out=90,in=270] (3.5,2.5);
\draw[webs] (1.5,3.5) to (1.5,4.5) node[above]{$m$};
\node[draw, fill=white, minimum width=1cm] at (2,4) {$\alpha_v$};
\draw[webs] (3.5,2.5) to (3.5,4.5) node[above]{$1$};
\end{tikzpicture}
\end{gather*}
\end{center}

\textcolor{revisions}{We may now apply $Tr^y_{k, \rho}$ to the final complex. By locality of $Tr^y$ (Proposition \ref{prop: try_local}) and the renormalized Markov move for a positive crossing (Property (2) of Proposition \ref{prop: renorm.markov}), we obtain}

\begin{center}
\begin{gather*}
\textcolor{revisions}{Tr^y_{k, \rho} \left( \hat{\textbf{C}}^y_{\sigma}(v0, w1) \right) \simeq (Qt^{-1})^{c - k - 1}}
\begin{tikzpicture}[baseline=(current bounding box.center), scale=0.75]
\draw[webs] (1.5,0) node[below]{$n$} to (1.5,2);
\draw[webs] (2.5,0) node[below]{$l - 1$} to (2.5,4.5) node[above]{$l$};
\node[draw, fill=white, minimum width=1cm] at (2,.5) {$\beta_w$};
\node[draw, fill=white, minimum width=1cm] at (2.75,3) {$\hat{K}^y_{l, \sigma}$};
\draw[webs] (0,0) node[below]{$m$} to (0,2);
\draw[webs] (.5,0) node[below]{$1$} to (.5,1);
\draw[webs] (1.5,2) to[out=90,in=270] (0,4.5) node[above]{$n$};
\draw[color=white, line width=5pt] (0,2) to[out=90,in=270] (1.5,3.5);
\draw[webs] (0,2) to[out=90,in=270] (1.5,3.5);
\draw[color=white, line width=5pt] (.5,1) to[out=90,in=270] (3,2.5);
\draw[webs] (.5,1) to[out=90,in=270] (3,2.5);
\draw[webs] (1.5,3.5) to (1.5,4.5) node[above]{$m$};
\node[draw, fill=white, minimum width=1cm] at (2,4) {$\alpha_v$};
\end{tikzpicture}
\ \textcolor{revisions}{\simeq (Qt^{-1})^{c - k - 1}}
\begin{tikzpicture}[baseline=(current bounding box.center), scale=0.75]
\draw[webs] (1.5,0) node[below]{$n$} to (1.5,2);
\draw[webs] (2.5,0) node[below]{$l - 1$} to (2.5,4.5) node[above]{$l$};
\node[draw, fill=white, minimum width=1cm] at (2,.5) {$\beta_w$};
\draw[webs] (0,0) node[below]{$m$} to (0,2);
\draw[webs] (.5,0) node[below]{$1$} to (.5,1);
\draw[webs] (1.5,2) to[out=90,in=270] (0,4.5) node[above]{$n$};
\draw[color=white, line width=5pt] (0,2) to[out=90,in=270] (1.5,3.5);
\draw[webs] (0,2) to[out=90,in=270] (1.5,3.5);
\draw[color=white, line width=5pt] (.5,1) to[out=90,in=270] (2.25,2.5);
\draw[webs] (.5,1) to[out=90,in=270] (2.25,2.5);
\draw[webs] (1.5,3.5) to (1.5,4.5) node[above]{$n$};
\node[draw, fill=white, minimum width=1cm] at (2,4) {$\alpha_v$};
\node[draw, fill=white, minimum width=1cm] at (2.75,3) {$\hat{K}^y_{l, \sigma \omega}$};
\end{tikzpicture}
\ \textcolor{revisions}{= \hat{\textbf{C}}^y_{\sigma \pi_{(10^{l - 1})}} (v, 1w)}
\end{gather*}
\end{center}

\textcolor{revisions}{Here we adopt the notation $\omega := \pi_{(10^{l - 1})} \in \mathfrak{S}^n$ to save space, and we have used Proposition \ref{prop: ky_absorb_renorm} to absorb the braid $\alpha_{(10^{l - 1})}$ into $K^y_{l, \sigma}$.}
\end{proof}

\textcolor{revisions}{By a completely symmetric argument, we also obtain the following analog of Lemma 3.8 in \cite{HM19}:}

\begin{lemma} \label{lem: hm3.8}
\textcolor{revisions}{Let $v \in \{0, 1\}^{m + l}$ and $w \in \{0, 1\}^{n + l}$ be sequences with $|v| + 1 = |w| = l \geq 1$. Then}

\[
\textcolor{revisions}{Tr^y_{k, \rho} \left( \hat{\textbf{C}}^y_{\sigma}(v1, w0) \right) \simeq \hat{\textbf{C}}^y_{\pi_{(10^{l - 1})} \sigma} (1v, w).}
\]
\end{lemma}

\textcolor{revisions}{We will also need an analog of Lemma 3.9 in \cite{HM19}.}

\begin{lemma} \label{lem: hm3.9}
\textcolor{revisions}{Let $v \in \{0, 1\}^{m + l}$ and $w \in \{0, 1\}^{n + l}$ be sequences with $|v| + 1 = |w| = l \geq 1$. Then}

\[
\textcolor{revisions}{HH^y \left( \hat{\textbf{C}}^y_{\sigma}(v0, w0) \right) \simeq Q^{-l} \left( HH^y \left( \hat{\textbf{C}}^y_{e_1 \sqcup\sigma}(1v, 1w) \right) \rightarrow T HH^y \left( \hat{\textbf{C}}^y_{\sigma}(0v, 0w) \right) \right).}
\]
\end{lemma}

\begin{proof}
\textcolor{revisions}{As in Lemma \ref{lem: hm3.7}, a straightforward computation shows that the underlying permutations in all complexes involved have the same number $c$ of disjoint cycles. Meanwhile, the complexes $\hat{\textbf{C}}^y_{\sigma}(v0,w0)$ and $\hat{\textbf{C}}^y_{\sigma}(0v,0w)$ each have an additional strand compared to $\hat{\textbf{C}}^y_{e_1 \sqcup \sigma}(1v,1w)$, so we should expect a differing factor of $Q^{-1}t$ in the respective normalizations.}

\textcolor{revisions}{Again using Proposition \ref{prop: ky_cross_slide}, we have}

\begin{center}
\begin{gather*}
\textcolor{revisions}{\hat{\textbf{C}}^y_{\sigma}(v0, w0) := (Qt^{-1})^{c - k}} \ 
\begin{tikzpicture}[baseline=(current bounding box.center), scale=0.75]
\draw[webs] (1.5,-.5) node[below]{$n$} to (1.5,1);
\draw[webs] (2.5,-.5) node[below]{$l$} to (2.5,1);
\node[draw, fill=white, minimum width=1cm] at (2,.5) {$\beta_w$};
\draw[webs] (3.5,-.5) node[below]{$1$} to (3.5,.5);
\draw[webs] (3.5,.5) to[out=90,in=270] (.5,2.5);
\draw[color=white,line width=5pt] (2.5,1) to[out=90,in=270] (3,2);
\draw[webs] (2.5,1) to[out=90,in=270] (3,2);
\draw[webs] (3,2) to[out=90,in=270] (2.5,3.5);
\draw[webs] (1.5,1) to[out=90,in=270] (0,2.5);
\draw[webs] (0,-.5) node[below]{$m$} to (0,1);
\draw[color=white, line width=5pt] (0,1) to[out=90,in=270] (1.5,3.5);
\draw[webs] (0,1) to[out=90,in=270] (1.5,3.5);
\draw[webs] (1.5,3.5) to (1.5,4.5) node[above]{$m$};
\draw[webs] (2.5,3.5) to (2.5,4.5) node[above]{$l$};
\node[draw, fill=white, minimum width=1cm] at (2,3.5) {$\alpha_v$};
\draw[webs] (0,2.5) to (0,4.5) node[above]{$n$};
\draw[webs] (.5,2.5) to (.5,4.5) node[above]{$1$};
\draw[webs] (3.5,4.5) node[above]{$1$} to (3.5,3.5);
\draw[color=white, line width=5pt] (3.5,3.5) to[out=270,in=90] (.5,1.5);
\draw[webs] (3.5,3.5) to[out=270,in=90] (.5,1.5);
\draw[webs] (.5,1.5) to (.5,-.5) node[below]{$1$};
\node[draw, fill=white, minimum width=1cm] at (3,2) {$\hat{K}^y_{l, \sigma}$};
\end{tikzpicture}
\quad \textcolor{revisions}{\simeq (Qt^{-1})^{c - k} \quad}
\begin{tikzpicture}[baseline=(current bounding box.center), scale=0.75]
\draw[webs] (1.5,-.5) node[below]{$n$} to (1.5,1);
\draw[webs] (2.5,-.5) node[below]{$l$} to (2.5,1);
\node[draw, fill=white, minimum width=1cm] at (2,.5) {$\beta_w$};
\draw[webs] (3.5,-.5) node[below]{$1$} to (3.5,1);
\draw[webs] (3.5,1) to[out=90,in=270] (.5,2.5);
\draw[color=white,line width=5pt] (2.5,1) to[out=90,in=270] (2.5,2.5);
\draw[webs] (2.5,1) to[out=90,in=270] (2.5,2.5);
\draw[webs] (2.5,2.5) to[out=90,in=270] (2.5,3.5);
\node[draw, fill=white, minimum width=1cm] at (2.5,2.5) {$\hat{K}^y_{l, \sigma}$};
\draw[webs] (1.5,1) to[out=90,in=270] (0,2.5);
\draw[webs] (0,-.5) node[below]{$m$} to (0,1);
\draw[color=white, line width=5pt] (0,1) to[out=90,in=270] (1.5,3.5);
\draw[webs] (0,1) to[out=90,in=270] (1.5,3.5);
\draw[webs] (1.5,3.5) to (1.5,4.5) node[above]{$m$};
\draw[webs] (2.5,3.5) to (2.5,4.5) node[above]{$l$};
\node[draw, fill=white, minimum width=1cm] at (2,3.5) {$\alpha_v$};
\draw[webs] (0,2.5) to (0,4.5) node[above]{$n$};
\draw[webs] (.5,2.5) to (.5,4.5) node[above]{$1$};
\draw[webs] (3.5,4.5) node[above]{$1$} to (3.5,2);
\draw[color=white, line width=5pt] (3.5,2) to[out=270,in=90] (.5,1);
\draw[webs] (3.5,2) to[out=270,in=90] (.5,1);
\draw[webs] (.5,1) to (.5,-.5) node[below]{$1$};
\end{tikzpicture}
\end{gather*}
\end{center}

\textcolor{revisions}{We now apply $Tr^y_{k, \rho}$. Again using locality and the renormalized Markov move, we obtain}

\begin{center}
\begin{gather*}
\textcolor{revisions}{Tr^y_{k, \rho}\left( \hat{\textbf{C}}^y_{\sigma}(v0, w0) \right) \simeq (Qt^{-1})^{c - k - 1} \ }
\begin{tikzpicture}[baseline=(current bounding box.center), scale=0.75]
\draw[webs] (1.5,-1) node[below]{$n$} to (1.5,1);
\draw[webs] (2.5,-1) node[below]{$l$} to (2.5,1);
\node[draw, fill=white, minimum width=1cm] at (2,0) {$\beta_w$};
\draw[webs] (2.75,1.25) to[out=90,in=270] (.5,2.5);
\draw[color=white,line width=5pt] (2.5,1) to[out=90,in=270] (2.5,2.5);
\draw[webs] (2.5,1) to[out=90,in=270] (2.5,2.5);
\draw[webs] (2.5,2.5) to[out=90,in=270] (2.5,3.5);
\node[draw, fill=white, minimum width=1cm] at (2.5,2.5) {$\hat{K}^y_{l, \sigma}$};
\draw[webs] (1.5,1) to[out=90,in=270] (0,2.5);
\draw[webs] (0,-1) node[below]{$m$} to (0,1);
\draw[color=white, line width=5pt] (0,1) to[out=90,in=270] (1.5,3.5);
\draw[webs] (0,1) to[out=90,in=270] (1.5,3.5);
\draw[webs] (1.5,3.5) to (1.5,4.5) node[above]{$m$};
\draw[webs] (2.5,3.5) to (2.5,4.5) node[above]{$l$};
\node[draw, fill=white, minimum width=1cm] at (2,3.5) {$\alpha_v$};
\draw[webs] (0,2.5) to (0,4.5) node[above]{$n$};
\draw[webs] (.5,2.5) to (.5,4.5) node[above]{$1$};
\draw[color=white, line width=5pt] (2.75,1.25) to[out=270,in=90] (.5,.5);
\draw[webs] (2.75,1.25) to[out=270,in=90] (.5,.5);
\draw[webs] (.5,.5) to (.5,-1) node[below]{$1$};
\end{tikzpicture}
\end{gather*}
\end{center}

\textcolor{revisions}{Here we see a Jucys-Murphy braid wrapping around\footnote{\textcolor{revisions}{From the \textit{left} rather than the right, but one can easily adapt Proposition \ref{prop: renorm.markov} to this setting.}} $\hat{K}^y_{l, \sigma}$. We can apply identity (6) of Proposition \ref{prop: renorm.markov} to replace this region with a two term convolution. The first term in this convolution is exactly}

\begin{center}
\begin{gather*}
\textcolor{revisions}{(Qt^{-1})^{c - k - 1} Q^{-l}} \ 
\begin{tikzpicture}[baseline=(current bounding box.center), scale=0.75]
\draw[webs] (1.5,-1) node[below]{$n$} to (1.5,1);
\draw[webs] (2.5,-1) node[below]{$l$} to (2.5,1);
\node[draw, fill=white, minimum width=1cm] at (2,0) {$\beta_w$};
\draw[webs] (2,2.5) to[out=90,in=270] (.5,3.5);
\draw[color=white,line width=5pt] (2.5,1) to[out=90,in=270] (2.5,2.5);
\draw[webs] (2.5,1) to[out=90,in=270] (2.5,2.5);
\draw[webs] (2.5,2.5) to[out=90,in=270] (2.5,3.5);
\node[draw, fill=white, minimum width=1.5cm] at (2.5,2) {$\hat{K}^y_{l + 1, e_1 \sqcup \sigma}$};
\draw[webs] (1.5,1) to[out=90,in=270] (0,2.5);
\draw[webs] (0,-1) node[below]{$m$} to (0,1);
\draw[color=white, line width=5pt] (0,1) to[out=90,in=270] (1.5,3.5);
\draw[webs] (0,1) to[out=90,in=270] (1.5,3.5);
\draw[webs] (1.5,3.5) to (1.5,4.5) node[above]{$m$};
\draw[webs] (2.5,3.5) to (2.5,4.5) node[above]{$l$};
\node[draw, fill=white, minimum width=1cm] at (2,3.5) {$\alpha_v$};
\draw[webs] (0,2.5) to (0,4.5) node[above]{$n$};
\draw[webs] (.5,3.5) to (.5,4.5) node[above]{$1$};
\draw[color=white, line width=5pt] (2,1.5) to[out=270,in=90] (.5,.5);
\draw[webs] (2,1.5) to[out=270,in=90] (.5,.5);
\draw[webs] (.5,.5) to (.5,-1) node[below]{$1$};
\end{tikzpicture}
\quad \textcolor{revisions}{= Q^{-l} \hat{\textbf{C}}^y_{e_1 \sqcup \sigma}(1v,1w).}
\end{gather*}
\end{center}

\textcolor{revisions}{Meanwhile, we can manipulate the second term using an isotopy and an inverse Markov move to obtain}

\begin{center}
\begin{gather*}
\textcolor{revisions}{(Qt^{-1})^{c - k - 1} Q^{-l}T} \ 
\begin{tikzpicture}[baseline=(current bounding box.center), scale=0.75]
\draw[webs] (1.5,-1) node[below]{$n$} to (1.5,1);
\draw[webs] (2.5,-1) node[below]{$l$} to (2.5,1);
\node[draw, fill=white, minimum width=1cm] at (2,0) {$\beta_w$};
\draw[webs] (1.75,2) to[out=90,in=270] (.5,3.5);
\draw[color=white,line width=5pt] (2.5,1) to[out=90,in=270] (2.5,2.5);
\draw[webs] (2.5,1) to[out=90,in=270] (2.5,2.5);
\draw[webs] (2.5,2.5) to[out=90,in=270] (2.5,3.5);
\draw[webs] (1.5,1) to[out=90,in=270] (0,2.5);
\draw[webs] (0,-1) node[below]{$m$} to (0,1);
\draw[color=white, line width=5pt] (0,1) to[out=90,in=270] (1.5,3.5);
\draw[webs] (0,1) to[out=90,in=270] (1.5,3.5);
\draw[webs] (1.5,3.5) to (1.5,4.5) node[above]{$m$};
\draw[webs] (2.5,3.5) to (2.5,4.5) node[above]{$l$};
\node[draw, fill=white, minimum width=1cm] at (2,3.5) {$\alpha_v$};
\draw[webs] (0,2.5) to (0,4.5) node[above]{$n$};
\draw[webs] (.5,3.5) to (.5,4.5) node[above]{$1$};
\draw[color=white, line width=5pt] (1.75,2) to[out=270,in=90] (.5,.5);
\draw[webs] (1.75,2) to[out=270,in=90] (.5,.5);
\draw[webs] (.5,.5) to (.5,-1) node[below]{$1$};
\node[draw, fill=white, minimum width=1cm] at (2.5,2) {$\hat{K}^y_{l, \sigma}$};
\end{tikzpicture}
\ \textcolor{revisions}{\simeq (Qt^{-1})^{c - k - 1}Q^{-l}T}
%%%%%%
\begin{tikzpicture}[baseline=(current bounding box.center), scale=0.75]
\draw[webs] (1.5,-1) node[below]{$n$} to (1.5,1);
\draw[webs] (2.5,-1) node[below]{$l$} to (2.5,1);
\node[draw, fill=white, minimum width=1cm] at (2,0) {$\beta_w$};
\draw[color=white,line width=5pt] (2.5,1) to[out=90,in=270] (2.5,2.5);
\draw[webs] (2.5,1) to[out=90,in=270] (2.5,2.5);
\draw[webs] (2.5,2.5) to[out=90,in=270] (2.5,3.5);
\node[draw, fill=white, minimum width=1cm] at (2.5,2) {$\hat{K}^y_{l, \sigma}$};
\draw[webs] (1.5,1) to[out=90,in=270] (0,2.5);
\draw[webs] (1.5,3.5) to (1.5,4.5) node[above]{$m$};
\draw[webs] (2.5,3.5) to (2.5,4.5) node[above]{$l$};
\draw[webs] (0,2.5) to (0,4.5) node[above]{$n$};
\draw[webs] (.5,3.5) to (.5,4.5) node[above]{$1$};
\draw[webs] (0,2) to[out=270,in=90] (.5,.5);
\draw[webs] (.5,.5) to (.5,-1) node[below]{$1$};
\draw[color=white, line width=5pt] (0,2) to[out=90,in=270] (.5,3.5);
\draw[webs] (0,2) to[out=90,in=270] (.5,3.5);
\draw[webs] (0,-1) node[below]{$m$} to (0,1);
\draw[color=white, line width=5pt] (0,1) to[out=90,in=270] (1.5,3.5);
\draw[webs] (0,1) to[out=90,in=270] (1.5,3.5);
\node[draw, fill=white, minimum width=1cm] at (2,3.5) {$\alpha_v$};
\end{tikzpicture}
\end{gather*}
\end{center}

\begin{center}
\begin{gather*}
\textcolor{revisions}{\simeq (Qt^{-1})^{c - k}Q^{-l}T}
%%%%%%
\begin{tikzpicture}[baseline=(current bounding box.center), scale=0.75]
\draw[webs] (1.5,-1) node[below]{$n$} to (1.5,1);
\draw[webs] (2.5,-1) node[below]{$l$} to (2.5,1);
\node[draw, fill=white, minimum width=1cm] at (2,0) {$\beta_w$};
\draw[color=white,line width=5pt] (2.5,1) to[out=90,in=270] (2.5,2.5);
\draw[webs] (2.5,1) to[out=90,in=270] (2.5,2.5);
\draw[webs] (2.5,2.5) to[out=90,in=270] (2.5,3.5);
\node[draw, fill=white, minimum width=1cm] at (2.5,2) {$\hat{K}^y_{l, \sigma}$};
\draw[webs] (1.5,1) to[out=90,in=270] (0,2.5);
\draw[webs] (1.5,3.5) to (1.5,4.5) node[above]{$m$};
\draw[webs] (2.5,3.5) to (2.5,4.5) node[above]{$l$};
\draw[webs] (0,2.5) to (0,4.5) node[above]{$n$};
\draw[webs] (.5,3.5) to (.5,4.5) node[above]{$1$};
\draw[webs] (-.5,3) to[out=270,in=90] (.5,.5);
\draw[webs] (-.5,3) to (-.5,4.5) node[above]{$1$};
\draw[webs] (.5,.5) to (.5,-1) node[below]{$1$};
\draw[color=white, line width=5pt] (-.5,1) to[out=90,in=270] (.5,3.5);
\draw[webs] (-.5,1) to[out=90,in=270] (.5,3.5);
\draw[webs] (-.5,1) to (-.5,-1) node[below]{$1$};
\draw[webs] (0,-1) node[below]{$m$} to (0,1);
\draw[color=white, line width=5pt] (0,1) to[out=90,in=270] (1.5,3.5);
\draw[webs] (0,1) to[out=90,in=270] (1.5,3.5);
\node[draw, fill=white, minimum width=1cm] at (2,3.5) {$\alpha_v$};
\end{tikzpicture}
\textcolor{revisions}{\quad = Q^{-l}T \hat{\textbf{C}}^y_{\sigma}(0v,0w)}.
\end{gather*}
\end{center}
\end{proof}

\textcolor{revisions}{The following is our analog of Theorem 3.5 in \cite{HM19}.}

\begin{theorem} \label{thm: mainrecursion}
For each pair $v \in \{0, 1\}^{l + m}$, $w \in \{0, 1\}^{l + n}$ and each $\sigma \in \mathfrak{S}^l$, \textcolor{revisions}{the $y$-ified Hochschild cohomology $HH^y(\hat{\textbf{C}}^y(v, w)_{\sigma})$ is homotopy equivalent to the free $\mathbb{Z}^3$-graded $R$-module of graded dimension $p_{\sigma}(v, w)$ with zero differential. In particular, we have $\mathcal{P}^y(T(m, n)) = p_e(0^m, 0^n)$ up to an overall normalization.}
\end{theorem}

\begin{proof}
\textcolor{revisions}{As in \cite{HM19}, we prove the theorem by induction on the set of pairs $(v, w)$. In the base case, $\hat{\textbf{C}}^y_e(0^m, \emptyset) = \hat{\textbf{C}}^y_e(\emptyset, 0^m) = R_m[\mathbb{Y}]$ up to an overall normalization, and we can repeatedly apply identity (1) from Proposition \ref{prop: renorm.markov} to establish Property (1) of Lemma \ref{lem: recursion_well_defined}.}

\textcolor{revisions}{Now, suppose both $v$ and $w$ are nonempty. Here we break into several cases.}

\textcolor{revisions}{\textbf{Case 0}: $(v, w) = (0^m, 0^n)$. In this case, we proceed via a graphical argument. Setting $d := gcd(m, n)$, we manipulate $\hat{\textbf{C}}^y_e(0^m, 0^n)$ as follows:}

\begin{center}
\begin{gather*}
\textcolor{revisions}{(Qt^{-1})^{d - m - n}}
%%%%%%%%%%%%%
\begin{tikzpicture}[baseline=(current bounding box.center), scale=0.5, tinynodes]
\draw[webs] (1.5,0) node[below]{$n$} to[out=90,in=270] (0,1.5) node[above]{$n$};
\draw[color=white, line width=5pt] (0,0) to[out=90,in=270] (1.5,1.5);
\draw[webs] (0,0) node[below]{$m$} to[out=90,in=270] (1.5,1.5) node[above]{$m$};
\end{tikzpicture}
%%%%%%%%%%%%%%
\ \textcolor{revisions}{\sim (Qt^{-1})^{d - m - n + 1}} \
%%%%%%%%%%%%%
\begin{tikzpicture}[baseline=(current bounding box.center),scale=0.75,tinynodes]
\draw[webs] (1.5,0) node[below]{$n-1$} to[out=90,in=270] (0,1.5) node[above]{$n-1$};
\draw[color=white, line width=5pt] (.75,1.5) to[out=270,in=90] (0,.75);
\draw[webs] (.75,1.5) node[above]{$1$} to[out=270,in=90] (0,.75);
\draw[webs] (0,.75) to[out=270,in=90] (.75,0) node[below]{$1$};
\draw[color=white, line width=5pt] (0,0) to[out=90,in=270] (1.5,1.5);
\draw[webs] (0,0) node[below]{$m-1$} to[out=90,in=270] (1.5,1.5) node[above]{$m-1$};
\end{tikzpicture}
%%%%%%%%%%%%%%
\ \textcolor{revisions}{\simeq (Qt^{-1})^{d - m - n + 1}} \
%%%%%%%%%%%%%
\begin{tikzpicture}[baseline=(current bounding box.center),scale=0.75,tinynodes]
\draw[webs] (1.5,0) node[below]{$n-1$} to[out=90,in=270] (0,1.5) node[above]{$n-1$};
\draw[webs] (.75,1.5) node[above]{$1$} to[out=270,in=90] (1.5,.75);
\draw[color=white, line width=5pt] (1.5,.75) to[out=270,in=90] (.75,0);
\draw[webs] (1.5,.75) to[out=270,in=90] (.75,0) node[below]{$1$};
\draw[color=white, line width=5pt] (0,0) to[out=90,in=270] (1.5,1.5);
\draw[webs] (0,0) node[below]{$m-1$} to[out=90,in=270] (1.5,1.5) node[above]{$m-1$};
\end{tikzpicture}
%%%%%%%%%%%%%%
\ \textcolor{revisions}{= \hat{\textbf{C}}^y(10^{m - 1}, 10^{n - 1}).}
\end{gather*}
\end{center}

\textcolor{revisions}{Here $\sim$ denotes two complexes which agree under $HH^y$. The first relation comes from an inverse Markov move (Identity (2) of Proposition \ref{prop: renorm.markov}) and the second comes from an isotopy. This establishes Property (7) of Lemma \ref{lem: recursion_well_defined}.}

\textcolor{revisions}{\textbf{Case 1:} $(v, w) = (v'1, w'1)$ for $v' < v, w' < w$. In this case, $\textbf{C}^y_{\sigma}(v, w)$ is of the form}

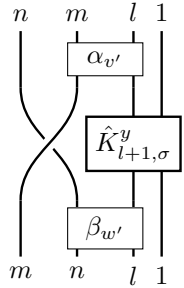
\begin{center}
\begin{tikzpicture}[baseline=(current bounding box.center), scale=0.75]
	\draw[webs] (2.5,0) node[below]{$1$} to (2.5,4) node[above]{$1$};
	\draw[webs] (1,0) node[below]{$n$} to (1,1);
	\draw[webs] (1,1) to[out=90,in=270] (0,3);
	\draw[webs] (0,0) node[below]{$m$} to (0,1);
	\draw[color=white, line width=5pt] (0,1) to[out=90,in=270] (1,3);
	\draw[webs] (0,1) to[out=90,in=270] (1,3);
	\draw[webs] (2,0) node[below]{$l$} to (2,1);
	\node[draw, fill=white, minimum width=1cm] at (1.5,.5) {$\beta_{w'}$};
	\draw[webs] (2,1) to node[pos=.5, draw, fill=white]{$\hat{K}^y_{l + 1, \sigma}$} (2,3);
	\draw[webs] (0,3) to (0,4) node[above]{$n$};
	\draw[webs] (1,3) to (1,4) node[above]{$m$};
	\draw[webs] (2,3) to (2,4) node[above]{$l$};
	\node[draw, fill=white, minimum width=1cm] at (1.5,3.5) {$\alpha_{v'}$};
\end{tikzpicture}
\end{center}

\textcolor{revisions}{By locality of $Tr^y_{\rho}$, we can close off the right-most $1$-labeled strand exiting $\hat{K}^y_{l + 1, \sigma}$, applying Identity (1), (4), or (5) from Proposition \ref{prop: renorm.markov}. This decreases the strand count by $1$ and decreases the cycle count by $1$ if and only if $\sigma(l + 1) = l + 1$. Comparing coefficients, this establishes Properties (2), (3), and (4) of Lemma \ref{lem: recursion_well_defined}, respectively.}

\textcolor{revisions}{\textbf{Case 2:} $(v, w) = (v'0, w'1)$ for $v' < v, w' < w$. Here Property (5) of Lemma \ref{lem: recursion_well_defined} follows from Lemma \ref{lem: hm3.7} and the inductive hypothesis.}

\textcolor{revisions}{\textbf{Case 3:} $(v, w) = (v'1, w'0)$ for $v' < v, w' < w$. Here Property (6) of Lemma \ref{lem: recursion_well_defined} follows from Lemma \ref{lem: hm3.8} and the inductive hypothesis.}

\textcolor{revisions}{\textbf{Case 4:} $(v, w) = (v'0, w'0)$ for $v' < v, w' < w$ with $|v'| = |w'| \geq 1$. In this case, we can apply Lemma \ref{lem: hm3.9} to rewrite $HH^y \left( \hat{\textbf{C}}^y_{\sigma}(v, w) \right)$ as a convolution of $HH^y$ of smaller sequences with an even relative shift in homological degree. By the inductive hypothesis, each of the terms in this convolution itself has $HH^y$ concentrated in even homological degrees. Since the component of the differential connecting these terms must have homological degree $1$, it must vanish, and the convolution splits as a direct sum. This establishes Property (8) of Lemma \ref{lem: recursion_well_defined}.}
\end{proof}

We now turn to a comparison of our recursion to that of \cite{HM19}. The authors of that work define a similar family of power series $p'(v, w) \in \mathbb{N}[[a, q, t]]$ which also satisfy $p'(0^m, 0^n) = \mathcal{P}(T(m, n))$ up to an overall normalization. In particular, their computations demonstrate that $HHH(T(m, n))$ is concentrated in \textit{even} homological degrees. By Theorem 1.17 in \cite{GH22}, it follows that $\mathcal{P}^y(T(m, n)) = (1 - T)^{-c} \mathcal{P}(T(m, n))$, where $c$ denotes the number of components of $T(m, n)$. Because of this, the transition from $\mathcal{P}$ to $\mathcal{P}^y$ can be easily accounted for in their recursion by tracking permutations $\sigma$ throughout as in our Theorem \ref{thm: mainrecursion} and multiplying by $(1 - T)^{-1}$ whenever the number of components decreases.

Applying these modifications to their recursion gives the following result:

\begin{theorem} \label{thm: hmrecursion}
For each pair $v \in \{0, 1\}^{l + m}$, $w \in \{0, 1\}^{l + n}$ and each $\sigma \in \mathfrak{S}^l$, let $p'_{y, \sigma}(v, w) \in \mathbb{N}[[a, q, t]]$ denote the unique power series satisfying

\begin{enumerate}
	\item \textcolor{revisions}{$p'_e(0^m, \emptyset) = \left( \cfrac{1 + A}{(1 - Q)(1 - T)} \right)^m$ and $p'_e(\emptyset, 0^n) = \left( \cfrac{1 + A}{(1 - Q)(1 - T)} \right)^n$.}
	
    \item $p'_e(0^m1, 0^n1) = \left( \cfrac{1 + A}{(1 - Q)(1 - T)} \right) p'_e(0^m, 0^n)$ for all $m, n \geq 0$;
    
    \item $p'_{\sigma}(v1, w1) = \left( \cfrac{T^l + A}{1 - T} \right) p'_{Tr_{l + 1}(\sigma)}(v, w)$ if $|v| = |w| = l \geq 1$ and $\sigma(n) = n$;
    
    \item $p'_{\sigma}(v1, w1) = \left( T^l + A \right) p'_{Tr_{l + 1}(\sigma)}(v, w)$ if $|v| = |w| = l \geq 1$ and $\sigma(n) \neq n$;
    
    \item $p'_{\sigma}(v0, w1) = p'_{\sigma \textcolor{revisions}{\pi_{(10^{l - 1})}}}(v, 1w)$ if $|v| = |w| + 1 = l \geq 1$;
    
    \item $p'_{\sigma}(v1, w0) = p'_{\textcolor{revisions}{\pi_{(10^{l - 1})}} \sigma}(1v, w)$ if $|v| + 1 = |w| = l \geq 1$;
    
    \item $p'_e(0^m, 0^n) = p'_e(10^{m - 1}, 10^{n - 1})$ \textcolor{revisions}{for all $m, n \geq 1$};
    
    \item $p'_{\sigma}(v0, w0) = T^{-l}p'_{\textcolor{revisions}{e_1 \sqcup \sigma}}(1v, 1w) + QT^{-l}p'_{\sigma}(0v, 0w)$ if $|v| = |w| = l \geq 1$.   
\end{enumerate}

Then $p'_{y, e}(0^m, 0^n) = \mathcal{P}^y(T(m, n))$ up to an overall normalization.
\end{theorem}

Observe that the recursion of Theorem $\ref{thm: hmrecursion}$ is obtained from that of Theorem $\ref{thm: mainrecursion}$ by interchanging $Q \leftrightarrow T$. It follows immediately that $p'_{y, \sigma}(v, w)(A, Q, T) = p_{\sigma}(A, T, Q)$ for all possible input triples $(v, w, \sigma)$. In particular, since both $p_e(0^m, 0^n)$ and $p'_{y, e}(0^m, 0^n)$ compute $\mathcal{P}^y(T(m, n))$ up to an overall normalization, we immediately obtain

\begin{theorem} \label{thm: uncolor_mirror_sym}
For each $m, n \geq 1$, we have
\[
\mathcal{P}^y(T(m, n))(A, Q, T) = \mathcal{P}^y(T(m, n))(A, T, Q)
\]
up to an overall normalization.
\end{theorem}

\subsection{Column Colored Torus Link Homology} \label{sec: col_tor_hom}

As discussed in the introduction, one can define a colored homology theory categorifying the $\bigwedge^k$-colored HOMFLYPT polynomial using the projector $(P_{1^k}^y)^{\vee}$ as in \cite{CK12}, \cite{Cau17}, and \cite{HM19}. Let $\mathcal{L} = L_1 \sqcup L_2 \sqcup \dots \sqcup L_n$ be a (framed)\footnote{The choice of framing only affects the overall normalization of the resulting invariant. To simplify matters, we will always assume our links are in blackboard framing.} link, and let $\underline{l} = (l_1, l_2, \dots, l_n)$ be positive integers assigned to each component. We call $l_1, \dots, l_n$ the \textit{colors} assigned to the components. Given a braid respresentative $\beta$ of $\mathcal{L}$, the strands of $\beta$ inherit colors from the corresponding link components. Then the following procedure produces a triply-graded $R$-module:

\begin{enumerate}
    \item Choose a collection of $n$ marked points, one on each component of $\mathcal{L}$, away from the crossings of $\beta$.
    
    \item Replace each $l_i$-colored strands with $l_i$ parallel copies of that strand.
    
    \item Consider the $y$-ified Rouquier complex of the resulting braid. Insert the projector $(P_{1^{l_i}}^y)^{\vee}$ at each marked point so that it spans the parallel copies of the previous step. We refer to this curved complex as $F^y_{\bigwedge^{l_1}, \dots, \bigwedge^{l_n}}(\beta)$.
    
    \item Apply $HHH^y$ to the resulting curved complex.
\end{enumerate}

We denote the resulting $R$-module by $HHH^y_{\bigwedge^{l_1}, \dots, \bigwedge^{l_n}}(\mathcal{L})$.

\begin{definition} \label{def: col_hom}
Let $\beta \in Br_n$ be given, and let $\mathcal{L} = \hat{\beta}$ denote the link obtained by closing $\beta$. Let $\underline{l} = (l_1, \dots, l_n)$ be a coloring of the \textcolor{revisions}{components} of $\mathcal{L}$ as above. Then the \textit{$(\bigwedge^{l_1}, \dots, \bigwedge^{l_n})$-colored y-ified homology of $\mathcal{L}$} is the triply-graded $R$-module

\[
HHH^y_{\bigwedge^{l_1}, \dots, \bigwedge^{l_n}}(\mathcal{L}) := HHH^y(F^y_{\bigwedge^{l_1}, \dots, \bigwedge^{l_n}}(\beta))
\]

\end{definition}

\begin{remark}
In principle, this module could depend on both a choice of braid representative and the choice of marked point on each component\textcolor{revisions}{, but in fact the result is independent of each of these choices. Invariance under two distinct choices of marked points on the same strand of $\beta$ follows immediately from Proposition \ref{prop: inf_cross_slide}, and Proposition \ref{prop: hy_tracelike} ensures invariance under sliding a choice of marked point from the bottom of one strand of $\beta$ to the top of another. Given two braids representing $\mathcal{L}$, they must be related through a finite sequence of Markov moves, and by invariance of choice of marked point, we may always insert projectors at a distant point from the regions where these moves are applied.}
\end{remark}

Computations of this colored homology theory appear significantly more complicated than those of the uncolored theory at first glance, as the bounded $y$-ified Rouquier complexes $F^y(\beta)$ are replaced with semi-infinite curved complexes. Fortunately, the $\mathbb{Z}[u_2, \dots, u_n]$ filtration on $(P_{1^n}^y)^{\vee}$ allows us to reduce to a finite complex upon applying $HH^y$ via a homological parity argument.

\begin{proposition} \label{prop: inf_reduction}
Fix integers $i, j, k \geq 0$, let $n = i + j + k$, and let $\sigma \in \mathfrak{S}^n$ be given. Let $X[\mathbb{Y}] \in \mathcal{Y}_{\sigma}^b(R_n-\text{Bim})$ be such that $HHH^y(\textcolor{revisions}{X[\mathbb{Y}]} \otimes (R_i[\mathbb{Y}] \sqcup (K^y_j)^{\vee} \sqcup R_k[\mathbb{Y}]))$ is supported in even homological degrees. Then up to an overall normalization, we have

\begin{align*}
    \textcolor{revisions}{HHH^y} \big(\textcolor{revisions}{X[\mathbb{Y}]} \otimes (R_i[\mathbb{Y}] \sqcup (P^y_{1^j})^{\vee} \sqcup R_k[\mathbb{Y}])\big)
    & \simeq \textcolor{revisions}{HHH^y}\big(\textcolor{revisions}{X[\mathbb{Y}]} \otimes (R_i[\mathbb{Y}] \sqcup (K^y_j)^{\vee} \sqcup R_k[\mathbb{Y}])\big) \otimes_R R[u_2, \dots, \textcolor{revisions}{u_j}] \\
    & \simeq \left( \prod_{\textcolor{revisions}{l = 2}}^{\textcolor{revisions}{j}} \cfrac{1}{1 - q^{-2l}t^2} \right) \textcolor{revisions}{HHH^y}\big(\textcolor{revisions}{X[\mathbb{Y}]} \otimes (R_i[\mathbb{Y}] \sqcup (K^y_j)^{\vee} \sqcup R_k[\mathbb{Y}])\big)
\end{align*}

\end{proposition}

\begin{proof}
\textcolor{revisions}{In Proposition 4.12 of \cite{EH19}, it is shown that given a complex $M$ with an action of a polynomial ring $R[x]$, there is a homotopy equivalence}

\begin{center}
\begin{equation} \label{eq: eh_parity}
\begin{tikzcd}
\textcolor{revisions}{M \simeq \mathrm{tw}_{\gamma} (R[x] \otimes \mathrm{Cone}( M} \arrow[r, "x", color=revisions] & \textcolor{revisions}{M))}
\end{tikzcd}
\end{equation}
\end{center}

\textcolor{revisions}{for some twist $\gamma$ which strictly increases the degree of $x$. By Proposition \ref{prop: dual_y_struc}, there is an action of the polynomial ring $R[u_2, \dots, u_j]$ on $(P^y_{1^j})^{\vee}$, with the action of each variable $u_l$ given by inclusion of the subcomplex consisting of terms with nonzero $u_l$-degree. By Gaussian elimination, the cone of this endomorphism is homotopy equivalent to the quotient complex of $(P^y_{1^j})^{\vee}$ consisting of terms with $u_l$-degree $0$.}

\textcolor{revisions}{We apply Equation \eqref{eq: eh_parity} to the action of each of these variables in turn. Repeatedly passing to mapping cones in this way results in the quotient complex of $(P^y_{1^j})^{\vee}$ consisting of terms with periodic degree $0$ in all variables; this is exactly $\text{Cone}(\overline{\alpha})$. In total, we obtain a homotopy equivalence}

\begin{center}
\begin{equation} \label{eq: our_parity}
\textcolor{revisions}{(P^y_{1^j})^{\vee} \simeq \mathrm{tw}_{\gamma}(R[u_2, \dots, u_j] \otimes \text{Cone}(\overline{\alpha}))}.
\end{equation}
\end{center}

\textcolor{revisions}{Applying $(-)^{\vee}$ to both sides of the homotopy equivalence of Proposition \ref{prop: fynite} gives $\text{Cone}(\overline{\alpha})^{\vee} \simeq  q^{1 - j}t^{-1} (K_j^y)^{\vee}$. Upon making this replacement, up to an overall normalization, \eqref{eq: our_parity} induces a homotopy equivalence}

\begin{center}
\begin{equation} \label{eq: temp_g1}
\textcolor{revisions}{X[\mathbb{Y}] \otimes (R_i[\mathbb{Y}] \sqcup (P^y_{1^j})^{\vee} \sqcup R_k[\mathbb{Y}]) \simeq \mathrm{tw}_{\gamma} \left( R[u_2, \dots, u_j] \otimes \left( X[\mathbb{Y}] \otimes (R_i[\mathbb{Y}] \sqcup (K^y_j)^{\vee} \sqcup R_k[\mathbb{Y}]) \right) \right)}.
\end{equation}
\end{center}

\textcolor{revisions}{Since the variables $u_l$ all have even homological degree, after applying $HHH^y$, Proposition \ref{prop: paritymiracle} ensures that this convolution becomes a direct sum.}
\end{proof}

\begin{remark}
Note that Proposition \ref{prop: inf_reduction} depends critically on the fact that multiplication by the periodic variables $u_2, \dots, u_n$ induces a closed endomorphism of $(P_{1^n}^y)^{\vee}$. On the other hand, it is easily verified that multiplication by $u_n^{-1}$ does \textit{not} induce a closed endomorphism of $P^y_{1^n}$. This is yet another reason why it is sensible to work with the unital idempotent instead of the counital idempotent.
\end{remark}

\begin{proposition} \label{prop: fin_self_dual}
The finite projector $K_n^y$ is self-dual up to an overall normalization; that is, $(K_n^y)^{\vee} \cong q^{-2(n - 1)}t^{n - 1}K_n^y$
\end{proposition}

\begin{proof}
Reverse all morphisms and grading shifts in the definition of $K_n^y$ and recall that $B_{w_0}$ is self-dual.
\end{proof}

\begin{example} \label{ex: unknot_comp}
We compute \textcolor{revisions}{our} invariant in the simplest possible setting. Let $\mathcal{L}$ be the unknot, and fix a choice of color $k \geq 1$. An obvious choice of braid representative is the trivial braid $\beta = e \in Br_1$. Any choice of marked point on $\beta$ gives $F^y_{\bigwedge^k}(\beta) = (P_{1^k}^y)^{\vee}$. By Proposition \ref{prop: inf_reduction} below, we have 

\[
HH^y(F^y_{\bigwedge^k}(\beta)) \simeq HH^y((K^y_k)^{\vee}) \otimes_R R[u_2, \dots, u_k]
\]

Applying \textcolor{revisions}{$H^{\bullet}$} to both sides of this homotopy equivalence gives an isomorphism 

\[
HHH^y(\mathcal{L}) \cong HHH^y((K^y_k)^{\vee}) \otimes_R R[u_2, \dots, u_k]
\]

By Proposition \ref{prop: fin_self_dual}, we have $(K^y_k)^{\vee} \cong K^y_k$ up to an overall grading shift. Note that, again up to an overall grading shift, $K^y_k$ is isomorphic to the complex $\hat{\bf{C}}^y(1^k, 1^k)_e$ of Theorem \ref{thm: mainrecursion}; the graded dimension of $HHH^y$ applied to this complex is by definition the polynomial $p_e(1^k, 1^k)$, which is computable using the recursion provided there. Up to an overall grading shift, we obtain

\begin{align}
\text{dim}(HHH^y_{\bigwedge^k}(\mathcal{L})) = \prod_{i = 1}^k \cfrac{Q^{i - 1} + A}{(1 - Q)(1 - Q^{i - 1}T)} \label{result: unlink}
\end{align}

\end{example}

As in Example \ref{ex: unknot_comp}, Propositions \ref{prop: inf_reduction} and \ref{prop: fin_self_dual} often allow us to reduce the computation of $HHH^y_{\bigwedge^k}$ to the evaluation of $HHH^y$ on a finite curved complex. In fact, we have already carried out this computation for a broad class of examples in Theorem \ref{thm: mainrecursion}. From these computations, we immediately obtain this colored invariant for positive torus links with a single $k$-colored component as below:

\begin{theorem} \label{thm: pos_tmn_column_inv}
For each $k, m, n \geq 1$ with $m, n$ coprime, up to an overall normalization, the positive torus knot $T(m, n)$ satisfies

\[
\text{dim}(HHH^y_{\bigwedge^k}(T(m, n))) = \left( \prod_{i = 2}^k \cfrac{1}{1 - Q^{1 - i}T} \right) p_{e}(1^k0^{k(m - 1)}, 1^k0^{k(n - 1)})
\]

More generally, for any $m, n \geq 1$, let $d = \text{gcd}(m, n)$. Then up to an overall normalization, the positive torus link $T(m, n)$ satisfies

\[
\text{dim}(HHH^y_{\bigwedge^k, \bigwedge^1, \dots, \bigwedge^1}(T(m, n))) = \left( \prod_{i = 2}^k \cfrac{1}{1 - Q^{1 - i}T} \right) p_{e}(1^k0^{md^{-1}(d - 1) + k(md^{-1} - 1)}, 1^k0^{nd^{-1}(d - 1) + k(nd^{-1} - 1)})
\]
\end{theorem}

\subsection{Module Structure} \label{sec: mod_sheet}

Before turning our attention to row colored homology, we pause to consider the module structure on $\bigwedge^k$-colored $y$-ified homology. Let $\beta \in Br_n$ be a braid presentation of $\hat{\beta} = \mathcal{L}$. To simplify our discussion, we assume throughout that $\mathcal{L}$ is a knot; the theory carries over almost identically in the case that $\mathcal{L}$ is a link (though some care must be taken when treating multiple sets of periodic variables).

Because $(P_{1^n}^y)^{\vee}$ is a (unital) idempotent, we have a natural isomorphism

\[
\psi_k \colon (P_{1^k}^y)^{\vee} \otimes (P_{1^k}^y)^{\vee} \cong (P_{1^k}^y)^{\vee}
\]

in $\mathcal{CY}_e(R_k-\text{Bim})$; we refer the interested reader to \cite{Hog17} for details. For a given choice of marked point $p \in \beta$, let $\tilde{F}^y_{\bigwedge^k}(\beta_p)$ denote the curved complex obtained from cabling and inserting $(P_{1^k}^y)^{\vee} \otimes (P_{1^k}^y)^{\vee}$, rather than $(P_{1^k}^y)^{\vee}$, at $p$. Then $\psi$ furnishes a natural isomorphism $\tilde{F}^y_{\bigwedge^k}(\beta_p) \cong F^y_{\bigwedge^k}(\beta)$; we denote this isomorphism also by $\psi_k$ for notational convenience.

Recall from Proposition \ref{prop: hyrep} that $HHH^y$ is a representable functor. In particular, we have an isomorphism

\[
HHH^y_{\bigwedge^k}(\mathcal{L}) \cong \text{Hom}_{\mathcal{CY}_{\beta}(R_n-\text{Bim})} (\mathbbm{1}^y_{\beta}, F^y_{\bigwedge^k}(\beta))
\]

The tensor product structure on $\mathcal{CY}(R_n-\text{Bim})$ applied to morphisms induces a map

\begin{align*}
\otimes \colon \text{Hom}_{\mathcal{CY}_e(R_k-\text{Bim})} (\mathbbm{1}^y, (P_{1^k}^y)^{\vee}) \otimes_R \text{Hom}_{\mathcal{CY}_{\beta}(R_n-\text{Bim})} (\mathbbm{1}_{\beta}^y, F^y_{\bigwedge^k}(\beta)) \to \text{Hom}_{\mathcal{CY}_{\beta}(R_n-\text{Bim})}(\mathbbm{1}_{\beta}^y, \tilde{F}^y_{\bigwedge^k}(\beta_p))
\end{align*}

Postcomposing with $\psi_k$ induces a natural isomorphism

\[
\text{Hom}_{\mathcal{CY}_{\beta}(R_n-\text{Bim})}(\mathbbm{1}_{\beta}^y, \tilde{F}^y_{\bigwedge^k}(\beta_p)) \cong \text{Hom}_{\mathcal{CY}_{\beta}(R_n-\text{Bim})}(\mathbbm{1}_{\beta}^y, F^y_{\bigwedge^k}(\beta))
\]

In total, we obtain a map

\[
\psi_k \circ \otimes \colon HHH^y_{\bigwedge^k}(\bigcirc) \otimes_R HHH^y_{\bigwedge^k}(\mathcal{L}) \to HHH^y_{\bigwedge^k}(\mathcal{L})
\]

We rephrase this conclusion below.

\begin{theorem}
Let $\mathcal{L}$ be a knot, and let $\bigcirc$ be the unknot. Then a given choice of marked point $p \in \mathcal{L}$ induces a left $HHH^y_{\bigwedge^k}(\bigcirc)$-module structure on $HHH^y_{\bigwedge^k}(\mathcal{L})$.
\end{theorem}

In keeping with the discussion in the introduction (c.f. \S \ref{sec: mod_sheet_spec}), we refer to $HHH^y_{\bigwedge^k}(\bigcirc)$ as the \textit{derived sheet algebra} of $HHH^y_{\bigwedge^k}$. One can similarly define a right-module structure over the derived sheet algebra via precomposition with $\psi_k$; naturality of $\psi_k$ ensures that the left and right module structures defined in this way agree.

At least two questions naturally arise:

\begin{enumerate}
    \item Is the module structure above independent of a choice of marked point?
    
    \item If so, can the module structure above be extended to a $dg$-module structure on $HH^y(F^y_{\bigwedge^k}(\beta))$, independent of choice of marked point up to quasi-isomorphism?
\end{enumerate}

As this work is already quite long, we delay considerations of these two questions to future work. We do, however, point out that Proposition \ref{prop: inf_reduction} furnishes a sort of ``freeness" result over the factor $R[u_2, \dots, u_k]$ of the derived sheet algebra for positive torus knots. On the other hand, a simple dimension count in Equation \eqref{result: unlink} reveals that, in contrast with the results of \cite{GH22} (and Proposition \ref{prop: row_color_free} below), the derived sheet algebra does \textit{not} itself carry a free action of the polynomial ring $R[y_1, \dots, y_k]$.

\subsection{Row Colored Torus Link Homology} \label{subsec: Row_tor_link_hom}

By the results of \cite{Hog18}, there is a family of complexes $P_n \in \text{Ch}^-(SBim_n)$ that play an analogous role in categorifying the $\text{Sym}^n$-colored HOMFLYPT polynomial as do our $P_{1^n}$ in the $\bigwedge^n$-colored invariant. These complexes are convolutions of the form $P_n \simeq \text{tw}(K_n' \otimes \mathbb{Z}[v_2, \dots, v_n])$ for a bounded complex $K_n'$ and $v_i$ formal variables of degree $\text{deg}(v_i) = q^{2i}t^{2 - 2i}$. Each $P_n$ is also a unital idempotent in its native category and can therefore be used to define a link invariant via the method outlined above.

We wish to $y$-ify this $\text{Sym}$-colored invariant. To do so requires a $y$-ification of the corresponding family of projectors.

\begin{theorem} \label{thm: yify_row}
There is a $y$-ification $P_n^y := (P_n[\mathbb{Y}], e, \delta_{P_n})$ of the infinite row projector $P_n$.
\end{theorem}

\begin{proof}
Throughout, let $FT_n^k$ denote the $k^{th}$ tensor power of the Rouquier complex associated to the full twist braid $FT_n := X_n^n$. In Theorem 2.30 of \cite{Hog18}, $P_n$ is constructed as a homotopy colimit of the directed system $\{FT_n^k, f_k\}_{k = 0}^\infty$ for a system of maps $f_k \colon FT_n^k \to FT_n^{k + 1}$. The maps $f_k$ are of the form $\text{id}_{FT_n^k} \otimes f_0$ for a degree $(q^{-1}t)^{n(n - 1)}$ chain map $f_0 \colon R_n \rightarrow FT_n$ given by inclusion of the unique copy of $R_n$ in $FT_n$ in highest homological degree. To construct $P_n^y$, it suffices to show that $f_0$ lifts to a map $\tilde{f}_0 \colon R_n[\mathbb{Y}] \rightarrow FT_n[\mathbb{Y}]$ of $y$-ified complexes; the homotopy colimit of the directed system $\{FT_n^k[\mathbb{Y}], \tilde{f}_k\}_{k = 0}^{\infty}$ with $\tilde{f}_k := \text{id}_{FT_n^k[\mathbb{Y}]} \otimes \tilde{f}_0$ will then be a $y$-ification of $P_n$.

We proceed by a similar obstruction-theoretic argument as in \cite{GH22} (c.f. our Proposition \ref{prop: GHobs}). Recall that the connection on $FT_n[\mathbb{Y}]$ is given by

\[
\delta_{FT_n} = d_{FT_n} + \sum_{i = 1}^n h_i y_i
\]

where $h_i \in \text{End}^{-1}_{\text{Ch}(SBim_n)}$ are dot-sliding homotopies satisfying $[d, h_i] = x_i - x_i'$. A direct computation gives

\begin{align*}
    [\delta, f_0] & = \left( d_{FT_n} + \sum_{i = 1}^n h_i y_i \right) f_0 \\
    & = \sum_{i = 1}^n h_if_0 y_i
\end{align*}

We claim that there exists a morphism $(f_0)_{y_i} \in \text{Hom}^{n(n - 1) - 2}(R_n, FT_n)$ satisfying $[d, (f_0)_{y_i}] = -h_i f_0$. Indeed, it was shown in \cite{EH19} that the complex $\text{Hom}(R_n, FT_n)$ has homology concentrated in even degrees. Since $-h_if_0$ has degree $n(n - 1) - 1$, which is odd for all $n \geq 1$, it suffices to show that $[d, -h_if_0] = 0$. We compute:

\[
[d, -h_if_0] = -[d, h_i]f_0 - h_i[d, f_0] = (x_i' - x_i)f_0 = 0
\]

Setting $f_{0, 1} := f_0 + \sum_{i = 1}^n (f_0)_{y_i} y_i$, we see that $[\delta, f_{0, 1}]$ is quadratic in the alphabet $\mathbb{Y}$.

We iterate this process. More precisely, suppose we have constructed a map

\[
f_{0, r} := \sum_{\substack {\textbf{v} \in \mathbb{Z}_{\geq 0}^n \\ |\textbf{v}| \leq r}} (f_0)_{y^\textbf{v}} y^{\textbf{v}}
\]

such that $[\delta, f_{0, r}]_{y^\textbf{v}} = 0$ for all $|\textbf{v}| \leq r$. By direct computation, we obtain

\[
[\delta, f_{0, r}] = \sum_{i = 1}^n h_i \sum_{\substack{\textbf{v} \in \mathbb{Z}^n_{\geq 0} \\ |\textbf{v}| = r}} (f_0)_{y^\textbf{v}} y_i y^\textbf{v} = \sum_{\substack{\textbf{v} \in \mathbb{Z}^n_{\geq 0} \\ |\textbf{v}| = r + 1}} \sum_{\textbf{v}_i \neq 0} h_i (f_0)_{y^{\textbf{v} - \textbf{e}_i}} y^\textbf{v}
\]

For each $i$ and $\textbf{v}$ in the latter summation above, $h_i (f_0)_{y^{\textbf{v} - \textbf{e}_i}}$ has homological degree $n(n - 1) - 2r - 1$. This is odd for all values of $i, r$, and $n$. Again by direct computation, for each $\textbf{v} \in \mathbb{Z}^n_{\geq 0}$ satisfying $|\textbf{v}| = r + 1$, we obtain

\begin{align*}
    \left[d, \sum_{\textbf{v}_i \neq 0} h_i (f_0)_{y^{\textbf{v} - \textbf{e}_i}} \right] & = \sum_{\textbf{v}_i \neq 0} \left( [d, h_i] (f_0)_{y^{\textbf{v} - \textbf{e}_i}} + h_i \left[d, (f_0)_{y^{\textbf{v} - \textbf{e}_i}} \right] \right) \\
    & = \sum_{\textbf{v}_i \neq 0} \left((x_i - x_i') (f_0)_{y^{\textbf{v} - \textbf{e}_i}} + h_i \left( \sum_{(\textbf{v} - \textbf{e}_i)_j \neq 0} h_j (f_0)_{y^{\textbf{v} - \textbf{e}_i - \textbf{e}_j}} \right) \right) \\
    & = \sum_{\textbf{v}_i \neq 0} (f_0)_{y^{\textbf{v} - \textbf{e}_i}} (x_i - x_i')  + \sum_{\textbf{v}_i \neq 0} \sum_{(\textbf{v} - \textbf{e}_i)_j \neq 0} \left( h_i h_j (f_0)_{y^{\textbf{v} - \textbf{e}_i - \textbf{e}_j}} \right) \\
    & = \sum_{\substack{i \leq j \\ \textbf{v}_i \neq 0 \\ (\textbf{v} - \textbf{e}_i)_j \neq 0}} (h_ih_j + h_jh_i) (f_0)_{y^{\textbf{v} - \textbf{e}_i - \textbf{e}_j}} \\
    & = 0
\end{align*}

Here the first equality follows from the graded Leibniz rule, the second equality follows by assumption on $[\delta, f_{0, r}]$, and the final equality follows from the identity $h_ih_j + h_jh_i = 0$ for dot-sliding homotopies. Since $\text{Hom}(R_n, FT_n)$ has homology concentrated in even degrees, for each $\textbf{v} \in \mathbb{Z}^n_{\geq 0}$ with $|\textbf{v}| = r + 1$, there exists some $(f_0)_{y^\textbf{v}} \in \text{Hom}(R_n, FT_n)$ satisfying 

\[
[d, (f_0)_{y^\textbf{v}}] = -\sum_{\textbf{v}_i \neq 0} h_i (f_0)_{y^{\textbf{v} - \textbf{e}_i}}
\]

Set $f_{0, r + 1} := f_{0, r} + \sum_{|\textbf{v}| = r + 1} (f_0)_{y^\textbf{v}} y^\textbf{v}$. Then $[\delta, f_{0, r+ 1}]_{y^\textbf{v}} = 0$ for all $|\textbf{v}| \leq r + 1$ by construction. Since $FT_n$ is bounded, this process eventually stabilizes; the resulting sum is the desired chain map $\tilde{f}_0$. (In fact we can always take $\tilde{f}_0 = f_{0, n(n - 1)/2}$.)

\end{proof}

Given a braid $\beta \in Br_n$ and a choice of marked point and color on each component of $\hat{\beta}$, the same cabling and insertion procedure described in \S \ref{sec: col_tor_hom} can be used to produce a curved complex $F^y_{(\text{Sym}^{l_1}, \dots, \text{Sym}^{l_n})}(\beta)$. As in the column-colored case, one hopes that the graded dimension of $HHH^y$ of the result is independent of a choice of braid representative for $\hat{\beta}$ and marked points on each strand. Note, however, that Theorem \ref{thm: yify_row} makes no claim that $P_n^y$ is a unital idempotent in the relevant $y$-ified category, nor is it claimed that $P_n^y$ slides past crossings. As a consequence, we cannot show that the resulting triply-graded $R$-module is an invariant via the usual arguments of \cite{CK12}. That being said, when working over a coefficient field $R$, there is sufficient structure present in this construction to uniquely determine the graded dimension of the conjectural invariant in the context of Theorem \ref{thm: pos_tmn_column_inv}. We carry this out below, beginning with a concrete definition of this conjectural invariant.

\begin{definition} \label{def: sym_y}
Let $\beta \in Br_n$ be given, let $\underline{l} = (l_1, \dots, l_n)$ be a coloring of the strands of $\hat{\beta}$, and choose one marked point on each strand of $\mathcal{L} := \hat{\beta}$ away from the crossings of $\beta$. Then the $(\text{Sym}^{l_1}, \dots, \text{Sym}^{l_n})$-colored $y$-ified homology of $\mathcal{L}$ is the triply-graded $R$-module

\[
HHH^y_{\text{Sym}^{l_1}, \dots, \text{Sym}^{l_n}}(\mathcal{L}) := HHH^y(F^y_{\text{Sym}^{l_1}, \dots, \text{Sym}^{l_n}}(\beta))
\]
\end{definition}

\begin{conjecture} \label{conj: sym_y}
The graded dimension of $HHH^y_{\text{Sym}^{l_1}, \dots, \text{Sym}^{l_n}}(\mathcal{L})$ is independent of all choices made in Definition \ref{def: sym_y} up to an overall normalization.
\end{conjecture}

\begin{example}
As in the column-colored case, we begin by computing this conjectural invariant for the trivial representation $\beta = e \in Br_1$ of the unknot $\mathcal{L} = \bigcirc$. We assume throughout that $R$ is a field. Any choice of marked point on $\beta$ gives $F^y_{\text{Sym}^k}(\beta) = P_k^y$. By Proposition \ref{prop: row_color_free} below, we obtain

\[
HHH^y_{\text{Sym}^k}(\mathcal{L}) \cong HHH(P_k) \otimes_R R[y_1, \dots, y_k]
\]

Up to an overall normalization, the $R$-module $HHH(P_k)$ is the un $y$-ified $\text{Sym}^k$-colored invariant of the unknot computed in \cite{HM19}. We therefore obtain

\[
\text{dim}(HHH^y_{\text{Sym}^k}(\mathcal{L})) = \prod_{i = 1}^k \cfrac{T^{i - 1} + A}{(1 - T)(1 - QT^{i - 1})}
\]

Note the similarity between this result and that of Example \ref{ex: unknot_comp}.

\end{example}

In \cite{GH22}, Gorsky-Hogancamp show that when working over a coefficient field, homological parity phenomena often ensure that $y$-ified homology is free in $y$-variables. Though those authors restrict their attention to finite complexes, their results extend to our setting, as outlined below.

\begin{proposition} \label{prop: row_color_free}
Let $R$ be a field. Fix integers $i, j, k \geq 0$, and let $n = i + j + k$. Let $X \in \text{Ch}^b(R_n-\text{Bim})$ be such that $HHH(X \otimes (R_i \sqcup P_j \sqcup R_k))$ is supported in even homological degrees, and let $X[\mathbb{Y}] \in \mathcal{Y}_{\sigma}(R_n-\text{Bim})$ be a $y$-ification of $X$ for some $\sigma \in \mathfrak{S}^n$. Let $\pi_0(\sigma)$ denote the orbits of $\sigma$ on the set $\{1, \dots, n\}$, and set $A := \mathbb{Q}[y_c]_{c \in \pi_0(\sigma)}$. Then there is an isomorphism of triply-graded vector spaces

\[
HHH^y(X[\mathbb{Y}] \otimes (R_i[\mathbb{Y}] \sqcup P_j^y \sqcup R_k[\mathbb{Y}])) \cong HHH(X \otimes (R_i \sqcup P_j \sqcup R_k)) \otimes_R A
\]
\end{proposition}

\begin{proof}
This follows almost immediately from Theorem 3.16 and Lemma 4.13 of \cite{GH22}; the only novelty is in adapting the proofs of those statements to the semi-infinite complex $C_0 := HH(X \otimes (R_i \sqcup P_j \sqcup R_k))$. This adaptation requires only that $C_0$ and $H(C_0) \otimes_R A$ be finite-dimensional in each homogeneous component with respect to the tri-grading $(a, q, t)$. The former claim that $C_0$ is finite-dimensional in each degree follows from the convolution description of $P_j$ and the fact that each periodic variable $u'_l$ has positive quantum degree. The latter claim follows from Lemma 4.3 of \cite{HM19} and the fact that no monomoial of the form $(u')^{\textbf{v}} y^\textbf{w}$ has degree $0$.
\end{proof}

\begin{remark} \label{rem: field_coeff}
Note that Proposition \ref{prop: row_color_free} is stated over a coefficient field. We require invertibility of certain coefficients whenever using parity phenomena to argue that the homology of a $y$-ified complex is free over $R[\mathbb{Y}]$; see \cite{GH22} for details.

On the other hand, the corresponding parity arguments that simplify the computation of the homology of periodic complexes as in Proposition \ref{prop: inf_reduction} does \textit{not} require working over a field. One could imagine a computation of $\text{Sym}^k$-colored homology in which these reductions are performed in the other order: first parity is established directly over $\mathbb{Z}$ for all computations involving $(K'_n)^y$, then periodicity is accounted for. Such a route would provide a proof of the results below with integral coefficients by avoiding all invertibility concerns.

To take this route to our results would require a more explicit construction of $P_n^y$ satisfying properties similar to our $(P_{1^n}^y)^{\vee}$. These properties would necessarily include:

\begin{enumerate}
    \item $P^y_n$ is a strict $y$-ification (allowing an application of Proposition \ref{prop: hom_from_tr} to compute $y$-ified homology directly).
    
    \item There is a strict $y$-ification $(K_n')^y$ of the finite row projector satisfying $(K_n')^y \simeq \text{Cone}(u_2) \dots \text{Cone}(u_n)$ (to allow reduction of periodic behavior as in our Proposition \ref{prop: inf_reduction}).
    
    \item $(K_n')^y$ (and its twisted variants) are well-behaved with respect to Markov moves and twisting as in Proposition \ref{prop: renorm.markov} (to allow a recursive computation using convenient complexes).
\end{enumerate}

The failure of Property (3) for the un $y$-ified finite projector $K_n$, as established in Proposition \ref{prop: ylessktrace}, requires us to establish these properties to have any hope of computing $\bigwedge^k$-colored homology, $y$-ified or otherwise. Indeed, without this obstruction, the technical portion of this work would not be much longer than our proof of Theorem \ref{thm: yify_row}! In the mutual interests of length and coherence, we content ourselves in this work with stating our comparison theorem over a coefficient field. We plan to address the properties of $P_n^y$ conjectured in this remark in future work.
\end{remark}

\textcolor{revisions}{In Theorem 4.6 of \cite{HM19}, those authors give a formula for the un $y$-ified, Sym-colored homology of all positive torus knots which is manifestly concentrated in even homological degrees. As a consequence, we can directly apply Proposition \ref{prop: row_color_free} to compute the $y$-ified, Sym-colored homology of these knots  \textit{independently of Conjecture \ref{conj: sym_y}}.}

\textcolor{revisions}{In analogy with our Proposition \ref{prop: inf_reduction}, the relation between the uncolored computations of \cite{HM19} and their colored computations is given by multiplication by the dimension of the polynomial ring $\mathbb{Z}[v_2, \dots, v_n]$. As a consequence, we can adapt their computations to the $y$-ified setting by tracking permutations exactly as in Theorem \ref{thm: hmrecursion}. We state the result below.}

\begin{theorem} \label{thm: row_y_well_def}

For each $k, m, n \geq 1$ with $m, n$ coprime, $\text{dim}(HHH^y_{\text{Sym}^k}(T(m, n)))$ is independent of all choices made in Definition \ref{def: sym_y}. Up to an overall normalization, we have

\[
\text{dim}(HHH^y_{\text{Sym}^k}(T(m, n))) = \left( \prod_{i = 2}^k \cfrac{1}{1 - QT^{1 - i}} \right) p_{y,e}'(1^k0^{k(m - 1)}, 1^k0^{k(n - 1)})
\]

More generally, for any $k, m, n \geq 1$, $\text{dim}(HHH^y_{\text{Sym}^k, \text{Sym}^1, \dots, \text{Sym}^1}(T(m, n)))$ is independent of all choices made in Definition \ref{def: sym_y}. Let $d = \text{gcd}(m, n)$. Then up to an overall normalization, we have

\[
\text{dim}(HHH^y_{\text{Sym}^k, \text{Sym}^1, \dots, \text{Sym}^1}(T(m, n))) = \left( \prod_{i = 2}^k \cfrac{1}{1 - Q^{1 - i}T} \right) p'_{y, e}(1^k0^{md^{-1}(d - 1) + k(md^{-1} - 1)}, 1^k0^{nd^{-1}(d - 1) + k(nd^{-1} - 1)})
\]

\end{theorem}

As in the proof of Theorem \ref{thm: uncolor_mirror_sym}, we immediately obtain

\begin{theorem} \label{thm: row_col_mirror_sym}
For each $k, m, n \geq 1$, we have 
\[
\text{dim}(HHH^y_{\text{Sym}^k, \text{Sym}^1, \dots, \text{Sym}^1}(T(m, n)))(A, Q, T) = \text{dim}(HHH^y_{\bigwedge^k, \bigwedge^1, \dots, \bigwedge^1}(T(m, n)))(A, T, Q)
\]
up to an overall normalization.
\end{theorem}

% To our knowledge, this is the first explicit confirmation of the mirror symmetry relationship conjectured by Gukov-Stošić, Gorsky-Gukov-Stošić, and Gorsky-Negut-Rasmussen for colored homology of nontrivial links.

\appendix
\section{Singular Soergel Bimodules and Rickard Complexes} \label{app: ssbim}

\textcolor{revisions}{In this Appendix we provide a brief introduction to the category $SSBim$ of singular Soergel bimodules, which can be viewed as an enlargement of the category $SBim$ of Soergel bimodules considered in the main body of the paper. Working in $SSBim$ allows us to to perform computations that would be vastly more challenging in $SBim$; we leverage this fact in Appendix \ref{app: dot_slide_nat} to prove Propositions \ref{prop: mu_nu} and \ref{prop: forkslyde}. We provide only the minimal discussion that allows us to prove these results and direct the reader interested in a more thorough introduction to this category to Section 3 of \cite{HRW21}.}

Given a tuple $I = (m_1, \dots, m_r) \in \mathbb{Z}_{\geq 1}^r$ with $m_1 + \dots + m_r = n$, the corresponding \textit{parabolic subgroup} of $\mathfrak{S}^n$ is defined as follows:

$$
\mathfrak{S}_I^n := \mathfrak{S}^{m_1} \times \dots \times \mathfrak{S}^{m_r}
$$

Since parabolic subgroups are in bijection with tuples $I \in \mathbb{Z}_{\geq 1}^{\bullet}$, we will implicitly identify the two in our notation, using $I$ to denote $\mathfrak{S}_I^n$. \textcolor{revisions}{For each such parabolic subgroup $I$, we denote by} $R_n^I$ the subring of $R_n$ \textcolor{revisions}{consisting} of polynomials invariant under the action of \textcolor{revisions}{$I$}. 

Given parabolic subgroups $I \subset J$, we have a reverse inclusion of invariant subrings $R_n^J \subset R_n^I$. Notice that $R_n^I$ is a left $R_n^J$-module by restriction of the usual multiplication action under this inclusion. Let $\ell(I)$, $\ell(J)$ denote the lengths of the longest words of the parabolic subgroups above. In this situation, we have distinguished \textit{merge} and \textit{split} bimodules:

$$
_JM_I := q^{\ell(I) - \ell(J)} R_n^J \otimes_{R_n^J} R_n^I; \quad \quad _IS_J := R_n^I \otimes_{R_n^J} R_n^J
$$

Note that $_JM_I$ is naturally a graded $(R_n^J, R_n^I)$-bimodule under left and right multiplication, as indicated by the subscript notation. Similarly, $_IS_J$ is a graded $(R_n^I, R_n^J)$-bimodule. Given an $(R_n^I, R_n^J)$- and $(R_n^J, R_n^L)$-bimodule for some parabolic subgroup $L$, there is a natural composition given by tensor product over the ring $R_n^J$. We refer to this as \textit{\textcolor{revisions}{horizontal} composition}; we denote this operation by $\otimes$ and the resulting $1$-morphisms e.g. by $ _IS_JM_K := (_IS_J) \otimes (_JM_K)$ for $I \subset J \supset K$. When the ground ring is irrelevant, we often abbreviate this operation simply by concatenation; in these instances, all ground rings are assumed to be compatible.

\begin{definition}
Let $SSBim_n$ denote the quantum-graded $R$-linear $2$-category with objects parabolic subgroups of $\mathfrak{S}^n$, $1$-morphisms generated by the bimodules

$$
R_n^I: I \rightarrow I; \quad \textcolor{revisions}{_JM_I: I \rightarrow J;} \quad \textcolor{revisions}{_IS_J: J \rightarrow I}
$$

under composition, quantum shift, direct sum, and direct summands, and $2$-morphisms all graded bimodule homomorphisms. We call both the category $SSBim_n$ and its $1$-morphisms (\textcolor{revisions}{type $A_{n - 1}$}) \textit{Singular Soergel Bimodules}. The collection of categories $SSBim := \bigsqcup_{n \geq 1} SSBim_n$ is itself a monoidal $2$-category with external tensor product given on objects by concatenation of sequences and on $1$- and $2$-morphisms by tensor product over $R$. We denote this external product by $\sqcup$.

%For each $n \geq 1$, there is an inclusion functor $(-) \sqcup R[x_{n+ 1}]: SSBim_n \hookrightarrow SSBim_{n + 1}$; we consider objects, $1$-morphisms, and $2$-morphisms of $SSBim_n$ as living in $SSBim_{n + 1}$ via this inclusion without further comment.
\end{definition}

\begin{remark}
\textcolor{revisions}{Notice that an object of the Hom category $\mathrm{End}_{SSBim_n}((1^n))$ is exactly an $R_n$-bimodule. In particular, there is an obvious inclusion of $SBim_n$ as a full subcategory of $\mathrm{End}_{SSBim_n}((1^n))$. In fact this is an equivalence of categories, but we will not require this fact here.}
\end{remark}

\textcolor{revisions}{The diagrammatic calculus for $SBim$ introduced in Section \ref{sec: SoergBim} has a natural extension to singular Soergel bimodules. Give parabolic subgroups $I = (j, k)$, $J = (j + k) \subset \mathfrak{S}^n$, we denote the possible identity, merge, and split bimodules as follows:}

\begin{gather*}
R^J := 
\begin{tikzpicture}[anchorbase,scale=.5,tinynodes]
	\draw[webs] (1,0) node[below]{$j + k$}
	 to[out=90,in=270] (1,1) node[above,yshift=-2pt]{$j + k$};
\end{tikzpicture}
; \quad\quad
_JM_I := 
\begin{tikzpicture}[anchorbase,scale=.5,tinynodes]
	\draw[webs] (0,0) node[below]{$j$} to[out=90,in=180] (.5,.5);
	\draw[webs] (1,0) node[below]{$k$} to[out=90,in=0] (.5,.5);
	\draw[webs] (.5,.5) to[out=90,in=270] (.5,1) node[above,yshift=-2pt]{$j + k$};
\end{tikzpicture}
; \quad\quad
_IS_J :=
\begin{tikzpicture}[anchorbase,scale=.5,tinynodes]
	\draw[webs] (.5,.5) to[out=180,in=270] (0,1) node[above,yshift=-2pt]{$j$};
	\draw[webs] (.5,.5) to[out=0,in=270] (1,1) node[above,yshift=-2pt]{$k$};
	\draw[webs] (.5,0) node[below]{$j + k$} to[out=90,in=270] (.5,.5);
\end{tikzpicture}
\end{gather*}

\textcolor{revisions}{As was the case for $SBim$, we depict external tensor product $\sqcup$ by horizontal concatenation and horizontal composition $\otimes$ by vertical concatenation.} We also use the following diagrammatic shorthand for \textit{ladder rung bimodules}:

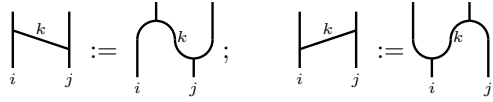
\begin{gather*}
\begin{tikzpicture}[anchorbase,scale=.5,tinynodes]
    \draw[webs] (0,0) node[below]{$i$} to (0, 1.5);
    \draw[webs] (1.5,0) node[below]{$j$} to (1.5, 1.5);
    \draw[webs] (0,1) to (0.75,0.75) node[above,yshift=-2pt]{$k$} to (1.5,.5);
\end{tikzpicture}
:=
\begin{tikzpicture}[anchorbase,scale=.5,tinynodes]
    \draw[webs] (0,0) node[below]{$i$} to (0,1);
    \draw[webs] (0,1) to[out=90,in=180] (.5,1.5);
    \draw[webs] (1,1) node[right,xshift=-3pt]{$k$} to[out=90,in=0] (.5,1.5);
    \draw[webs] (.5,1.5) to (.5,2);
    \draw[webs] (1,1) to[out=270,in=180] (1.5,.5);
    \draw[webs] (2,1) to[out=270,in=0] (1.5,.5);
    \draw[webs] (2,1) to (2,2);
    \draw[webs] (1.5,.5) to (1.5,0) node[below]{$j$};
\end{tikzpicture}
\ ; \quad \quad
\begin{tikzpicture}[anchorbase,scale=.5,tinynodes]
    \draw[webs] (0,0) node[below]{$i$} to (0, 1.5);
    \draw[webs] (1.5,0) node[below]{$j$} to (1.5, 1.5);
    \draw[webs] (0,.5) to (0.75,0.75) node[above,yshift=-2pt]{$k$} to (1.5,1);
\end{tikzpicture}
:=
\begin{tikzpicture}[anchorbase,scale=.5,tinynodes]
    \draw[webs] (0.5,0) node[below]{$i$} to (0.5,0.5);
    \draw[webs] (0.5,0.5) to[out=180,in=270] (0,1);
    \draw[webs] (0,1) to (0,2);
    \draw[webs] (0.5,0.5) to[out=0,in=270] (1,1) node[right,xshift=-3pt]{$k$};
    \draw[webs] (1.5,1.5) to[out=180,in=90] (1,1);
    \draw[webs] (1.5,1.5) to (1.5,2);
    \draw[webs] (1.5,1.5) to[out=0,in=90] (2,1);
    \draw[webs] (2,1) to (2,0) node[below]{$j$};
\end{tikzpicture}
\end{gather*}

Note that we often suppress colors on strands when there is no ambiguity (i.e. when all colors can be determined by balancing colors above and below each trivalent vertex). \textcolor{revisions}{There are a number of of $q$-degree $0$ isomorphisms that can be encoded diagramatically, giving a presentation of $SBim_n$ purely by generators and relations. We refer the interested reader to \cite{CKM14} for a list of these relations\footnote{\textcolor{revisions}{To obtain these relations from their Section 2.2, ignore all relations involving downward oriented strands or tags and take $n \to \infty$.}}. The only such relation we use here is an associativity relation allowing }for unambiguous representation of merges and splits involving more than two strands using vertices of higher valence; for example, letting $I = (a, b, c), J = (a, b + c), K = (a + b + c) \subset \mathfrak{S}^{a + b + c}$, we have:

\begin{gather*}
_IM_K \cong _IM_JM_K =
\begin{tikzpicture}[anchorbase,scale=.5,tinynodes]
    \draw[webs] (0,0) node[below]{$a$} to (0,1);
    \draw[webs] (.5,0) node[below]{$b$} to[out=90,in=180] (1,.5);
    \draw[webs] (1.5,0) node[below]{$c$} to[out=90,in=0] (1,.5);
    \draw[webs] (1,.5) to (1,1);
    \draw[webs] (0,1) to[out=90,in=180] (.5,1.5);
    \draw[webs] (1,1) to[out=90,in=0] (.5,1.5);
    \draw[webs] (.5,1.5) to (.5,2) node[above]{$a + b + c$};
\end{tikzpicture}
\cong
\begin{tikzpicture}[anchorbase,scale=.5,tinynodes]
    \draw[webs] (0,0) node[below]{$a$} to[out=90,in=180] (.5,1);
    \draw[webs] (.5,0) node[below]{$b$} to (.5,1);
    \draw[webs] (1,0) node[below]{$c$} to[out=90,in=0] (.5,1);
    \draw[webs] (.5,1) to (.5,2) node[above,yshift=-2pt]{$a + b + c$};
\end{tikzpicture}
\end{gather*}

We will also be interested in complexes of singular Soergel bimodules. \textcolor{revisions}As $SSBim$ is a $2$-category, the standard constructions of dg categories above do not apply verbatim. We remedy this by applying these constructions to \textcolor{revisions}{$1$-morphism categories; see e.g. Definition 3.20 \cite{HRW21} for further details on this point.}

\textcolor{revisions}{One major benefit in working with \textit{singular} Soergel bimodules is that they allow for the assignment of chain complexes to \textit{colored} braids. We briefly recall these complexes here and again refer the reader to \cite{HRW21} for further details.}

\begin{definition}
The colored braid groupoid $\mathfrak{Br}_n$ on $n$ strands is the set of all pairs of braids $\beta \in Br_n$ and assignments $\textbf{a} = (a_1, \dots, a_n) \in \mathbb{Z}_{\geq 1}^n$ of colors to each strand of $\beta$. Given two colored braids $b = (\beta, \textbf{a})$, $b' = (\beta', \textbf{b})$, their composition is well-defined if and only if $a_{\beta(i)} = b_i$ for each $1 \leq i \leq n$. In this case, we have $bb' = (\beta\beta', \textbf{a})$.
\end{definition}

We employ a graphical calculus for colored braids exactly analogous to that of braids in which each strand is labelled with its corresponding color. For instance, given a sequence of colors $\textbf{a} = (a_1, \dots, a_n)$, we have the \textit{colored Artin generators}:

\begin{gather*}
\sigma_{i, {\textbf{a}}} =
\begin{tikzpicture}[anchorbase,scale=.5,tinynodes]
    \draw[color=white] (0,-1.9) to (1,-1.9);
    \draw[webs] (-1.5,-1) to (-1.5,1) node[above]{$a_1$};
    \node at (-1,0) {$\dots$};
	\draw[webs] (.5,-1) to [out=90,in=270] (-.5,1) node[above]{$a_i$};
	\draw[line width=5pt,color=white] (-.5,-1) to [out=90,in=270] (.5,1);
	\draw[webs] (-.5,-1) to [out=90,in=270] (.5,1) node[above]{$a_{i + 1}$};
    \node at (1,0) {$\dots$};
    \draw[webs] (1.5,-1) to (1.5,1) node[above]{$a_n$};
\end{tikzpicture}
; \quad \quad \sigma_{i, \textbf{a}}^{-1} = 
\begin{tikzpicture}[anchorbase,scale=.5,tinynodes]
    \draw[color=white] (0,-1.9) to (1,-1.9);
    \draw[webs] (-1.5,-1) to (-1.5,1) node[above]{$a_1$};
    \node at (-1,0) {$\dots$};
    \draw[webs] (-.5,-1) to [out=90,in=270] (.5,1) node[above]{$a_{i + 1}$};
	\draw[line width=5pt,color=white] (.5,-1) to [out=90,in=270] (-.5,1);
	\draw[webs] (.5,-1) to [out=90,in=270] (-.5,1) node[above]{$a_i$};
    \node at (1,0) {$\dots$};
    \draw[webs] (1.5,-1) to (1.5,1) node[above]{$a_n$};
\end{tikzpicture}
\end{gather*}

As in the construction of the Rouquier complex, we may assign a \textit{Rickard complex} $F(\beta, \textbf{a}) \in K^b(SSBim_n)$ to each colored braid $(\beta, \textbf{a}) \in \mathfrak{Br}_n$ by specifying this assignment on the Artin generators $\sigma_{i, \textbf{a}}^{\pm 1}$ and extending by vertical composition. The resulting complexes satisfy the (colored) braid relations up to homotopy equivalence, giving a well-defined homomorphism $F: \mathfrak{Br}_n \rightarrow K^b(SSBim_n)$. We record the case $n = 2$, $\textbf{a} = (a, 1)$ below:

\begin{gather*}
F(\sigma_{1, \textbf{a}}) := 
\begin{tikzpicture}[anchorbase,scale=.5,tinynodes]
    \draw[webs] (0,0) node[below]{$1$} to[out=90,in=180] (.5,.5);
    \draw[webs] (1,0) node[below]{$a$} to[out=90,in=0] (.5,.5);
    \draw[webs] (.5,.5) to (.5,1);
    \draw[webs] (.5,1) to[out=180,in=270] (0,1.5) node[above,yshift=-2pt]{$a$};
    \draw[webs] (.5,1) to[out=0,in=270] (1,1.5) node[above,yshift=-2pt]{$1$};
    \draw[thick,->] (1.5,.75) to (3.5,.75);
\end{tikzpicture}
\ q^{-1}t \ 
\begin{tikzpicture}[anchorbase,scale=.5,tinynodes]
    \draw[webs] (0,0) node[below]{$1$} to (0,1.5) node[above,yshift=-2pt]{$a$};
    \draw[webs] (1,0) node[below]{$a$} to (1,1.5) node[above,yshift=-2pt]{$1$};
    \draw[webs] (0,1) to (1,.5);
\end{tikzpicture}
; \quad \quad
F(\sigma_{1, \textbf{a}}^{-1}) := qt^{-1}
\begin{tikzpicture}[anchorbase,scale=.5,tinynodes]
    \draw[webs] (0,0) node[below]{$1$} to (0,1.5) node[above,yshift=-2pt]{$a$};
    \draw[webs] (1,0) node[below]{$a$} to (1,1.5) node[above,yshift=-2pt]{$1$};
    \draw[webs] (0,1) to (1,.5);
    \draw[thick,->] (1.5,.75) to (3.5,.75);
\end{tikzpicture}
\begin{tikzpicture}[anchorbase,scale=.5,tinynodes]
    \draw[webs] (0,0) node[below]{$1$} to[out=90,in=180] (.5,.5);
    \draw[webs] (1,0) node[below]{$a$} to[out=90,in=0] (.5,.5);
    \draw[webs] (.5,.5) to (.5,1);
    \draw[webs] (.5,1) to[out=180,in=270] (0,1.5) node[above,yshift=-2pt]{$a$};
    \draw[webs] (.5,1) to[out=0,in=270] (1,1.5) node[above,yshift=-2pt]{$1$};
\end{tikzpicture}
\end{gather*}

In case $\textbf{a} = (1, a)$, the corresponding Rickard complexes are obtained from the above by reflecting each chain bimodule about a central vertical axis. We will not require explicit formulas for the differentials on these complexes; rather than quoting that result, we mention only that if $a = 1$, we recover the Rouquier complexes of \textcolor{revisions}{Section \ref{sec: rouq}}. The reader interested in the specific form of these differentials can see \cite{HRW21} for more detail.

As in the case of Rouquier complexes, we depict the Rickard complex of a colored braid diagramatically and consider this as an extension of the graphical calculus for $K^b(SSBim)$. As an example, we have the \textit{fork slide} relation in $K^b(SSBim)$:

\begin{proposition}[\cite{RT21}] \label{prop: fork_moves}
\textcolor{revisions}{For each $a, b, c \geq 1$, there exists a $q$-degree $0$ homotopy equivalence}

\begin{gather*}
\begin{tikzpicture}[anchorbase,scale=.5,tinynodes]
    \draw[webs] (1,2) node[above,yshift=-2pt]{$a + b$} to (1,1.5);
    \draw[webs] (1,1.5) to[out=180,in=90] (.5,0) node[below]{$a$};
    \draw[webs] (1,1.5) to[out=0,in=90] (1.5,0) node[below]{$b$};
    \draw[line width=5pt, color=white] (2.5,2) to[out=270,in=90] (-.5,0);
    \draw[webs] (2.5,2) node[above,yshift=-2pt]{$c$} to[out=270,in=90] (-.5,0) node[below]{$c$};
\end{tikzpicture}
\simeq
\begin{tikzpicture}[anchorbase,scale=.5,tinynodes]
    \draw[webs] (1,2) node[above,yshift=-2pt]{$a + b$} to (1,.5);
    \draw[webs] (1,.5) to[out=180,in=90] (.5,0) node[below]{$a$};
    \draw[webs] (1,.5) to[out=0,in=90] (1.5,0) node[below]{$b$};
    \draw[line width=5pt, color=white] (2.5,2) to[out=270,in=90] (-.5,0);
    \draw[webs] (2.5,2) node[above,yshift=-2pt]{$c$} to[out=270,in=90] (-.5,0) node[below]{$c$};
\end{tikzpicture}
\end{gather*}

as well as rotations, reflections, and analogs involving negative crossings. Moreover, each of these homotopy equivalences is a strong deformation retraction from the complex on the left to the complex on the right obtained by Gaussian elimination.
\end{proposition}

The dot sliding homotopies \textcolor{revisions}{of Proposition \ref{prop: dot_slide}} extend to Rickard complexes, though some care must be taken to account for the action of higher degree symmetric polynomials on colored strands. We only use dot slides on $1$-colored strands in this paper; see \cite{HRW21} for more detail.

\section{Proof of Propositions \ref{prop: mu_nu} and \ref{prop: forkslyde}} \label{app: dot_slide_nat}

\textcolor{revisions}{To prove Proposition \ref{prop: mu_nu}, we will make use of the following minimal model for the full twist on two strands; this is a special case of Theorem 3.24 of \cite{HRW21b}.}

\begin{proposition} \label{prop: ft_min}
Let $k \geq 1$ be given. Then there is a strong deformation retraction in $K^b(SSBim_n)$ obtained from Gaussian elimination:

\begin{gather*}
FT_{k, 1} := \ 
\begin{tikzpicture}[anchorbase,scale=.5,tinynodes]
    \draw[webs] (1.5,0) node[below]{$1$} to[out=90,in=270] (0,1.5);
    \draw[line width=5pt,color=white] (0,0) to[out=90,in=270] (1.5,1.5);
    \draw[webs] (0,0) node[below]{$k$} to[out=90,in=270] (1.5,1.5);
    \draw[webs] (1.5,1.5) to[out=90,in=270] (0,3) node[above,yshift=-2pt]{$k$};
    \draw[line width=5pt,color=white] (0,1.5) to[out=90,in=270] (1.5,3);
    \draw[webs] (0,1.5) to[out=90,in=270] (1.5,3) node[above,yshift=-2pt]{$1$};
\end{tikzpicture}
\simeq q^k
\begin{tikzpicture}[anchorbase,scale=.5,tinynodes]
    \draw[webs] (0,0) node[below]{$k$} to[out=90,in=180] (.5,.5);
    \draw[webs] (1,0) node[below]{$1$} to[out=90,in=0] (.5,.5);
    \draw[webs] (.5,.5) to (.5,1);
    \draw[webs] (.5,1) to[out=180,in=270] (0,1.5) node[above,yshift=-2pt]{$k$};
    \draw[webs] (.5,1) to[out=0,in=270] (1,1.5) node[above,yshift=-2pt]{$1$};
    \draw[thick] (1.5,.75) to (2.5,.75) node[above]{$x_n - x_n'$};
    \draw[thick,->] (2.5,.75) to (3.5,.75);
\end{tikzpicture}
q^{k - 2}t
\begin{tikzpicture}[anchorbase,scale=.5,tinynodes]
    \draw[webs] (0,0) node[below]{$k$} to[out=90,in=180] (.5,.5);
    \draw[webs] (1,0) node[below]{$1$} to[out=90,in=0] (.5,.5);
    \draw[webs] (.5,.5) to (.5,1);
    \draw[webs] (.5,1) to[out=180,in=270] (0,1.5) node[above,yshift=-2pt]{$k$};
    \draw[webs] (.5,1) to[out=0,in=270] (1,1.5) node[above,yshift=-2pt]{$1$};
    \draw[thick] (1.5,.75) to (2.5,.75) node[above]{$unzip$};
    \draw[thick,->] (2.5,.75) to (3.5,.75);
\end{tikzpicture}
\ q^{-2}t^2
\begin{tikzpicture}[anchorbase,scale=.5,tinynodes]
    \draw[webs] (0,0) node[below]{$k$} to (0,1.5) node[above,yshift=-2pt]{$k$};
    \draw[webs] (.75,0) node[below]{$1$} to (.75,1.5) node[above,yshift=-2pt]{$1$};
\end{tikzpicture}
\end{gather*}
\end{proposition}

\begin{definition} \label{def: stmn}
For each $n \geq 2$, we let $STM_n := _{1^n}S_{\mathfrak{S}^{n - 1}} \otimes FT_{n - 1, 1} \otimes _{\mathfrak{S}^{n - 1}}M_{1^n}$. This complex is depicted below.

\begin{center}
$STM_n := $
\begin{tikzpicture}[anchorbase,scale=.5,tinynodes]
    \draw [decorate,decoration = {brace}] (0,1.2) -- (1,1.2) node[pos=.5,above]{$n - 1$};
    \draw[webs] (0,1) to[out=270,in=180] (.5,.5);
    \node at (.5,1) {$\dots$};
    \draw[webs] (1,1) to[out=270,in=0] (.5,.5);
    \draw[webs] (1.5,1) to (1.5,.5);
    \draw[webs] (.5,.5) to[out=270,in=90] (1.5,-.5) node[right]{$n - 1$};
    \draw[line width=5pt,color=white] (1.5,.5) to[out=270,in=90] (.5,-.5);
    \draw[webs] (1.5,.5) to[out=270,in=90] (.5,-.5) node[left]{$1$};
    \draw[webs] (.5,-.5) to[out=270,in=90] (1.5,-1.5);
    \draw[line width=5pt,color=white] (1.5,-.5) to[out=270,in=90] (.5,-1.5);
    \draw[webs] (1.5,-.5) to[out=270,in=90] (.5,-1.5);
    \node at (.5,-2) {$\dots$};
    \draw[webs] (.5,-1.5) to[out=180,in=90] (0,-2);
    \draw[webs] (.5,-1.5) to[out=0,in=90] (1,-2);
    \draw[webs] (1.5,-1.5) to (1.5,-2);
    \draw [decorate,decoration = {brace,mirror}] (0,-2.2) -- (1,-2.2) node[pos=.5,below,yshift=-2pt]{$n - 1$};
\end{tikzpicture}
\end{center}
\end{definition}

\begin{proof}[\textcolor{revisions}{Proof of Proposition \ref{prop: mu_nu}}]

By \textcolor{revisions}{repeated application of Proposition \ref{prop: fork_moves}}, we obtain a strong deformation retraction $(\overline{\mu}, \overline{\nu})$ from $B_{w_1}J_n$ to $STM_n$; see the diagram below. The result then follows from the minimal complex for $FT_{n - 1,1}$ in Proposition \ref{prop: ft_min}.

\begin{center}
$\overline{\nu} \colon$ 
\begin{tikzpicture}[anchorbase,scale=.5,tinynodes]
    \draw [decorate,decoration = {brace}] (0,1.7) -- (1,1.7) node[pos=.5,above]{$n - 1$};
    \draw[webs] (0,0) to[out=90,in=180] (.5,.5);
    \node at (.5,0) {$\dots$};
    \draw[webs] (1,0) to[out=90,in=0] (.5,.5);
    \draw[webs] (.5,.5) to (.5,1);
    \draw[webs] (.5,1) to[out=180,in=270] (0,1.5);
    \node at (.5,1.5) {$\dots$};
    \draw[webs] (.5,1) to[out=0,in=270] (1,1.5);
    \draw[webs] (1.5,0) to (1.5,1.5);
    \draw[webs] (0,0) to[out=270,in=90] (.5,-1);
    \draw[webs] (1,0) to[out=270,in=90] (1.5,-1);
    \node at (1,-1) {$\dots$};
    \draw[line width=5pt,color=white] (1.5,0) to[out=270,in=90] (0,-1);
    \draw[webs] (1.5,0) to[out=270,in=90] (0,-1);
    \draw[webs] (0,-1) to[out=270,in=90] (1.5,-2);
    \draw[line width=5pt,color=white] (.5,-1) to[out=270,in=90] (0,-2);
    \draw[webs] (.5,-1) to[out=270,in=90] (0,-2);
    \draw[line width=5pt,color=white] (1.5,-1) to[out=270,in=90] (1,-2);
    \draw[webs] (1.5,-1) to[out=270,in=90] (1,-2);
    \node at (.5,-2) {$\dots$};
    \draw [decorate,decoration = {brace,mirror}] (0,-2.2) -- (1,-2.2) node[pos=.5,below,yshift=-2pt]{$n - 1$};
\end{tikzpicture}
$\longleftrightarrow$
\begin{tikzpicture}[anchorbase,scale=.5,tinynodes]
    \draw [decorate,decoration = {brace}] (0,1.2) -- (1,1.2) node[pos=.5,above]{$n - 1$};
    \draw[webs] (0,1) to[out=270,in=180] (.5,.5);
    \node at (.5,1) {$\dots$};
    \draw[webs] (1,1) to[out=270,in=0] (.5,.5);
    \draw[webs] (1.5,1) to (1.5,.5);
    \draw[webs] (.5,.5) to[out=270,in=90] (1.5,-.5) node[right]{$n - 1$};
    \draw[line width=5pt,color=white] (1.5,.5) to[out=270,in=90] (.5,-.5);
    \draw[webs] (1.5,.5) to[out=270,in=90] (.5,-.5) node[left]{$1$};
    \draw[webs] (.5,-.5) to[out=270,in=90] (1.5,-1.5);
    \draw[line width=5pt,color=white] (1.5,-.5) to[out=270,in=90] (.5,-1.5);
    \draw[webs] (1.5,-.5) to[out=270,in=90] (.5,-1.5);
    \node at (.5,-2) {$\dots$};
    \draw[webs] (.5,-1.5) to[out=180,in=90] (0,-2);
    \draw[webs] (.5,-1.5) to[out=0,in=90] (1,-2);
    \draw[webs] (1.5,-1.5) to (1.5,-2);
    \draw [decorate,decoration = {brace,mirror}] (0,-2.2) -- (1,-2.2) node[pos=.5,below,yshift=-2pt]{$n - 1$};
\end{tikzpicture}
$\colon \overline{\mu}$
\end{center}
\end{proof}

\textcolor{revisions}{Next, we turn our attention to Proposition \ref{prop: forkslyde}, beginning by verifying a certain $y$-ified version of Proposition \ref{prop: fork_moves}.} Before doing this, we pause to establish notation. For each $k \geq 0$, let $\Sigma_k^{\pm}, \tilde{\Sigma}_k^{\pm}$ denote the complexes of singular Soergel bimodules depicted below.

\begin{gather*}
\Sigma_k^+ := \ 
\begin{tikzpicture}[anchorbase,scale=.5,tinynodes]
    \draw[webs] (1,2) node[above,yshift=-2pt]{$k + 1$} to (1,1.5);
    \draw[webs] (1,1.5) to[out=180,in=90] (.5,0) node[below]{$1$};
    \draw[webs] (1,1.5) to[out=0,in=90] (1.5,0) node[below]{$k$};
    \draw[line width=5pt, color=white] (2.5,2) to[out=270,in=90] (-.5,0);
    \draw[webs] (2.5,2) node[above,yshift=-2pt]{$1$} to[out=270,in=90] (-.5,0) node[below]{$1$};
\end{tikzpicture}
; \quad \tilde{\Sigma}_k^+ := \ 
\begin{tikzpicture}[anchorbase,scale=.5,tinynodes]
    \draw[webs] (1,2) node[above,yshift=-2pt]{$k + 1$} to (1,.5);
    \draw[webs] (1,.5) to[out=180,in=90] (.5,0) node[below]{$1$};
    \draw[webs] (1,.5) to[out=0,in=90] (1.5,0) node[below]{$k$};
    \draw[line width=5pt, color=white] (2.5,2) to[out=270,in=90] (-.5,0);
    \draw[webs] (2.5,2) node[above,yshift=-2pt]{$1$} to[out=270,in=90] (-.5,0) node[below]{$1$};
\end{tikzpicture}
\end{gather*}

\begin{gather*}
\Sigma_k^- := \ 
\begin{tikzpicture}[anchorbase,scale=.5,tinynodes]
    \draw[webs] (-.5,2) node[above,yshift=-2pt]{$1$} to[out=270,in=90] (2.5,0) node[below]{$1$};
    \draw[line width=5pt,color=white] (1,1.5) to[out=180,in=90] (.5,0);
    \draw[line width=5pt,color=white] (1,1.5) to[out=0,in=90] (1.5,0);
    \draw[webs] (1,2) node[above,yshift=-2pt]{$k + 1$} to (1,1.5);
    \draw[webs] (1,1.5) to[out=180,in=90] (.5,0) node[below]{$k$};
    \draw[webs] (1,1.5) to[out=0,in=90] (1.5,0) node[below]{$1$};
\end{tikzpicture}
; \quad \tilde{\Sigma}_k^- := \ 
\begin{tikzpicture}[anchorbase,scale=.5,tinynodes]
    \draw[webs] (-.5,2) node[above,yshift=-2pt]{$1$} to[out=270,in=90] (2.5,0) node[below]{$1$};
    \draw[line width=5pt,color=white] (1,2) to (1,.5);
    \draw[webs] (1,2) node[above,yshift=-2pt]{$k + 1$} to (1,.5);
    \draw[webs] (1,.5) to[out=180,in=90] (.5,0) node[below]{$k$};
    \draw[webs] (1,.5) to[out=0,in=90] (1.5,0) node[below]{$1$};
\end{tikzpicture}
\end{gather*}

As usual, there are dot-sliding homotopies $h^{\pm} \in \text{End}^{-1}(\Sigma_k^{\pm})$, $\tilde{h}^{\pm} \in \text{End}^{-1}(\tilde{\Sigma}_k^{\pm})$ satisfying $[d, h^+] = [d, \tilde{h}^+] = x_{k + 2} - x_1''$, $[d, h^-] = [d, \tilde{h}^-] = x_1 - x_{k + 2}''$. Let $f_k^{\pm} \in \text{Hom}^0_{\text{Ch}(SSBim_{k + 2})}(\Sigma_k^{\pm}, \tilde{\Sigma}_k^{\pm})$ and $g_k^{\pm} \in \text{Hom}^0_{\text{Ch}(SSBim_{k + 2})}(\tilde{\Sigma}_k^{\pm}, \Sigma_k^{\pm})$ denote the fork-slide homotopy equivalences of Proposition \ref{prop: fork_moves}.

\begin{lemma} \label{lem: dot_slide_fork_slide}
The maps defined above satisfy $f_k^{\pm} h_k^{\pm} - \tilde{h}_k^{\pm} f_k^{\pm} = 0$ and $g_k^{\pm} \tilde{h}_k^{\pm} - h_k^{\pm} g_k^{\pm} = 0$.
\end{lemma}

\begin{proof}
We include a proof only of the overhanded relations $f_k^+ h_k^+ - \tilde{h}_k^+ f_k+ = 0$ and $g_k^+ \tilde{h}_k^+ - h_k^+ g_k^+ = 0$; the underhanded relations are exactly analogous. Let $X, Y, Z \in SSBim_{k + 2}$ be as below.

\begin{gather*}
X := \ 
\begin{tikzpicture}[anchorbase,scale=.5,tinynodes]
    \draw[webs] (0,2) node[above,yshift=-2pt]{$k + 1$} to[out=270,in=180] (.5,1.5);
    \draw[webs] (1,2) node[above,yshift=-2pt]{$1$} to[out=270,in=0] (.5,1.5);
    \draw[webs] (.5,1.5) to (.5,1);
    \draw[webs] (.5,1) to[out=0,in=90] (1,.5);
    \draw[webs] (1,.5) to (1,0) node[below]{$k$};
    \draw[webs] (.5,1) to[out=180,in=90] (0,.5);
    \draw[webs] (0,.5) to[out=180,in=90] (-.5,0) node[below]{$1$};
    \draw[webs] (0,.5) to[out=0,in=90] (.5,0) node[below]{$1$};
\end{tikzpicture}
; \quad Y := \ 
\begin{tikzpicture}[anchorbase,scale=.5,tinynodes]
    \draw[webs] (0,2) node[above,yshift=-2pt]{$k + 1$} to (0,.5);
    \draw[webs] (1,2) node[above,yshift=-2pt]{$1$} to (1,0) node[below]{$k$};
    \draw[webs] (0,1.25) to (1,.75);
    \draw[webs] (0,.5) to[out=180,in=90] (-.5,0) node[below]{$1$};
    \draw[webs] (0,.5) to[out=0,in=90] (.5,0) node[below]{$1$};
\end{tikzpicture}
; \quad Z := \ 
\begin{tikzpicture}[anchorbase,scale=.5,tinynodes]
    \draw[webs] (0,2) node[above,yshift=-2pt]{$k + 1$} to (0,0) node[below]{$1$};
    \draw[webs] (1,2) node[above,yshift=-2pt]{$1$} to (1,.5);
    \draw[webs] (0,1.25) to (1,.75);
    \draw[webs] (1,.5) to[out=180,in=90] (.5,0) node[below]{$1$};
    \draw[webs] (1,.5) to[out=0,in=90] (1.5,0) node[below]{$k$};
\end{tikzpicture}
\end{gather*}

By expanding each Rickard complex and repeatedly applying the \textcolor{revisions}{diagrammtic relations of \cite{CKM14}}, we obtain the explicit descriptions below of the complexes $\Sigma^+_k$ and $\tilde{\Sigma}^+_k$.

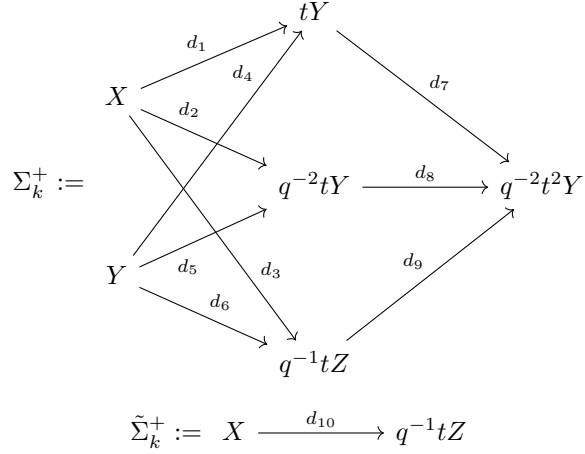
\begin{center}
$\Sigma^+_k := $ 
\begin{tikzcd}
                                                             &  & tY \arrow[rrdd, "d_7"]       &  &            \\
X \arrow[rru, "d_1"] \arrow[rrd, "d_2", near start] \arrow[rrddd, "d_3", near end] &  &                              &  &            \\
                                                             &  & q^{-2}tY \arrow[rr, "d_8"]   &  & q^{-2}t^2Y \\
Y \arrow[rruuu, "d_4", near end] \arrow[rrd, "d_6"] \arrow[rru, "d_5"', near start] &  &                              &  &            \\
                                                             &  & q^{-1}tZ \arrow[rruu, "d_9"] &  &           
\end{tikzcd}
\end{center}

\begin{center}
$\tilde{\Sigma}^+_k := $
\begin{tikzcd}
X \arrow[rr, "d_{10}"] & & q^{-1}tZ
\end{tikzcd}
\end{center}

Since $f_k^+$, $g_k^+$ are obtained from Gaussian elimination, Proposition \ref{prop: gauss_elim} furnishes an explicit description of these maps on each component in terms of the differentials on each complex. We depict these maps as vertical blue arrows below. We will also find it convenient to visualize each component of the homotopies $h_k^+$ and $\tilde{h}_k^+$; we depict those components as backwards horizontal red arrows. In total, we see the following collection of morphisms:

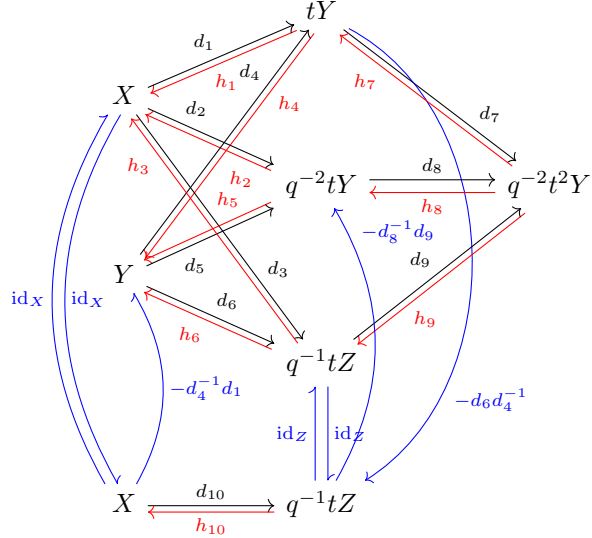
\begin{center}
\begin{tikzcd}
                                                                                                        &  & tY \arrow[rrdd, "d_7", near end] \arrow[dddddd, "-d_6d_4^{-1}", bend left=60, near end,blue]  \arrow[lld, "h_1", pos=.6, red, shift left] \arrow[llddd, "h_4", red, shift left, near start]                             &  &            \\
X \arrow[rrd, "d_2", near start] \arrow[rrddd, "d_3", near end] \arrow[rru, "d_1"] \arrow[ddddd, "\text{id}_X", harpoon, bend right, shift left, blue] &  &                                                                                                   &  &            \\
                                                                                                        &  & q^{-2}tY \arrow[rr, "d_8", shift left] \arrow[lld, "h_5"', pos=.2, shift right, red] \arrow[llu, "h_2", red, pos=.1,  shift left]                                                                      &  & q^{-2}t^2Y \arrow[lluu, "h_7", red, shift left, near end] \arrow[ll, "h_8", red, shift left] \arrow[lldd, "h_9", red, shift left, pos=.7] \\
Y \arrow[rruuu, "d_4", near end] \arrow[rrd, "d_6"] \arrow[rru, "d_5"', near start]                                            &  &                                                                                                   &  &            \\
                                                                                                        &  & q^{-1}tZ \arrow[rruu, "d_9"] \arrow[dd, "\text{id}_Z", harpoon, shift left, blue] \arrow[llu, "h_6", shift left, red] \arrow[lluuu, "h_3", red, shift left, pos=.85]                     &  &            \\
                                                                                                        &  &                                                                                                   &  &            \\
X \arrow[uuu, "-d_4^{-1}d_1"', harpoon, bend right, blue] \arrow[uuuuu, "\text{id}_X", harpoon, bend left, shift left, blue] \arrow[rr, "d_{10}"]                                          &  & q^{-1}tZ \arrow[uu, "\text{id}_Z", harpoon, shift left, blue] \arrow[uuuu, "-d_8^{-1}d_9"', pos=.85, bend right, blue] \arrow[ll, "h_{10}", red, shift left] &  &           
\end{tikzcd}
\end{center}

In fact several of these components must vanish \textcolor{revisions}{by degree considerations}. \textcolor{revisions}{There is an analog of Lemma \ref{lem: morph_dim} which allows us to compute the graded dimensions of morphism spaces between singular Soergel bimodules using duality functors; see Proposition 3.19 of \cite{HRW21} for details. Using this technology, one can compute} the following graded dimensions for various relevant morphism spaces:

\begin{align*}
    \text{dim}\left(\text{Hom}_{SSBim_{k + 2}}(X, Y)\right) & = [k] \left( \prod_{i = 1}^k \cfrac{q^2}{1 - q^{2i}} \right) \left( \cfrac{q}{1 - q^2} \right) \left( \cfrac{1}{1 - q^2} \right) q^{-k} \\
    & = (q^2 + q^4 + \dots + q^{2k}) \left( \prod_{i = 1}^k \cfrac{1}{1 - q^{2i}} \right) \left( \cfrac{1}{1 - q^2} \right)^2; \\
    \text{dim}(\text{Hom}_{SSBim_{k + 2}}(Y, X)) & = (q^2 + q^4 + \dots + q^{2k}) \left( \prod_{i = 1}^k \cfrac{1}{1 - q^{2i}} \right) \left( \cfrac{1}{1 - q^2} \right)^2; \\
    \text{dim}(\text{Hom}_{SSBim_{k + 2}}(Z, X)) & = (q + q^3 + \dots + q^{2k + 1}) \left( \prod_{i = 1}^k \cfrac{1}{1 - q^{2i}} \right) \left( \cfrac{1}{1 - q^2} \right)^2
\end{align*}

Observe that the lowest quantum degree of any nonzero morphism from $X$ to $Y$ is $q^2$. Since $d_1$ is a degree $q^0$ morphism from $X$ to $Y$, it must vanish. Similarly, since the lowest quantum degree of any nonzero morphism from $Y$ to $X$ is $q^2$ and $h_2$ is a degree $q^0$ morphism from $X$ to $Y$, it too must vanish. We update the diagram to reflect these requirements below:

\begin{center}
\begin{tikzcd}
                                                                                                        &  & tY \arrow[rrdd, "d_7", near end] \arrow[dddddd, "-d_6d_4^{-1}", bend left=60, near end,blue]  \arrow[lld, "h_1"', pos=.6, red, shift left] \arrow[llddd, "h_4", red, shift left, near start]                             &  &            \\
X \arrow[rrd, "d_2", near start] \arrow[rrddd, "d_3", near end] \arrow[ddddd, "\text{id}_X", harpoon, bend right, shift left, blue] &  &                                                                                                   &  &            \\
                                                                                                        &  & q^{-2}tY \arrow[rr, "d_8", shift left] \arrow[lld, "h_5"', pos=.2, shift right, red]                                                                     &  & q^{-2}t^2Y \arrow[lluu, "h_7", red, shift left, near end] \arrow[ll, "h_8", red, shift left] \arrow[lldd, "h_9", red, shift left, pos=.7] \\
Y \arrow[rruuu, "d_4", near end] \arrow[rrd, "d_6"] \arrow[rru, "d_5"', near start]                                            &  &                                                                                                   &  &            \\
                                                                                                        &  & q^{-1}tZ \arrow[rruu, "d_9"] \arrow[dd, "\text{id}_Z", harpoon, shift left, blue] \arrow[llu, "h_6", shift left, red] \arrow[lluuu, "h_3", red, shift left, pos=.85]                     &  &            \\
                                                                                                        &  &                                                                                                   &  &            \\
X \arrow[uuuuu, "\text{id}_X", harpoon, bend left, shift left, blue] \arrow[rr, "d_{10}"]                                          &  & q^{-1}tZ \arrow[uu, "\text{id}_Z", harpoon, shift left, blue] \arrow[uuuu, "-d_8^{-1}d_9"', pos=.85, bend right, blue] \arrow[ll, "h_{10}", red, shift left] &  &           
\end{tikzcd}
\end{center}

Various relationships among these components also follow from properties of the morphisms they constitute. In order for $f_k^+$ to be a chain map, we must have $d_3 = d_{10}$. Additionally, since $\text{Hom}_{SSBim_{k + 2}}(Z, X)$ is $1$-dimensional in degree $q$, we must have $h_3 = \lambda h_{10}$ for some scalar $\lambda \in R$. (The case $h_{10} = 0$ is excluded by the identity $h_{10}d_{10} = x_{k + 2} - x_1'' \neq 0$ on $\tilde{\Sigma}_k^+$.) Since the only component of $[d_{\Sigma_k^+}, h_k^+]$ mapping from $X$ to itself is $h_3d_3$, we must have $x_{k + 2} - x_1'' = h_3d_3 = \lambda h_{10}d_{10} = \lambda(x_{k + 2} - x_1'')$. Then $\lambda = 1$, so $h_3 = h_{10}$.

Since $[d_{\Sigma_k^+}, h_k^+] = x_{k + 2} - x_1''$, all components of this map that are not endomorphisms must vanish. We record useful consequences of this vanishing in particular components below.

\begin{align*}
Y \to X & : h_1d_4 + h_3d_6 = 0 \Longrightarrow h_1 = -h_3d_6d_4^{-1} \\
q^{-2}tY \to tY & : d_4h_5 + h_7d_8 = 0 \Longrightarrow h_5 = -d_4^{-1}h_7d_8, \ h_7 = -d_4h_5d_8^{-1} \\
q^{-2}tY \to q^{-1}tZ & : d_6h_5 + h_9d_8 = 0 \Longrightarrow h_9 = -d_6h_5d_8^{-1} \\
q^{-1}tZ \to tY & : d_4h_6 + h_7d_9 = 0 \Longrightarrow h_6 = -d_4^{-1}h_7d_9 \\
\end{align*}

Finally, we check that $f_k^+ h_k^+ - \tilde{h}_k^+ f_k+ = 0$ and $g_k^+ \tilde{h}_k^+ - h_k^+ g_k^+ = 0$ via explicit computation. The component of $g_k^+ \tilde{h}_k^+ - h_k^+ g_k^+$ from $q^{-1}tZ$ to $X$ is given by $h_{10} - h_3$; this was shown to vanish above. The other component of $g_k^+ \tilde{h}_k^+ - h_k^+ g_k^+$ maps from $q^{-1}tZ$ to $Y$ and is given by $h_5d_8^{-1}d_9 - h_6$; this vanishes by lines 2 and 4 of the above list of identities.

Next, the component of $f_k^+ h_k^+ - \tilde{h}_k^+ f_k^+$ mapping from $q^{-2}t^2Y$ to $q^{-1}tZ$ is given by $h_9 - d_6d_4^{-1}h_7$; this vanishes by lines 2 and 3 above. The component mapping from $tY$ to $X$ is given by $h_1 + h_{10}d_6d_4^{-1} = h_1 + h_3d_6d_4^{-1}$; this vanishes by line 1 above. The component mapping from $q^{-2}Y$ to $X$ is already $0$. Finally, the component mapping from $q^{-1}tZ$ to $X$ is given by $h_3 - h_{10}$, which again vanishes.
\end{proof}

Let $h \in \text{End}^{-1}(FT_{n - 1, 1})$ denote the dot-sliding homotopy satisfying $[d, h] = x_n - x_n''$. Observe that both $h$ and the dot-sliding homotopy $h_n \in \text{End}^{-1}(B_{w_1}J_n)$ satisfying $[d, h_n] = x_n - x_n''$ are sums of dot-sliding homotopies of the form in Lemma \ref{lem: dot_slide_fork_slide}. Additionally, the strong deformation retraction $\overline{\mu} \colon B_{w_1}J_n \simeq STM_n \colon \overline{\nu}$ of Proposition \ref{prop: mu_nu} consists of compositions of fork-slides of the form of Lemma \ref{lem: dot_slide_fork_slide}. Applying the result of Lemma \ref{lem: dot_slide_fork_slide} repeatedly to these compositions, we immediately obtain

\begin{proposition} \label{prop: half_forkslyde}
The maps above satisfy $h \overline{\mu} = \overline{\mu} h_n$, $h_n \overline{\nu} = \overline{\nu} h$.
\end{proposition}

Now, let $STM_n^{y_n}$ denote the curved complex $STM_n[\mathbb{Y}]$ with connection $\delta_{STM_n} := d_{STM_n} + h(y_n - y_1)$. A direct computation shows that $STM_n^{y_n}$ has curvature $Z_{y_n}$.

\begin{corollary} \label{cor: half_forkslyde}
The strong deformation retraction $\overline{\mu} \colon B_{w_1}J_n \simeq STM_n \colon \overline{\nu}$ of Proposition \ref{prop: mu_nu} lifts without modification to a strong deformation retraction from $B_{w_1}J_n^{y_n}$ to $STM_n^{y_n}$.
\end{corollary}

\begin{proof}
We verify directly that $[\delta, \overline{\mu}] = [\delta, \overline{\nu}] = 0$, beginning with the former.

\begin{align*}
    [\delta, \overline{\mu}] & = (d_{STM_n} + h(y_n - y_1)) \overline{\mu} - \overline{\mu} (d_{B_{w_1}J_n} + h_n(y_n - y_1)) \\
    & = [d, \overline{\mu}] + (h \overline{\mu} - \overline{\mu} h_n) (y_n - y_1)
\end{align*}

The first term above vanishes by Proposition \ref{prop: mu_nu}, and the second vanishes by \ref{prop: half_forkslyde}. Similarly, we have

\begin{align*}
    [\delta, \overline{\nu}] & = (d_{B_{w_1}J_n} + h_n(y_n - y_1)) \overline{\nu} - \overline{\nu} (d_{STM_n} + h(y_n - y_1))  \\
    & = [d, \overline{\nu}] + (h_n \overline{\nu} - \overline{\nu} h) (y_n - y_1)
\end{align*}

Again, the first term above vanishes by Proposition \ref{prop: mu_nu}, and the second vanishes by \ref{prop: half_forkslyde}.

It remains to show that $\overline{\mu}$, $\overline{\nu}$ constitute a strong deformation retraction. The identity $\overline{\nu} \overline{\mu} = \text{id}_{STM_n^{y_n}}$ follows immediately from the corresponding uncurved identity. We have a homotopy equivalence

\begin{align*}
\text{End}_{\text{Ch}(SSBim_n)}(STM_n) & \simeq  \text{End}_{\text{Ch}(SBim_n)}(B_{w_1}J_n) \\
& \simeq \text{End}_{\text{Ch}(SBim_n)}(B_{w_1}J_nJ_n^{-1}) \\
& \simeq \text{End}_{\text{Ch}(SBim_n)}(B_{w_1})
\end{align*}

The first homotopy equivalence above is induced from conjugation by $\overline{\mu}$ and $\overline{\nu}$, the second is induced by the invertible functor $- J_n^{-1}$, and the third is induced by the homotopy equivalence $J_n J_n^{-1} \simeq R_n$. Since $B_{w_1}$ is concentrated in homological degree $0$, we must have $H^i(\text{End}_{\text{Ch}(SSBim_n)}(STM_n)) = 0$ for $i \neq 0$. Then $\overline{\mu} \overline{\nu} - \text{id}_{STM_n^{y_n}} \in \text{End}^0_{Z_{y_n}-\text{Fac}(R_n-\text{Bim})}(STM_n^{y_n})$ is nullhomotopic by a direct application of Corollary \ref{corr: GHobs2}.
\end{proof}

Next, we denote by $T_{k, 1}$ the minimal complex for $FT_{k, 1}$. We give names to the bimodules appearing in $T_{k, 1}$ below.

\begin{gather*}
B := \ 
\begin{tikzpicture}[anchorbase,scale=.5,tinynodes]
    \draw[webs] (0,2) node[above,yshift=-2pt]{$k$} to[out=270,in=180] (.5,1.5);
    \draw[webs] (1,2) node[above,yshift=-2pt]{$1$} to[out=270,in=0] (.5,1.5);
    \draw[webs] (.5,1.5) to (.5,.5);
    \draw[webs] (.5,.5) to[out=0,in=90] (1,0) node[below]{$1$};
    \draw[webs] (.5,.5) to[out=180,in=90] (0,0) node[below]{$k$};
\end{tikzpicture}
; \quad I := \ 
\begin{tikzpicture}[anchorbase,scale=.5,tinynodes]
    \draw[webs] (0,2) node[above,yshift=-2pt]{$k$} to (0,0) node[below]{$k$};
    \draw[webs] (1,2) node[above,yshift=-2pt]{$1$} to (1,0) node[below]{$1$};
\end{tikzpicture}
\end{gather*}

Let $h' \in \text{End}^{-1}_{\text{Ch}(SSBim_{k + 1})}(T_{k, 1})$ be the degree $q^2t^{-1}$ morphism depicted in red below.

\begin{center}
\begin{tikzcd}[sep=large]
q^k B \arrow[r, "x_{k + 1} - x_{k + 1}''", harpoon, shift left] & q^{k - 2}t B \arrow[l, "\text{id}_J", red, harpoon, shift left] \arrow[r, "unzip"] & q^{-2}t^2 I
\end{tikzcd}
\end{center}

Observe that $[d, h'] = x_{k + 1} - x_{k + 1}''$ (which vanishes on $I$), so that $h'$ behaves like a dot-sliding homotopy on $T_{k, 1}$. Let $f \in \text{Hom}^0_{\text{Ch}(SSBim_{k + 1})}(FT_{k, 1}, T_{k, 1})$, $g \in \text{Hom}^0_{\text{Ch}(SSBim_{k + 1})}(T_{k, 1}, FT_{k, 1})$ denote the homotopy equivalences of Proposition \ref{prop: ft_min}.

\begin{lemma} \label{lem: dot_slide_ft}
Suppose $k \geq 2$. Then $fh - h'f = 0$, and there is a morphism $\tilde{\chi} \in \text{Hom}^{-2}_{\text{Ch}(SSBim_{k + 1})}(T_{k, 1}, FT_{k, 1})$ satisfying $[d, \tilde{\chi}] = gh' - hg$ and $f \tilde{\chi} = 0$.
\end{lemma}

\begin{proof}
Our strategy will be largely the same as in Lemma \ref{lem: dot_slide_fork_slide}. Again applying \textcolor{revisions}{diagrammatic relations} and the identities $[k + 1] = q^k + q^{-1}[k]$, $q^{-1}[k] = q^{k - 2} + q^{-2}[k - 1]$, we obtain an explicit description of $FT_{k, 1}$. We depict the components of all relevant morphisms according to the conventions of Lemma \ref{lem: dot_slide_fork_slide} above; here $FT_{k, 1}$ is the upper complex and $T_{k, 1}$ the lower complex.

\begin{center}
\begin{tikzcd}
& & {q^{-2}[k - 1]tB} \arrow[rrd, "d_7"] \arrow[dll, "h_1", red] \arrow[llddd, "h_4", red, pos=.3] \arrow[rrddd, "d_8", pos=.85] & & \\
{q^{-1}[k] B} \arrow[rru, "d_1", shift left] \arrow[rrd, "d_2", pos=.2, shift left] \arrow[rrddd, "d_3", near end] & & & & {q^{-2}[k - 1]t^2 B} \arrow[ull, "h_7", red, shift left] \arrow[dll, "h_9", red, shift left, near end] \arrow[dddll, "h_{11}", red, shift left, pos=.15] \arrow[ddddd, "-d_8d_7^{-1}", blue, bend left] \\
& & {q^{-1}[k]t B} \arrow[rru, "d_9", near end] \arrow[rrd, "d_{10}", pos=.15] \arrow[dddd, "-d_3d_2^{-1}"', pos=.2, bend left, blue] \arrow[ull, "h_2", red, pos=.17] \arrow[dll, "h_5", red, near start] & & \\
q^k B \arrow[rruuu, "d_4", pos=.65, shift left] \arrow[rru, "d_5", pos=.9, shift left] \arrow[ddd, "\text{id}_B", harpoon, shift left, blue] \arrow[rrd, "d_6"] & & & & q^{-2}t^2I \arrow[ddd, "\text{id}_I", harpoon, shift left, blue] \arrow[ull, "h_{10}", red, shift left, near start] \arrow[uuull, "h_8", shift left, red, pos=.85] \arrow[lld, "h_{12}", red, shift left] \\
& & q^{k - 2}t B \arrow[dd, "\text{id}_B", harpoon, shift left, blue] \arrow[rruuu, "d_{11}", near start] \arrow[rru, "d_{12}"] \arrow[uuull, "h_3", red, shift left, pos=.85] \arrow[ull, "h_6", red, shift left] & & \\
\\
q^k B \arrow[rr, "x_{k + 1} - x_{k + 1}''"] \arrow[uuuuu, "-d_2^{-1}d_5", bend left, blue] \arrow[uuu, "\text{id}_B", harpoon, shift left, blue] & & q^{k - 2}tB \arrow[uu, "\text{id}_B", harpoon, shift left, blue] \arrow[rr, "unzip"] \arrow[uuuuuu, "-d_7^{-1}d_{11}"', bend left, pos=.9, blue] \arrow[ll, "\text{id}_B", red, shift left] & & q^{-2}t^2I \arrow[uuu, "\text{id}_I", harpoon, shift left, blue]
\end{tikzcd}
\end{center}

We can again eliminate many components of these maps by computing graded dimensions of various morphism spaces. We quote the results below.

\begin{align*}
\text{dim}(\text{End}_{SSBim_{k + 1}}(B)) & = (1 + q^2 + \dots + q^{2k}) \left(\prod_{i = 1}^k \cfrac{1}{1-q^{2i}} \right) \left( \cfrac{1}{1 - q^2} \right) \\
\text{dim}(\text{Hom}_{SSBim_{k + 1}}(B, I)) & = q^k \left( \prod_{i = 1}^k \cfrac{1}{1 - q^{2i}} \right) \left( \cfrac{1}{1 - q^2} \right) \\
\text{dim}(\text{Hom}_{SSBim_{k + 1}}(I, B)) & = q^k \left( \prod_{i = 1}^k \cfrac{1}{1 - q^{2i}} \right) \left( \cfrac{1}{1 - q^2} \right)
\end{align*}

In particular, the lowest quantum degree of any morphism between $B$ and $I$ is $k$. Since the highest degree shift appearing in $q^{-2}[k - 1]$ is $q^{k - 4}$, the maximal quantum degree of any component of $d_8$ is $q^{k - 2}$; hence $d_8$ vanishes. The maximal quantum degree of any component of $h_4$ is $q^{-2}$; since $B$ has no negative degree endomorphisms, this component vanishes. The degree of $h_{12}$ must be $q^{2 - k}$; for $k \geq 2$, we have $2 - k < k$, so this component also vanishes. We reproduce the diagram above updated with these restrictions.

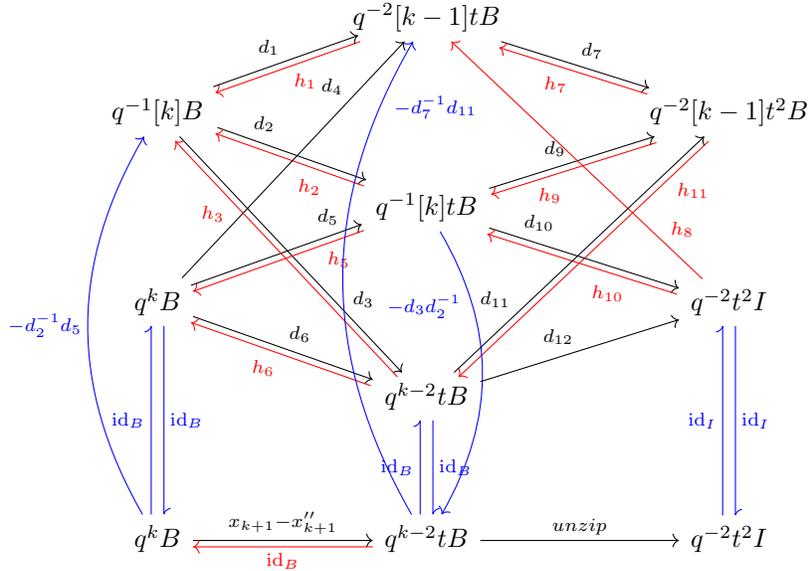
\begin{center}
\begin{tikzcd}
& & {q^{-2}[k - 1]tB} \arrow[rrd, "d_7"] \arrow[dll, "h_1", red] & & \\
{q^{-1}[k] B} \arrow[rru, "d_1", shift left] \arrow[rrd, "d_2", pos=.2, shift left] \arrow[rrddd, "d_3", near end] & & & & {q^{-2}[k - 1]t^2 B} \arrow[ull, "h_7", red, shift left] \arrow[dll, "h_9", red, shift left, near end] \arrow[dddll, "h_{11}", red, shift left, pos=.15] \\
& & {q^{-1}[k]t B} \arrow[rru, "d_9"] \arrow[rrd, "d_{10}", pos=.15] \arrow[dddd, "-d_3d_2^{-1}"', pos=.2, bend left, blue] \arrow[ull, "h_2", red, near start] \arrow[dll, "h_5", red, near start] & & \\
q^k B \arrow[rruuu, "d_4", near end] \arrow[rru, "d_5", pos=.9, shift left] \arrow[ddd, "\text{id}_B", harpoon, shift left, blue] \arrow[rrd, "d_6"] & & & & q^{-2}t^2I \arrow[ddd, "\text{id}_I", harpoon, shift left, blue] \arrow[ull, "h_{10}", red, shift left, near start] \arrow[uuull, "h_8"', red, pos=.15] \\
& & q^{k - 2}t B \arrow[dd, "\text{id}_B", harpoon, shift left, blue] \arrow[rruuu, "d_{11}", near start] \arrow[rru, "d_{12}"] \arrow[uuull, "h_3", red, shift left, near end] \arrow[ull, "h_6", red, shift left] & & \\
\\
q^k B \arrow[rr, "x_{k + 1} - x_{k + 1}''"] \arrow[uuuuu, "-d_2^{-1}d_5", bend left, blue] \arrow[uuu, "\text{id}_B", harpoon, shift left, blue] & & q^{k - 2}tB \arrow[uu, "\text{id}_B", harpoon, shift left, blue] \arrow[rr, "unzip"] \arrow[uuuuuu, "-d_7^{-1}d_{11}"', bend left, pos=.9, blue] \arrow[ll, "\text{id}_B", red, shift left] & & q^{-2}t^2I \arrow[uuu, "\text{id}_I", harpoon, shift left, blue]
\end{tikzcd}
\end{center}

As before, we leverage properties of the morphisms $f, g, h, h'$ to establish various useful relations among these components. First, we verify that $gh'f - gfh \in \text{End}^{-1}_{\text{Ch}(SSBim_{k + 1})}(FT_{k, 1})$ is closed via direct computation:

\[
[d, gh'f - gfh] = g[d, h']f - gf[d, h] = g(x_{k + 2} - x_{k + 2}'') - gf(x_{k + 2} - x_{k + 2}'') = 0
\]

By computations in \cite{HRW21}, $H^{\bullet} (\text{End}_{\text{Ch}(SSBim_{k + 2})} (FT_{k, 1}))$ is concentrated in even degrees. It follows that there exists some $\zeta \in \text{End}^{-2}_{\text{Ch}(SSBim_{k + 1})}(FT_{k, 1})$ such that $[d, \zeta] = gh'f - gfh$. We turn our attention to showing that $\zeta = 0$. Below we reproduce the complex for $FT_{k, 1}$ with backwards cyan arrows labeling components of $gh'f - gfh$ and backwards purple arrows labeling components of $\zeta$. We include only the vanishing components of $gh'f - gfh$ for visual clarity; as we will see, no other components will be relevant.

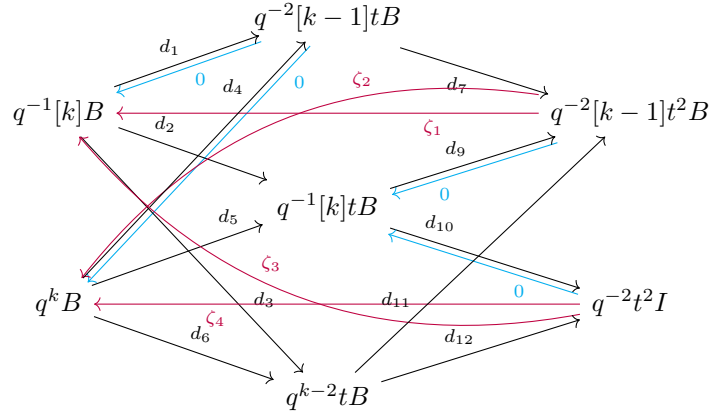
\begin{center}
\begin{tikzcd}
& & {q^{-2}[k - 1]tB} \arrow[rrd, "d_7"', shift right] \arrow[dll, "0", cyan] \arrow[llddd, "0", shift left, pos=.1, cyan] & & \\
{q^{-1}[k] B} \arrow[rru, "d_1", shift left] \arrow[rrd, "d_2", pos=.2, shift left] \arrow[rrddd, "d_3", near end] & & & & {q^{-2}[k - 1]t^2 B}  \arrow[dll, "0", cyan, shift left, near end] \arrow[llll, "\zeta_1", purple, near start] \arrow[lllldd, "\zeta_2"', purple, pos=.3, bend right] \\
& & {q^{-1}[k]t B} \arrow[rru, "d_9"] \arrow[rrd, "d_{10}", pos=.15] & & \\
q^k B \arrow[rruuu, "d_4", near end] \arrow[rru, "d_5", pos=.9, shift left] \arrow[rrd, "d_6"] & & & & q^{-2}t^2I \arrow[ull, "0", cyan, shift left, near start] \arrow[lllluu, "\zeta_3"', purple, bend left, pos=.6] \arrow[llll, "\zeta_4", purple, near end] \\
& & q^{k - 2}t B \arrow[rruuu, "d_{11}", near start] \arrow[rru, "d_{12}"] & &
\end{tikzcd}
\end{center}

Observe that $\zeta_1d_7 = 0$; since $d_7$ is invertible, we have $\zeta_1 = 0$. Similarly, since $\zeta_2d_7 = 0$, we must have $\zeta_2 = 0$. Since $d_2\zeta_3 = 0$ and $d_2$ is invertible, we must have $\zeta_3 = 0$. Finally, $\zeta_4$ is a degree $-k$ morphism from $I$ to $B$; this vanishes by the dimension computations above. Since every component of $\zeta$ vanishes, in fact $\zeta = 0$. It follows immediately that $gh'f - gfh = 0$. Precomposing with $f$, we obtain $fg(h'f - fh) = h'f - fh = 0$, where the first equality follows from the status of the pair $f, g$ as a strong deformation retract. This proves the first claim.

A similar computation shows that $gh'f - hgf \in \text{End}^{-1}_{\text{Ch}(SSBim_{k + 1})}(FT_{k, 1})$ is closed, so there exists some $\chi \in \text{End}^{-2}_{\text{Ch}(SSBim_{k + 1})}(FT_{k, 1})$ satisfying $[d, \chi] = gh'f - hgf$. We reproduce the complex $FT_{k, 1}$ with backwards cyan arrows labeling components of $gh'f - hgf$ and backwards purple arrows labeling components of $\xi$. Again, we include only those components of $gh'f - hgf$ which are of relevance to us.

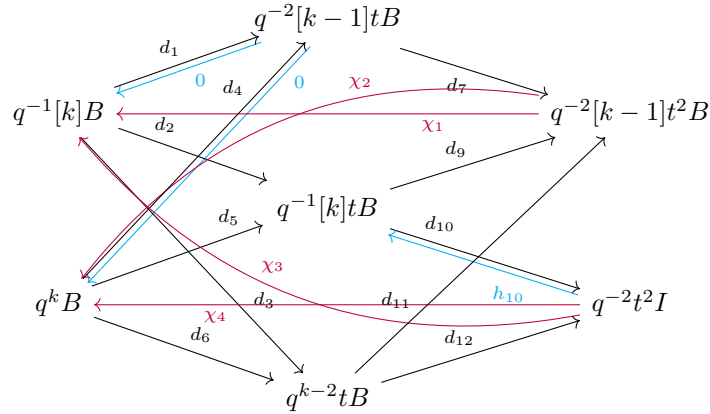
\begin{center}
\begin{tikzcd}
& & {q^{-2}[k - 1]tB} \arrow[rrd, "d_7"', shift right] \arrow[dll, "0", cyan] \arrow[llddd, "0", shift left, pos=.1, cyan] & & \\
{q^{-1}[k] B} \arrow[rru, "d_1", shift left] \arrow[rrd, "d_2", pos=.2, shift left] \arrow[rrddd, "d_3", near end] & & & & {q^{-2}[k - 1]t^2 B} \arrow[llll, "\chi_1", purple, near start] \arrow[lllldd, "\chi_2"', purple, pos=.3, bend right] \\
& & {q^{-1}[k]t B} \arrow[rru, "d_9"] \arrow[rrd, "d_{10}", pos=.15] & & \\
q^k B \arrow[rruuu, "d_4", near end] \arrow[rru, "d_5", pos=.9, shift left] \arrow[rrd, "d_6"] & & & & q^{-2}t^2I \arrow[ull, "h_{10}", cyan, shift left, near start] \arrow[lllluu, "\chi_3"', purple, bend left, pos=.6] \arrow[llll, "\chi_4", purple, near end] \\
& & q^{k - 2}t B \arrow[rruuu, "d_{11}", near start] \arrow[rru, "d_{12}"] & &
\end{tikzcd}
\end{center}

The same considerations as before show that $\chi_1, \chi_2$, and $\chi_4$ vanish\footnote{One would like to conclude that $\chi = 0$ so that $gh' - hg$ is not just nullhomotopic but identically $0$. Unfortunately, this is not true. Since $d_2\chi_3 = h_{10}$ and $d_2$ is invertible, we must have $\chi_3 = d_2^{-1}h_{10}$. A direct computation shows that $h_{10}$ is the zip map in lowest quantum degree, which is decidedly nonzero.}. In particular, $\chi_3$ is the only nonzero component of $\chi$, so $f \chi = 0$ by inspection. Setting $\tilde{\chi} := \chi g$, we obtain $f \tilde{\chi} = 0$ and

\[
[d, \tilde{\chi}] = [d, \chi] g + \chi [d, g] = (gh'f - hgf)g = (gh' - hg)fg = gh' - hg.
\]
\end{proof}

Applying the functor $_{1^n}S_{\mathfrak{S}^{n - 1}} \otimes - \otimes _{\mathfrak{S}^{n - 1}}M_{1^n}$ to the maps $f, g$ of Lemma \ref{lem: dot_slide_ft} induces a strong deformation retraction from $STM_n$ to $C_n$ as in Proposition \ref{prop: mu_nu}. We refer to the resulting maps also as $f$ and $g$.

\begin{corollary}
The map $f \in \text{Hom}^0_{\text{Ch}(SSBim_n)}(STM_n, C_n)$ lifts without modification to a chain map $f \colon STM_n^{y_n} \to C_n^{y_n}$ of curved complexes. The map $g \in \text{Hom}^0_{\text{Ch}(SSBim_n)}(C_n, STM_n)$ has a strict lift to a chain map $\tilde{g} \colon C_n^{y_n} \to STM_n^{y_n}$ of curved complexes. The maps $f, \tilde{g}$ constitute a strong deformation retraction from $STM_n^{y_n}$ to $C_n^{y_n}$.
\end{corollary}

\begin{proof}
It is easily verified from the explicit description of Proposition \ref{prop: forkslyde} that the connection on $C_n^{y_n}$ is given by $\delta_{C_n} = d_{C_n} + h'(y_n - y_1)$. That $f$ lifts to a chain map of curved complexes follows from Proposition \ref{lem: dot_slide_ft} exactly as in the proof of Corollary \ref{cor: half_forkslyde}. Set $\tilde{g} := g + \tilde{\chi} (y_n - y_1)$. We check that $\tilde{g}$ is a chain map via direct computation:

\begin{align*}
    [\delta, \tilde{g}] & = (d_{STM_n} + h(y_n - y_1)) (g + \tilde{\chi} (y_n - y_1)) - (g + \tilde{\chi} (y_n - y_1)) (d_{C_n} + h' (y_n - y_1)) \\
    & = [d, g] + ([d, \tilde{\chi}] + (hg - gh'))(y_n - y_1) + (h \tilde{\chi} - \tilde{\chi} h') (y_n - y_1)^2 \\
    & = (h \tilde{\chi} - \tilde{\chi} h') (y_n - y_1)^2
\end{align*}

Note that $h \tilde{\chi} - \tilde{\chi} h' \in \text{Hom}^{-3}_{\text{Ch}(SSBim_n)}(C_n, STM_n)$. Since both $C_n$ and $STM_n$ are concentrated in homological degrees $0$, $1$, and $2$, we must have $h \tilde{\chi} - \tilde{\chi} h' = 0$. Hence $\tilde{g}$ is closed.

Since $f \tilde{\chi} = 0$ by Lemma \ref{lem: dot_slide_ft}, we have $f \tilde{g} = f (g + \tilde{\chi} (y_n - y_1)) = fg = \text{id}_{C_n^{y_n}}$. That $\tilde{g} f \sim \text{id}_{STM_n^{y_n}}$ follows from Corollary \ref{corr: GHobs2} exactly as in the proof of Corollary \ref{cor: half_forkslyde}.
\end{proof}

\begin{proof}[Proof of Proposition \ref{prop: forkslyde}]
Set $\tilde{\mu} = \overline{\mu} \tilde{g}$, and observe that $\mu = \overline{\mu} g$ and $\nu = f \overline{\nu}$.
\end{proof}

\bibliographystyle{alpha}
\bibliography{Columns_Homfly}

\begin{thebibliography}{EMTW20}

\bibitem[AH17]{AH17}
M.~Abel and M.~Hogancamp.
\newblock Categorified {Y}oung symmetrizers and stable homology of torus links
  {II}.
\newblock {\em Sel. Math. (N.S.)}, 23(3):1739--1801, 2017.
\newblock \href{https://arxiv.org/abs/1510.05330}{arxiv:1510.05330}.

\bibitem[Ale23]{Al23}
J.~W. Alexander.
\newblock A lemma on systems of knotted curves.
\newblock {\em Proc. Natl. Acad. Sci. U.S.A.}, 9:93--95, 1923.

\bibitem[AW19]{AW19}
M.~Abel and M.~Willis.
\newblock Colored {K}hovanov-{R}ozansky homology for infinite braids.
\newblock {\em Algebr. Geom. Topol}, 19:2401--2438, 2019.
\newblock \href{https://arxiv.org/abs/1709.06666}{arxiv:1709.06666}.

\bibitem[BN07]{BN07}
D.~Bar-Natan.
\newblock Fast {K}hovanov homology computations.
\newblock {\em J. Knot Theory Ramif.}, 16(3):243 -- 255, 2007.

\bibitem[Cau15]{Cau15}
S.~Cautis.
\newblock Clasp technology to knot homology via the affine {G}rassmannian.
\newblock {\em Math. Ann.}, 363:1053--1115, 2015.
\newblock \href{https://arxiv.org/abs/1207.2074}{arxiv:1207.2074}.

\bibitem[Cau17]{Cau17}
S.~Cautis.
\newblock Remarks on coloured triply graded link invariants.
\newblock {\em Algebr. Geom. Topol}, 17(6):3811--3836, 2017.
\newblock \href{https://arxiv.org/abs/1611.09924}{arxiv:1611.09924}.

\bibitem[CK12]{CK12}
B.~Cooper and V.~Krushkal.
\newblock Categorification of the {J}ones-{W}enzl projectors.
\newblock {\em Quantum Topol.}, 3(2):139--180, 2012.
\newblock \href{https://arxiv.org/abs/1005.5117}{arxiv:1005.5117}.

\bibitem[CKM14]{CKM14}
S.~Cautis, J.~Kamnitzer, and S.~Morrison.
\newblock Webs and quantum skew {H}owe duality.
\newblock {\em Math. Ann.}, 360(1-2):351--390, 2014.
\newblock \href{https://arxiv.org/abs/1210.6437}{arxiv:1210.6437}.

\bibitem[DGR06]{DGR06}
N.~Dunfield, S.~Gukov, and J.~Rasmussen.
\newblock The superpolynomial for knot homologies.
\newblock {\em Exp. Math.}, 15(2):129--159, 2006.
\newblock \href{https://arxiv.org/abs/math/0505662}{arxiv:math/0505662}.

\bibitem[EH17a]{EH17a}
B.~Elias and M.~Hogancamp.
\newblock Categorical diagonalization.
\newblock \href{https://arxiv.org/abs/1707.04349}{arxiv:1707.04349}, 2017.

\bibitem[EH17b]{EH17b}
B.~Elias and M.~Hogancamp.
\newblock Categorical diagonalization of full twists.
\newblock \href{https://arxiv.org/abs/1801.00191}{arXiv:1801.00191}, 2017.

\bibitem[EH19]{EH19}
B.~Elias and M.~Hogancamp.
\newblock On the computation of torus link homology.
\newblock {\em Compos. Math.}, 2019.
\newblock \href{https://arxiv.org/abs/1603.00407}{arxiv:1603.00407}.

\bibitem[Elb22]{Elb22}
M.~Elbehiry.
\newblock {\em Minimal categorical idempotents in the {S}oergel category of
  finite {C}oxeter groups}.
\newblock 2022.
\newblock Thesis (Ph.D.)--Northeastern University.

\bibitem[EMTW20]{EMTW20}
B.~Elias, S.~Makisumi, U.~Thiel, and G.~Williamson.
\newblock {\em Introduction to {S}oergel bimodules}.
\newblock RSME Springer Series. Springer Cham, Cham, Switzerland, 2020.

\bibitem[ESW14]{ESW14}
B.~Elias, N.~Snyder, and G.~Williamson.
\newblock On cubes of {F}robenius extensions.
\newblock \href{https://arxiv.org/abs/1308.5994}{arxiv:1308.5994}, 2014.

\bibitem[EW14]{EW14}
B.~Elias and G.~Williamson.
\newblock The {H}odge theory of {S}oergel bimodules.
\newblock {\em Ann. Math.}, 180:1089--1136, 2014.
\newblock \href{https://arxiv.org/abs/1212.0791}{arxiv:1212.0791}.

\bibitem[GGS18]{GGS18}
E.~Gorsky, S.~Gukov, and M.~{S}to{\v{s}i\'c}.
\newblock Quadruply-graded colored homology of knots.
\newblock {\em Fund. Math.}, 243:209--299, 2018.
\newblock \href{https://arxiv.org/abs/1304.3481}{arxiv:1304.3481}.

\bibitem[GH22]{GH22}
E.~Gorsky and M.~Hogancamp.
\newblock Hilbert schemes and y-ification of {K}hovanov-{R}ozansky homology.
\newblock {\em Geom. Topol.}, 26:587--678, 2022.
\newblock \href{https://arxiv.org/abs/1712.03938}{arxiv:1712.03938}.

\bibitem[GHM21]{GHM21}
E.~Gorsky, M.~Hogancamp, and A.~Mellit.
\newblock Tautological classes and symmetry in {K}hovanov-{R}ozansky homology.
\newblock \href{http://arxiv.org/abs/2103.01212}{arXiv:2103.01212}, 2021.

\bibitem[GS12]{GS12}
S.~Gukov and M.~Sto{\v{s}}i{\'c}.
\newblock Homological algebra of knots and {BPS} states.
\newblock {\em Geometry \& Topology Monographs}, 18:309--367, 2012.
\newblock \href{https://arxiv.org/abs/1112.0030}{arxiv:1112.0030}.

\bibitem[HM19]{HM19}
M.~Hogancamp and A.~Mellit.
\newblock Torus link homology.
\newblock \href{https://arxiv.org/abs/1909.00418}{arxiv:1909.00418}, 2019.

\bibitem[Hog17]{Hog17}
M.~Hogancamp.
\newblock Idempotents in triangulated monoidal categories.
\newblock \href{https://arxiv.org/abs/1703.01001}{arxiv:1703.01001}, 2017.

\bibitem[Hog18]{Hog18}
M.~Hogancamp.
\newblock Categorified {Y}oung symmetrizers and stable homology of torus links.
\newblock {\em Geom. Topol.}, 22:2943--3002, 2018.
\newblock \href{https://arxiv.org/abs/1505.08148}{arxiv:1505.08148}.

\bibitem[Hog20]{Hog20}
M.~Hogancamp.
\newblock Homological perturbation theory with curvature.
\newblock \href{https://arxiv.org/abs/1912.03843v2}{arxiv:1912.03843v2}, 2020.

\bibitem[HRW21a]{HRW21}
M.~Hogancamp, D.~E.~V. Rose, and P.~Wedrich.
\newblock Link splitting deformation of colored {K}hovanov--{R}ozansky
  homology.
\newblock \href{http://arxiv.org/abs/2107.09590}{arXiv:2107.09590}, 2021.

\bibitem[HRW21b]{HRW21b}
M.~Hogancamp, D.~E.~V. Rose, and P.~Wedrich.
\newblock A skein relation for singular {S}oergel bimodules.
\newblock \href{https://arxiv.org/abs/2107.08117}{arXiv:2107.08117}, 2021.

\bibitem[Jon87]{Jon87}
V.~F.~R. Jones.
\newblock Hecke algebra representations of braid groups and link polynomials.
\newblock {\em Ann. Math.}, 126:335--388, 1987.

\bibitem[Kho07]{Kh07}
M.~Khovanov.
\newblock Triply-graded link homology and {H}ochschild homology of {S}oergel
  bimodules.
\newblock {\em Int. J. Math.}, 18(8):869--885, 2007.
\newblock \href{https://arxiv.org/abs/math/0510265}{arxiv:math/0510265}.

\bibitem[KR08]{KhR08}
M.~Khovanov and L.~Rozansky.
\newblock Matrix factorizations and link homology. {II}.
\newblock {\em Geom. Topol.}, 12(3):1387--1425, 2008.
\newblock \href{https://arxiv.org/abs/math/0401268}{arxiv:math/0401268}.

\bibitem[LP09]{LP09}
K.~Liu and P.~Peng.
\newblock Proof of the {L}abastida-{M}ari\~{n}o-{O}oguri-{V}afa conjecture.
\newblock \href{https://arxiv.org/abs/0704.1526}{arxiv:0704.1526}, 2009.

\bibitem[LP11]{LP11}
K.~Liu and P.~Peng.
\newblock New structures of knot invariants.
\newblock {\em Commun. Number Theory}, 5(3):601--615, 2011.

\bibitem[Mel22]{Mel22}
A.~Mellit.
\newblock Homology of torus knots.
\newblock {\em Geom. Topol.}, 26:47--70, 2022.
\newblock \href{https://arxiv.org/abs/1704.07630}{arxiv:1704.07630}.

\bibitem[MSV11]{MSV11}
M.~Mackaay, M.~Stošić, and P.~Vaz.
\newblock The 1,2-coloured {HOMFLY-PT} link homology.
\newblock {\em Trans. Am. Math. Soc.}, 363(4):2091--2124, 2011.
\newblock \href{https://arxiv.org/abs/0809.0193}{arxiv:0809.0193}.

\bibitem[OR17]{OR17}
A.~Oblomkov and L.~Rozansky.
\newblock {HOMFLYPT} homology of {C}oxeter links.
\newblock \href{https://arxiv.org/abs/1706.00124}{arxiv:1706.00124}, 2017.

\bibitem[OR18]{OR18}
A.~Oblomkov and L.~Rozansky.
\newblock Knot homology and sheaves on the {H}ilbert scheme of points in the
  plane.
\newblock {\em Sel. Math.}, 24:2351--2454, 2018.
\newblock \href{https://arxiv.org/abs/1608.03227}{arxiv:1608.03227}.

\bibitem[OR19a]{OR19a}
A.~Oblomkov and L.~Rozansky.
\newblock Affine braid group, {JM} elements and knot homology.
\newblock {\em Transform. Groups}, 24:531--544, 2019.
\newblock \href{https://arxiv.org/abs/1702.03569}{arxiv:1702.03569}.

\bibitem[OR19b]{OR19b}
A.~Oblomkov and L.~Rozansky.
\newblock Dualizable link homology.
\newblock \href{https://arxiv.org/abs/1905.06511}{arxiv:1905.06511}, 2019.

\bibitem[OR20]{OR20}
A.~Oblomkov and L.~Rozansky.
\newblock Soergel bimiodules and matrix factorizations.
\newblock \href{https://arxiv.org/abs/2010.14546}{arxiv:2010.14546}, 2020.

\bibitem[Ras15]{Ras15}
J.~Rasmussen.
\newblock Some differentials on {K}hovanov-{R}ozansky homology.
\newblock {\em Geom. Topol.}, 19(6):3031--3104, 2015.
\newblock \href{http://arxiv.org/abs/0607544}{arXiv:0607544}.

\bibitem[Ros14]{Ros14}
D.~Rose.
\newblock A categorification of quantum $\mathfrak{sl}_3$ projectors and the
  $\mathfrak{sl}_3$ {R}eshetikhin-{T}uraev invariant of tangles.
\newblock {\em Quantum Topol.}, 5(1):1--59, 2014.
\newblock \href{https://arxiv.org/abs/1109.1745}{arxiv:1109.1745}.

\bibitem[Rou04]{Rou04}
R.~Rouquier.
\newblock Categorification of the braid groups.
\newblock \href{http://arxiv.org/abs/math/0409593}{arXiv:math/0409593}, 2004.

\bibitem[Roz10]{Roz10}
L.~Rozansky.
\newblock An infinite torus braid yields a categorified {J}ones-{W}enzl
  projector.
\newblock \href{https://arxiv.org/abs/1005.3266}{arxiv:1005.3266}, 2010.

\bibitem[RT90]{RT90}
N.~Yu. Reshtikhin and V.~G. Turaev.
\newblock Ribbon graphs and their invariants derived from quantum groups.
\newblock {\em Comm. Math. Phys.}, 1:1--26, 1990.

\bibitem[RT21]{RT21}
D.~Rose and D.~Tubbenhauer.
\newblock Homflypt homology for links in handlebodies via type {A} {S}oergel
  bimodules.
\newblock {\em Quantum Topology}, 2:373--410, 2021.
\newblock \href{https://arxiv.org/abs/1908.06878}{arxiv:1908.06878}.

\bibitem[SS11]{SS11}
C.~Stroppel and J.~Sussan.
\newblock Categorified {J}ones-{W}enzl projectors: a comparison.
\newblock \href{https://arxiv.org/abs/1105.3038}{arxiv:1105.3038}, 2011.

\bibitem[Wed19]{Wed19}
P.~Wedrich.
\newblock Exponential growth of colored {HOMFLY-PT} homology.
\newblock {\em Adv. Math.}, 2019.
\newblock \href{http://arxiv.org/abs/1602.02769}{arxiv:1602.02769}.

\bibitem[Wil21]{Wil20}
M.~Willis.
\newblock Khovanov-{R}ozansky homology for infinite multicolored braids.
\newblock {\em Cand. J. Math.}, 73:1239--1277, 2021.
\newblock \href{https://arxiv.org/abs/1904.09055}{arxiv:1904.09055}.

\bibitem[WW17]{WW17}
B.~Webster and G.~Williamson.
\newblock A geometric construction of colored {HOMFLYPT} homology.
\newblock {\em Geom. Topol.}, 21:2557--2600, 2017.
\newblock \href{https://arxiv.org/abs/0905.0486}{arxiv:0905.0486}.

\end{thebibliography}

\end{document}